\documentclass[10pt,twoside, openright]{report} 

\usepackage{color}
\definecolor{myurlcolor}{rgb}{0,0,0.5}
\usepackage{amsmath}
\usepackage{amsfonts}
\usepackage{amssymb}
\usepackage{amsthm}
\usepackage[all,2cell]{xy}
\UseAllTwocells
\usepackage{stmaryrd}
\usepackage{bbm}
\usepackage{enumerate}
\usepackage{accsupp}
\usepackage[title]{appendix}
\usepackage{comment}
\usepackage{mathabx}
\usepackage{booktabs}
\renewcommand{\bibname}{References}

\def\slashedrightarrow{\relbar\joinrel\mapstochar\joinrel\rightarrow}
\newcommand{\proto}{\slashedrightarrow}

\newcommand{\from}{\colon} % For indicating type of a map
\newcommand{\cat}[1]{\mathcal{#1}} % Font for arbitrary category
\newcommand{\scat}[1]{\mathbbm{#1}} % Font for small categoryt
\newcommand{\fcat}[1]{\mathbf{#1}} % Font for fixed category
\newcommand{\demph}[1]{\textbf{#1}} % Font for defining terms 
\renewcommand{\emptyset}{\varnothing} % TA's preferred empty set
\renewcommand{\epsilon}{\varepsilon} 
\newcommand{\such}{\mathrel{|}} % For use in set-builder notation
\newcommand{\hyph}{\text{-}} % Hyphen for use in math mode
\newcommand{\iso}{\cong} % Isomorphism
\newcommand{\eqv}{\simeq} % Equivalence of categories
\newcommand{\toby}[1]{\stackrel{#1}{\longrightarrow}} % Labelled arrow
\newcommand{\incl}{\hookrightarrow} % Inclusion
\newcommand{\bref}[1]{(\ref{#1})} % Cross-reference in brackets
\newcommand{\of}{\mathbin{\circ}} % Composition
\DeclareMathOperator{\ev}{ev} % Evaluation functor
\newcommand{\id}{\mathrm{id}}
\newcommand{\op}{\mathrm{op}}

\newcommand{\mnd}[1]{\mathbb{#1}}
\newcommand{\lpair}{\langle}
\newcommand{\rpair}{\rangle}

\newcommand{\ra}{\mathrm{r.a.}}

\newcommand{\cm}{\mathrm{c.m.}}
\newcommand{\str}{\operatorname{Str}}
\newcommand{\sem}{\operatorname{Sem}}

\newcommand{\thr}{\operatorname{Th}}

\renewcommand{\mod}{\operatorname{Mod}}
\newcommand{\kl}[1]{\operatorname{Kl}(#1)}
\newcommand{\kle}{\operatorname{Kl}}
\newcommand{\monad}{\fcat{Mnd}}
\newcommand{\alg}{\operatorname{Alg}}
\newcommand{\tensor}{\otimes}
\newcommand{\tagqed}{\tag*{\scshape \qed}} %for putting a qed symbol on the same line as a display, when NOT in a proof environment

\newcommand{\Set}{\fcat{Set}}
\newcommand{\SET}{\fcat{SET}}
\newcommand{\Top}{\fcat{Top}}
\newcommand{\TOP}{\fcat{TOP}}
\newcommand{\TOPCAT}{\fcat{TOP}\hyph\fcat{CAT}}
\newcommand{\Vect}{\fcat{Vect}}
\newcommand{\finvect}{\fcat{FinVect}}
\newcommand{\finset}{\fcat{FinSet}}
\newcommand{\fin}{\mathbb{F}}
\newcommand{\finbij}{\mathbb{B}}
\newcommand{\Cat}{\fcat{Cat}}
\newcommand{\CAT}{\fcat{CAT}}
\newcommand{\ACAT}{\fcat{CATwA}}
\newcommand{\Amonad}{\fcat{MndwA}}
\newcommand{\finprod}{\fcat{FinProdCAT}}
\newcommand{\symmon}{\fcat{SymMonCAT}}
\newcommand{\monCAT}{\fcat{MonCAT}}
\newcommand{\multi}{\fcat{MultiCAT}}

\newcommand{\law}{\operatorname{Law}}
\newcommand{\Law}{\fcat{Law}}
\newcommand{\LAW}{\fcat{LAW}}
\newcommand{\PROP}{\fcat{PROP}}
\newcommand{\PRO}{\fcat{PRO}}

\newcommand{\PROF}{\fcat{PROF}}

\newcommand{\multicat}{\fcat{MultiCAT}}
\newcommand{\operad}{\fcat{OPD}}
\newcommand{\tract}{\mathrm{tract}}
\newcommand{\convt}{\fcat{Convt}}

\newcommand{\vect}{\mathbf}
\newcommand{\Gp}{\fcat{Gp}}
\newcommand{\FinGp}{\fcat{FinGp}}
\newcommand{\TopGp}{\fcat{TopGp}}
\newcommand{\profGp}{\fcat{ProfGp}}

\newcommand{\monoid}{\fcat{Mon}}

\newcommand{\TopMonoid}{\fcat{TopMon}}
\newcommand{\MONOID}{\fcat{MON}}

\newcommand{\obj}{\hyph\fcat{obj}}

\newcommand{\catover}[1]{\CAT/#1}

\newcommand{\radjover}[1]{(\CAT/#1)_{\ra}}
\newcommand{\cmover}[1]{(\CAT/#1)_{\cm}}
\newcommand{\currylo}{H_{\bullet}}
\newcommand{\curryhi}{H^{\bullet}}

\newcommand{\const}[1]{\widebar{#1}}

\newcommand{\nat}{\mathbb{N}}

\newcommand{\Gfinset}{G\text{-}\fcat{FinSet}}
\newcommand{\pc}{\widehat}
\newcommand{\sym}{\operatorname{Sym}}
\newcommand{\ob}{\operatorname{ob}}
\newcommand{\mor}{\operatorname{mor}}

\newcommand{\pow}{\mathcal{P}}
\newcommand{\Union}{\bigcup}
\newcommand{\ladj}{\dashv}
\newcommand{\radj}{\vdash}
\newcommand{\ofsfont}{\mathcal}
\newcommand{\procomp}[1]{\widehat{#1}}

\newcommand{\und}{\operatorname{Und}}

\newcommand{\talg}{\mnd{T}\hyph \alg}

\newcommand{\cplt}{\operatorname{Cplt}}
\newcommand{\inc}{\operatorname{Inc}}
\newcommand{\disc}{\operatorname{Disc}}
\newcommand{\oper}{\mathcal{O}}
\newcommand{\proth}{\fcat{PTh}}

\newcommand{\chu}{\operatorname{Chu}}

\newcommand{\pullbackcorner}[1][dr]{\save*!/#1-1.2pc/#1:(-1,1)@^{|-}\restore}

\newcommand{\setting}{setting}
\newcommand{\Setting}{Setting}

\newtheorem{lem}[subsection]{Lemma}
\newtheorem{prop}[subsection]{Proposition}
\newtheorem{thm}[subsection]{Theorem}
\newtheorem{cor}[subsection]{Corollary}
\theoremstyle{definition}
\newtheorem{defn}[subsection]{Definition}
\newtheorem{ex}[subsection]{Example}

\newtheorem{remark}[subsection]{Remark}
\newtheorem{question}[subsection]{Question}
\newtheorem*{defn*}{Definition}

\title{Structure and Semantics}
\author{Tom Avery}
\date{2017}

\usepackage[phd,onehalf,hyperref,
colour
]{edmaths}

\begin{document}
\flushbottom
\pagenumbering{roman}
\sloppy

\maketitle

\declaration

\cleardoublepage
\phantomsection
\addcontentsline{toc}{chapter}{Lay summary}
\chapter*{Lay summary}
It is common in mathematics to consider structures that consist of a collection of elements that can be combined in various ways, which we call operations, such that certain equations always hold. For example, two integers can be combined by adding them together, and the order in which they are added does not change the result. Thus, the collection of all integers has addition as an operation, and it satisfies the equation $x + y = y + x$, which says that the order in which we add two numbers does not matter.

Often we do not consider such structures in isolation, but consider an entire class of structures with similar operations and equations; this allows us to prove things about many structures simultaneously. To continue our example from above, the collection of rational numbers (i.e. fractions) also has an addition operation that satisfies the equation $x + y = y + x$. Thus if we prove some result that only depends on the fact that we can combine two elements and the order in which they are combined does not matter, that result will apply equally to both integers and to rational numbers.

An algebraic theory is a way of describing such a class of structures by specifying some operations abstractly, and some equations that they should satisfy. A model of an algebraic theory consists of a collection of elements that can be combined in the ways prescribed by the operations of the theory such that the equations hold. Thus there is an algebraic theory with ``$+$'' as an operation and the equation $x + y = y + x$, and both the collection of all integers and the collection of all rational numbers are models of this theory.

Algebraic theories are very useful, but they have limitations. Sometimes we want to define a class of structures using operations and equations, but where the object underlying the structure is not just a collection of elements, but something more complicated. In some contexts we may want to place restrictions on the kinds of operations and equations allowed, or allow more general types of operation. Over the years, many different variants of the notion of an algebraic theory have been developed to cope with all these different situations.

The goal of this thesis is to develop a general notion of algebraic theory that unifies many of these variants. We then use this general notion to give an extension of one of the variants in particular, called monads, adding certain desirable properties that monads themselves lack. We do this using ideas from topology, the branch of mathematics that studies spaces with a notion of continuity. Although algebraic theories are used to \emph{describe} mathematical structures, they are also structures in their own right. We draw an analogy between our general notion of algebraic theories as structures, and another kind of structure called a group.

\cleardoublepage
\phantomsection
\begin{abstract}
Algebraic theories describe mathematical structures that are defined in terms of operations and equations, and are extremely important throughout mathematics. Many generalisations of the classical notion of an algebraic theory have sprung up for use in different mathematical contexts; some examples include Lawvere theories, monads, PROPs and operads. The first central notion of this thesis is a common generalisation of these, which we call a \emph{proto-theory}.

The purpose of an algebraic theory is to describe its models, which are structures in which each of the abstract operations of the theory is given a concrete interpretation such that the equations of the theory hold. The process of going from a theory to its models is called semantics, and is encapsulated in a \emph{semantics functor}. In order to define a model of a theory in a given category, it is necessary to have some structure that relates the arities of the operations in the theory with the objects of the category. This leads to the second central notion of this thesis, that of an interpretation of arities, or \emph{aritation} for short. We show that any aritation gives rise to a semantics functor from the appropriate category of proto-theories, and that this functor has a left adjoint called the \emph{structure functor}, giving rise to a \emph{structure--semantics adjunction}. Furthermore, we show that the usual semantics for many existing notions of algebraic theory arises in this way by choosing an appropriate aritation.

Another aim of this thesis is to find a \emph{convenient category of monads} in the following sense. Every right adjoint into a category gives rise to a monad on that category, and in fact some functors that are not right adjoints do too, namely their codensity monads. This is the structure part of the structure--semantics adjunction for monads. However, the fact that not every functor has a codensity monad means that the structure functor is not defined on the category of all functors into the base category, but only on a full subcategory of it.

This deficiency is solved when passing to general proto-theories with a canonical choice of aritation whose structure--semantics adjunction restricts to the usual one for monads. However, this comes at a cost: the semantics functor for general proto-theories is not full and faithful, unlike the one for monads. The condition that a semantics functor be full and faithful can be thought of as a kind of \emph{completeness theorem} --- it says that no information is lost when passing from a theory to its models. It is therefore desirable to retain this property of the semantics of monads if possible.

The goal then, is to find a notion of algebraic theory that generalises monads for which the semantics functor is full and faithful with a left adjoint; equivalently the semantics functor should exhibit the category of theories as a reflective subcategory of the category of all functors into the base category. We achieve this (for well-behaved base categories) with a special kind of proto-theory enriched in topological spaces, which we call a \emph{complete topological proto-theory}.

We also pursue an analogy between the theory of proto-theories and that of groups. Under this analogy, monads correspond to finite groups, and complete topological proto-theories correspond to profinite groups. We give several characterisations of complete topological proto-theories in terms of monads, mirroring characterisations of profinite groups in terms of finite groups.
\end{abstract}

\cleardoublepage
\phantomsection
\addcontentsline{toc}{chapter}{Acknowledgements}
\chapter*{Acknowledgements}
I would like to thank my supervisor Tom Leinster for the countless hours of guidance and discussion that made this thesis possible. His unique way of thinking has had a profound and lasting influence on my own, for which I am very grateful.

I would also like to thank the School of Mathematics at the University of Edinburgh, and especially my many fellow PhD students (too many to name) for making my time here so enjoyable. I will look back fondly on the many long conversations over tea and biscuits that took place while I should have been working on this thesis.

Finally I would like to thank Dorota and my family for their constant love, support and encouragement.

\cleardoublepage
\phantomsection
\addcontentsline{toc}{chapter}{Table of contents}
\tableofcontents 
\cleardoublepage
\phantomsection
\pagenumbering{arabic}

\chapter{Introduction}
\label{chap:intro}
This thesis is concerned with algebraic theories and their category-theoretic generalisations. The ``structure'' and ``semantics'' of the title refer to a certain type of adjunction that arises naturally for many different notions of algebraic theory including Lawvere theories, monads, PROPs, PROs, operads and monoids. We have two broad objectives. First, we wish to find a common generalisation of all of these in such a way that their structure--semantics adjunctions arise naturally via the same mechanism. Second, we search for a ``convenient category of monads'', that is, an extension of the category of monads on a given base category that remedies certain deficiencies of the category of monads, while maintaining other desirable properties.

\section{Algebraic Theories}
Algebraic theories (in the classical sense of universal algebra, also sometimes called equational presentations) are logical theories of an extremely simple type. They describe structures defined by a collection of operations of various arities and equations between terms built up from these operations. Despite their simplicity, many of the structures of greatest interest to mathematicians are described by algebraic theories. For example, the theory of groups has three operations, namely multiplication of arity 2 (denoted by concatenation), inversion of arity 1 (denoted $(-)^{-1}$) and a constant identity element (denoted $e$), which is thought of as an operation of arity 0. The axioms of the theory of groups are
\begin{align*}
x(yz) &= (xy)z, \\
ex &= x, \\
xe&= x \text{ and} \\
x x^{-1} &= e.
\end{align*}
A model of this theory is just a group; it is a set equipped with instantiations of these operations of the appropriate arities such that the axioms hold universally.

The theory of fields is an example of a logical theory that is not algebraic. The theory of fields is the same as the theory of rings (which is an algebraic theory) but with the additional axioms
\[
\forall x. (\neg(x=0) \implies \exists y. (xy= 1)) \quad \text{and} \quad \neg(0=1).
\]
Since these involve logical quantifiers and connectives (besides the implicit universal quantification present in algebraic theories), the usual axiomatisation of fields is not an algebraic theory, and indeed one can show that there is no algebraic theory that describes fields.

The simplicity of algebraic theories makes them very amenable to category-theoretic generalisation. Let us consider the two most well-known categorical notions of algebraic theory: Lawvere theories and monads.

Lawvere theories provide perhaps the most direct translation of algebraic theories into category-theoretic terms. Indeed, when they were first introduced by Lawvere in~\cite{lawvere63}, he referred to what we now call a Lawvere theory simply as an algebraic theory. Every algebraic theory gives rise to a Lawvere theory and vice versa, but this is not quite a one-to-one correspondence: many algebraic theories can give rise to the same Lawvere theory.

Given an algebraic theory, the Lawvere theory it gives rise to describes not the algebraic theory itself, but the structure possessed by the collection of all terms-up-to-equivalence of the theory. It is this structure that is relevant when talking about models of a theory, and so in some sense Lawvere theories are the more fundamental notion, with algebraic theories merely providing presentations of their Lawvere theories. Just as a group may have many different presentations in terms of generators and relations, so a Lawvere theory may have many different presentations in terms of algebraic theories.

Monads are the second major category-theoretic notion of algebraic theory, and they are closely related to Lawvere theories; for a historical overview, see Hyland and Power~\cite{hylandPower07}. Indeed, monads actually generalise Lawvere theories: the category of Lawvere theories is equivalent to the category of finitary monads on $\Set$ (that is, the monads whose underlying endofunctor preserves filtered colimits). This result is due to Linton~\cite{linton66}. Thus one might wonder why we would look for a common generalisation of Lawvere theories and monads when we already have one, namely monads themselves.

The answer (aside from the fact that there are other notions of algebraic theory that we would also like to generalise) is that there is an important sense in which monads do \emph{not} generalise Lawvere theories: their semantics. A monad naturally exists attached to a particular base category, and algebras (i.e.\ models) for the monad are objects of that base category equipped with structure defined in terms of the monad. Thus if we view a Lawvere theory as a finitary monad on $\Set$, then \emph{a priori} it only makes sense to talk about models of the Lawvere theory in $\Set$. However, there is a natural notion of a model of a Lawvere theory in any finite product category, not just in $\Set$. Of course, models of the Lawvere theory in $\Set$ do coincide with algebras for the corresponding monad, but the more flexible semantics available to Lawvere theories cannot be explained by viewing them just as a special kind of monad.

We can now clarify what we mean when we say that we are looking for a common generalisation of Lawvere theories and monads that is compatible with their semantics.  We would like a general notion of algebraic theory with its own notion of semantics, with Lawvere theories and monads appearing as special cases, in such a way that when the general semantics is specialised to these cases, we recover the semantics of Lawvere theories and of monads in their full generality.

As mentioned above, there are other notions of algebraic theory we would like to generalise. These include:
\begin{itemize}
\item PROPs and PROs, which are analogues of Lawvere theories that take models in symmetric and non-symmetric monoidal categories as opposed to finite product categories;
\item operads, whose models take values in arbitrary multicategories;
\item monads with arities, which are monads on a category that are determined by their values on a given subcategory, in the same way that a finitary monad on $\Set$ is determined by its values on the subcategory of finite sets; and
\item monoids, which can be thought of as very simple algebraic theories with their actions as models.
\end{itemize}
These notions are reviewed in Chapter~\ref{chap:notions}, alongside classical algebraic theories, Lawvere theories and monads.

\section{Proto-theories}

Let us consider in more detail why we might want to find a common generalisation of these notions. The goal is not to replace them with something superior or necessarily to prove a large number of results about them simultaneously. Indeed, this would probably be impossible; there are very significant differences between these notions, for the good reason that they were developed for use in different contexts. In the words of Saunders Mac Lane (\cite{maclane71}, Chapter~4), ``good theory does not search for the maximum generality, but the right generality''.

Rather, we would like to see what they have in common and where they diverge from one another, and in particular what it is that they share that allows them all to be called notions of algebraic theory. Because these notions differ from each other in so many important respects, this common core will necessarily be very simple, almost trivial; nevertheless, it is enough to develop a good notion of semantics. Our common generalisation should be seen not as something that supersedes the existing notions of algebraic theory, but as something that precedes them, like a common ancestor from which they have all evolved (albeit in a conceptual sense, not a historical one). The name we use for this common generalisation, proto-theory, is intended to evoke this idea.

The definition of a proto-theory is extremely simple; it is simply a 1-cell in some 2-category that lies in the left class of a given factorisation system. In practice what this often amounts to is a bijective-on-objects functor between categories, possibly preserving some extra structure. Why then give it a special name? The idea is to promote a point of view that makes certain constructions more intuitive. Consider the definition of a generalised element: a generalised element of an object of a category is simply a morphism with that object as its codomain. Nevertheless, in some contexts thinking of morphisms as element-like-things makes certain constructions more intuitive. Similarly, thinking of bijective-on-objects functors as theory-like-things makes the construction of structure--semantics adjunctions more intuitive.

We think of a bijective-on-objects functor $L \from \cat{A} \to \cat{L}$ as an algebraic theory as follows. The objects of $\cat{A}$ are thought of as shapes for the inputs and outputs for some operations. The morphisms in $\cat{A}$ are then canonical ways of transforming one shape into another. The morphisms in $\cat{L}$ are the operations (or terms-up-to-equivalence) of the theory; each one has an input shape (also called its arity) and an output shape. Composition in $\cat{L}$ corresponds to substitution of terms, and the equations of the theory are encoded in the equations that hold between composites in $\cat{L}$.

Suppose we have some proto-theory $L \from \cat{A} \to \cat{L}$ and we wish to consider models of $L$ in a category $\cat{B}$. Intuitively, a model of $L$ in $\cat{B}$ should consist of an object $b$ of $\cat{B}$ together with an interpretation of each operation of the theory. More precisely, if $l \from La \to La'$ is an operation of $L$ with arity $a$ and output shape $a'$, then the $L$-model structure on $b$ should give us a way of transforming ``$a$-indexed families of elements of $b$'' into $a'$-indexed families.

This does not make sense \emph{a priori}; for arbitrary categories $\cat{A}$ and $\cat{B}$ there is no canonical notion of a family of elements of an objects in $\cat{B}$ indexed by an object of $\cat{A}$. Instead, we need to specify such a notion as extra structure. This leads to the second major definition of this thesis, that of an \emph{interpretation of arities} or \emph{aritation} for short. Once we have specified an interpretation of arities from $\cat{A}$ in $\cat{B}$, then we can define models of $L$ in $\cat{B}$, and such models form a category equipped with a canonical forgetful functor to $\cat{B}$. This is encapsulated in a \emph{semantics functor}
\[
\sem \from \proth(\cat{A})^{\op} \to \catover{\cat{B}},
\]
where $\proth(\cat{A})$ denotes the category of proto-theories with arities in $\cat{A}$. This semantics functor always has a left adjoint which is called a \emph{structure functor}, and together they form the \emph{structure--semantics adjunction} for the chosen aritation.

The existence of the structure functor comes at a cost, namely that we must be willing to tolerate categories that are much larger than those that are commonly dealt with. The reason is as follows. Let $\cat{B}$ be some category, and $\cat{M}$ some category of objects of $\cat{B}$ equipped with extra structure, with forgetful functor $U \from \cat{M} \to \cat{B}$. Then if we have an interpretation of arities from $\cat{A}$ in $\cat{B}$, we can apply the corresponding structure functor to $U$ to obtain a proto-theory with arities in $\cat{A}$, which is a bijective-on-objects functor $\str(U) \from \cat{A} \to \thr(U)$ for some category $\thr(U)$ defined in terms of $U$. Suppose the cardinality of the set $\ob(\cat{M})$ of objects of $\cat{M}$ is $\kappa$; then each hom-set in $\thr(U)$ can have cardinality as large as $2^{\kappa}$. In particular, even if $\cat{M}$ is locally small but has a large set of objects, the hom-sets of $\thr(U)$ can not only fail to be small, but can have cardinality as large as the \emph{power set} of $\ob(\cat{M})$.

There are several ways to avoid having to deal with such large sets. The first would be to restrict our attention to small categories. If $\kappa$ above is small, then so is $2^{\kappa}$, and we never have to deal with large categories at all. However this is undesirable because many of the categories we are most interested in are not in fact small; in particular we are often interested in the category of all small sets, or categories of all small sets equipped with some structure and these categories are of course not small.

The second way is to give up on the existence of a left adjoint to the semantics functor and restrict our attention to proto-theories and categories of structures satisfying some size constraints. This seems like a reasonable approach, although not the one we pursue in this thesis. We have chosen rather to see how the theory develops naturally without imposing size restrictions, allowing sets to get as large as they need to in order for the constructions we are interested in to make sense. Having gained this ``big picture'' view, one can later impose whatever size conditions are appropriate for the particular situation one is interested in, but if this were done from the start one might miss out on useful insights granted by a broader perspective.

In Chapter~\ref{chap:adjunction}, we develop the notions of proto-theories, aritations and their structure--semantics adjunctions in the special case of proto-theories in $\CAT$, which are just bijective-on-objects functors. Aside from the definitions, the main content here is the construction of the structure--semantics adjunction for an arbitrary aritation. We repeat this process in Chapter~\ref{chap:structure}, but now for proto-theories in the full generality of an arbitrary 2-category. Again we define the appropriate notions of proto-theory and aritation and construct the structure--semantics adjunction. We then show that all the examples of notions of algebraic theory from Chapter~\ref{chap:notions} arise in this way, with the exception of monoids, which are dealt with in Section~\ref{sec:str-sem-adj-monoids}, and monads (possibly with arities), which are dealt with in Chapter~\ref{chap:canonical1}.

As mentioned above, since the definition of proto-theory is extremely simple one would not necessarily expect there to be many interesting theorems that hold in the full generality of completely arbitrary proto-theories. However there are a few results that can be proved at this level of abstraction, or at least with mild assumptions on the proto-theories and aritations in question; a few such results are explored in Chapter~\ref{chap:general}. In particular, the bird's-eye-view provided by proto-theories allows us to give a uniform proof of the fact that forgetful functors from categories of algebras for both monads and Lawvere theories create all limits, while also explaining why this is not the case for other classes of proto-theory. We also prove that the forgetful functors from categories of models of proto-theories have the property of being \emph{amnestic isofibrations}, at least in the examples we are most interested in. Along the way we show that a choice we made when defining the semantics of proto-theories that may have seemed somewhat arbitrary (namely that semantics is defined by a strict, rather than weak pullback in $\CAT$) was not arbitrary after all, in the sense that the two choices result in equivalent categories of models.

\section{A convenient category of monads}

Our second broad objective in this thesis is to find a well-behaved extension of the category of monads on a given category. In particular, we would like a notion of semantics for this extension, generalising that of monads, with certain desirable properties that the usual semantics of monads lacks.

Recall that every adjunction gives rise to a monad. More precisely, there is a canonical (contravariant) functor from the category of right adjoints into a given category $\cat{B}$ (with commuting triangles as morphisms) to the category of monads on $\cat{B}$. This functor is adjoint to the semantics functor that sends a monad to the forgetful functor from its category of Eilenberg--Moore algebras. This is the classical structure--semantics adjunction for monads.

We could instead regard the semantics functor as a functor into the category of \emph{all} functors into $\cat{B}$ rather than just the right adjoints, and ask whether this version of the semantics functor has an adjoint; that is, do arbitrary functors have a best approximation by a monadic functor? One can show that such an approximation exists for a given functor if and only if the right Kan extension of that functor along itself exists and is a pointwise Kan extension; in this case the resulting monad is the pointwise codensity monad of the functor. However, codensity monads do not always exist and so the answer to our question is negative. Nevertheless, one can ask whether there is some generalisation of the notion of a monad, with a semantics extending the usual semantics of monads such that such a left adjoint to the semantics functor does exist. In Chapter~\ref{chap:canonical1} we show that the notion of proto-theories with arities in $\cat{B}$ provides such a notion, where the semantics is provided by an aritation that we call the \emph{canonical aritation} on $\cat{B}$.

However, when we pass from monads to proto-theories, a desirable property of the semantics of monads is lost. The semantics functor from the category of monads on $\cat{B}$ to the category of right adjoints into $\cat{B}$ is full and faithful. This is a kind of completeness theorem; it implies that no information is lost when passing from a monad to its category of algebras. More precisely, since the semantics functor has a left adjoint, it being full and faithful is equivalent to the counit of this adjunction being an isomorphism, so the semantics functor exhibits (the opposite of) the category of monads as a reflective subcategory of the category of right adjoints into $\cat{B}$. Since the counit is an isomorphism, we can recover a monad (up to isomorphism) from its category of algebras. Thinking of monads as algebraic theories, this means that a monad gives a complete description of the algebraic structure possessed by its models. There is no superfluous information in the monad that is not reflected in its algebras, and its algebras do not possess any additional algebraic structure other than that described by the monad; this is a very desirable property for a notion of algebraic theory to have.

We can now say precisely what we mean when we say that we are looking for a convenient category of monads. A convenient category of monads on $\cat{B}$ is a category $\cat{C}$ in which the category $\monad(\cat{B})$ of monads on $\cat{B}$ can be embedded as a full subcategory, equipped with an adjunction
\[
\xymatrix{
{\catover{\cat{B} }}\ar@<5pt>[r]_-{\perp}^-{\str}\ & {\cat{C}^{\op}}\ar@<5pt>[l]^-{\sem}
}
\]
that extends the structure--semantics adjunction for monads, with $\sem$ full and faithful.

Unfortunately, the semantics of proto-theories induced by the canonical aritation on $\cat{B}$ does not have this property. We prove this in Chapter~\ref{chap:canonical2}, by establishing a relationship between proto-theories and groups. One can define an action of a group $G$ on an object $b$ of an arbitrary category $\cat{B}$ --- it is simply a monoid homomorphism from $G$ to the monoid $\cat{B}(b,b)$ of endomorphisms of that object. There is also an appropriate notion of equivariant map between such $G$-objects, and so they form a category with a forgetful functor to $\cat{B}$. This describes the semantics part of a structure--semantics adjunction with groups as the notion of algebraic theory.

We define a full and faithful functor from the category of groups to the category of proto-theories with arities in $\finset$, the category of finite sets, in such a way that the structure--semantics adjunction for proto-theories induced by the canonical aritation on $\finset$ extends the structure--semantics adjunction for groups described above. In particular, the monad on the category of proto-theories on $\finset$ induced by the former adjunction restricts to the monad on the category of groups induced by the latter adjunction. We then show that this monad on the category of groups is the profinite completion monad. This monad is known not to be idempotent, meaning that the monad on the category of proto-theories is not idempotent. It follows that the structure--semantics adjunction for proto-theories on $\finset$ is not idempotent, and in particular the semantics functor is not full and faithful. Thus proto-theories do not in general satisfy the completeness theorem.

However, this negative result suggests an analogy between proto-theories and groups that turns out to be very fruitful, and the rest of this thesis is spent pursuing it with the ultimate goal of finding a convenient category of monads. Under this analogy, the structure--semantics monad on the category of proto-theories corresponds to the profinite completion monad on the category of groups. The profinite completion monad is the codensity monad of the inclusion of the category of finite groups, and we might wonder whether there is a similar characterisation of the structure--semantics monad as a codensity monad, and if so what the analogue of the category of finite groups is. In the second part of Chapter~\ref{chap:canonical2}, we show that, under mild assumptions on $\cat{B}$, the structure--semantics monad is the codensity monad of the inclusion of the category of monads into the category of proto-theories, and so in some sense monads play a role analogous to that of finite groups.

Although the profinite completion monad on the category of groups is not idempotent, there is a closely related monad which is, namely the profinite completion monad on the category of topological groups. This suggests that by considering some notion of topological proto-theories, analogous to topological groups, we may find a structure--semantics monad that is idempotent, which is a first step towards a convenient category of monads. We do this in Chapter~\ref{chap:topology}, giving a definition of topological proto-theory and showing that their semantics extends the semantics of monads. Then we show that, under certain conditions on the base category $\cat{B}$, the topological structure--semantics adjunction is idempotent. The conditions we impose on $\cat{B}$ appear to be quite restrictive, however they hold in the most important examples, namely the categories of sets and finite sets, as well as in the category of vector spaces over any field.

Any idempotent adjunction can be factored as reflection and a coreflection. Thus we have a reflective subcategory of the category of topological proto-theories, and the restriction of the topological semantics functor to this subcategory is full and faithful. We call the objects of this subcategory complete topological proto-theories. As the algebras for the topological structure--semantics monad, they are analogous to profinite groups, which are the algebras for the topological profinite completion monad.

In Chapter~\ref{chap:complete} we first show that monads are complete topological proto-theories,  from which it follows that the category of complete topological proto-theories is a convenient category of monads. We then pursue the analogy between complete topological proto-theories and profinite groups, giving several characterisations of the category of complete topological proto-theories that mirror similar characterisations of the category of profinite groups. In particular we can define complete topological proto-theories without even mentioning the structure--semantics adjunction: they are precisely the topological proto-theories that can be written as limits of diagrams of monads. In addition, the category of complete topological proto-theories is the smallest reflective subcategory of the category of topological proto-theories that contains the monads.

The final section of Chapter~\ref{chap:complete} deals with some examples of categories of models of complete topological proto-theories that are not monadic. These categories are described by \emph{equational presentations} in the sense of Manes~\cite{manes76}, and include the categories of complete lattices and complete Boolean algebras. These are structures that can be defined by operations and equations that are highly infinitary in nature. Indeed, they may have operations of arbitrarily high arity, and as a result free algebras do not exist and so these categories are not monadic. Nonetheless, we show that every category that is equationally presentable over $\Set$ is the category of models for some complete topological proto-theory on $\Set$.

Our use of the term ``convenient category of monads'' is inspired by the idea of a ``convenient category of topological spaces'' from Steenrod~\cite{steenrod67}. In both cases a ``convenient category of $X$'' refers to a modified version of the category of $X$, that has certain desirable properties that the category of $X$ itself lacks. However, the specific requirements we ask for in a convenient category of monads are unrelated to the requirements for a convenient category of topological spaces.

\section{Further work}

Finally, in Chapter~\ref{chap:open}, we discuss some questions that remain unanswered and which could provide interesting directions for further work. First there is the question of what are the most appropriate notions of morphisms between proto-theories and aritations. Many of the existing notions of algebraic theory are closely related; for example Lawvere theories can be described by finitary monads, and there are various canonical functors between the categories of Lawvere theories, PROPs, PROs and operads, and these are compatible with their semantics to varying degrees. It would be illuminating to understand these relationships in terms of morphisms between aritations or proto-theories. 

There is a sense in which the theory of proto-theories, aritations and structure--semantics adjunctions can be generalised from $\CAT$ to other symmetric monoidal categories; we do not emphasise this generalisation in this thesis because all of the known examples of interest are in the context of $\CAT$. Our second open question is whether there are examples of structure--semantics adjunctions in this more general context that have mathematical significance.

The third open question concerns the analogy between groups and proto-theories, and specifically between profinite groups and complete topological proto-theories. There are many characterisations of the category of profinite groups. In Chapter~\ref{chap:complete} we prove the proto-theory analogues of some, but not all of these. Thus it remains an open question whether proto-theoretic analogues of the other characterisations of profinite groups exist.

\chapter{Background material}
\label{chap:background}

In this chapter we review some background material that will be used throughout the rest of this thesis. In Section~\ref{sec:sets-background}, we introduce the set-theoretic assumptions necessary for dealing with the large categories that appear later in the thesis. Section~\ref{sec:2-monad-background} covers 2-categories and 2-monads, which will be used in Chapter~\ref{chap:structure} to describe certain notions of algebraic theory in terms of proto-theories. In Section~\ref{sec:fact-background} we recall the notions of factorisation systems and enhanced factorisation systems, and in particular the bijective-on-objects/full and faithful factorisation system on $\CAT$, and in Section~\ref{sec:bo-properties-background} we describe some additional properties of bijective-on-objects functors. Section~\ref{sec:dense-codense-background} covers density and codensity, including codensity monads, and Section~\ref{sec:idem-background} recalls the notions of idempotent monads and adjunctions. Finally in Section~\ref{sec:profinite-background} we recall the definition of a profinite group and some equivalent ways of characterising them.

\section{Set-theoretic preliminaries}
\label{sec:sets-background}

As mentioned in the introduction, in order to define and prove results about general structure--semantics adjunctions, we will need to deal with categories that are larger than usual. The appropriate way to do this is using the notion of a Grothendieck universe. Informally, this means that there is a set $U$ of sets that is closed under all the usual set forming operations, such as unions, products, power sets, and so on.

We think of the elements of $U$ as ``small'' sets. We can then do most ordinary mathematics while only ever referring to small sets --- it is usually only necessary to talk about sets that are not elements of $U$ when we wish to discuss the totality of all small structures of a given type as a mathematical structure in its own right. For example, we could talk about the collection of all groups that have small underlying sets. Since $U$ is a set, such collections are themselves sets (albeit not small) and we can manipulate them using the usual set-forming operations.

More precisely, in the context of Zermelo--Fraenkel set theory with the axiom of choice (ZFC), a Grothendieck universe is defined as follows.

\begin{defn}
A set $U$ is a \demph{Grothendieck universe} if
\begin{enumerate}
\item whenever $X \in U$ and $Y \in X$, then $Y \in U$;
\item whenever $X, Y \in U$, then $\{X,Y\} \in U$;
\item whenever $X \in U$, then $\pow (X) \in U$, where $\pow(X)$ is the power set of $X$; and
\item whenever $I \in U$ and $X_i \in U$ for each $i \in I$, then $\Union_{i \in I} X_i \in U$.
\end{enumerate}
\end{defn}

For the rest of this thesis, we will assume the existence of a Grothendieck universe $U$.

\begin{defn}
\begin{enumerate}
\item A set $X$ is \demph{small} if $X$ is in bijection with some $X' \in U$.
\item We define \demph{large set} to be synonymous with ``set'', and use it when we wish to emphasise that the set in question is not necessarily small.
\item A \demph{properly large set} is a large set that is not small.
\item A \demph{class} is a collection of sets defined by some first-order formula, not necessarily forming a set.
\end{enumerate}
\end{defn}

By default, the collections of objects and morphisms of a category may be classes. If we wished to avoid this, and only deal with categories with sets of objects and morphisms, we could posit the existence of a second Grothendieck universe above the first. However, since we will not need to perform any complex set-theoretic manipulations on categories with proper classes of morphisms, we prefer to avoid this and deal with these categories on a somewhat informal basis. 

\begin{defn}
Let $\cat{C}$ be a category with object class $\ob(C)$ and morphism class $\mor(C)$. Then:
\begin{itemize}
\item if $\mor(C)$ (and hence $\ob(C)$) is a small set, then $\cat{C}$ is \demph{small};
\item if $\mor(C)$ is a large set then $\cat{C}$ is \demph{large} (this implies that $\ob(C)$ is a large set);
\item if $\ob(C)$ is a large set and each $\cat{C}(c,c')$ is a small set, then $\cat{C}$ is \demph{locally small}; and
\item we call a category $\cat{C}$ \demph{huge} when we wish to emphasise that it does not necessarily satisfy any of the above size conditions.
\end{itemize}
If $\cat{C}$ is large and not small it is called \demph{properly large}, and if $\cat{C}$ is huge and not large it is called \demph{properly huge}. If $\cat{C}$ is equivalent to a small category then $\cat{C}$ is called \demph{essentially small}, and if $\cat{C}$ is equivalent to a large category then $\cat{C}$ is called \demph{essentially large}. Similarly, if $\cat{C}$ is essentially large and each $\cat{C}(c,c')$ is small, then $\cat{C}$ is called \demph{essentially locally small}.

For almost all purposes an essentially small category may be treated as if it were small. We may sometimes abuse terminology slightly by calling categories small when in fact they are only essentially small, and similarly for large and essentially large categories.
\end{defn}

We now define notation for some categories of various sizes that we shall use frequently.

\begin{defn}
\begin{enumerate}
\item Write $\SET$ for the (properly huge) category of sets.
\item Write $\Set$ for the (properly large) category of small sets.
\item Write $\finset$ for the (essentially small) category of finite sets.
\item Write $\CAT$ for the (properly huge) category of large categories.
\item Write $\Cat$ for the (properly large) category of small categories.
\item Write $\TOP$ for the (properly huge) category of topological spaces.
\item Write $\Top$ for the (properly large) category of small topological spaces.
\item Write $\MONOID$ for the (properly huge) category of large monoids.
\item Write $\monoid$ for the (properly large) category of small monoids.
\item Write $\TopMonoid$ for the (properly large) category of small topological monoids.
\item Write $\Gp$ for the (properly large) category of small groups.
\item Write $\FinGp$ for the (essentially small) category of finite groups.
\item Write $\TopGp$ for the (properly large) category of small topological groups.
\end{enumerate}
\end{defn}

\begin{comment}
\begin{remark}
A common alternative to the use of Grothendieck universes in discussing large categories is to use G\"odel--Bernays set theory, which axiomatises the notion of classes in a rigorous way. However, even the classes of G\"odel--Bernays set theory are not large enough to describe the all of the categories we will be interested in for the following reason. A class in G\"odel--Bernays set theory can be thought of as a sub-''set'' of the collection of all sets. 

However, this is not sufficient for our purposes --- in the language of Grothendieck universes, this corresponds to allowing discussion of \emph{subsets} of the universe $U$ (these are the classes of the corresponding model of G\"odel--Bernays set theory), rather than just the \emph{elements} of $U$ (which are the sets of G\"odel--Bernays). In particular, we can only form collections with cardinality at most that of $U$ itself. However, in order for the structure--semantics adjunction to be defined, we need to allow categories to have collections of morphisms with cardinality strictly greater than that of $U$.
\end{remark}
\end{comment}

\section{Categories with algebraic structure and 2-monads}
\label{sec:2-monad-background}

At several points in this thesis we will have reason to consider categories equipped with some kind of extra structure. The kinds of structure we are most interested in are best characterised in terms of 2-monads on the category $\CAT$ of all large categories. We collect here the basic 2-categorical definitions and notation that we shall use in later chapters. The definitions of 2-category, 2-functor and 2-natural transformation were first developed by Kelly and Street in~\cite{kellyStreet74}, and the theory of 2-monads was developed by Blackwell, Kelly and Power in~\cite{blackwellKellyPower89}.

\begin{defn}
A \demph{2-monad} on the 2-category $\CAT$ of large categories consists of a 2-functor $T \from \CAT \to \CAT$ together with 2-natural transformations $\eta \from \id_{\CAT} \to T$ and $\mu \from TT \to T$ such that the usual monad axioms hold strictly.
\end{defn}

\begin{defn}
Let $\mnd{T} = (T, \eta, \mu)$ be a 2-monad on $\CAT$. A \demph{$\mnd{T}$-algebra} consists of a category $\cat{C}$ together with a functor $Y \from T\cat{C} \to \cat{C}$ such that the usual axioms for an algebra for a monad hold strictly.
\end{defn}

\begin{defn}
Let $\mnd{T} = (T, \eta, \mu)$ be a 2-monad on $\CAT$, and $(\cat{C}, Y)$ and $(\cat{D}, Z)$ be $\mnd{T}$-algebras. A \demph{pseudo-$\mnd{T}$-morphism} $(\cat{C}, Y) \to (\cat{D}, Z)$ consists of a functor $F \from \cat{C} \to \cat{D}$ together with a natural isomorphism $f \from Z \of TF \to F \of Y$ such that
\[
\vcenter{
\xymatrix{
\cat{C}\ar[r]^{\eta_{\cat{C}}}\ar[d]_{F} & T\cat{C}\ar[r]^Y\ar[d]_{TF}\drtwocell\omit{^f} & \cat{C}\ar[d]^F \\
\cat{D}\ar[r]_{\eta_{\cat{D}}} & T\cat{D}\ar[r]_Z& \cat{D}
}}
\quad = \quad
\vcenter{
\xymatrix{
\cat{C}\ar@{=}[r]\ar[d]_F & \cat{C}\ar[d]^F \\
\cat{D}\ar@{=}[r] & \cat{D}
}}
\]
and
\[
\vcenter{
\xymatrix{
TT\cat{C}\ar[r]^{\mu_{\cat{C}}}\ar[d]_{TTF} & T\cat{C}\ar[r]^{Y}\ar[d]_{TF}\drtwocell\omit{^f} & \cat{C}\ar[d]^F \\
TT\cat{D}\ar[r]_{\mu_{\cat{D}}} & T\cat{D}\ar[r]_{Z} & \cat{D}
}}
\quad = \quad
\vcenter{
\xymatrix{
TT\cat{C}\ar[r]^{TY}\ar[d]_{TTF}\drtwocell\omit{^Tf\:\:} & T\cat{C}\ar[r]^{Y}\ar[d]_{TF}\drtwocell\omit{^f} & \cat{C}\ar[d]^F \\
TT\cat{D}\ar[r]_{TZ} & T\cat{D}\ar[r]_{Z} & \cat{D}.
}}
\]
\end{defn}

\begin{defn}
Let $\mnd{T} = (T, \eta,\mu)$ be a 2-monad on $\CAT$, let $(\cat{C}, Y)$ and $(\cat{D}, Z)$ be $\mnd{T}$-algebras and let $(F,f)$ and $(G,g)$ be pseudo-$\mnd{T}$-morphisms $(\cat{C}, Y) \to (\cat{D}, Z)$. A \demph{$\mnd{T}$-transformation} $(F,f) \to (G,g)$ consists of a natural transformation $\phi \from F \to G$ such that
\[
\vcenter{
\xymatrix
@R=40pt
@C=40pt{
T\cat{C}\rtwocell^{TF}_{TG}{\:\:T\phi}\ar[d]_Y\drtwocell\omit{g} & T\cat{D}\ar[d]^Z \\
\cat{C}\ar[r]_G & \cat{D}
}}
\quad = \quad
\vcenter{
\xymatrix
@R=40pt
@C=40pt{
T\cat{C}\ar[r]^{TF}\ar[d]_{Y}\drtwocell\omit{f} & T\cat{D}\ar[d]^Z \\
\cat{C}\rtwocell^F_G{\phi} & \cat{D}.
}}
\]
\end{defn}

\begin{defn}
\label{defn:2-monad-algebras}
Let $\mnd{T}$ be a 2-monad on $\CAT$. We write $\talg$ for the 2-category of $\mnd{T}$-algebras, pseudo-$\mnd{T}$-morphisms and $\mnd{T}$-transformations.
\end{defn}

\section{Factorisation systems}
\label{sec:fact-background}

Factorisation systems generalise some of the important properties of the classes of surjective and injective functions between sets. The notion of a factorisation system was introduced by Freyd and Kelly in~\cite{freydKelly72}. Over the years many variants have been defined, however, when we write ``factorisation system'', we always refer to this original notion, which is also sometimes called an \emph{orthogonal factorisation system}.

\begin{defn}
Let $e \from a \to c$ and $n \from b \to d$ be morphisms in a category $\cat{C}$. Then we say that \demph{$e$ is left orthogonal to $n$} or \demph{$n$ is right orthogonal to $e$} and write $e \perp n$ if, for every commutative square of the form
\[
\xymatrix{
a \ar[r]^f\ar[d]_e & b\ar[d]^n \\
c\ar[r]_g & d
}
\]
there is a unique $h \from  c \to b$ such that $h \of e = f$ and $n \of h =g$. We call such an $h$ a \demph{fill-in} for this square.
\end{defn}

\begin{defn}
\label{defn:fact-system}
Let $\cat{C}$ be a category. A \demph{factorisation system} on $\cat{C}$ consists of two classes $\ofsfont{E}$ and $\ofsfont{N}$ of morphisms in $\cat{C}$ that are each closed under composition and each contain all the isomorphisms, such that
\begin{enumerate}
\item every morphism $f$ in $\cat{C}$ can be written as a composite $e \of n$ where $e \in \ofsfont{E}$ and $n \in \ofsfont{N}$; and
\item for every $e \in \ofsfont{E}$ and $n \in \ofsfont{N}$ we have $e \perp n$.
\end{enumerate}
\end{defn}

It is possible to define factorisation systems without reference to the orthogonality relation; the following characterisation is due to Joyal (Definition C.0.19 in~\cite{joyal08}).

\begin{lem}
Let $\ofsfont{E}$ and $\ofsfont{N}$ be two classes of morphisms in a category $\cat{C}$. Then $(\ofsfont{E},\ofsfont{N})$ is a factorisation system on $\cat{C}$ if and only if $\ofsfont{E}$ and $\ofsfont{N}$ are both closed under composition and contain all the isomorphisms, and in addition every morphism in $\cat{C}$ can be factored as a map in $\cat{E}$ followed by a map in $\cat{N}$ and this factorisation is unique up to unique isomorphism. \hfill \qed
\end{lem}

\begin{remark}
\label{rem:ofs-pullback}
Another way to express the fact that a morphism $e \from a \to c$ is left orthogonal to a morphism $n \from b \to d$ is that the square
\[
\xymatrix{
\cat{C}(c,b)\ar[r]^{n_*}\ar[d]_{e^*} & \cat{C}(c,d)\ar[d]^{e^*} \\
\cat{C}(a,b)\ar[r]_{n_*} & \cat{C}(a,d)
}
\]
is a pullback in $\SET$.
\end{remark}

The prototypical factorisation system is on the category of sets, with $\ofsfont{E}$ being the class of all surjections and $\ofsfont{N}$ being the class of all injections. The main example that shall concern us in this thesis is as follows.

\begin{lem}
\label{lem:bo-ff-factorisation}
There is a factorisation system $(\ofsfont{E},\ofsfont{N})$ on the category $\CAT$ of all large categories, with $\ofsfont{E}$ being the class of all functors that are bijective on objects, and $\ofsfont{N}$ being the class of full and faithful functors.
\end{lem}
\begin{proof}
This is well-known and the proof is elementary; we therefore omit it.
\end{proof}

The observation in Remark~\ref{rem:ofs-pullback} that orthogonality can be expressed in terms of pullbacks in $\SET$ allows us to generalise the notion of a factorisation system to enriched categories, and in particular to 2-categories, which are categories enriched in $\CAT$.

\begin{defn}
\label{defn:cat-fact-sys}
Let $\cat{X}$ be a 2-category, and let $(\ofsfont{E}, \ofsfont{N})$ be a factorisation system on the underlying 1-category of $\cat{X}$. The $(\ofsfont{E},\ofsfont{N})$ is a $\CAT$-factorisation system on $\cat{X}$ if, for every $(E\from \cat{A} \to \cat{C}) \in \ofsfont{E}$ and $(N\from \cat{B} \to \cat{D}) \in \ofsfont{N}$, the square

\[
\xymatrix{
\cat{X}(\cat{C},\cat{B})\ar[r]^{N_*}\ar[d]_{E^*} & \cat{X}(\cat{C},\cat{D})\ar[d]^{E^*} \\
\cat{X}(\cat{A},\cat{B})\ar[r]_{N_*} & \cat{X}(\cat{A},\cat{D})
}
\]
is a pullback in $\CAT$.
\end{defn}

On the level on objects, the fact that this square is a pullback is orthogonality of $\ofsfont{E}$ and $\ofsfont{N}$ in the unenriched sense, as per Remark~\ref{rem:ofs-pullback}. On the level of morphisms however, the condition of being a pullback says the following.

Let $E \from \cat{A} \to \cat{C}$ and $N \from \cat{B} \to \cat{D}$ be 1-cells in $\cat{X}$ with $E \in \ofsfont{E}$ and $N\in \ofsfont{N}$. Let $F_1, F_2 \from \cat{A} \to \cat{B}$ and $G_1, G_2 \from \cat{C} \to \cat{D}$ be 1-cells such that the square
\[
\xymatrix{
\cat{A}\ar[r]^{F_i} \ar[d]_{E} & \cat{B}\ar[d]^{N} \\
\cat{C}\ar[r]_{G_i} & \cat{D}
}
\]
commutes for $i = 1,2$, and let $H_1, H_2 \from \cat{C} \to \cat{B}$ be the fill-ins for these two squares respectively. Then, given 2-cells $\alpha \from F_1 \to F_2$ and $\beta \from G_1 \to G_2$ such that
\[
\vcenter{
\xymatrix
@R=40pt
@C=40pt{
\cat{A}\rtwocell^{F_2}_{F_1}{^\alpha} \ar[d]_E & \cat{B}\ar[d]^N \\
\cat{C}\ar[r]_{G_1} & \cat{D}
}}
\quad = \quad
\vcenter{
\xymatrix
@R=40pt
@C=40pt{
\cat{A}\ar[r]^{F_2} \ar[d]_E & \cat{B}\ar[d]^N \\
\cat{C}\rtwocell_{G_1}^{G_2}{^\beta} & \cat{D},
}}
\]
there is a unique 2-cell $\gamma \from H_1 \to H_2$ such that
\[
\vcenter{
\xymatrix
@R=40pt
@C=40pt{
\cat{A}\ar[r]^{F_2} \ar[d]_E & \cat{B} \\
\cat{C}\urtwocell_{H_1}^{H_2}{^\gamma}&
}}
\quad = \quad
\vcenter{
\xymatrix
@R=40pt
@C=40pt{
\cat{A}\rtwocell_{F_1}^{F_2}{^\alpha} & \cat{B}
}}
\]
and
\[
\vcenter{
\xymatrix
@R=40pt
@C=40pt{
& \cat{B}\ar[d]^N \\
\cat{C}\urtwocell_{H_1}^{H_2}{^\gamma}\ar[r]_{G_1} & \cat{D}
}}
\quad = \quad
\vcenter{
\xymatrix
@R=40pt
@C=40pt{
\cat{C}\rtwocell_{G_1}^{G_2}{^\beta} & \cat{D}.
}}
\]

There is a further strengthening of the notion of factorisation system that is available in the setting of 2-categories (and not in general enriched categories).

\begin{defn}
\label{defn:strong-orthog}
Let $\cat{X}$ be a 2-category, and let $E \from \cat{A} \to \cat{C}$ and $N \from \cat{B} \to \cat{D}$ be 2-cells in $\cat{X}$. We say that $E$ is \demph{strongly left orthogonal to $N$} or that \demph{$N$ is strongly right orthogonal to $E$} if, for all 1-cells $F \from \cat{A} \to \cat{B}$ and $G \from \cat{C} \to \cat{D}$ and invertible 2-cells
\[
\xymatrix{
\cat{A}\ar[r]^F\ar[d]_E\drtwocell\omit{^\phi} & \cat{B}\ar[d]^N \\
\cat{C}\ar[r]_G & \cat{D}
}
\]
there is a unique 1-cell $H \from \cat{C} \to \cat{B}$ and invertible 2-cell $\theta \from G \to N \of H$ such that $H \of E = F$ and
\[
\vcenter{
\xymatrix{
\cat{A}\ar[r]^F\ar[d]_E & \cat{B}\ar[d]^N \\
\cat{C}\ar[ur]^H\ar[r]_G^*!/ur3pt/{\labelstyle{\theta} \objectstyle\Uparrow } & \cat{D}
}}
\quad = \quad
\vcenter{
\xymatrix{
\cat{A}\ar[r]^F\ar[d]_E\drtwocell\omit{^\phi} & \cat{B}\ar[d]^N \\
\cat{C}\ar[r]_G & \cat{D}.
}}
\]
We say that $(H, \theta)$ is a \demph{fill-in} for $\phi$.
\end{defn}

\begin{defn}
Let $(\ofsfont{E}, \ofsfont{N})$ be a $\CAT$-factorisation system on a 2-category $\cat{X}$. We say that $(\ofsfont{E}, \ofsfont{N})$ is an \demph{enhanced factorisation system} if in addition every element of $\ofsfont{E}$ is strongly left orthogonal to every element of $\ofsfont{N}$.
\end{defn}

\begin{lem}
\label{lem:bo-ff-factorisation-enhanced}
The bijective-on-objects/full-and-faithful factorisation system on $\CAT$ is an enhanced factorisation system.
\end{lem}
\begin{proof}
This follows from Proposition~23 in Street and Walters~\cite{streetWalters78}.
\end{proof}

An enhanced factorisation system on a 2-category $\cat{X}$ is a factorisation system on the underlying 1-category of $\cat{X}$ with two additional properties: we have the 2-dimensional orthogonality property as described after Definition~\ref{defn:cat-fact-sys}, and we have the strong orthogonality property of Definition~\ref{defn:strong-orthog}. In the case of the bijective-on-objects/full-and-faithful factorisation we have an additional ``two-dimensional strong orthogonality'' property that combines the two. I was unable to find any mention of this additional property in the literature, although it may be known. I also do not know whether the analogous result holds in \emph{any} enhanced factorisation system; in any case, we shall only need it for the bijective-on-objects/full-and-faithful factorisation system on $\CAT$.

\begin{lem}
\label{lem:enhanced-2-dim}
Let $E \from \cat{A} \to \cat{C}$ be a bijective-on-objects functor and let $N \from \cat{B} \to \cat{D}$ be a full and faithful functor. Let $F_1, F_2 \from \cat{A} \to \cat{B}$ and $G_1, G_2 \from \cat{C} \to \cat{D}$ be functors. For $i = 1, 2,$ let
\[
\xymatrix{
\cat{A}\ar[r]^{F_i}\ar[d]_E\drtwocell\omit{^\phi_i} & \cat{B}\ar[d]^N \\
\cat{C}\ar[r]_{G_i} & \cat{D}
}
\]
be a natural isomorphism with fill-in
\[
\xymatrix{
& \cat{B}\ar[d]^N \\
\cat{C}\ar[r]_{G_i}^*!/ur5pt/{\labelstyle{\theta_i} \!\! \objectstyle\Uparrow } \ar[ur]^{H_i} & \cat{D},
}
\]
and let
\[
\vcenter{
\xymatrix{
\cat{A}\rtwocell^{F_2}_{F_1} {^\alpha} & \cat{B}
}}
\quad \text{and} \quad
\vcenter{
\xymatrix{
\cat{C}\rtwocell^{G_2}_{G_1} {^\beta} & \cat{D}
}}
\]
be natural transformations such that
\begin{equation}
\label{eq:enhanced-compat}
\vcenter{
\xymatrix
@C=40pt
@R=40pt{
\cat{A}\rtwocell^{F_2}_{F_1}{^\alpha}\ar[d]_E\drtwocell\omit{^\phi_1} & \cat{B}\ar[d]^N \\
\cat{C}\ar[r]_{G_1} & \cat{D}
}}
\quad = \quad
\vcenter{
\xymatrix
@C=40pt
@R=40pt{
\cat{A}\ar[r]^{F_2}\ar[d]_E \drtwocell\omit{^\phi_2} & \cat{B}\ar[d]^N \\
\cat{B}\rtwocell^{G_2}_{G_1}{^\beta} & \cat{D}.
}}
\end{equation}
Then there is a unique natural transformation $\gamma \from H_1 \to H_2$ such that
\[
\vcenter{
\xymatrix
@C=40pt
@R=40pt{
\cat{A}\ar[r]^{F_2}\ar[d]_{E} & \cat{B} \\
\cat{C}\urtwocell^{H_2}_{H_1}{^\gamma} &
}}
\quad = \quad
\vcenter{
\xymatrix
@C=40pt
@R=40pt{
\cat{A}\rtwocell^{F_2}_{F_1}{^\alpha}\ar[d]_E & \cat{B} \\
\cat{C}\ar[ur]_{H_1} & 
}}
\]
and
\[
\vcenter{
\xymatrix
@C=50pt
@R=50pt{
& \cat{B}\ar[d]^N \\
\cat{C}\urtwocell^{H_2}_{H_1}{^\gamma} \ar[r]_{G_1}^*!/ur6pt/{\labelstyle{\theta_1} \!\! \objectstyle\Uparrow } & \cat{D}
}}
\quad = \quad
\vcenter{
\xymatrix
@C=50pt
@R=50pt{
& \cat{B}\ar[d]^N_*!/dl10pt/{\labelstyle{\theta_2} \!\! \objectstyle\Uparrow } \\
\cat{C}\rtwocell^{G_2}_{G_1}{^\beta}\ar[ur]^{H_2} & \cat{D}.
}}
\]
\end{lem}
\begin{proof}
Let us define $\gamma$ component-wise. Given $c \in \cat{C}$, there is a unique $a \in \cat{A}$ such that $c = E(a)$, and then $H_i c = F_i a$. Thus $\alpha_a$ gives a map $H_1 c \to H_2 c$; we define $\gamma_c = \alpha_a$. We must check that this does define a natural transformation $\gamma \from H_1 \to H_2$.

Let $ f \from c \to c'$ in $\cat{C}$, and let $a, a' \in \cat{A}$ such that $Ea = c$ and $Ea' = c'$. We wish to show that
\[
\xymatrix{
(H_1 E a = F_1 a) \ar[r]^{\alpha_a} \ar[d]_{H_1 f} & (H_2 Ea = F_2 a)\ar[d]^{H_2 f} \\
(H_1 E a' = F_1 a') \ar[r]^{\alpha_{a'}} & (H_2 Ea' = F_2 a')
}
\]
commutes. Since $N$ is full and faithful, it is sufficient to check that this square commutes after applying $N$ to it. Consider the following cube:
\[
\xymatrix{
& G_1 E a\ar[rr]^{\beta_{Ea}}\ar[dl]_{(\phi_1)_a}\ar'[d][dd]_{G_1 f} & & G_2 Ea\ar[dl]^{(\phi_2)_a}\ar[dd]^{G_2 f} \\
NH_1\ar[rr]^(.65){N\alpha_a}\ar[dd]_{NH_1 f} E a & & NH_2 E a\ar[dd]_(.3){NH_2 f} & \\
& G_1 E a'\ar'[r][rr]^(.4){\beta_{Ea'}}\ar[dl]_{ (\phi_1)_{a'}} & & G_2 E a'\ar[dl]^{(\phi_2)_{a'}} \\
NH_1 E a'\ar[rr]_{N\alpha_{a'}} & & NH_2 E a'. & 
}
\]
The back square commutes by naturality of $\beta$, and the top and bottom squares commute by Equation~\bref{eq:enhanced-compat}. Recall that $ \phi_i = \theta_i E$ as part of what it means for $\theta_i$ to be a fill-in for $\phi_i$. Noting this, the left and right-hand squares commute by naturality of $\theta_1$ and $\theta_2$ respectively. Since all the morphisms from the back of the cube to the front are isomorphisms, it follows that the front face of the cube commutes, and this is precisely what was required to show that $\gamma$ is a natural transformation $H_1 \to H_2$.

It is clear from the definition that $\gamma$ is unique such that $\gamma E = \alpha$. Thus, all that remains is to establish the equality
\[
\vcenter{
\xymatrix
@C=50pt
@R=50pt{
& \cat{B}\ar[d]^N \\
\cat{C}\urtwocell^{H_2}_{H_1}{^\gamma} \ar[r]_{G_1}^*!/ur6pt/{\labelstyle{\theta_1} \!\! \objectstyle\Uparrow } & \cat{D}
}}
\quad = \quad
\vcenter{
\xymatrix
@C=50pt
@R=50pt{
& \cat{B}\ar[d]^N_*!/dl10pt/{\labelstyle{\theta_2} \!\! \objectstyle\Uparrow } \\
\cat{C}\rtwocell^{G_2}_{G_1}{^\beta}\ar[ur]^{H_2} & \cat{D}.
}}
\]
Since $E\from \cat{A} \to \cat{C}$ is bijective on objects, it is sufficient to show that these two natural transformations become equal when whiskered with $E$. But we have
\[
\vcenter{
\xymatrix
@C=50pt
@R=50pt{
\cat{A}\ar[r]^{F_2}\ar[d]_E& \cat{B}\ar[d]^N \\
\cat{C}\urtwocell^{H_2}_{H_1}{^\gamma} \ar[r]_{G_1}^*!/ur6pt/{\labelstyle{\theta_1} \!\! \objectstyle\Uparrow } & \cat{D}
}}
\quad = \quad
\vcenter{
\xymatrix
@C=50pt
@R=50pt{
\cat{A}\rtwocell^{F_2}_{F_1}{^\alpha} \ar[d]_E & \cat{B}\ar[d]^N \\
\cat{C}\ar[r]_{G_1}^*!/ur10pt/{\labelstyle{\theta_1} \!\! \objectstyle\Uparrow }\ar[ur]^{H_1} & \cat{D}
}}
\quad = \quad
\vcenter{
\xymatrix
@C=50pt
@R=50pt{
\cat{A}\rtwocell^{F_2}_{F_1}{^\alpha}\ar[d]_E\drtwocell\omit{^\phi_1} & \cat{B}\ar[d]^N \\
\cat{C}\ar[r]_{G_1} & \cat{D}
}}
\]
and
\[
\vcenter{
\xymatrix
@C=50pt
@R=50pt{
\cat{A}\ar[r]^{F_2}\ar[d]_E& \cat{B}\ar[d]^N_*!/dl10pt/{\labelstyle{\theta_2} \!\! \objectstyle\Uparrow } \\
\cat{C}\rtwocell^{G_2}_{G_1}{^\beta}\ar[ur]^{H_2} & \cat{D}
}}
\quad = \quad
\vcenter{
\xymatrix
@C=50pt
@R=50pt{
\cat{A}\ar[r]^{F_2}\ar[d]_E \drtwocell\omit{^\phi_2} & \cat{B}\ar[d]^N \\
\cat{B}\rtwocell^{G_2}_{G_1}{^\beta} & \cat{D},
}}
\]
so these two natural transformations are equal by Equation~\bref{eq:enhanced-compat}, as required.
\end{proof}

\section{Bijective-on-objects functors}
\label{sec:bo-properties-background}
Let $\cat{A}$ and $\cat{L}$ be large categories and suppose $L \from \cat{A} \to \cat{L}$ is a bijective-on-objects functor. We shall often consider functors of the form
\[
L^* \from [\cat{L},\cat{C}] \to [\cat{A},\cat{C}]
\]
where $\cat{C}$ is some other category. Such functors enjoy several useful properties. These properties are likely to be well-known but I was not able to find them in the literature.

\begin{defn}
\label{defn:creation-of-limits}
Let $D \from \cat{I} \to \cat{A}$ and $G \from \cat{A} \to \cat{C}$ be functors. We say that \demph{$G$ creates limits of $D$} if, for every limit cone $(\lambda_i \from c \to GDi)_{i \in \cat{I}}$ for $G \of D$, there is a unique cone $(\mu_i \from a \to Di)_{i \in \cat{I}}$ on $D$ such that $Ga = c$ and $G\mu_i = \lambda_i$ for each $i \in \cat{I}$, and this cone is a limit cone.

We say that \demph{$G$ creates limits of shape $\cat{I}$} if $G$ creates limits of $D \from \cat{I} \to \cat{A}$ for all such $D$.
\end{defn}
This is the definition of creation of limits from Section~V.1 of~\cite{maclane71}; note that this is somewhat stricter than the definition that is sometimes used by more recent authors.

Similarly, when we speak of monadic functors we mean this in the sense of~VI.7 of~\cite{maclane71}, rather than the slightly weaker sense that is commonly used by modern authors. Explicitly:

\begin{defn}
A functor $G \from \cat{C} \to \cat{B}$ is \demph{monadic} if it has a left adjoint and the canonical comparison functor from $\cat{C}$ to the category of algebras for the induced monad on $\cat{B}$ is an isomorphism of categories. We say that $G$ is \demph{weakly monadic} if it has a right adjoint and the comparison functor is an equivalence.
\end{defn}

\begin{lem}
\label{lem:bo-restrict-create-lims}
Let $\cat{C}$ and $\cat{I}$ be categories and suppose $\cat{C}$ has limits of shape $\cat{I}$. Let $L \from \cat{A} \to \cat{L}$ be a bijective-on-objects functor between large categories. Then $L^* \from [\cat{L},\cat{C}] \to [\cat{A},\cat{C}]$ creates limits of shape $\cat{I}$. Dually, if $\cat{C}$ has colimits of shape $\cat{I}$, then $L^*$ creates such colimits.
\end{lem}
\begin{proof}
Recall that if $\cat{C}$ has limits of shape $\cat{I}$, so does the functor category $[\cat{B},\cat{C}]$ for any category $\cat{B}$. Furthermore, given a limit cone $\lambda_b$ on
\[
\cat{I}\toby{D} [\cat{B}, \cat{C}] \toby{\ev_b} \cat{C}
\]
for each $b \in \cat{B}$, then there is a unique functor $X \from \cat{B} \to \cat{C}$ and a unique cone $\lambda$ on $D$ with vertex $X$ such that each $\ev_b$ sends $\lambda$ to $\lambda_b$, and furthermore this cone is a limit cone.

Let $D \from \cat{I} \to [\cat{L},\cat{C}]$, and suppose 
\[
(\mu_i \from Y \to L^*(D(i)))_{i \in \cat{I}}
\]
is a limit cone on $L^* \of D$. Then, for each $a \in \cat{A}$,
\[
((\mu_i)_a \from Ya \to D(i) (La))_{i \in \cat{I}}
\]
is a limit cone for $\ev_a \of L^* \of D \from \cat{I} \to \cat{C}$. But
\[
\xymatrix{
[\cat{L},\cat{C}]\ar[r]^{L^*}\ar[dr]_{\ev_{La}} & [\cat{A},\cat{C}]\ar[d]^{\ev_a} \\
& \cat{C}
}
\]
commutes, and so the $(\mu_i)_a$ also define limit cones on each $\ev_{La} \of D$. Since every object of $\cat{L}$ is of the form $La$ for a unique $a \in \cat{A}$, we therefore have a unique functor $X \from \cat{L} \to \cat{C}$ and a unique cone $\lambda$ on $D$ with vertex $X$ such that
\[
(\lambda_i)_{La} = (\mu_i)_a \from XLa = Ya \to D(i)(La)
\]
for each $a \in \cat{A}$ and $i \in \cat{I}$, and furthermore $\lambda$ is a limit cone. But then by construction, this cone is unique such that $L^*(\lambda) = \mu$, as required.
\end{proof}

This has the following immediate consequence.

\begin{cor}
\label{cor:rest-bo-monadic}
Let $L \from \cat{A} \to \cat{L}$ be a bijective-on-objects functor, and let $\cat{C}$ be a category with coequalisers. Then the functor
\[
L^* \from [\cat{L}, \cat{C}] \to [\cat{A}, \cat{C}]
\]
is monadic if and only if it has a left adjoint.
\end{cor}
\begin{proof}
The monadicity theorem (Theorem~1 in~VI.7 of~\cite{maclane71}) states that a functor is monadic if and only if it has a left adjoint and creates certain coequalisers. But by the above lemma, $L^*$ creates \emph{all} coequalisers.
\end{proof}

Thus functors of the form $L^*$ with $L$ bijective on objects are closely related to monadic functors. Indeed, even when they fail to have a left adjoint they have the following properties in common with monadic functors.

\begin{defn}
\label{defn:isofibration}
Let $U \from \cat{D} \to \cat{C}$ be a functor. We say that $U$ is an \demph{isofibration} if, for every $d \in \cat{D}$, $c \in \cat{C}$ and isomorphism $i \from Ud \to c$, there is an object $d'$ in $\cat{D}$ and isomorphism $j \from d \to d'$ such that $Ud' = c$ and $Uj = i$.
\end{defn}

\begin{defn}
Let $U \from \cat{D} \to \cat{C}$ be a functor. We say that $U$ is \demph{amnestic} if $U$ reflects identities in the following sense: an isomorphism in $\cat{D}$ is an identity if and only if it is sent to one by $U$.
\end{defn}

\begin{lem}
\label{lem:amnestic-isofib}
A functor $U \from \cat{D} \to \cat{C}$ is an amnestic isofibration if and only if, for every $d \in \cat{D}$, $c \in \cat{C}$ and isomorphism $i \from Ud \to c$, there is a \emph{unique} pair $(d', j)$ where $d' \in \cat{D}$ and $j \from d \to d'$ is an isomorphism such that $Ud' = c$ and $Uj = i$.
\end{lem}

\begin{proof}
Suppose $U$ has this property; then clearly $U$ is an isofibration. Suppose $j \from d \to d'$ is an isomorphism such that $Uj = \id_{Ud}$. Then by assumption $j$ is unique such, but $\id_d \from d \to d$ is another such isomorphism, so $d = d'$ and $j = \id_d$.

Conversely suppose $U$ is an amnestic isofibration, and let $d \in \cat{D}$, $c \in \cat{C}$ and $i \from Ud \to c$ be an isomorphism. By the isofibration property there is \emph{some} $d' \in \cat{D}$ and isomorphism $j \from d \to d'$ such that $Ud' = c$ and $uj = i$; let us show that they are unique. Suppose $d'' \in \cat{D}$ and $j' \from d \to d''$ such that $Ud'' = c$ and $Uj' = i$. Then $j' \of j^{-1} \from d' \to d''$ is an isomorphism and
\[
U(j' \of j^{-1}) = U(j') \of U(j)^{-1} = i \of i^{-1} = \id_c.
\]
Since $U$ is amnestic, it follows that $d'=d''$ and $j' \of j^{-1} = \id_{d'}$, so $j' = j$.
\end{proof}

\begin{lem}
\label{lem:bo-restrict-isofib}
Let $L \from \cat{A} \to \cat{L}$ be a bijective-on-objects functor. Then for any category $\cat{C}$, the functor
\[
L^* \from [\cat{L}, \cat{C}] \to [\cat{A},\cat{C}]
\]
is an amnestic isofibration.
\end{lem}
\begin{proof}
We will show that $L^*$ satisfies the condition in the previous lemma. Let $F \from \cat{L} \to \cat{C}$ and $G \from \cat{A} \to \cat{C}$ be functors and $\phi \from F \of L \to G$ be a natural isomorphism. We define a functor $G' \from \cat{L} \to \cat{C}$ as follows.

Given an object $La \in \cat{L}$ (every object of $\cat{L}$ is of this form for a unique $a \in \cat{A}$), define $G'(La) = Ga$. Given a morphism $l \from La \to La'$ in $\cat{L}$, define $G'l$ to be the composite
\[
G'(La) = Ga \toby{\phi_a^{-1}} FLa \toby{Fl} FLa' \toby{\phi_{a'}} Ga' = G'(La').
\]
This is clearly functorial, and defining $\theta_{La} = \phi_a$ for $a \in \cat{A}$ makes $\theta$ into a natural isomorphism $F \to G'$, and it is unique such that $L^*(\theta) = \phi$.
\end{proof}

Before continuing we pause to make note of the relationship between monadic and weakly monadic functors; this will be used in Section~\ref{sec:profinite-background} to show that the category of profinite groups is monadic over various categories.

\begin{lem}
\label{lem:monadic-amnestic-isofib}
A functor $U \from \cat{D} \to \cat{C}$ is monadic if and only if it is a weakly monadic amnestic isofibration.
\end{lem}
\begin{proof}
The following argument appears at~\cite{nlabMonadic}. Suppose $U$ has a left adjoint, inducing a monad $\mnd{T}$ on $\cat{C}$, and $K \from \cat{D} \to \cat{C}^{\mnd{T}}$ is  the comparison functor. It is clear that $U^{\mnd{T}} \from \cat{C}^{\mnd{T}} \to \cat{C}$ is an amnestic isofibration, as is any isomorphism. Since amnestic isofibrations are closed under composition and $U = U^\mnd{T} \of K$, it follows that if $U$ is monadic, it is an amnestic isofibration.

Conversely suppose $U$ is a weakly monadic amnestic isofibration. The facts that $U= U^{\mnd{T}} \of K$ and $U^{\mnd{T}}$ is an amnestic isofibration implies that $K$ is also an amnestic isofibration. Since $K$ is an equivalence it is essentially surjective on objects, but then the fact that it is an isofibration implies that it is \emph{actually} surjective on objects. Meanwhile, the fact that $K$ is full and faithful and amnestic implies that it must be injective on objects. Thus it is full and faithful and bijective on objects, so it is an isomorphism.
\end{proof}

\section{Density and codensity}
\label{sec:dense-codense-background}
The notions of density and codensity were introduced by Isbell in~\cite{isbell60} under the names left adequacy and right adequacy respectively. We will make use of both of these notions; density when discussing monads with arities in Sections~\ref{sec:notions-monads-arites} and~\ref{sec:canonical1-monads-arities}, and codensity in its relation to codensity monads. Throughout this section let $\cat{A}$ and $\cat{B}$ be locally large categories and let $F \from \cat{A} \to \cat{B}$ be a functor.

\begin{defn}
We define the \demph{nerve functor} of $F$ to be the composite
\[
N_F \from \cat{B} \incl [\cat{B}^{\op}, \SET] \toby{(F^{\op})^*} [\cat{A}^{\op},\SET],
\]
where the first factor is the Yoneda embedding. Dually, the \demph{conerve functor} of $F$ is the composite
\[
N^F \from \cat{B} \incl [ \cat{B}, \SET]^{\op} \toby{(F^*)^{\op}} [\cat{A}, \SET]^{\op}.
\]
\end{defn}

\begin{defn}
\label{defn:dense-codense}
We say that $F$ is \demph{dense} if $N_F$ is full and faithful, and that $F$ is \demph{codense} if $N^F$ is full and faithful.
\end{defn}

An important special case is when $\cat{A}$ is a full subcategory of $\cat{B}$ and $F$ is the inclusion; in this situation we call $\cat{A}$ a dense (respectively codense) subcategory of $\cat{B}$. In particular, dense subcategories are always assumed to be full.

Density of $F$ is equivalent to the condition that every object of $\cat{B}$ is canonically a colimit of objects in $\cat{A}$, in a sense that we now make precise.

\begin{defn}
For every object $b \in \cat{B}$ there is a canonical functor
\[
(F \downarrow b) \to \cat{A} \toby{F} \cat{B}
\]
where the functor $(F \downarrow b) \to \cat{A}$ is the evident forgetful functor. There is a canonical cocone on this diagram with vertex $b$, and whose component at $(f \from Fa \to b) \in (F\downarrow b)$ is $f$ itself. We call this the \demph{$F$-cocone on $b$}.

Dually, there is a functor
\[
(b \downarrow F) \to \cat{A} \toby{F} \cat{B}
\]
and a canonical cone on this diagram with vertex $b$, which we call the \demph{$F$-cone on $b$}.
\end{defn}

\begin{lem}
The following are equivalent:
\begin{enumerate}
\item the functor $F$ is dense;
\item for every $b \in \cat{B}$, the $F$-cocone on $\cat{B}$ is a colimit cocone; and
\item the identity functor $\cat{B} \to \cat{B}$ is the pointwise left Kan extension of $F$ along itself.
\end{enumerate}
Dually, the following are equivalent:
\begin{enumerate}
\item the functor $F$ is codense;
\item for every $b \in \cat{B}$, the $F$-cone on $\cat{B}$ is a limit cone; and
\item the identity functor $\cat{B} \to \cat{B}$ is the pointwise right Kan extension of $F$ along itself.
\end{enumerate}
\end{lem}
\begin{proof}
This is well-known; see for example Propositions~1 and~2 in X.6 of~\cite{maclane71}.
\end{proof}
The third of these conditions makes it clear that there is a connection between density and codensity and Kan extensions. In particular we can use the left and right Kan extensions of a functor along itself to measure the failure of a functor to be dense or codense. It turns out that these Kan extensions naturally come equipped with the structure of a comonad or monad respectively; this was observed by Kock in~\cite{kock66}.

\begin{defn}
\label{defn:codensity-monad}
Let $T \from \cat{B} \to \cat{B}$ be a functor and $\kappa \from T \of F \to F$ be a natural transformation exhibiting $T$ as the right Kan extension of $F$ along itself. We define natural transformations $\eta \from \id_{\cat{B}} \to T$ and $\mu \from T \of T \to T$, using the universal property of Kan extensions, to be the unique natural transformations such that we have
\[
\vcenter{
\xymatrix@=50pt{
{\cat{A}}\ar[r]^{F}\ar[dr]_F\druppertwocell\omit{<-3.3>\kappa} & {\cat{B}}\dtwocell^{<2>\id_{\cat{B}}}_{<1.2>T}{\eta} \\
& {\cat{B}}
}}
\quad = \quad
\vcenter{
\xymatrix@=50pt{
{\cat{A}}\ar[r]^F\druppertwocell\omit{=<-3.3>}\ar[dr]_F & {\cat{B}}\ar@{=}[d] \\
& {\cat{B}}
}}
\]
and
\[
\vcenter{
\xymatrix@C=20pt@R=40pt{
{\cat{A}}\ar[ddrr]_F\ar[rr]^F\ddrruppertwocell\omit{<-3>\kappa}  & & {\cat{B}}\ar[dd]_{T}\ar[dr]^{T}\dduppertwocell\omit{<-3>\mu} &\\
 & & & {\cat{B}}\ar[dl]^{T} \\
 & &{\cat{B}} &
}}
\quad = \quad
\vcenter{
\xymatrix@C=70pt@R=40pt{
{\cat{A}}\ar[r]^F\ar[dr]_F\ar[ddr]_F\druppertwocell\omit{<-2.5> \; \kappa}\ddruppertwocell\omit{<-2.5>\kappa} & {\cat{B}}\ar[d]^{T} \\
& {\cat{B}}\ar[d]^{T} \\
&{\cat{B}.}
}}
\]
It is straightforward to check that $(T, \eta,\mu)$ is a monad, and we call it the \demph{codensity monad} of $F$. If the Kan extension $T$ is a pointwise Kan extension (as defined in e.g.\ Definition~1.3.4 of~\cite{riehl14}), then we call $(T, \eta,\mu)$ the \demph{pointwise codensity monad} of $F$. Dually there is a notion of a \demph{(pointwise) density comonad}.
\end{defn}
Thus a functor is codense if and only if its codensity monad is trivial.

\begin{defn}
\label{defn:codensity-monad-comparison}
Let $\mnd{T} = (T, \eta, \mu)$ be the codensity monad of $F$ with $\kappa \from T \of F \to F$ the natural transformation making $T$ the right Kan extension of $F$ along itself. Then for each $a \in \cat{A}$ the map
\[
\kappa_a \from TFa \to Fa
\]
makes $Fa$ into a $\mnd{T}$-algebra, and if $f \from a \to a'$ in $\cat{A}$, then $Ff$ is a $\mnd{T}$-algebra homomorphism $(Fa, \kappa_a) \to (Fa', \kappa_{a'})$. Thus the assignments $a \mapsto (Fa, \kappa_a)$ and $f \mapsto Ff$ define a functor $K \from \cat{A} \to \cat{B}^{\mnd{T}}$ such that
\[
\xymatrix{
\cat{A}\ar[r]^K\ar[dr]_{F} & \cat{B}^{\mnd{T}}\ar[d]^{U^{\mnd{T}}} \\
& \cat{B}
}
\]
commutes, where $U^{\mnd{T}} \from \cat{B}^{\mnd{T}} \to \cat{B}$ is the forgetful functor from the category of $\mnd{T}$-algebras. We call $K$ the \demph{canonical comparison functor} for $F$.
\end{defn}

We now record here some lemmas that will aid us in identifying codensity monads.

\begin{lem}
\label{lem:codensity-iso-kleisli}
Let $U \from \cat{A} \to \cat{C}$ and $G \from \cat{D} \to \cat{C}$ be functors and suppose $G$ has a left adjoint $F$. Then the monad induced by the adjunction $F \dashv G$ is the pointwise codensity monad of $U$ if and only if for each $c, c' \in \cat{C}$ there is a bijection
\begin{equation}
\label{eq:codensity-iso-kleisli}
\Phi \from \cat{D}(Fc, Fc') \to [\cat{A},\Set](\cat{C}(c', U-), \cat{C}(c,U-))
\end{equation}
such that
\begin{enumerate}
\item if $f \from Fc \to Fc'$ and $f' \from Fc' \to Fc''$ then $\Phi(f' \of f) = \Phi (f) \of \Phi(f')$ and
\item if $g \from c \to c'$ then $\Phi(Fg) = g^* \from \cat{C}(c', U-) \to \cat{C}(c,U-)$.
\end{enumerate}
\end{lem}
\begin{proof}
The pointwise codensity monad $\mnd{T} = (T, \eta, \mu)$ of $U$ exists if and only if for each $c'$, the canonical functor
\[
(c' \downarrow U) \to \cat{A} \toby{U} \cat{C}
\]
has a limit. Unpacking the definition of a cone on this diagram, this is equivalent to the existence of an object $Tc'$ such that natural transformations $\cat{C}(c',U-) \to \cat{C}(c,U-)$ correspond to morphisms $c \to Tc'$, naturally in $c \in \cat{C}$. But if there is a correspondence as in Equation~\bref{eq:codensity-iso-kleisli}, then $GFc'$ is such an object, so the pointwise codensity monad exists.

But now the category with the same objects as $\cat{C}$, and whose morphisms are natural transformations $\cat{C}(c', U-) \to \cat{C}(c, U-)$ is precisely the Kleisli category of the codensity monad. And the category with the same objects as $\cat{C}$ and whose morphisms are morphisms $Fc \to Fc'$ is the Kleisli category of the monad induced by $F \ladj G$. But if two monads have isomorphic Kleisli categories, and the isomorphism is compatible with their respective free functors, then the monads are isomorphic.

Conversely, if the pointwise codensity monad exists and is isomorphic to the monad induced by $F \dashv G$, then the two monads have isomorphic Kleisli categories, yielding the required functorial correspondence.
\end{proof}

\begin{lem}
\label{lem:codense-radj-monad}
Let $I \from \cat{A} \to \cat{D}$ be a codense functor, and let $G \from \cat{D} \to \cat{C}$ be a functor with a left adjoint $F$. Then the pointwise codensity monad of $G \of I$ exists and is isomorphic to the monad induced by $F \ladj G$.
\end{lem}
\begin{proof}
It is sufficient to establish a bijection
\[
[\cat{A},\SET](\cat{C}(c', G \of I - ), \cat{C}(c, G \of I-)) \iso \cat{D}(Fc, Fc')
\]
satisfying the conditions of Lemma~\ref{lem:codensity-iso-kleisli}. But we have
\begin{align*}
[\cat{A},\SET](\cat{C}(c', G \of I - ), \cat{C}(c, G \of I-)) & \iso [\cat{A},\SET](\cat{C}(Fc',  I - ), \cat{C}(Fc,  I-)) && \text{(since $F \ladj G$)}\\
&\iso \cat{D}(Fc, Fc') && \text{(since $I$ is codense)},
\end{align*}
and this bijection is compatible with composition. Furthermore, tracing $f^*$ through this sequence of bijections gives $Ff$.
\end{proof}

\section{Idempotent adjunctions and monads}
\label{sec:idem-background}

In this section we review what it means for an adjunction or monad to be idempotent, and some of the consequences of these properties.

\begin{lem}
Let $\mnd{T} = (T, \eta, \mu)$ be a monad on a category $\cat{C}$, and let $U^{\mnd{T}}\from \cat{C}^{\mnd{T}} \to \cat{C}$ be the forgetful functor from the category of $\mnd{T}$-algebras. Then the following are equivalent:
\begin{enumerate}
\item the functor $U^{\mnd{T}} \from \cat{C}^{\mnd{T}} \to \cat{C}$ is full and faithful;
\item the natural transformation $\mu \from TT \to T$ is an isomorphism; and
\item for every $\mnd{T}$-algebra $(a, \alpha)$, the map $\alpha \from Ta \to a$ is an isomorphism.
\end{enumerate}
\end{lem}
\begin{proof}
See Proposition~4.2.3 in~\cite{borceux94v2}.
\end{proof}

\begin{defn}
A monad satisfying the conditions of the previous lemma is called an \demph{idempotent monad}.
\end{defn}

Recall the following definitions.

\begin{defn}
Let $\cat{C}$ be a category and $\cat{A}$ a full subcategory of $\cat{C}$. We say that $\cat{A}$ is \demph{replete} in $\cat{C}$ if, whenever we have an isomorphism $a \iso c$ in $\cat{C}$ where $a \in \cat{A}$ then $c \in \cat{A}$. We say that $\cat{A}$ is \demph{reflective} in $\cat{C}$ if the inclusion $\cat{A} \incl \cat{C}$ has a left adjoint.
\end{defn}

\begin{prop}
Let $\cat{C}$ be a category. There is a bijective correspondence between idempotent monads on $\cat{C}$ and reflective, replete subcategories of $\cat{C}$. This correspondence sends an idempotent monad to its category of algebras, and sends a reflective subcategory to the monad induced by the reflection.
\end{prop}
\begin{proof}
See Corollary~4.2.4 of~\cite{borceux94v2}.
\end{proof}

Let us now consider a type of adjunction that is closely related to the notion of an idempotent monad.

\begin{lem}
\label{lem:idem-adj-conditions}
Let $F \from \cat{C} \to \cat{D}$ be a functor with right adjoint $G$, with unit $\eta$ and counit $\epsilon$. Then the following conditions are equivalent:
\begin{enumerate}
\item $F\eta$ is an isomorphism;
\item $\epsilon F$ is an isomorphism;
\item $G\epsilon F$ is an isomorphism, that is, the monad induced by the adjunction is idempotent;
\item $GF \eta = \eta GF$;
\item $G\epsilon$ is an isomorphism;
\item \label{part:idem-adj-etaG} $\eta G$ is an isomorphism;
\item $F \eta G$ is an isomorphism, that is, the comonad induced by the adjunction is idempotent;
\item $FG \epsilon = \epsilon FG$; and
\item $FG\epsilon F = \epsilon FGF$.
\end{enumerate}
\end{lem}
\begin{proof}
This is well-known; see for example~3.4 in~\cite{clarkWisbauer11}.
\end{proof}
An adjunction satisfying the conditions of the above lemma is called an \demph{idempotent adjunction.}

\begin{lem}
\label{lem:idem-adj-factor}
Let $F \from \cat{C} \to \cat{D}$ be a functor with a right adjoint $G$ with unit $\eta$ and counit $\epsilon$, and suppose the adjunction $F \ladj G$ is idempotent. Define $\cat{A}$ to be the full subcategory of $\cat{C}$ on those objects $c$ for which $\eta_c \from c \to GFc$ is an isomorphism. 
Write:
\begin{itemize}
\item $R \from \cat{A} \incl \cat{C}$ for the inclusion;
\item $F' \from \cat{A} \to \cat{D}$ for $F \of R$;
\item $G' \from \cat{D} \to \cat{A}$ for the factorisation of $G$ through $\cat{A}$, which exists by Lemma~\ref{lem:idem-adj-conditions}.\bref{part:idem-adj-etaG}; and
\item $L \from \cat{C} \to \cat{A}$ for the composite $G' \of F$.
\end{itemize}
Then we have $L \ladj R$ and $F' \ladj G'$, with $R$ and $F'$ full and faithful, and the adjunction $F \ladj G$ is isomorphic to the composite adjunction
\[
\xymatrix{
{\cat{C}}\ar@<5pt>[r]_-{\perp}^-{L} & \cat{A}\ar@<5pt>[l]^-{R} \ar@<5pt>[r]_-{\perp}^-{F'} & \cat{D}.\ar@<5pt>[l]^-{G'}
}
\]
Furthermore $\cat{A}$ is replete in $\cat{C}$, and can be identified up to isomorphism with the category of algebras for the monad induced by $F \ladj G$.
\end{lem}
\begin{proof}
See~3.6 of~\cite{clarkWisbauer11}.
\end{proof}

Finally we note some of the consequences of a codensity monad being idempotent. This result may already be known, but I am not aware of it in the literature.

\begin{lem}
\label{lem:cm-idem-lem}
Let $U \from \cat{M} \to \cat{B}$ be a functor with a codensity monad $\mnd{T}$. Suppose that $\mnd{T}$ is idempotent, so that its category of algebras $\cat{B}^{\mnd{T}}$ can be identified with a reflective, replete subcategory of $\cat{B}$. Then
\begin{enumerate}
\item
\label{part:cm-idem-lem-1}
the canonical comparison functor $K \from \cat{M} \to \cat{B}^{\mnd{T}}$ is codense;
\item
\label{part:cm-idem-lem-2}
the full subcategory $\cat{B}^{\mnd{T}} \incl \cat{B}$ consists precisely of those objects of $\cat{B}$ of the form $\lim_{i \in \cat{I}} UDi$, where $D \from \cat{I} \to \cat{M}$ is a diagram in $\cat{M}$; and
\item
\label{part:cm-idem-lem-3}
the full subcategory $\cat{B}^{\mnd{T}} \incl \cat{B}$ is the smallest reflective, replete subcategory of $\cat{B}$ through which $U$ factors.
\end{enumerate}
\end{lem}

\begin{proof}
\begin{enumerate}
\item
Let $b \in \cat{B}^{\mnd{T}} \subseteq \cat{B}$. To show that $K$ is codense, we must show that $b$ is the limit of the canonical diagram
\[
(b \downarrow K) \to \cat{M} \toby{K} \cat{B}^{\mnd{T}}.
\]
But since $\cat{B}^{\mnd{T}}$ is reflective in $\cat{B}$, limits computed in $\cat{B}^{\mnd{T}}$ coincide with limits computed in $\cat{B}$, and the limit of this diagram in $\cat{B}$ is by definition $Tb$. But since $\mnd{T}$ is idempotent and $b \in \cat{B}^{\mnd{T}} \subseteq \cat{B}$, we have $b \iso Tb$, and so $b$ is a limit of this diagram.
\item
As a reflective, replete subcategory, $\cat{B}^{\mnd{T}}$ is closed under limits in $\cat{B}$ and so contains every object of this form. On the other hand, since $K$ is codense, every object $b$ of $\cat{B}^{\mnd{T}}$ is the limit of the diagram
\[
(b \downarrow K) \to \cat{M} \toby{K} \cat{B}^{\mnd{T}}.
\]
which is of the form described.
\item
Any reflective, replete subcategory must be closed under limits, and so by~\bref{part:cm-idem-lem-2}, if $U$ factors through such a subcategory then it must contain $\cat{B}^{\mnd{T}}$. Thus, since $\cat{B}^{\mnd{T}}$ is itself reflective and replete, it is the smallest such. \qedhere
\end{enumerate}
\end{proof}

\section{Profinite groups}
\label{sec:profinite-background}
In Chapters~\ref{chap:canonical2},~\ref{chap:topology} and~\ref{chap:complete}, we will develop an analogy between the notion of algebraic theory developed in this thesis and some aspects of group theory. As part of this comparison, we will make frequent reference to \emph{profinite groups} and we take the opportunity here to collect some basic definitions and results concerning these.

In particular, there are many ways of characterising the category of profinite groups up to equivalence. Although these are well-known, I could not find a comprehensive list of these characterisations.

\begin{defn}
\label{defn:profinite-group}
A \demph{profinite group} is a small topological group that can be written as a small limit of finite discrete groups in the category of small topological groups. We write $\profGp$ for the full subcategory of $\TopGp$ consisting of the profinite groups.
\end{defn}

\begin{prop}
A small topological group is profinite if and only if it is compact, Hausdorff and totally disconnected.
\end{prop}
\begin{proof}
See for example Corollary~1.2.4 in~\cite{wilson98}.
\end{proof}

\begin{prop}
The codensity monad of the functor $\FinGp \incl \TopGp$ that sends a finite group to the corresponding discrete group is idempotent.
\end{prop}
\begin{proof}
This was proved by Deleanu in Theorem~3.1 of~\cite{deleanu83}.
\end{proof}

\begin{prop}
The category of algebras for the codensity monad of $\FinGp \incl \TopGp$ is $\profGp$, and the forgetful functor to $\TopGp$ is the usual inclusion. Furthermore, $\FinGp$ is codense in $\profGp$, and $\profGp$ is the smallest reflective subcategory of $\TopGp$ containing $\FinGp$.
\end{prop}
\begin{proof}
By the previous proposition, this codensity monad is idempotent, and so its category of algebras can be identified with the closure of $\FinGp \incl \TopGp$ under small limits by Lemma~\ref{lem:cm-idem-lem}.\bref{part:cm-idem-lem-2}, but this is precisely the definition of $\profGp$. The other assertions then follow from Lemma~\ref{lem:cm-idem-lem}.\bref{part:cm-idem-lem-1} and Lemma~\ref{lem:cm-idem-lem}.\bref{part:cm-idem-lem-3}.
\end{proof}

We will show that $\profGp$ is monadic over $\Set$, $\Top$ and $\Gp$ by applying a result of Gildenhuys and Kennison, namely Theorem~3.1 from~\cite{kennisonGildenhuys71}, which we will restate for convenience. First however, we recall some definitions.

\begin{defn}
Let $\mnd{T}$ be a monad on $\Set$. Then a \demph{Birkhoff subcategory} of $\Set^{\mnd{T}}$ is a full subcategory closed under products, subalgebras and homomorphic images.
\end{defn}

The celebrated Birkhoff Variety Theorem states that every Birkhoff subcategory of the category of algebras for a finitary algebraic theory is itself the category of algebras for a finitary algebraic theory. However, we shall not need this result, but only the fact that any Birkhoff subcategory is in particular a reflective subcategory, which can be seen by a routine application of the General Adjoint Functor Theorem.

The following definition is from Section~2 of Gildenhuys and Kennison~\cite{kennisonGildenhuys71}.
\begin{defn}
Let $U \from \scat{M} \to \Set$ with $\scat{M}$ small, and let $\mnd{T}$ be the codensity monad of $U$. Then the \demph{category of $\scat{M}$-objects}, denoted $\scat{M}\obj$, is defined to be the smallest full subcategory of $\Set^{\mnd{T}}$ through which the comparison functor $K \from \scat{M} \to \Set^{\mnd{T}}$ factors and which is closed under small limits.
\end{defn}

The following notion is defined in Section~1 of~\cite{kennisonGildenhuys71}, under the name ``separating triple'' rather than ``separating monad''.
\begin{defn}
Let $U \from \scat{M} \to \Set$ be a functor with $\scat{M}$ small. A \demph{separating monad} for $U$ consists of a monad $\mnd{T}_0$ and a full and faithful functor $S \from \scat{M} \to \Set^{\mnd{T}_0}$ whose image is closed under the formation of subalgebras, such that
\[
\xymatrix{
\scat{M}\ar[r]^S\ar[dr]_U & \Set^{\mnd{T}_0}\ar[d]^{U^{\mnd{T}_0}} \\
& \Set
}
\]
commutes.
\end{defn}

\begin{defn}
We say that a monad $\mnd{T}$ on $\Set$ \demph{admits a group operation} if there exists a morphism of monads from the free group monad on $\Set$ to $\mnd{T}$.
\end{defn}

\begin{prop}
\label{prop:kennison-gildenhuys}
Let $U \from \scat{M} \to \Set$ with $\scat{M}$ small, let $\mnd{T}$ be the codensity monad of $U$, and let $\tilde{\mnd{T}}$ be the codensity monad of the composite
\[
\scat{M}\toby{U} \Set \toby{D} \Top
\]
where $D$ is the discrete space functor. Let $\mnd{T}_0$ be a finitary separating monad for $U$ such that $\mnd{T}_0$ admits a group operation. Suppose $\scat{M}$ has and $U$ preserves finite products, and suppose $U$ takes values in the finite sets. Then $\Set^{\mnd{T}}$  can be identified with the smallest Birkhoff subcategory of the category of compact Hausdorff $\mnd{T}_0$-algebras through which the comparison functor from $\scat{M}$ factors. Furthermore, we have equivalences
\[
\scat{M}\obj \eqv \Top^{\tilde{\mnd{T}}} \eqv \Set^{\mnd{T}},
\]
compatible with the forgetful functors to $\Set$.
\end{prop}
Note that speaking of Birkhoff subcategories of the category of compact Hausdorff $\mnd{T}_0$-algebras does make sense, because this category is monadic over $\Set$; this is Proposition~7.1 in Manes~\cite{manes69}.

\begin{proof}
This is Theorem~3.1 in~\cite{kennisonGildenhuys71}.
\end{proof}

\begin{prop}
\label{prop:profinite-codensity-set-top-gp}
The category $\profGp$ of profinite groups is monadic over $\Set$, $\Top$ and $\Gp$. Furthermore, in each case the corresponding monad is the codensity monad of the forgetful functor from $\FinGp$ to each of these categories.
\end{prop}
\begin{proof}
It is clear that the forgetful functors to each of these categories are amnestic isofibrations, therefore it is enough to show that these functors are \emph{weakly} monadic by Lemma~\ref{lem:monadic-amnestic-isofib}.

Consider Proposition~\ref{prop:kennison-gildenhuys} with $\scat{M}$ being $\FinGp$ and $\mnd{T}_0$ the free group monad (which certainly admits a group operation and is finitary), and write $\mnd{T}$ for the codensity monad of $\FinGp$ over $\Set$. The proposition then tells us that $\Set^{\mnd{T}}$ can be identified with a Birkhoff subcategory, and thus a reflective subcategory, of the category of compact Hausdorff groups. Limits in the category of compact Hausdorff groups are computed as in the category of topological groups, and hence so are limits in $\Set^{\mnd{T}}$.

But in addition, we have $\FinGp\obj \eqv \Set^{\mnd{T}}$, that is every object of $\Set^{\mnd{T}}$ is a small limit (in $\Set^{\mnd{T}}$) of finite groups. But $\Set^{\mnd{T}}$ is complete and limits are computed as in the category of topological groups as noted above. Thus we can identify $\Set^{\mnd{T}}$ up to equivalence with the category of topological groups that are limits of finite discrete groups, which by Definition~\ref{defn:profinite-group} is exactly $\profGp$. Thus $\profGp$ is weakly monadic over $\Set$.

Again by Proposition~\ref{prop:kennison-gildenhuys}, we have
\[
\Set^{\mnd{T}} \eqv \Top^{\tilde{\mnd{T}}}
\]
where $\tilde{\mnd{T}}$ is the codensity monad of the forgetful functor from the category of finite discrete groups to $\Top$. In particular the latter is also equivalent to $\profGp$ in a way compatible with the forgetful functors, so $\profGp$ is weakly monadic over $\Top$.

To see that $\profGp$ is monadic over $\Gp$ we use a standard argument, applying the monadicity theorem (Theorem~1 in~VI.7 of~\cite{maclane71}). The forgetful functor $V \from \profGp \to \Gp$ has a left adjoint by a standard application of the General Adjoint Functor Theorem. To see that $V$ creates coequalisers of $V$-split pairs; consider the commuting diagram of forgetful functors
\[
\xymatrix{
\profGp \ar[rr]^V\ar[dr]_{V'} && \Gp\ar[dl]^{V''}. \\
& \Set &
}
\]
Let $f,g \from G \to H$ be a $V$-split pair. Then $f,g$ are also $V'$ split, and so their coequaliser is created by $V'$ since $V'$ is monadic. Thus $f$ and $g$ have a coequaliser $e$ in $\profGp$ and it is preserved by $V'$. Thus $Ve \of Vf = Ve \of Vg$ ,and $V'' Ve$ is the coequaliser of $V'' V f$ and $V'' Vg$, so since $V''$ is monadic, it follows that $Ve$ is the coequaliser of $Vf$ and $Vg$; that is, the equaliser of $f$ and $g$ is preserved by $V$.

Furthermore, if $e'$ is some morphism in $\profGp$ such that $e' \of f = e' \of g$ and $Ve'$ is a coequaliser of $Vf$ and $Vg$, then $V'e' = V'' Ve'$ is a coequaliser of $V'f$ and $V'g$ (since the coequaliser of $Vf$ and $Vg$ is split and so is absolute). Thus, since $V'$ creates coequalisers of $V'$-split pairs, $e'$ must be a coequaliser of $f$ and $g$. Thus $V$ creates coequalisers of $V$-split pairs, so is monadic. The fact that the corresponding monad on $\Gp$ is the codensity monad of $\FinGp \incl \Gp$ follows from the fact that $\FinGp$ is codense in $\profGp$ and Lemma~\ref{lem:codense-radj-monad}.
\end{proof}

\begin{remark}
\label{rem:prof-gp-chars}
To summarise the results of this section, the category of profinite groups can be characterised up to isomorphism in the following equivalent ways:
\begin{enumerate}
\item the full subcategory of $\TopGp$ consisting of the small limits of finite discrete groups;
\item \label{part:prof-gp-chars-chaus-td} the full subcategory of $\TopGp$ consisting of the compact Hausdorff, totally disconnected groups;
\item the smallest replete reflective subcategory of $\TopGp$ containing the finite discrete groups;
\item the category of algebras for the codensity monad of $\FinGp \incl \TopGp$;
\item the category of algebras for the codensity monad of $\FinGp \incl \Gp$;
\item the category of algebras for the codensity monad of the forgetful functor $\FinGp \to \Top$; and
\item the category of algebras for the codensity monad of the forgetful functor $\FinGp \to \Set$.
\end{enumerate}
In fact there are many variants of~\bref{part:prof-gp-chars-chaus-td}, characterising profinite groups in terms of other properties of their underlying topological spaces. These characterisations will not be relevant for our purposes so we omit them here; for a full account see Section II.4 of~\cite{johnstone82}.
\end{remark}

\chapter{Notions of algebraic theory}
\label{chap:notions}

In this chapter we recall various notions of algebraic theory, and their associated structure--semantics adjunctions. One of the main goals of this thesis is to find a common generalisation of these, in order to identify which features unite them and allow them all to be regarded as notions of algebraic theory, and which features distinguish them from one another. A first approximation to this common generalisation will be defined in Chapter~\ref{chap:adjunction}, and this will be refined in Chapter~\ref{chap:structure}.

In Section~\ref{sec:notions-classical} we recall the classical notion of algebraic theory. We then discuss the various categorical generalisations of this notion that have been developed. Specifically we review Lawvere theories (Section~\ref{sec:notions-lawvere}), monads (Section~\ref{sec:notions-monads}), PROPs and PROs (Section~\ref{sec:notions-PROP}), operads (Section~\ref{sec:notions-operads}), monads with arities (Section~\ref{sec:notions-monads-arites}) and monoids (Section~\ref{sec:notions-monoids}). Finally in Section~\ref{sec:notions-compare} we summarise and compare all of these.

\section{Classical algebraic theories}
\label{sec:notions-classical}

First, let us recall the original, non-categorical definition of an algebraic theory. The following definitions are adapted from Chapter~1 of~\cite{johnstone87}, but a similar presentation can be found in Chapter~1 of~\cite{manes76}, or any textbook on universal algebra.

\begin{defn}
An \demph{operational type} consists of a set $\Omega$, whose elements we call operation symbols, together with, for each $\omega \in \Omega$, a natural number $\alpha (\omega)$ called the \demph{arity} of $\omega$. We write $\Omega_n$ for the subset of $\Omega$ consisting of all operation symbols with arity $n \in \nat$.
\end{defn}

\begin{defn}
Given a natural number $n$, we define the set $F_{\Omega} (n)$ of \demph{$\Omega$-terms in the variables $x_1, \ldots, x_n$} to be the smallest set such that:
\begin{enumerate}
\item for $i = 1, \ldots, n$, we have $x_i \in F_{\Omega} (n)$ (where $x_i$ is just thought of as an abstract variable), and
\item if $\omega \in \Omega_m$ and $t_1, \ldots, t_m \in F_{\Omega} (n)$ then $\omega ( t_1, \ldots , t_m) \in F_{\Omega}(n)$.
\end{enumerate}
If $t \in F_{\Omega}(n)$, we call $n$ the \demph{arity} of $t$.
\end{defn}

\begin{remark}
If $t \in F_{\Omega}(n)$ and $n \leq m$, then $t$ can also be regarded as an element of $F_{\Omega} (m)$, so the arity of $t$ is not in fact well defined. Strictly speaking then, we should define a term to be a pair $(n, t)$ where $n$ is a natural number and $t \in F_{\Omega}(n)$, and then define the arity of such a term to be $n$. However we usually omit explicit mention of $n$ for the sake of brevity, in a mild abuse of notation. Note however, that when we speak of the arity of a term $t$, we do \emph{not} necessarily mean the \emph{smallest} number $n$ such that $t \in F_{\Omega}(n)$.
\end{remark}

\begin{defn}
\label{defn:alg-theory}
An \demph{algebraic theory} consists of an operational type $\Omega$ together with a set $E$ of pairs of $\Omega$-terms where, within each pair, both terms have the same arity (but different pairs may have different arities). 
\end{defn}

One should think of the pairs in $E$ as ``formal equations'' between $\Omega$-terms --- these are the axioms of the algebraic theory. We now turn to semantics.

\begin{defn}
Let $\Omega$ be an operational type. Then an \demph{$\Omega$-structure} consists of a set $A$ together with, for each $\omega \in \Omega_n$, a function
\[
[\omega]_A \from A^n \to A
\]
called the \demph{interpretation} of $\omega$ in $A$.

If $A$ and $B$ are $\Omega$-structures, a function $h \from A \to B$ is called a $\Omega$-structure homomorphism if, for every $\omega \in \Omega_n$, the diagram
\[
\xymatrix{
A^n \ar[r]^{h^n}\ar[d]_{[\omega]_A} & B^n\ar[d]^{[\omega]_B} \\
A \ar[r]_{h} & B
}
\]
commutes.
\end{defn}

In order to say what it means for a structure to satisfy certain equations, we must extend the interpretations of operation symbols to interpretations of terms.

\begin{defn}
Let $A$ be an $\Omega$-structure, for an operational type $\Omega$. The interpretation of a term $t \in F_{\Omega} (n)$ in $A$ is a function
\[
[t]_A \from A^n \to A
\]
defined recursively as follows:
\begin{enumerate}
\item if $t = x_i$ for some $i$, then $[t]_A = \pi_i \from A^n \to A$, the product projection onto the $i$-th factor; and
\item if $t = \omega (t_1, \ldots, t_m)$ for some $\omega \in \Omega_m$ and $t_i \in F_{\Omega} (n)$, then $[t]_A$ is defined to be the composite
\[
A^n \toby{\Delta} (A^n)^m \toby{ [t_1]_A \times \cdots \times [t_m]_A} A^m \toby{[\omega]_A} A,
\]
where $\Delta \from A^n \to (A^n)^m$ is the diagonal map.
\end{enumerate}
\end{defn}

\begin{defn}
Let $(\Omega, E)$ be an algebraic theory, and $A$ an $\Omega$-structure. Then $A$ is a \demph{model} of, or \demph{algebra} for $(\Omega, E)$ if, for each $(s, t) \in E$, with $s, t \in F_{\Omega}(n)$, we have an equality of functions
\[
[s]_A = [t]_A \from  A^n \to A.
\]
A \demph{$(\Omega,E)$-model homomorphism} between $(\Omega, E)$-models is an $\Omega$-structure homomorphism between the underlying $\Omega$-structures.
\end{defn}

\begin{ex}
\label{ex:theory-groups}
Define an operational type $\Omega$ with
\[
\Omega_0 = \{ e \},\quad \Omega_1 = \{i \},\quad \Omega_2 = \{ m \},
\]
and $\Omega_n = \emptyset$ for all other values of $n$. Define
\begin{align*}
E = \{ & (m(x_1, m(x_2,x_3)), m(m(x_1,x_2),x_3) ), \\
& (m(e, x_1), x_1 ), \\
& (m(x_1, e),x_1 ), \\
& (m(x_1, i(x_1)), e) \}.
\end{align*}
Then $(\Omega, E)$ is the \demph{theory of groups} and a model of $(\Omega, E)$ is a \demph{group}.
\end{ex}

\begin{remark}
There is a problem with the definition of algebraic theory given in Definition~\ref{defn:alg-theory}. Suppose $A$ is a group, and define a binary operation $q \from A \times A \to A$ by
\[
q(a,b) = [m]_A (a, [i]_A(b)).
\]
Then we can describe the multiplication and inverses of $A$ in terms of $q$ as follows:
\begin{align*}
[m]_A (a,b) &= q (a, q( q( b, b), b)) \\
[i]_A ( a) &= q( q(a,a), a). \\
\end{align*}
Therefore, we could express the group axioms entirely in terms of the identity and $q$, and obtain a new algebraic theory with a single constant $e$ and a single binary operation symbol $q$, but whose category of models is isomorphic to the category of groups. Thus it seems incorrect to describe the theory described in Example~\ref{ex:theory-groups} as \emph{the} theory of groups, since there are in fact multiple theories whose models are groups. Rather, we should think of what is described as an algebraic theory in Definition~\ref{defn:alg-theory} as a \emph{presentation} (or axiomatisation) of a theory, and any given theory (whatever we ultimately decide that means) as potentially having multiple distinct presentations. In the rest of this section, we try to identify what features such a ``presentation-independent'' notion of algebraic theory should have, so that we can generalise them to other categorical contexts.
\end{remark}

\begin{defn}
Let $\Omega$ be an operational type, and let $t \in F_{\Omega} (n)$ and $(s_i)_{i = 1}^n$ be a family of terms $s_i \in F_{\Omega}(m)$ indexed by $i = 1, \ldots, n$. Then we write
\[
t[(s_i / x_i)_{i =1}^n ] \in F_{\Omega} (m)
\]
for the term obtained by replacing each occurrence of $x_i$ in $t$ with the corresponding term $s_i$.
\end{defn}

\begin{defn}
\label{defn:equiv-terms}
Given an algebraic theory $(\Omega, E)$, define a family of relations $\sim_n \subseteq F_{\Omega}(n) \times F_{\Omega}(n)$ indexed by all natural numbers $n$ to be the smallest such that
\begin{enumerate}
\item
\label{part:equiv-terms-equiv}
each $\sim_n$ is an equivalence relation on $F_{\Omega}(n)$;
\item
\label{part:equiv-terms-axioms}
if $(s, t) \in E$, where $s, t \in F_{\Omega}(n)$, then $s \sim_n t$;
\item
\label{part:equiv-terms-sub1}
if $s \sim_n t$ and $u_i \in F_{\Omega}(m)$ for each $i = 1, \ldots, n$, then $s[(u_i/x_i)_{i=1}^n] \sim_m t[(u_i/x_i)_{i = 1}^n]$; and
\item
\label{part:equiv-terms-sub2}
if $u \in F_{\Omega}(n)$ and $s_i \sim_m t_i$ for each $i = 1, \ldots n$, where $s_i, t_i \in F_{\Omega} (m)$, then $u[(s_i/x_i)_{i = 1}^n] \sim_m u[(t_i/x_i)_{i = 1}^n]$,
\end{enumerate}
in the sense that if $(\sim'_n)_{n \in \nat}$ is another such family of equivalence relations, then $\sim_n \subseteq \sim'_n \subseteq F_{\Omega}(n) \times F_{\Omega} (n)$ for each $n \in \nat$.

An equivalence class of $\sim_n$ is called an \demph{operation} of the theory $(\Omega, E)$ with arity $n$. Write $\oper(n)$ for the set of operations of $(\Omega, E)$ of arity $n$.
\end{defn}

\begin{remark}
\label{rem:deduction-rules}
It is helpful to think of the conditions \bref{part:equiv-terms-equiv}--\bref{part:equiv-terms-sub2} above as describing deduction rules for the logic of algebraic theories. Then the elements of $E$ are the axioms of the theory, and an expression of the form $s \sim_n t$ is a theorem, and a proof is a sequence of applications of the rules \bref{part:equiv-terms-equiv}--\bref{part:equiv-terms-sub2}, starting from the axioms.
\end{remark}

\begin{thm}[The Completeness Theorem]
\label{thm:classical-completeness}
Let $(\Omega, E)$ be an algebraic theory, and let $s, t \in F_{\Omega}(n)$. Then $s \sim_n t$ if and only if, for every model $A$ of $(\Omega, E)$ we have an equality of functions
\[
[s]_A = [t]_A \from A^n \to A.
\]
\end{thm}
\begin{proof}
This is Corollary~1.5 in~\cite{johnstone87}.
\end{proof}

Informally, the completeness theorem states that the properties possessed by all models of $(\Omega, E)$ (that is, the only equations that hold between interpretations of terms in all models) are precisely those that can be derived syntactically via the process described in Remark~\ref{rem:deduction-rules}.

Let us note some features of the collection of all operations of a theory.
\begin{enumerate}
\item Every operation has an arity $n \in \nat$, since every term does.
\item Every $i = 1, \ldots, n$ gives rise to an $n$-ary operation (namely the equivalence class of $x_i \in F_{\Omega}(n)$).
\item If $t$ is an operation with arity $n$, and $(s_i)_{i = 1}^n$ is a family of operations of arity $m$, indexed by $i = 1, \ldots, n$, then there is an operation $t[(s_i/ x_i)_{i = 1}^n]$ of arity $m$, obtained by substituting $s_i$ in for each occurrence of $x_i$ in $t$, for each $i = 1, \ldots, n$ (note that this is well-defined, since conditions~\bref{part:equiv-terms-sub1} and~\bref{part:equiv-terms-sub2} of Definition~\ref{defn:equiv-terms} imply that equivalence of terms is preserved by substitution from both sides respectively).
\end{enumerate}

\begin{remark}
\label{rem:alg-th-features}
We may abstract from the observations above to obtain a list of some features that we might expect a notion of algebraic theory to have; we will later examine how both Lawvere theories and monads display these features. First of all, independently of any \emph{particular} theory:
\begin{enumerate}
\item[0.] there is a collection of \emph{arities} $\cat{A}$.
\end{enumerate}
Next, any individual theory should have the following features.
\begin{enumerate}
\item\label{part:alg-th-features-arity} For each arity $a$, there is a totality $\oper(a)$ of $a$-ary operations of the theory.
\item Every element of an arity $a$ gives rise to an $a$-ary operation; that is, there is in some sense a map $a \to \oper(a)$;
\item \label{part:alg-th-features-sub} Given an $a$-ary operation $\tau \in \oper(a)$ and a family $(\sigma_i)_{i \in a}$ of $b$-ary operations indexed by the arity $a$, we can substitute the $\sigma_i$'s into $\tau$ to form a new $b$-ary operation. That is, in some sense we have a map $\oper(a) \times \oper(b)^a \to \oper(b)$.
\end{enumerate}
The points above are left deliberately vague, and we make no assertions at this stage about what kind of entity the arities, operations and totalities of operations are. In particular, the notation and terminology may suggest that we treat each arity $a$ and totality of $a$-ary operations as if they are sets, but this is intended as a guide for intuition only. When we make these features precise for particular notions of algebraic theory, we will take $a$ and $\oper(a)$ to be various different kinds of mathematical objects.
\end{remark}

\section{Lawvere theories}
\label{sec:notions-lawvere}

Lawvere theories provide the most direct translation of the classical notion of algebraic theory into category-theoretic terms. Indeed, every algebraic theory gives rise to a Lawvere theory (and every Lawvere theory arises in this way), and two algebraic theories give rise to isomorphic Lawvere theories if and only if they have isomorphic categories of models. Lawvere theories were first defined in Lawvere's PhD thesis~\cite{lawvere63}, where they are called simply ``algebraic theories''. Lawvere theories have been generalised to an enriched setting in Power~\cite{power99} and Nishizawa and Power~\cite{nishizawaPower09}, however we will deal only with the ordinary non-enriched version. Some of the differences between our approach and that of~\cite{nishizawaPower09} Are discussed in Section~\ref{sec:structure-examples} in the subsection on Lawvere theories.

\begin{defn}
Let $\fin$ be a skeleton of $\finset$, so that the objects of $\fin$ are sets of the form
\[
n = \{ x_1, \ldots x_n \},
\]
where the elements $x_i$ are arbitrary. We think of the $x_i$ as abstract variables.
\end{defn}

\begin{defn}
\label{defn:lawvere-theory}
A \demph{Lawvere theory} consists of a large category $\cat{L}$ together with a functor $L \from \fin^{\op} \to \cat{L}$ that is bijective on objects and preserves finite products (that is, it sends coproducts in $\fin$ to products in $\cat{L}$). A Lawvere theory morphism from $L' \from \fin^{\op} \to \cat{L}'$ to $L \from \fin^{\op} \to \cat{L}$ is a functor $P \from \cat{L}' \to \cat{L}$ such that $P \of L' = L$. We write $\LAW$ for the category of Lawvere theories and $\Law$ for the category of locally small Lawvere theories.
\end{defn}

It is common to require Lawvere theories to be locally small; note that we do not make this restriction.

\begin{remark}
Let $(\Omega, E)$ be an algebraic theory in the sense of Definition~\ref{defn:alg-theory}. Then we may define an associated Lawvere theory $L \from \fin^{\op} \to \cat{L}$ as follows. The objects of $\cat{L}$ are the same as the objects of $\fin$, and the hom-set $\cat{L}(m,n)$ is given by the set $\oper(m)^n$, the set of $n$-tuples of $m$-ary operations (that is, equivalence classes of $m$-ary terms) for the theory $(\Omega, E)$. Composition is induced by substitution of terms: if $(\tau_i)_{i = 1}^n \in \cat{L}(m, n) = \oper(m)^n$ and $(\sigma_j)_{j=1}^m \in \oper(p)^m = \cat{L}(p,m)$, their composite is the $n$-tuple of $p$-ary operations whose $i$-th member is $\tau_i[(\sigma_j/x_j)_{i = j}^m]$. The functor $L \from \fin^{\op} \to \cat{L}$ sends $f \from n \to m$ to the $n$-tuple of $m$-ary operations whose $i$-th member is (the equivalence class of the term) $f(x_i)$; this also determines the identities in $\cat{L}$.

Conversely, given a Lawvere theory $L \from \fin^{\op} \to \cat{L}$, we may define an algebraic theory $(\Omega, E)$, where $\Omega_n = \cat{L}(Ln, L1)$, and where the equations in $E$ are precisely those that hold in $L$, when a variable $x_i$ is interpreted as a projection $Ln \to L1$, and formal substitution of terms is interpreted as composition in $L$.
\end{remark}

\begin{remark}
Let us highlight how the features of a general notion of algebraic theory described in Remark~\ref{rem:alg-th-features} manifest themselves in the case of Lawvere theories. Let $L \from \fin^{\op} \to \cat{L}$ be a Lawvere theory.
\begin{enumerate}
\item[0.] The arities for Lawvere theories are the natural numbers.
\item For $n \in \nat$, the $n$-ary operations of a Lawvere theory form the set $\cat{L}(Ln, L1)$.
\item For each $n$, we have a function $n \to \cat{L}(Ln, L1)$ that sends $x_i$ to the $n$-ary operation it represents.
\item Suppose we have an $n$-ary operation $\tau \in \cat{L}(Ln, L1)$, and a family $(\sigma_i)_{i = 1}^n$ of $m$-ary operations indexed by $i=1, \ldots, n$. Then the $\sigma_i$ correspond to a morphism $\sigma \in \cat{L}(Lm, Ln)$, since $Ln$ is the $n$-fold power of $L1$ (because in $\fin$, the object $n$ is the $n$-th copower of $1$). The result of substituting the $\sigma_i$ into $\tau$ is given by the composite $\tau \of \sigma \in \cat{L}(Lm, L1)$.
\end{enumerate}
\end{remark}

We now turn to the semantics of Lawvere theories.

\begin{defn}
By a \demph{finite product category} we mean a category $\cat{B}$ equipped with, for each finite family $b_1, \ldots, b_n$ of objects of $\cat{B}$, a specified object $b_1 \times \cdots \times b_n$ and maps $\pi_i \from b_1 \times \cdots \times b_n \to b_i$ exhibiting $b_1 \times \cdots \times b_n$ as the product of $b_1, \ldots, b_n$. This is in contrast to a \demph{category with finite products}, which is a category in which all finite products exist, but are not specified.

A \demph{finite product preserving functor} between finite product categories, however, is still only required to preserve the property of being a product, not the distinguished choice of product.
\end{defn}

For the rest of this section, fix a large finite product category $\cat{B}$.

\begin{remark}
Note that for each $b \in \cat{B}$, the assignment $n \mapsto b^n$ has a unique extension to a product preserving functor $b^{(-)} \from \fin^{\op} \to \cat{B}$, and every morphism $f \from b \to c$ extends uniquely to a natural transformation $f^{(-)} \from b^{(-)} \to c^{(-)}$.
\end{remark}

\begin{defn}
\label{defn:law-models}
Let $L \from \fin^{\op} \to \cat{L}$ be a Lawvere theory. Then a \demph{model} $x$ of $L$ in $\cat{B}$ consists of an object $d^x$ and a functor $\Gamma^x \from \cat{L} \to \cat{B}$ such that
\[
\xymatrix{
\fin^{\op}\ar[r]^{L} \ar[dr]_{(d^x)^{(-)}} & \cat{L}\ar[d]^{\Gamma^x}\\
& \cat{B}
}
\]
commutes.

A \demph{$L$-model homomorphism} $x \to y$ consists of a morphism $h \from d^x \to d^y$ in $\cat{B}$ together with a natural transformation $\kappa \from \Gamma^x \to \Gamma^y$ such that
\[
\vcenter{
\xymatrix
@R=50pt
@C=50pt{
\fin^{\op} \ar[r]^{L}&  \cat{L}\rtwocell^{\Gamma^x}_{\Gamma^y}<5> {\kappa} & \cat{B}
}}
\quad = \quad
\vcenter{
\xymatrix
@R=50pt
@C=50pt{
\fin^{\op} \rtwocell^{(d^x)^{(-)}}_{(d^y)^{(-)}} <5>{\:\:\:\:\:\: h^{(-)}} & \cat{B}.
}}
\]
\end{defn}

\begin{remark}
\label{rem:Law-model-hom}
In the definition of an $L$-model homomorphism, note that the components of $\kappa$ are completely determined by the components of $h^{(-)}$ because $L$ is bijective on objects, and these are in turn determined by $h$ itself. Thus, being an $L$-model homomorphism is in fact a \emph{property} of a morphism $h$ between $L$-models, rather than an additional \emph{structure}. Furthermore, given \emph{any} natural transformation $\kappa \from \Gamma^x \to \Gamma^y$, we can define $h = \kappa_{L1}$ and then $h$ is an $L$-model homomorphism. Thus we could have defined a homomorphism of $L$-models to be a natural transformation, but we want to emphasise that such a natural transformation is determined by its component at $L1$, which is a morphism between the underlying models of the $L$-models.
\end{remark}

\begin{defn}
\label{defn:lawvere-sem-objects}
Let $L \from \fin^{\op} \to \cat{L}$ be a Lawvere theory. We write $\mod_{\law}(L)$ for the category whose objects are models of $L$ in $\cat{B}$, and whose morphisms are $L$-model homomorphisms. We write $\sem_{\law}(L)\from \mod_{\law}(L) \to \cat{B}$ for the obvious forgetful functor.
\end{defn}

\begin{remark}
\label{rem:law-model-non-standard}
Our definition of a model of a Lawvere theory in Definition~\ref{defn:law-models} is slightly non-standard --- usually a model is defined as a finite-product preserving functor out of $\cat{L}$. However, Lemma~\ref{lem:law-model-non-standard} below shows that the two definitions yield equivalent categories of models.

We have chosen this non-standard definition partly for pragmatic reasons; it will make it easier to fit Lawvere theories into the general framework to be described in subsequent chapters. However there is also a conceptual reason, namely that the forgetful functor from the category of models not only reflects isomorphisms, but also reflects \emph{equalities}. That is, if an object can be equipped with two model structures such that the identity on that object is a homomorphism between them then the two model structures are in fact identical. This is not the case with the less restrictive definition. This matches mathematical practice: for example, one does not distinguish between two group structures on the same underlying set that differ only in the choice of which binary cartesian product of the underlying set was used to define the multiplication.
\end{remark}

\begin{lem}
\label{lem:law-model-non-standard}
Let $L \from \fin^{\op} \to \cat{L}$ be a Lawvere theory. Then $\mod_{\law}(L)$ is equivalent to the full subcategory of $[\cat{L}, \cat{B}]$ consisting of the finite product preserving functors.
\end{lem}
\begin{proof}
Suppose $x = (d^x, \Gamma^x)$ is a model of $L$, so by definition $\Gamma^x \of L = (d^x)^{(-)}$. The product projections in $\cat{L}$ all lie in the image of $L \from \fin^{\op} \to \cat{L}$, and each $(d^x)^{(-)} \from \fin^{\op} \to \cat{B}$ preserves finite products, so the condition that $\Gamma^x \of L = (d^x)^{(-)}$ implies that $\Gamma$ does send the product projections to product projections in $\cat{B}$. Hence $\Gamma^x$ is a finite product preserving functor. As noted in Remark~\ref{rem:Law-model-hom}, every natural transformation between $L$-models defines a homomorphism, so we can identify $\mod_{\law}(L)$ with a full subcategory of the category of finite product preserving functors $\cat{L} \to \cat{B}$. Thus to show that the two categories are equivalent, it is sufficient to show that every finite product preserving functor is isomorphic to an $L$-model.

Let $\Gamma \from \cat{L} \to \cat{B}$ be a functor that preserves finite products. Let $d = \Gamma(L1)$. Every $n \in \fin$ is the $n$-th copower of $1$; write $\iota_i \from 1 \to n$ for the $i$-th copower coprojection for $i= 1, \ldots, n$. Then, since $\Gamma$ preserves finite products, we have isomorphisms $\gamma_n \from \Gamma(Ln) \to d^n$ such that
\[
\xymatrix{
\Gamma(Ln)\ar[r]^{\gamma_n} \ar[dr]_{\Gamma(L\iota_i)} & d^n\ar[d]^{\pi_i} \\
& \Gamma(L1) = d
}
\]
commutes for each $n \in \fin$ and $i = 1, \ldots, n$. We define a functor $\Gamma' \from \cat{L} \to \cat{B}$ by setting $\Gamma' (Ln) = d^n$, and, given $l \from Ln \to Lm$, by setting $\Gamma'(l)$ to be the unique morphism $d^n \to d^m$ such that
\[
\xymatrix{
\Gamma(Ln)\ar[r]^{\gamma_n}\ar[d]_{\Gamma(l)} & d^n\ar[d]^{\Gamma'(l)} \\
\Gamma(Lm)\ar[r]^{\gamma_m} \ar[dr]_{\Gamma(L\iota_i)} & d^m\ar[d]^{\pi_i} \\
& d
}
\]
commutes for each $i = 1, \ldots, m$. It is then clear that $(d, \Gamma')$ is an $L$-model and the $\gamma_n$ form the components of a natural isomorphism $\Gamma \to \Gamma'$.
\end{proof}

Let us unpack the definition of a model of a Lawvere theory. Let $L \from \fin^{\op} \to \cat{L}$ be a Lawvere theory and $(d, \Gamma)$ a model of $L$ in $\cat{B}$. By the lemma above, $\Gamma$ preserves finite products and so, since every object of $\fin^{\op}$ is a finite power of $1$, it follows that $\Gamma$ is determined by its values on morphisms of the form $l \from Ln \to L1$ in $\cat{L}$. Such morphisms are the $n$-ary operations of the corresponding algebraic theory, so to equip $d$ with an $L$-model structure is to provide an interpretation of each $n$-ary operation of the theory.

Functoriality of $\Gamma$ says that, if $l \from Ln \to L1$ and $k_i \from Lm \to L1$ in $\cat{L}$ for $i = 1, \ldots, n$ then
\[
\Gamma(l) \of (\Gamma(k_1), \ldots , \Gamma( k_n) ) = \Gamma (l \of (k_1, \ldots, k_n)) \from d^m \to d,
\]
or in other words, the interpretations of the operations of the theory are compatible with the process of substituting operations into one another.

\begin{defn}
\label{defn:lawvere-sem-morphisms}
Let
\[
\xymatrix{
\cat{L}'\ar[rr]^P & &\cat{L}\\
& \fin^{\op}\ar[ul]^{L'}\ar[ur]_{L}
}
\]
be a morphism of Lawvere theories, and define a functor $\sem_{\law}(P) \from \mod_{\law}(L) \to \mod_{\law}(L')$ as follows. Given $x = (d^x, \Gamma^x) \in \mod_{\law}(L)$, we set $d^{\sem_{\law}(P) (x)} = d^x$ and set $\Gamma^{\sem_{\law}(P)(x)} $ to be the composite
\[
\cat{L}' \toby{P} \cat{L} \toby{\Gamma^x} \cat{B}.
\]
If $h \from x \to y$ is an $L$-model homomorphism then $\sem_{\law}(P) (h)$ is the $L'$-model homomorphism $\sem_{\law}(P)(x) \to \sem_{\law}(P)(y)$ with the same underlying morphism in $\cat{B}$ (and this is necessarily a homomorphism).
\end{defn}

\begin{prop}
\label{prop:lawvere-sem}
Together, Definitions~\ref{defn:lawvere-sem-objects} and~\ref{defn:lawvere-sem-morphisms} define a functor $\LAW^{\op} \to \catover{\cat{B}}$.
\end{prop}
\begin{proof}
This is immediate.
\end{proof}

Let $(U \from \cat{M} \to \cat{B}) \in \catover{\cat{B}}$. Imagine that $U$ is the forgetful functor from the category of models of some Lawvere theory $L \from \fin^{\op} \to \cat{L}$, but we do not know anything about $L$. How might we attempt to recover it from $\cat{M}$ and $U$?

We know that for every morphism $l \from Ln \to Lm$ in $\cat{L}$ we would have a map $\Gamma^x (l) \from (Ux)^n \to (Ux)^m$ for each $x \in \cat{M}$ (that is, for each $L$-model). Furthermore, we know that for each morphism $h \from x \to y$ in $\cat{M}$, the square
\[
\xymatrix{
(Ux)^n \ar[r]^{\Gamma^x(l)}\ar[d]_{(Uh)^n} & (Ux)^m\ar[d]^{(Uh)^m} \\
(Uy)^n \ar[r]_{\Gamma^y(l)} & (Uy)^m
}
\]
commutes. In other words, every morphism $l \from Ln \to Lm$ in $\cat{L}$ gives rise to a natural transformation $U^n \to U^m$. Therefore, if we wanted to recover a Lawvere theory from $U$, we might reasonably attempt to do so by defining the morphisms $Ln \to Lm$ in our Lawvere theory to be the natural transformations $U^n \to U^m$. This motivates the following definition.

\begin{defn}
\label{defn:lawvere-str-objects}
Let $U \from \cat{M} \to \cat{B}$ be an object of $\catover{\cat{B}}$. Then we define a category $\thr_{\law}(U)$ to have the same objects as $\fin$, and with hom-sets given by
\[
\thr_{\law}(U) ( n, m) = [\cat{M}, \cat{B}](U^n, U^m),
\]
with composition as in $[\cat{M},\cat{B}]$. We define a functor $\str_{\law}(U) \from \fin^{\op} \to \thr_{\law}(U)$ to be the identity on objects, and sending $f \from n \to m$ in $\fin$ to the unique natural transformation $f^* \from U^m \to U^n$ such that
\[
\xymatrix{
U^m\ar[r]^{f^*}\ar[dr]_{\pi_{f(i)}} & U^n\ar[d]^{\pi_i} \\
& U
}
\]
commutes for each $i = 1, \ldots, n$.
\end{defn}
Here, and elsewhere, $\str$ stands for ``structure'' and $\thr$ stands for ``theory''. Together $\str_{\law}(U)$ and $\thr_{\law}(U)$ give the Lawvere theory of the most general kind of algebraic structure possessed by all the objects of $\cat{M}$, when we regard them as objects of $\cat{B}$ equipped with extra structure.
\begin{defn}
\label{defn:lawvere-str-morphisms}
Let
\[
\xymatrix{
\cat{M}'\ar[rr]^Q \ar[dr]_{U'} && \cat{M}\ar[dl]^{U} \\
& \cat{B} & 
}
 \]
 be a morphism in $\catover{\cat{B}}$. We define a functor $\str_{\law}(Q) \from \thr_{\law}(U) \to \thr_{\law}(U')$, sending a natural transformation
\[
\gamma \from U^m \to U^n
\]
to the natural transformation
\[
\gamma Q \from U^m \of Q = (U')^m \to U^n \of Q = (U')^n.
\]
\end{defn}

\begin{prop}
\label{prop:lawvere-str}
Together, Definitions~\ref{defn:lawvere-str-objects} and~\ref{defn:lawvere-str-morphisms} define a functor $\str_{\law} \from \catover{\cat{B}} \to \LAW^{\op}.$
\end{prop}

\begin{proof}
The only part that is not obvious is that $\str_{\law}(U) \from \fin^{\op} \to \thr_{\law}(U)$ preserves finite products. Note that
\[
\xymatrix
@R=35pt
@C=35pt{
\fin^{\op}\ar[r]^{\str_{\law}(U)}\ar[d] & \thr_{\law}(U)\ar[d] \\
[\cat{B}, \cat{B}] \ar[r]_{U^*} & [\cat{M}, \cat{B}]
}
\]
commutes, where the left-hand vertical arrow sends $n \in \fin$ to $\id_{\cat{B}}^n \from \cat{B} \to \cat{B}$, and the right-hand vertical arrow is a full inclusion. But the left and bottom sides of this square preserve finite products (since finite products commute with finite powers), and hence so does the top-right composite. Since the inclusion $\thr_{\law}(U) \incl [\cat{M}, \cat{B}]$ is full and faithful, and so reflects limits, it follows that $\str_{\law}(U)$ preserves finite products.
\end{proof}

\begin{prop}
\label{prop:str-sem-adj-law}
The functors from Propositions~\ref{prop:lawvere-sem} and~\ref{prop:lawvere-str} form an adjunction
\[
\xymatrix{
{\catover{\cat{B}}}\ar@<5pt>[r]_-{\perp}^{\str_{\law}}\ & {\LAW^{\op}.}\ar@<5pt>[l]^{\sem_{\law}}
}
\]
\end{prop}

\begin{proof}
A version of this is Theorem~2 in~III.1 of~\cite{lawvere63}, which deals with the case when $\cat{B} = \Set$, and with slightly stronger size restrictions on the categories and functors involved. The general result is presumably known, but I was unable to find it in the literature. In fact it will follow from Proposition~\ref{prop:law-sem-from-proth}, but it is also straightforward to prove directly, and we sketch such a proof here.

Let $L \from \fin^{\op} \to \cat{L}$ be a Lawvere theory and $(U \from \cat{M} \to \cat{B}) \in \catover{\cat{B}}$. We will describe a bijection
\[
\catover{\cat{B}}(U, \sem_{\law}(L)) \iso \LAW ( L, \str_{\law}(U)).
\]
Let
\[
\xymatrix{
\cat{M}\ar[rr]^Q\ar[dr]_U && \mod_{\law}(L)\ar[dl]^{\sem_{\law}(L)} \\
& \cat{B} &
}
\]
be a morphism in $\catover{\cat{B}}$. We will define a morphism $\widebar{Q} \from L \to \str_{\law}(U)$ in $\LAW$. Given $l \from Ln \to Ln'$ in $\cat{L}$, for every model $x= (d^x, \Gamma^x)$ of $L$ in $\cat{B}$, we have a map
\[
\Gamma^x (l) \from (d^x)^n = U^n (x) \to (d^x)^{n'} = U^{n'} (x),
\]
and by the definition of $L$-model homomorphisms, these are natural in $x \in \mod_{\law}(L)$. Thus $\Gamma^{(-)} (l)$ is a natural transformation $U^n \to U^{n'}$, that is, a morphism $n \to n'$ in $\thr_{\law}(U)$; we define $\widebar{Q} (l) = \Gamma^{(-)}(l)$.

In the other direction, let
\[
\xymatrix{
\cat{L}\ar[rr]^P & &\thr_{\law}(U)\\
& \fin^{\op}\ar[ul]^{L}\ar[ur]_{\str_{\law}(U)}
}
\]
be a morphism in $\LAW$; we must define a corresponding morphism $\widebar{P} \from U \to \sem_{\law}(L)$ in $\catover{\cat{B}}$. That is, for each object $m \in \cat{M}$, we must equip $Um$ with an $L$-model structure such that for any map $f \from m \to m'$ in $\cat{M}$, the map $Um$ becomes an $L$-model homomorphism. To define an $L$-model structure on $Um$ we must give, for each $l \from Ln \to Ln'$ in $\cat{L}$, a map $(Um)^n \to (Um)^{n'}$. We take this map to be $P(l)_m$.

It remains to check that these are inverse bijections natural in $U$ and $L$; this is straightforward and we omit it.
\end{proof}

\begin{defn}
Let $\cat{M}$ be a large category and $U \from \cat{M} \to \Set$ a functor. We say that $U$ is tractable if, for all $n, n' \in \nat$, the set of natural transformations $U^n \to U^{n'}$ is small. We write $\catover{\Set}_{\tract}$ for the full subcategory of $\catover{\Set}$ consisting of the tractable functors.
\end{defn}

Recall that $\Law$ denotes the full subcategory of $\LAW$ consisting of those Lawvere theories $L \from \fin^{\op} \to \cat{L}$ where $\cat{L}$ is locally small.

\begin{prop}
\label{prop:str-sem-law-set-ff}
In the case when $\cat{B} = \Set$, the adjunction of Proposition~\ref{prop:str-sem-adj-law} restricts to an adjunction
\[
\xymatrix{
{\catover{\Set}_{\tract}}\ar@<5pt>[r]_-{\perp}^{\str_{\law}}\ & {\Law^{\op},}\ar@<5pt>[l]^{\sem_{\law}}
}
\]
and $\sem_{\law} \from \Law^{\op} \to \catover{\Set}_{\tract}$ is full and faithful.
\end{prop}

\begin{proof}
The fact that the adjunction restricts in this way is Theorem~2 in~III.1 of~\cite{lawvere63}, and the fact that $\sem_{\law}$ is full and faithful is Theorem~1 in the same section.
\end{proof}

\begin{remark}
\label{rem:ff-completeness-interpret}
Recall that a right adjoint is faithful if and only if every component of the counit of the adjunction is an epimorphism, and is full and faithful if and only if the counit is an isomorphism (See Theorem~1 in~IV.3 of~\cite{maclane71}). Let us interpret each of these conditions in the case of the adjunction from Proposition~\ref{prop:str-sem-law-set-ff}. The monomorphisms in $\Law$ (corresponding to epimorphisms in $\Law^{\op}$) are precisely the morphisms of Lawvere theories given by faithful functors. The component of the counit at a Lawvere theory $L \from \fin^{\op} \to \cat{L}$ is the functor
\[
E_L \from \cat{L} \to \thr_{\law}(\sem_{\law}(L))
\]
that sends $l \from Lm \to Ln$ to the natural transformation $E_L(l) = \Gamma^{(-)} (l) \from \sem_{\law}(L)^m \to \sem_{\law}(L)^n$ with $x$-component
\[
\Gamma^x (l) \from \sem_{\law}(L)^m (x) = (d^x)^m \to \sem_{\law}(L)^n (x) = (d^x)^n
\]
for each $x \in \mod_{\law}(L)$. Since $L$ preserves finite products and every object of $\fin$ is a copower of $1$, this functor is faithful if and only if, whenever $l, l' \from Ln \to L1$ in $\cat{L}$ are such that $\Gamma^x (l) = \Gamma^x(l')$ for every $x \in \mod_{\law}(L)$, we in fact have $l = l'$. That is, any two operations of the theory $L$ that have the same interpretation in every model, are in fact equal. Thus faithfulness of $\sem_{\law}$ is very closely analogous to the completeness theorem (Theorem~\ref{thm:classical-completeness}) for classical algebraic theories.

Next we interpret what it means for $\sem_{\law}$ to be \emph{full} and faithful, or equivalently for each $E_L$ to be an isomorphism. This occurs precisely when, for any Lawvere theory $L\from \fin^{\op} \to \cat{L}$, every natural transformation $\sem_{\law}(L)^n \to \sem_{\law}$ is of the form $\Gamma^{(-)}(l)$ for some $l \from Ln \to L1$. A natural transformation $\sem_{\law}(L)^n \to \sem_{\law}(L)$ can be thought of as an $n$-ary operation possessed by every $L$-model and preserved by every $L$-model homomorphism --- that is, it is a kind of additional algebraic structure possessed by $L$-models. Thus, the assertion that $\sem_{\law}$ is full and faithful, or equivalently that each $E_L$ is an isomorphism says that $L$-models do not possess any extra algebraic structure other than that already described by the theory $L$. This is a kind of ``structural completeness theorem''.
\end{remark}

\section{Monads}
\label{sec:notions-monads}

Throughout this section, let $\cat{B}$ be any large category. We assume the reader is familiar with the basic theory of monads, as described, for example, in Chapter~VI of~\cite{maclane71} or Chapter~4 of~\cite{borceux94v2}. Note however, that when we refer to ``morphisms of monads'', we mean this in the sense of Definition~4.5.8 of~\cite{borceux94v2}, rather than as in Section~1 of~\cite{street72}. In particular, morphisms of monads are always between monads on the same category, and a morphism $(T, \eta, \mu) \to (T', \eta', \mu')$ consists of a natural transformation $T \to T'$ making certain diagrams commute. We write $\monad(\cat{B})$ for the category of monads and monad morphisms on $\cat{B}$.

\begin{remark}
Let us make sense of the conditions of Remark~\ref{rem:alg-th-features} in the context of monads as algebraic theories. Let $\mnd{T} = (T, \eta, \mu)$ be a monad on $\cat{B}$.
\begin{enumerate}
\item[0.] The arities for monads on the category $\cat{B}$ are the objects of $\cat{B}$.
\item Given an arity $b$ (that is, an object of $\cat{B}$), the totality of all $b$-ary operations for $\mnd{T}$ is given by the object $Tb \in \cat{B}$.
\item The unit of the monad gives a map $\eta_b \from b \to Tb$ for each $b$.
\item One way of restating condition~\ref{rem:alg-th-features}.\bref{part:alg-th-features-sub} is that, any $a$-indexed family of $b$-ary operations should give rise to a way of turning $a$-ary operations into $b$-ary operations. A morphism $f \from a \to Tb$ can be thought of as an $a$-indexed family of $b$-ary operations in some sense, and every such morphism canonically gives rise to a morphism $Ta \to Tb$ (which is a way of turning $a$-ary operations into $b$-ary operations), namely the composite
\[
Ta \toby{Tf} TTb \toby{\mu_b} Tb.
\]
\end{enumerate}
\end{remark}

\begin{defn}
We write $\radjover{\cat{B}}$ for the category whose objects are functors $U \from \cat{M} \to \cat{B}$ with $\cat{M}$ large, equipped with a specified left adjoint and choice of unit and counit, and whose morphisms from $U' \from \cat{M}' \to \cat{B}$ to $U \from \cat{M} \to \cat{B}$ are functors $Q \from \cat{M}' \to \cat{M}$ such that $U \of Q = U'$. There is an evident forgetful functor $\radjover{\cat{B}} \to \catover{\cat{B}}$, and this is full and faithful.
\end{defn}

\begin{defn}
We write $\sem_{\monad} \from \monad(\cat{B})^{\op} \to \radjover{\cat{B}}$ for the functor defined as follows. On objects, $\sem_{\monad}$ sends a monad $\mnd{T}$ to the forgetful functor $U^{\mnd{T}} \from \cat{B}^{\mnd{T}} \to \cat{B}$ from the category of Eilenberg--Moore algebras for $\mnd{T}$ with its usual left adjoint, unit and counit. On morphisms, $\sem_{\monad}$ sends a monad morphism $\phi \from \mnd{T}' \to \mnd{T}$ to the functor $\sem_{\monad}(\phi) \from \cat{B}^{\mnd{T}} \to \cat{B}^{\mnd{T}'}$ that sends a $\mnd{T}$-algebra $a \from Td \to d$ to the $\mnd{T}'$-algebra $a \of \phi_b \from T'b \to Tb \to b$, and sends a $\mnd{T}$-algebra homomorphism to the $\mnd{T}'$-algebra homomorphism with the same underlying morphism in $\cat{B}$.
\end{defn}

\begin{prop}
\label{prop:sem-str-adj-monad}
There is a functor $\str_{\monad} \from \radjover{\cat{B}} \to \monad(\cat{B})^{\op}$ that, on objects, sends a right adjoint to the monad induced by the adjunction, and this gives a structure--semantics adjunction
\[
\xymatrix{
{\radjover{\cat{B}} }\ar@<5pt>[r]_-{\perp}^-{\str_{\monad}}\ & {\monad(\cat{B})^{\op}.}\ar@<5pt>[l]^-{\sem_{\monad}}
} \tagqed
\]
\end{prop}
\begin{proof}
We could prove this directly, however, it follows from the more general result Corollary~\ref{cor:str-sem-adj-monad}, and so we defer the proof until then.
\end{proof}
One may ask whether it is possible to extend the definition of $\str_{\monad}$ to a larger subcategory of $\catover{\cat{B}}$ than $\radjover{\cat{B}}$. In other words, is it possible that for a functor $U \from \cat{M} \to \cat{B}$ \emph{without} a left adjoint there may nonetheless exist a monad $\str_{\monad}(U)$ such that there is an isomorphism
\[
\catover{\cat{B}} (U, \sem_{\monad}(\mnd{T})) \iso \monad(\cat{B})(\mnd{T}, \str_{\monad}(U)),
\]
natural in $\mnd{T} \in \monad(\cat{B})$? Indeed this \emph{does} sometimes occur; more precisely it occurs whenever the codensity monad of $U$ exists, as defined in Definition~\ref{defn:codensity-monad}.

\begin{defn}
Write $\cmover{\cat{B}}$ for the full subcategory of $\catover{\cat{B}}$ on those functors into $\cat{B}$ that have a codensity monad.
\end{defn}

\begin{prop}
\label{prop:codensity-natural-bij}
Let $(U \from \cat{M} \to \cat{B}) \in \cmover{\cat{B}}$ have codensity monad $\mnd{T}$. Then there is a bijection
\[
\catover{\cat{B}} ( U, \sem_{\monad} (\mnd{S})) \iso \monad(\cat{B})(\mnd{S}, \mnd{T}),
\]
natural in $\mnd{S} \in \monad(\cat{B})$.
\end{prop}
\begin{proof}
This is Theorem~II.1.1 in Dubuc~\cite{dubuc70}.
\end{proof}

\begin{prop}
Let $(U \from \cat{M} \to \cat{B}) \in \radjover{\cat{B}}$ have left adjoint $F$ with unit $\eta$ and counit $\epsilon$. Then $(GF, \eta, G\epsilon F)$ is a codensity monad of $G$. In particular, there is a natural inclusion $\radjover{\cat{B}} \incl \cmover{\cat{B}}$.
\end{prop}
\begin{proof}
This is well-known; see for example Proposition~6.1 in~\cite{leinster13}.
\end{proof}

\begin{cor}
\label{cor:str-sem-adj-monad}
There is an adjunction
\[
\xymatrix{
{\cmover{\cat{B}} }\ar@<5pt>[r]_-{\perp}^-{\str_{\monad}}\ & {\monad(\cat{B})^{\op}}\ar@<5pt>[l]^-{\sem_{\monad}},
} 
\]
where $\str_{\monad}$ sends a functor to its codensity monad, and $\sem_{\monad}$ sends a monad to the forgetful functor from its category of algebras. Furthermore, this adjunction restricts to the adjunction
\[
\xymatrix{
{\radjover{\cat{B}} }\ar@<5pt>[r]_-{\perp}^-{\str_{\monad}}\ & {\monad(\cat{B})^{\op}}\ar@<5pt>[l]^-{\sem_{\monad}}
}
\]
described in Proposition~\ref{prop:sem-str-adj-monad}.
\end{cor}
\begin{proof}
The first part is immediate from the Proposition~\ref{prop:codensity-natural-bij}. The right adjoint $\sem_{\monad}$ evidently takes values in the full subcategory $\radjover{\cat{B}}$, and so the adjunction does restrict as claimed. The fact that the values of $\str_{\monad}$ on the objects of $\radjover{\cat{B}}$ are as claimed in Proposition~\ref{prop:sem-str-adj-monad} is the content of the previous proposition.
\end{proof}

Recall from Remark~\ref{rem:ff-completeness-interpret} that a semantics functor being full and faithful can be thought of as a kind of completeness theorem: it says that no information is lost when passing from theories to models, and that the models do not have any algebraic structure besides that specified by the theory. The following result shows that the semantics of monads satisfies such a completeness theorem.
\begin{prop}
\label{prop:sem-mnd-ff}
The semantics functor
\[
\sem_{\monad} \from \monad(\cat{B})^{\op} \to \cmover{\cat{B}}
\]
from Corollary~\ref{cor:str-sem-adj-monad} is full and faithful. Equivalently, the counit of the adjunction in Proposition~\ref{prop:sem-str-adj-monad} is an isomorphism.
\end{prop}
\begin{proof}
This is part of Theorem~6 in~Street~\cite{street72}.
\end{proof}

\section{PROPs and PROs}
\label{sec:notions-PROP}

Lawvere theories allow us to uniformly describe algebraic structures in finite product categories. However there are many algebraic structures that make sense in more general categories: for example, one can define monoids in arbitrary monoidal categories, and commutative monoids in any symmetric monoidal category. The notion of a PROP (which stands for PROduct and Permutation category) was first developed by Mac Lane in~\cite{maclane63}, and has a relationship to symmetric monoidal categories analogous to that between Lawvere theories and finite product categories. Likewise, PROs (dropping the permutations) play this role for (non-symmetric) monoidal categories.

\begin{defn}
A \demph{PROP} is a large strict symmetric monoidal category, whose objects are the natural numbers, and whose tensor product is given on objects by addition of natural numbers. A morphism of PROPs is a strict symmetric monoidal functor that is the identity on objects.

A \demph{PRO} is a large strict monoidal category whose objects are the natural numbers and whose tensor product is given by addition. A morphism of PROs is a strict monoidal functor that is the identity on objects.

We write $\PROP$  for the category of PROPs and their morphisms, and we write $\PRO$ for the category of PROs and their morphisms.
\end{defn}

\begin{remark}
In this thesis, when we speak of ``monoidal categories'', we implicitly mean ``unbiased monoidal categories''. That is, we assume that a monoidal category $\cat{B}$ is equipped, not just with a unit object and a binary tensor product $\cat{B} \times \cat{B} \to \cat{B}$, but with a choice of $n$-fold tensor product $\cat{B}^n \to \cat{B}$ for each $n \in \nat$. In particular, there is canonical $n$-th power functor $(-)^{\tensor n} \from \cat{B} \to \cat{B}$ for each $n$, rather than just an isomorphism class of such functors. This definition is equivalent to the usual definition in which only a binary tensor product and unit are specified, in that they yield equivalent categories of monoidal categories. This is Corollary~3.2.5 in~\cite{leinster04}.
\end{remark}

\begin{defn}
\label{defn:PROP-model}
Let $\cat{P}$ be a PROP. Then a \demph{model} of $\cat{P}$ in a symmetric monoidal category $\cat{B}$ is a symmetric monoidal functor (that is, a functor that preserves the tensor, unit and symmetry up to specified coherent isomorphisms) $\Gamma \from \cat{P} \to \cat{B}$ such that, for each $n \in \nat$, the distinguished isomorphism
\[
\Gamma(1)^{\tensor n} \toby{\iso} \Gamma(n)
\]
arising from the fact that $n$ is the $n$-th tensor power of $1$ in $\cat{P}$, is in fact an \emph{identity}.

Similarly, a \demph{model} of a PRO $\cat{P}$ in a monoidal category $\cat{B}$ is a monoidal functor $\Gamma \from \cat{P} \to \cat{B}$ (that is, a functor preserving the tensor and unit up to coherent isomorphism) such that the distinguished isomorphism
\[
\Gamma(1)^{\tensor n} \toby{\iso} \Gamma(n)
\]
is an identity for each $n \in \nat$.
\end{defn}

\begin{defn}
A homomorphism $\Gamma \to \Gamma'$ between models of a PROP (respectively PRO) $\cat{P}$ is a monoidal natural transformation $\Gamma \to \Gamma'$.
\end{defn}

\begin{defn}
\label{defn:sem-PROP-objects}
Let $\cat{B}$ be a large symmetric monoidal category and $\cat{P}$ a PROP. We write  $\mod_{\PROP} (\cat{B}$) for the category of models and model homomorphisms of $\cat{P}$ in $\cat{B}$. We write $\sem_{\PROP} (\cat{P}) \from \mod_{\PROP}(\cat{P}) \to \cat{B}$ for the functor sending a model $\Gamma$ to $\Gamma(1)$ and a homomorphism to its component at $1$.

Similarly if we let $\cat{B}$ be a large monoidal category and $\cat{P}$ a PRO, then we define
$\sem_{\PRO} (\cat{P}) \from \mod_{\PRO}(\cat{P}) \to \cat{B}$ similarly.
\end{defn}

\begin{remark}
\label{rem:prop-model-non-standard}
Definition~\ref{defn:PROP-model} is non-standard; it is more common to define a model of a PROP as a symmetric monoidal functor out of $\cat{P}$ without the requirement of strictly preserving tensor powers of $1$. The situation here is analogous to that in Remark~\ref{rem:law-model-non-standard}:  the categories of models obtained according to our definition and the standard definition are equivalent (the proof of which is similar to Lemma~\ref{lem:law-model-non-standard}) and our definition has the advantage that the forgetful functor reflects equalities as well as isomorphisms, which more closely matches how mathematicians usually think about algebraic structures.

Note that a model of a PROP as we have defined it is not the same as a \emph{strict} symmetric monoidal functor: tensor powers of $1$ are preserved strictly, but other tensor products need only be preserved up to coherent isomorphism.
\end{remark}

\begin{defn}
\label{defn:sem-PROP-morphisms}
Let $P \from \cat{P}' \to \cat{P}$ be a morphism of PROPs. We define a functor 
\[
\sem_{\PROP} (P) \from \mod_{\PROP}(\cat{P}) \to \mod_{\PROP} (\cat{P}')
\]
by sending a model $ \Gamma \from \cat{P} \to \cat{B}$ to $\Gamma \of P \from \cat{P}' \to \cat{B}$, and sending a homomorphism $\alpha \from \Gamma \to \Gamma'$ to $\alpha P$. If instead we let $P$ be a morphism of PROs, we define $\sem_{\PRO} (P) \from \mod_{\PRO} (\cat{P}) \to \mod_{\PRO}(\cat{P}')$ similarly.
\end{defn}

\begin{prop}
\label{prop:PROP-sem-functor}
Definitions~\ref{defn:sem-PROP-objects} and~\ref{defn:sem-PROP-morphisms} define functors
\[
\sem_{\PROP} \from \PROP^{\op} \to \catover{\cat{B}}
\]
and
\[
\sem_{\PRO} \from \PRO^{\op} \to \catover{\cat{B}}.
\]
\end{prop}
\begin{proof}
This is a straightforward check.
\end{proof}

These semantics functors have left adjoints that can be constructed in a straightforward manner. However, these left adjoints will fall out of the general machinery we develop in this thesis, so we postpone discussion of them until Section~\ref{sec:structure-examples}.

\section{Operads}
\label{sec:notions-operads}

Before defining operads, we first define multicategories, since these provide the context in which operads will take their models, and operads themselves are a special case of multicategories.

\begin{defn}
A \demph{multicategory} $\cat{B}$ consists of
\begin{itemize}
\item a (possibly large) set $\ob(\cat{B})$ of \demph{objects};
\item for all $n \in \nat$ and $b_1, \ldots, b_n, b \in \ob( \cat{B})$, a (possibly large) set $\cat{B}(b_1, \ldots, b_n ; b)$ of \demph{morphisms} with domain $b_1, \ldots, b_n$ and codomain $b$;
\item for all $b \in \ob (\cat{B})$ a distinguished \demph{identity morphism} $\id_b \in \cat{B}(b; b)$;
\item for all $n, k_1, \ldots, k_n \in \nat$ and $b, b_i, b_i^j \in \ob(\cat{B})$ for $i = 1, \ldots, n$ and $j = 1, \ldots, k_i$, a function
\begin{align*}
\cat{B}(b_1, \ldots, b_n; b) \times \cat{B}(b_1^1, \ldots, b_1^{k_1} ; b_1) \times \cdots \times \cat{B}(b_n^1, \ldots, b_n^{k_n} ; b_n) \\
\to \cat{B}( b_1^1, \ldots, b_1^{k_1}, \ldots, b_n^1, \ldots, b_n^{k_n} ; b)
\end{align*}
called composition,
\end{itemize}
satisfying identity and associativity axioms, that are described explicitly in Definition~3.5.1 of~\cite{leinster04}.
\end{defn}

\begin{defn}
A \demph{morphism of multicategories} $F \from \cat{B} \to \cat{B}'$ consists of
\begin{itemize}
\item a function $F \from \ob (\cat{B}) \to \ob (\cat{B}')$, and
\item for all $n \in \nat$ and $b_1, \ldots, b_n, b \in \ob (\cat{B} )$, a function, also denoted $F$,
\[
\cat{B}(b_1, \ldots, b_n; b) \to \cat{B}' (Fb_1, \ldots, Fb_n; Fb),
\]
\end{itemize}
that preserves identities and composition in the obvious sense.
\end{defn}

\begin{defn}
A transformation $F \to F'$ between morphisms of multicategories $\cat{B} \to \cat{B}'$ consists of, for each $b \in \ob (\cat{B})$, a unary morphism $\alpha_b \in \cat{B}'( Fb; F' b)$, such that, for any
\[
f \from b_1, \ldots, b_n \to b
\]
in $\cat{B}$, we have
\[
\alpha_b \of Ff = F'f \of (\alpha_{b_1}, \ldots, \alpha_{b_n}).
\]
\end{defn}

\begin{defn}
We write $\multicat$ for the 2-category of large multicategories, with their morphisms and transformations.
\end{defn}

\begin{defn}
An \demph{operad} is a multicategory with a single object. If $P$ is an operad with object $*$, we usually write $P(n)$ for
\[
P(\underbrace{*, \ldots,  *}_{n \text{ times}} ; *)
\]
in order to simplify notation. A morphism of operads is simply a morphism of multicategories. We write $\operad$ for the full subcategory of (the underlying $1$-category of) $\multicat$ consisting of the large operads.
\end{defn}
Note that the unique object of an operad regarded as a one-object multicategory is still part of the data defining that operad; in particular, if two operads have identical sets of operations and composition functions but distinct objects we still regard them as being distinct (albeit isomorphic). This is a minor technicality, but we will make use of it in Proposition~\ref{prop:iso-proth-operad} to obtain an isomorphism of categories, rather than an equivalence.

\begin{defn}
Given a multicategory $\cat{B}$, write $\cat{B}_0$ for the category whose objects are the objects of $\cat{B}$, and with $\cat{B}_0(b, b') = \cat{B}(b; b')$. That is, it is the category obtained by discarding all the morphism of $\cat{B}$ except for the unary ones.
\end{defn}

\begin{defn}
\label{defn:sem-operad-objects}
Let $P$ be an operad and $\cat{B}$ a multicategory. We write $\mod_{\operad} (P)$ to be the category $\multicat (P, \cat{B})$ of multicategory morphisms $P \to \cat{B}$ and transformations between them. We call the objects of $\mod_{\operad} (P)$ \demph{models} of $P$, and the morphisms \demph{homomorphisms} of $P$-models.

We write $\sem_{\operad} \from \mod_{\operad}(P) \to \cat{B}_0$ for the functor that sends a model $\Gamma \from P \to \cat{B}$ to $\Gamma(*)$ (where $*$ is the unique object of the operad $P$ regarded as a multicategory), and sends a transformation $\alpha \from \Gamma \to \Gamma'$ to its unique component $\alpha_*$.
\end{defn}

\begin{defn}
\label{defn:sem-operad-morphisms}
Given an operad morphism $F \from P' \to P$, define a functor $\sem_{\operad} (F) \from \mod_{\operad}(P) \to \mod_{\operad}(P')$ by sending a $P$-model $\Gamma \from P \to \cat{B}$ to $\Gamma \of F \from P' \to \cat{B}$, and sending a homomorphism $\alpha \from \Gamma \to \Gamma'$ to the result of whiskering $\alpha$ with $P$ in the 2-category $\multicat$, namely $\alpha P$.
\end{defn}

\begin{prop}
\label{prop:operad-sem-functor}
Definitions~\ref{defn:sem-operad-objects} and~\ref{defn:sem-operad-morphisms} together define a functor
\[
\sem_{\operad} \from \operad^{\op} \to \catover{\cat{B}_0}.
\]
\end{prop}
\begin{proof}
Again, this is a straightforward check.
\end{proof}

As with PROPs and PROs, we postpone discussion of the left adjoint to this structure functor until Section~\ref{sec:structure-examples}, although it is straightforward to describe explicitly.

\section{Monads with arities}
\label{sec:notions-monads-arites}
The theory of monads with arities generalises that of ordinary monads, and is developed in Weber \cite{weber2007} and Berger, Melli\`es and Weber~\cite{bergerMelliesWeber12}. All of the definitions and results of this section appear in \cite{bergerMelliesWeber12}. Roughly speaking, a monad with arities is a monad $\mnd{T}$ on some category $\cat{B}$ that is completely determined by the values $Ta$ for $a$ in some subcategory $\cat{A} \incl \cat{B}$, called the category of arities. The object $Tb$ for general $b \in \cat{B}$ is built up out of the $Ta$'s in a canonical way. The prototypical example to keep in mind is that of finitary monads on $\Set$ --- in this case, the category of arities consists of the finite sets.

\begin{comment}
\begin{defn}
Given a category $\cat{B}$ with full subcategory $I \from \cat{A} \incl \cat{B}$, for each $b \in \cat{B}$ there is a canonical cocone on the diagram
\[
(I \downarrow b) \to \cat{A} \incl \cat{B}
\]
whose $(Ia \toby{f} b)$-th component is just $f$, for each $Ia \in \cat{A}$ and $f \from Ia \to b$. Call this the \demph{$\cat{A}$-cocone on $b$}.
\end{defn}
\end{comment}

Recall from Definition~\ref{defn:dense-codense} that a functor $I \from \cat{A} \to \cat{B}$ is dense if its nerve functor
\[
N_{I} \from \cat{B} \to [\cat{B}^{\op}, \SET] \toby{I^*} [ \cat{A}^{\op}, \SET]
\]
is full and faithful. If $I$ is the inclusion of a full subcategory $\cat{A}$ of $\cat{B}$, then we say that $\cat{A}$ is a dense subcategory of $\cat{B}$, and write $N_{\cat{A}}$ for $N_I$.

\begin{defn}
\label{defn:cat-with-arities}
The 2-category $\ACAT$ has objects, morphisms and 2-cells as follows.
\begin{description}
\item[Objects] of $\ACAT$ are of the form $(\cat{B}, \cat{A})$ where $\cat{B}$ is a large category and $\cat{A}$ is a dense subcategory of $\cat{B}$; we call $(\cat{B},\cat{A})$ a \demph{large category with arities}.
\item[Morphisms] $(\cat{B}, \cat{A}) \to (\cat{B}',\cat{A}')$ are functors $F \from \cat{B} \to \cat{B}'$ such that the composite $N_{\cat{A}'} \of F$ sends the $\cat{A}$-cocones in $\cat{B}$ to colimit cocones in $[\cat{A}'^{\op},\SET]$. Such a functor is called \demph{arity-respecting}.
\item[2-cells] are just ordinary natural transformations between functors.
\end{description}
\end{defn}

\begin{defn}
A \demph{monad with arities} is a monad in the 2-category $\ACAT$. Explicitly, a monad with arities on a category with arities $(\cat{B}, \cat{A})$ is a monad $\mnd{T} = (T,\eta,\mu)$ on $\cat{B}$ such that the composite $N_{\cat{A}} \of T$ sends the $\cat{A}$-cones in $\cat{B}$ to colimit cocones in $[\cat{A}^{\op}, \SET]$.
\end{defn}

For the rest of this section, fix a large category with arities $(\cat{B},\cat{A})$.

\begin{defn}
\label{defn:arities-algebras}
Let $\mnd{T}$ be a monad with arities on $(\cat{B}, \cat{A})$. Write $\Theta_{\mnd{T}}$ for the full subcategory of $\cat{B}^{\mnd{T}}$ consisting of the algebras of the form $(Ta, \mu_a)$ where $a \in \cat{A}$. Write $j_{\mnd{T}} \from \cat{A} \to \Theta_{\mnd{T}}$ for $F^{\mnd{T}}\from \cat{B} \to \cat{B}^{\mnd{T}}$ with domain restricted to $\cat{A}$ and codomain restricted to $\Theta_{\mnd{T}}$.
\end{defn}

\begin{prop}
The subcategory $\Theta_{\mnd{T}}$ is dense in $\cat{B}^{\mnd{T}}$ so $(\cat{B}^{\mnd{T}}, \Theta_{\mnd{T}})$ is an object of $\ACAT$, and the forgetful functor $U^{\mnd{T}} \from \cat{B}^{\mnd{T}} \to \cat{B}$ is arity respecting. In addition, $U^{\mnd{T}}$ exhibits $(\cat{B}^{\mnd{T}}, \Theta_{\mnd{T}})$ as an Eilenberg--Moore object for the monad $\mnd{T}$ in the 2-category $\ACAT$, in the sense of Street~\cite{street72}.
\end{prop}
\begin{proof}
This is Proposition~2.3 in~\cite{bergerMelliesWeber12}.
\end{proof}

This concludes our short review of \cite{bergerMelliesWeber12}; we pause to note one consequence of the previous Proposition.

\begin{defn}
Write $(\ACAT / (\cat{B}, \cat{A}) )_{\ra}$ for the category whose objects are morphisms in $\ACAT$ whose codomain is $(\cat{B},\cat{A})$ that have a left adjoint (in the 2-category $\ACAT$), and whose morphisms are commutative triangles of arity respecting functors. Explicitly, an object of $(\ACAT/(\cat{B},\cat{A}))_{\ra}$ is an arity-respecting functor into $(\cat{B},\cat{A})$ with a left adjoint that is also arity-respecting.

Write $\Amonad (\cat{B},\cat{A})$ for the category of monads with arities on $(\cat{B},\cat{A})$, as a full subcategory of $\monad(\cat{B})$.
\end{defn}

\begin{prop}
\label{prop:monad-arities-adj}
There is an adjunction
\[
\xymatrix{
{(\ACAT/(\cat{B},\cat{A}))_{\ra} }\ar@<5pt>[r]_-{\perp}^-{\str_{\Amonad}}\ & {\Amonad(\cat{B}, \cat{A})^{\op}}\ar@<5pt>[l]^-{\sem_{\Amonad}}
}
\]
where $\sem_{\Amonad}$ sends a monad with arities $\mnd{T}$ to $(\cat{B}^{\mnd{T}}, \Theta_{\mnd{T}})$ as defined in Definition~\ref{defn:arities-algebras} with its forgetful functor to $\cat{B}$, and $\str_{\Amonad}$ sends an adjunction to the monad it induces.
\end{prop}
\begin{proof}
This follows from the fact $(\cat{B}^{\mnd{T}}, \Theta_{\mnd{T}})$ is an Eilenberg--Moore object for $\mnd{T}$ in $\ACAT$ and Theorem~6 of Street~\cite{street72}.
\end{proof}

This structure--semantics adjunction for monads with arities both generalises and specialises the adjunction for ordinary monads from Proposition~\ref{prop:sem-str-adj-monad}. It generalises it in the sense that by setting $\cat{A} = \cat{B}$ we recover the usual structure--semantics for monads, and it is a specialisation of it in the sense that both squares in
\[
\xymatrix{
{(\ACAT/(\cat{B},\cat{A}))_{\ra} }\ar@<5pt>[r]_-{\perp}^-{\str_{\Amonad}}\ar[d] & {\Amonad(\cat{B}, \cat{A})^{\op}}\ar@<5pt>[l]^-{\sem_{\Amonad}}\ar[d] \\
{(\catover{\cat{B}})_{\ra} }\ar@<5pt>[r]_-{\perp}^-{\str_{\monad}}\ & {\monad(\cat{B})^{\op}}\ar@<5pt>[l]^-{\sem_{\monad}}
}
\]
commute, where the vertical arrows are the evident forgetful functors.

\section{Monoids}
\label{sec:notions-monoids}

Our final example of a notion of algebraic theory is extremely simple: an ordinary monoid in $\SET$ can be viewed as an algebraic theory with only unary operations. A model of this theory is simply an action of the monoid. We deal with them separately rather simply noting that every monoid gives rise to, say, a Lawvere theory, because their simplicity makes their semantics much more widely applicable. Indeed, one can define actions of a monoid in any category whatsoever. For this section, we fix a large category $\cat{B}$.

\begin{defn}
\label{defn:sem-monoid-objects}
Let $M$ be a large monoid. A \demph{model} of $M$ (or an \demph{object equipped with an $M$-action}) in $\cat{B}$ consists of an object $b \in \cat{B}$ together with a monoid homomorphism $\alpha \from M \to \cat{B}(b, b)$. A \demph{homomorphism} of $M$-models from $(b, \alpha)$ to $(b', \alpha')$ consists of a morphism $h \from b \to b'$ such that
\[
\xymatrix{
M \ar[r]^{\alpha}\ar[d]_{\alpha'} & \cat{B}(b, b)\ar[d]^{h_*} \\
\cat{B}(b', b') \ar[r]_{h^*} & \cat{B}(b, b')
}
\]
commutes.

Together the $M$-models and homomorphisms in $\cat{B}$ form a category $\mod_{\MONOID}(\cat{B})$, and there is a natural forgetful functor $\sem_{\MONOID} \from \mod_{\MONOID} \to \cat{B}$.
\end{defn}

\begin{defn}
\label{defn:sem-monoid-morphisms}
Let $f \from M' \to M$ be a monoid homomorphism. We define a functor
\[
\sem_{\MONOID} (f) \from \mod_{\MONOID} (M) \to \mod_{\MONOID}(M')
\]
by sending an $M$-model $(b, \alpha)$ to the $M'$-model $(b, \alpha \of f)$, and sending an $M$-model homomorphism to the $M'$-model homomorphism with the same underlying morphism in $\cat{B}$.
\end{defn}

\begin{prop}
Together, Definitions~\ref{defn:sem-monoid-objects} and~\ref{defn:sem-monoid-morphisms} define a functor $\sem_{\MONOID} \from \MONOID^{\op} \to \catover{\cat{B}}$.
\end{prop}

\begin{defn}
Let $(U \from \cat{M} \to \cat{B}) \in \catover{\cat{B}}$. Define $\str_{\MONOID} (U)$ to be the monoid of natural transformations $U \to U$.

Now let
\[
\xymatrix{
\cat{M}' \ar[rr]^Q\ar[dr]_{U'}  & & \cat{M}\ar[dl]^{U} \\
& \cat{B}
 }
 \]
be a morphism in $\catover{\cat{B}}$. Define $\str_{\MONOID} (Q) \from \str_{\MONOID} (U) \to \str_{\MONOID}(U')$ by sending a natural transformation $\gamma \from U \to U$ to
\[
\gamma Q \from U \of Q = U' \to U \of Q = U'.
\]
This defines a functor $\str_{\MONOID} \from \catover{\cat{B}} \to \MONOID^{\op}$.
\end{defn}

\begin{prop}
\label{prop:str-sem-adj-monoid}
We have an adjunction
\[
\xymatrix{
{\catover{\cat{B}}}\ar@<5pt>[r]_-{\perp}^-{\str_{\MONOID}}\ & {\MONOID^{\op}.}\ar@<5pt>[l]^-{\sem_{\MONOID}}
}
\]
\end{prop}
\begin{proof}
Given a monoid $M$ and a functor $U \from \cat{M} \to \cat{B}$, we sketch a bijection
\[
\catover{\cat{B}}(U, \sem_{\MONOID} (M)) \iso \MONOID(M, \str_{\MONOID}(U));
\]
the remaining details are straightforward to fill in.

Let
\[
\xymatrix{
\cat{M}\ar[rr]^-{P}\ar[dr]_U && \mod_{\MONOID}(M)\ar[dl]^{\sem_{\MONOID}(M)} \\
& \cat{B} &
}
\]
be a morphism in $\catover{\cat{B}}$. Then we define a monoid homomorphism $\widebar{P} \from M \to \str_{\MONOID}(U)$ as follows. For $x \in M$ and $m \in \cat{M}$, the functor $P$ equips $m$ with an $M$-action, and in particular $x$ gives rise to a morphism $Um \to Um$. These morphisms, indexed by $m \in \cat{M}$, form the components of a natural transformation $U \to U$, i.e.\ an element of $\str_{\MONOID} (U)$. We define $\widebar{P}(x)$ to be this natural transformation.

In the other direction, suppose
\[
Q \from M \to \str_{\MONOID}(U)
\]
is a monoid homomorphism. Then for each $m \in \cat{M}$, we equip $Um$ with an $M$-action by letting $x \in M$ act on $Um$ via $Q(x)_m \from Um \to Um$. This assignment of an $M$-action to $Um$ for each $m \in \cat{M}$ defines a functor $\widebar{Q} \from \cat{M} \to \sem_{\MONOID}(M)$.
\end{proof}

\section{Comparison of notions of algebraic theory}
\label{sec:notions-compare}
In this chapter we have discussed several existing notions of algebraic theory. In subsequent chapters we will describe a general notion of algebraic theory that generalises all of these, but first let us pause to summarise and compare them to one another. Each has its own notion of arity, and its own context in which models of a theory make sense. These features of the various notions of algebraic theory are summarised in Table~\ref{tab:notions}.

\begin{table}[h]
\centering
\begin{tabular}{p{4.5cm}p{4.5cm}p{4.5cm}}
\toprule
Notion of algebraic theory& Takes models in & Arities \\
\midrule
Lawvere theories & A finite product category $\cat{B}$ & Natural numbers\\
Monads & A category $\cat{B}$ & Objects of $\cat{B}$ \\
PROPs & A symmetric monoidal category $\cat{B}$ & Natural numbers \\
PROs & A monoidal category $\cat{B}$ & Natural numbers \\
Operads & A multicategory $\cat{B}$ & Natural numbers \\
Monads with arities & A category $\cat{B}$ & Objects of a dense subcategory $\cat{A}$ of $\cat{B}$ \\
Monoids & A category $\cat{B}$ & A single arity $1$ \\
\bottomrule
\end{tabular}
\caption{Notions of algebraic theory.}
\label{tab:notions}
\end{table}

\chapter{The structure--semantics adjunction}
\label{chap:adjunction}

In this chapter we define a special case of the general notion of \emph{proto-theories} and their structure--semantics adjunctions, with the fully general version being introduced in the Chapter~\ref{chap:structure}. We begin with this special case because it is easier to motivate from a conceptual point of view, although the more general notion encompasses more examples.

In Section~\ref{sec:proto-theories} we give the definition of a proto-theory and compare them to Lawvere theories in order to indicate how they serve as a kind of algebraic theory. In Section~\ref{sec:semantics-proto-theory} we discuss how we need some extra data before we can interpret a proto-theory in a given category, leading up to the definition of an aritation and the induced semantics functor. We then define the structure functor for a given aritation in Section~\ref{sec:structure-def}, and in Section~\ref{sec:str-sem-adj} we show that structure is left adjoint to semantics. In Section~\ref{sec:profunctor} we investigate proto-theories from the point of view of profunctors, and in Section~\ref{sec:str-sem-adj-monoids} we discuss the simplest examples of proto-theories, namely monoids, and their semantics.

\section{Proto-theories}
\label{sec:proto-theories}

In this section we define proto-theories in $\CAT$, which will serve as a general notion of algebraic theory.

\begin{defn}
\label{defn:proto-theory}
Let $\cat{A}$ be a large category. A \demph{proto-theory} with arities $\cat{A}$ is a functor $L \from \cat{A} \to \cat{L}$ that is bijective on objects. A morphism of proto-theories from $L' \from \cat{A} \to \cat{L}' $ to $L \from \cat{A} \to \cat{L}$ consists of a functor $P \from \cat{L}' \to \cat{L}$ such that $P \of L' = L$. We write $\proth(\cat{A})$ for the category of proto-theories with arities $\cat{A}$ and their morphisms.
\end{defn}

Proto-theories clearly generalise Lawvere theories, as defined in Definition~\ref{defn:lawvere-theory}: a Lawvere theory is just a proto-theory with arities $\fin^{\op}$, with the extra condition that $L \from \fin^{\op} \to \cat{L}$ preserves finite products. In order to gain some intuition for how a proto-theory can be thought of as a kind of algebraic theory, it is therefore useful to consider how the various features of a proto-theory are interpreted in the special case of a Lawvere theory.

Fundamental to any notion of algebraic theory is a notion of \emph{arities}, that are thought of as shapes for possible configurations of elements. For classical algebraic theories, the arities are the natural numbers, and a ``configuration of elements of shape $n$'' in a set $X$ is simply an $n$-tuple of elements of $X$. Lawvere theories also have the natural numbers as their arities --- this is reflected in the fact the objects of $\fin^{\op}$ can be identified with the natural numbers. Thus, for a general proto-theory $L \from \cat{A} \to \cat{L}$, one should think of the objects of $\cat{A}$ as ``arities'' for possible configurations of elements. We will discuss what plays the role of such an ``$a$-ary configuration'' when we come to the semantics of proto-theories; for now keep in mind the intuition that for Lawvere theories, an $n$-ary configuration in a set is an $n$-tuple of elements.

In a Lawvere theory $L \from \fin^{\op} \to \cat{L}$, the operations of a particular arity $n$ are the morphisms $Ln \to L1$. An $n$-ary operation is thought of as something that transforms an $n$-tuple of elements (that is, a $n$-ary configuration) into a single element (in an abstract sense --- operations of a theory only have the ``potential'' to turn tuples into elements; this potential is only realised when they are given a concrete interpretation in a particular model). Morphisms in $\cat{L}$ with arbitrary codomain $Lm$ can then be identified with $m$-tuples of operations, since $Lm = (L1)^m$. However, for a general category of arities $\cat{A}$, there will not necessarily be a distinguished object to play the role that $1$ plays here. Thus, for a proto-theory $L \from \cat{A} \to \cat{L}$, if we want to think of a morphism $l \in \cat{L}(La', La)$ as an operation of the proto-theory, then such an operation will not only have an \emph{arity} of its input, given by $a'$, but also a \emph{shape} of its output, given by $a$. Thus, $l$ should be though of as something that has the potential to turn $a'$-ary configurations into $a$-ary configurations (again, in an abstract sense).

For Lawvere theories, the morphisms in $\fin^{\op}$ describe permissible ways of transforming a configuration of elements of one shape into another. More precisely, if $f$ is a function $ \{1, \ldots, n\} \to \{1, \ldots, m \}$ (defining a morphism $m \to n$ in $\fin^{\op}$), and we have an $m$-tuple $(e_i)_{i=1}^m$ of elements of some set, then we can define an $n$-tuple $(e_{f(i)})_{i=1}^n$. This allows such morphisms to act on operations of a proto-theory by transforming their input arities: if $\tau$ is an operation of an algebraic theory with arity $n$, and $f$ is as above, then we obtain an $m$-ary operation $\tau[(x_{f(i)}/x_i)_{i=1}^n]$. Similarly, if $L \from \cat{A} \to \cat{L}$ is a proto-theory and we are given a morphism $f \from a'' \to a'$ in $\cat{A}$ and an operation $l \from La' \to La$ of $L$ (with arity $a'$), then we obtain a new operation $l \of Lf$ with arity $a''$ (and the same shape). Similarly morphisms in $a$ can transform the shape of the output of an operation. Thus we should think of morphisms in $\cat{A}$ as ways of transforming configurations of one shape into configurations of another; not as part of any particular theory, but as part of the underlying logic that $\cat{A}$ represents.

Composition in a Lawvere theory represents the process of building compound operations by substituting operations into one another, and should be thought of similarly in a proto-theory. Then the ``axioms'' of the proto-theory are encoded in the equations that hold between composites in $\cat{L}$.

\section{The semantics of proto-theories}
\label{sec:semantics-proto-theory}

In classical universal algebra, an interpretation of an $n$-ary operation $\omega$ from some theory on a set $X$ consists of a function $X^n \to X$; that is, it is something that turns $n$-ary configurations of elements of $X$ into elements of $X$. Let $\cat{A}$ and $\cat{B}$ be large categories, and $L \from \cat{A} \to \cat{L}$ a proto-theory. Then we would like to say something like ``an interpretation of an $a'$-ary operation $l \in \cat{L}(La', La)$ on an object $b \in \cat{B}$ is something that turns $a'$-ary configurations of elements of $b$ into elements of $b$''.

There are two problems with this: firstly, it does not take the ``output shape'' $a$ of the operation $l$ into account. This is easily fixed by amending our statement to ``an interpretation of an $a'$-ary operation $l \in \cat{L}(La', La)$ on an object $b \in \cat{B}$ is something that turns $a'$-ary configurations of elements of $b$ into $a$-ary configurations of elements of $b$''. 

Secondly, we do not yet have a notion of an ``$a$-ary configuration of elements of $b$''. We solve this problem by fiat: we just suppose that for each $a \in \cat{A}$ and $b \in \cat{B}$ there is a totality of $a$-ary configurations of elements of $b$, denoted $\lpair a, b \rpair$. We remain completely agnostic as to what sort of thing this should be, except that we want it to be functorial in both $a$ and $b$; (recall that we interpret morphisms in $\cat{A}$ as abstract ways of transforming configurations of elements, and a morphism $b \to b'$ should extend to a map between the totalities of configurations of any given shape). Thus, $\lpair a, b \rpair$ could live in some third category $\cat{C}$. This motivates the following definition.

\begin{defn}
\label{defn:aritation}
Let $\cat{A}$, $\cat{B}$ and $\cat{C}$ be large categories. An \demph{interpretation of arities from $\cat{A}$ in $\cat{B}$, with values in $\cat{C}$} (called an \demph{aritation} for short) is a functor
\[
\lpair -, - \rpair \from \cat{A} \times \cat{B} \to \cat{C}.
\]
For such an aritation, $\cat{A}$ is called the \demph{category of arities} and $\cat{B}$ is called the $\demph{base category}$. For a given such aritation, write
\[
\currylo \from \cat{B} \to [\cat{A}, \cat{C}]
\]
and
\[
\curryhi \from \cat{A} \to [\cat{B}, \cat{C}]
\]
for the functors obtained by currying $\lpair - , - \rpair$.
\end{defn}

\begin{ex}
\label{ex:canon-aritation}
Let $\cat{B}$ be a locally small category. Then the hom-functor
\[
\cat{B}(-,-) \from \cat{B}^{\op} \times \cat{B} \to \Set
\]
can be regarded as an aritation $\lpair - , - \rpair = \cat{B}(-,-)$. The corresponding
\[
\currylo \from \cat{B} \to [\cat{B}^{\op}, \Set]
\]
and
\[
\curryhi \from \cat{B}^{\op} \to [\cat{B}, \Set]
\]
are the two Yoneda embeddings. This example will be explored in more detail later in this thesis, especially in Chapters~\ref{chap:canonical1} and~\ref{chap:canonical2}.
\end{ex}

\begin{ex}
\label{ex:fin-prod-aritation}
Let $\cat{B}$ be a large finite product category. Then define an aritation
\[
\lpair - , - \rpair \from \fin^{\op} \times \cat{B} \to \cat{B}
\]
sending $(n, b) \mapsto b^n$. A version of this aritation is closely related to the semantics of Lawvere theories, as we shall see in Section~\ref{sec:structure-examples}.
\end{ex}

Recall that, if $L \from \cat{A} \to \cat{L}$ is a proto-theory, then a morphism $l \in \cat{L}(La', La)$ is thought of as an operation of input arity $a'$ and output shape $a$. Thus, an interpretation of such an operation in some object $b \in \cat{B}$ should send $a'$-ary configurations of elements of $b$ to $a$-ary configurations; that is, it should be a morphism $\lpair a', b \rpair \to \lpair a, b \rpair$. A model of the proto-theory should be an object with an interpretation of each such operation, respecting the process of substitution of operations and the axioms of the proto-theory, which, recall, are encoded in the composition of $\cat{L}$. Putting this all together gives the following definition.

\begin{defn}
\label{defn:sem-objects}
Let $\lpair - , - \rpair \from \cat{A} \times \cat{B} \to \cat{C}$ be an aritation, and $L \from \cat{A} \to \cat{L}$ a proto-theory with arities $\cat{A}$. Then the category $\mod(L)$ and the functors $\sem(L) \from \mod(L) \to \cat{B}$ and $I(L) \from \mod(L) \to [\cat{L}, \cat{C}]$ are defined by the following pullback square in $\CAT$:
\[
\xymatrix{
{\mod(L)}\ar[r]^{I(L)}\ar[d]_{\sem(L)}\pullbackcorner  & {[\cat{L}, \cat{C} ] }\ar[d]^{L^*} \\
{\cat{B}}\ar[r]_{\currylo} & {[\cat{A}, \cat{C}].}
}
\]
We call $\mod(L)$ the \demph{category of models} of $L$.
\end{defn}

Note that $\mod(L)$ and $\sem(L)$ depend crucially on the aritation $\lpair-, - \rpair$. This dependence is usually clear from the context, so we do not make it explicit.

\begin{defn}
\label{defn:models-explicit}
Let $\lpair - , - \rpair \from \cat{A} \times \cat{B} \to \cat{C}$ be an aritation, and $(L \from \cat{A} \to \cat{L}) \in \proth(\cat{A})$. We introduce the following terminology for the objects and morphisms of $\mod(L)$.
\begin{enumerate}
\item
\label{part:models-explicit-models}
An object of $\mod(L)$ is called a \demph{model of $L$} or an \demph{$L$-model}. Explicitly, an $L$-model $x \in \mod(L)$ consists of an object $d^x \in \cat{B}$ and a functor $\Gamma^x \from \cat{L} \to \cat{C}$ such that
\[
\Gamma^x \of L = \lpair -, d^x \rpair \from \cat{A} \to \cat{C}.
\] 
\item
\label{part:models-explicit-homomorphisms}
A morphism in $\mod(L)$ is called an \demph{$L$-model homomorphism}. Explicitly, an $L$-model homomorphism $x \to y$ consists of a morphism $h \from d^x \to d^y$ in $\cat{B}$ such that for every $l \in \cat{L}(La' , La)$, the square
\[
\xymatrix{
{\big(\lpair a', d^x \rpair = \Gamma^x (La')\big)}\ar[r]^{\lpair a', h \rpair}\ar[d]_{\Gamma^x(l)} & {\big(\lpair a', d^y \rpair = \Gamma^y (La')\big)}\ar[d]^{\Gamma^y(l)} \\
{\big(\lpair a, d^x \rpair = \Gamma^x (La)\big)}\ar[r]^{\lpair a, h \rpair} & {\big(\lpair a, d^y \rpair = \Gamma^y (La)\big)}
}
\]
commutes.
\end{enumerate}
\end{defn}

Let $x = (d^x, \Gamma^x)$ be a model of a proto-theory $L \from \cat{A} \to \cat{L}$ for a particular aritation $\lpair-,- \rpair \from \cat{A} \times \cat{B} \to \cat{C}$. Then $\Gamma^x$ provides an interpretation of the operations of the proto-theory in the following sense. An operation of $L$ is a morphism $l \from La' \to La$ in $\cat{L}$. The functor $\Gamma^x$ gives an interpretation of such an operation as a morphism
\[
\lpair a', d^x \rpair \to \lpair a, d^x \rpair,
\]
which we can think of as a way of turning $a'$-ary configurations of elements of $d^x$ into $a$-ary configurations as desired. The functoriality of $\Gamma^x$ means that these interpretations respect the process of substituting one operation into another, which is encoded in the composition of $\cat{L}$.

A homomorphism of $L$-models is simply a map between their underlying objects that commutes with the interpretations of each operation of $L$. The forgetful functor $\sem(L) \from \mod(L) \to \cat{B}$ sends an $L$-model to its underlying object in $\cat{B}$ and sends a homomorphism to its underlying morphism.
\begin{defn}
\label{defn:sem-morphisms}
Let $\lpair -, - \rpair \from \cat{A} \times \cat{B} \to \cat{C}$ be an aritation, and let
\[
\xymatrix{
\cat{L}'\ar[rr]^P & & \cat{L} \\
& \cat{A}\ar[ul]^{L'}\ar[ur]_{L}  &
}
\]
be a morphism in $\proth(\cat{A})$. Then $\sem(P) \from \mod(L) \to \mod(L')$ is defined to be the unique functor making
\[
\xymatrix{
{\mod(L)}\ar[rrr]^{I(L)} \ar[dr]^{\sem(P)} \ar[dd]_{\sem(L)} & & & {[\cat{L}, \cat{C} ]}\ar[dl]_{P^*}\ar[dd]^{L^*} \\
& {\mod(L')}\ar[r]^{I(L')}\ar[dl]^{\sem(L')} & {[\cat{L}', \cat{C}]}\ar[dr]_{L'^*} & \\
 {\cat{B}}\ar[rrr]_{\currylo} & & &{[\cat{A}, \cat{C}]} 
}
\]
commute. The universal property of the pullback defining $\mod(L')$ ensures that such a $\sem(P)$ exists and is unique.

Explicitly, $\sem(P)$ sends an $L$-model $x$ to the $L'$-model $\sem(P)(x)$ with $d^{\sem(P)(x)} = d^x$ and $\Gamma^{\sem(P)(x)} = \Gamma^x \of P$, and is the identity on morphisms.
\end{defn}

A morphism of proto-theories 
\[
\xymatrix{
\cat{L}'\ar[rr]^P & & \cat{L} \\
& \cat{A}\ar[ul]^{L'}\ar[ur]_{L}  &
}
\]
is an interpretation of $L'$ in $L$; that is, it assigns to every operation of $L'$ a corresponding operation of $L$ of the appropriate arity and shape. This gives a canonical way of turning an $L$-model $x = (d^x, \Gamma^x)$ into an $L'$-model: given an operation for the proto-theory $L$, its interpretation for the new $L'$-model structure on $d^x$ is the interpretation (for the $L$-model $x$) of the operation of $L$ that it is sent to by $P$. This is precisely what the functor $\sem(P) \from \mod(L) \to \mod(L')$ does.

\begin{prop}
The assignments denoted $\sem$ in Definitions~\ref{defn:sem-objects} and~\ref{defn:sem-morphisms} together define a functor
\[
\sem \from \proth(\cat{A})^{\op} \to \catover{\cat{B}}.
\]
\end{prop}
\begin{proof}
Functoriality of $\sem$ is immediate from the universal property of pullbacks.
\end{proof}

Let us consider the semantics functors that arise from the aritations of Examples~\ref{ex:canon-aritation} and~\ref{ex:fin-prod-aritation}.

\begin{ex}
For a locally small category $\cat{B}$, the aritation defined in Example~\ref{ex:canon-aritation} gives rise to a functor
\[
\sem \from \proth(\cat{B}^{\op})^{\op} \to \catover{\cat{B}}.
\]
Given a proto-theory $L \from \cat{B}^{\op} \to \cat{L}$, a model of $L$ (that is, an object of $\mod(L)$) consists of an object $d^x \in \cat{B}$ together with a functor $\Gamma^x \from \cat{L} \to \Set$ such that the composite $\Gamma^x \of L$ is the representable $\cat{B}(-, d^x) \from \cat{B}^{\op} \to \Set$. An $L$-model homomorphism $(d^x, \Gamma^x) \to (d^y, \Gamma^y)$ consists of a morphism $h \from d^x \to d^y$ such that the natural transformation
\[
h_* \from \cat{B}(-, d^x) \to \cat{B}(-, d^y)
\]
extends to a (necessarily unique) natural transformation $\Gamma^x \to \Gamma^y$.
\end{ex}

\begin{ex}
Let $\cat{B}$ be a large finite product category, and consider the aritation defined in Example~\ref{ex:fin-prod-aritation}. This aritation gives rise to a functor
\[
\proth(\fin^{\op})^{\op} \to \catover{\cat{B}}.
\]
Given a proto-theory $L \from \fin^{\op} \to \cat{L}$, a model of $L$ consists of an object $d^x \in \cat{B}$ together with a functor $\Gamma^x \from \cat{L} \to \cat{B}$ such that the composite $\Gamma^x \of L$ is the functor $(d^x)^{(-)} \from \fin^{\op} \to \cat{B}$. An $L$-model homomorphism $(d^x, \Gamma^x) \to (d^y, \Gamma^y)$ consists of a morphism $h \from d^x \to d^y$ such that for every $l \from Ln \to Lm$ in $\cat{L}$, the square
\[
\xymatrix{
(d^x)^n \ar[r]^{h^n}\ar[d]_{\Gamma^x l} & (d^y)^n\ar[d]^{\Gamma^y l} \\
(d^x)^m \ar[r]_{h^m} & (d^y)^{m}
}
\]
commutes. In particular, if $L \in \proth(\fin^{\op})$ is a Lawvere theory (that is, $L$ preserves finite products) then the notions of $L$-model and $L$-model homomorphism agree with those for Lawvere theories.
\end{ex}

\section{The structure functor}
\label{sec:structure-def}

Let $\lpair -,- \rpair \from \cat{A} \times \cat{B} \to \cat{C}$ be an aritation. Let  $(U \from \cat{M} \to \cat{B})\in \catover{\cat{B}}$ be any functor, but let us think of it for now as a forgetful functor, so that $U$ sends objects and morphisms of $\cat{M}$ to their ``underlying'' objects and morphisms in $\cat{B}$. Let $L \from \cat{A} \to \cat{L}$ be a proto-theory, and consider a morphism $Q \from U \to \sem(L)$ in $\catover{\cat{B}}$. Such a morphism is a way of assigning, for each object of $\cat{M}$, an $L$-model structure to its underlying object in $\cat{B}$ in such a way that the underlying morphism of each morphism in $\cat{M}$ becomes a homomorphism between the corresponding $L$-models.

If there were an initial such $Q$ (that is, an initial object in the comma category $(U \downarrow \sem)$), then we could think of the corresponding $\mod(L)$ as the ``best approximation to $\cat{M}$ by algebraic structure''. The proto-theory $L$ would in some sense describe the most general kind of algebraic structure possessed by all the objects of $\cat{M}$, in that for any other proto-theory $L'$, an assignment of $L'$-model structures to all objects of $\cat{M}$ (that is, a morphism $U \to \sem(L')$ in $\catover{\cat{B}}$) would be the same thing as an interpretation of $L'$ in $L$.

Of course, the existence of such an algebraic approximation for each $U$ is the same thing as the existence of a left adjoint to $\sem \from \proth(\cat{A})^{\op} \to \catover{\cat{B}}$; in this section we construct such an adjoint.

\begin{defn}
\label{defn:str-objects}
Let $\lpair -, - \rpair \from \cat{A} \times \cat{B} \to \cat{C}$ be an aritation and let $(U \from \cat{M} \to \cat{B}) \in \catover{\cat{B}}$. Define a category $\thr(U)$ and functors $\str(U) \from \cat{A} \to \thr(U)$ and $J(U) \from \thr(U) \to [\cat{M}, \cat{C}]$ via the bijective-on-objects/full-and-faithful factorisation of the composite
\[
\cat{A} \toby{\curryhi} [\cat{B}, \cat{C}] \toby{U^*} [ \cat{M},\cat{C}].
\]
That is, $\str(U) \from \cat{A} \to \thr(U)$ is the identity on objects, and $J(U) \from \thr(U) \to [\cat{M}, \cat{C}]$ is full and faithful, and
\[
\xymatrix{
{\cat{A}}\ar[r]^{\curryhi} \ar[d]_{\str(U)} & {[\cat{B},\cat{C}]} \ar[d]^{U^*} \\
{\thr(U)}\ar[r]_{J(U)} & {[\cat{M}, \cat{C}]}
}
\]
commutes.
\end{defn}

 Explicitly, the objects of $\thr(U)$ are the objects of $\cat{A}$, and a morphism $a' \to a$ in $\thr(U)$ is a natural transformation
\[
\lpair a', U-\rpair \to \lpair a, U-\rpair,
\]
and composition is the usual composition of natural transformations. The functor $\str(U)$ is the identity on objects, and sends a morphism $f \from a' \to a$ to the natural transformation
\[
\lpair f, - \rpair \from \lpair a', U- \rpair \to \lpair a , U- \rpair.
\]
The functor $J(U) \from \thr(U) \to [\cat{M}, \cat{C}]$ sends an object $a$ to the functor $\lpair a, U- \rpair$ and is the identity on morphisms.

\begin{lem}
\label{lem:str-well-defined}
Let $\lpair -, - \rpair \from \cat{A} \times \cat{B} \to \cat{C}$ be an aritation and let 
\[
\xymatrix{
\cat{M}'\ar[rr]^Q\ar[dr]_{U'} & & \cat{M}\ar[dl]^{U} \\
& \cat{B} &
}
\]
be a morphism in $\catover{\cat{B}}$. Then the square
\[
\xymatrix{
{\cat{A}} \ar[rr]^{\str(U')} \ar[d]_{\str(U)} & & {\thr(U')}\ar[d]^{J(U')} \\
{\thr(U)} \ar[r]_{J(U)} & {[\cat{M}, \cat{C}]} \ar[r]_{Q^*} & {[\cat{M}', \cat{C}]}
} 
\]
commutes, and there is a unique functor $\str(Q) \from \thr(U) \to \thr(U')$ making
\[
\xymatrix{
{\cat{A}} \ar[rr]^{\str(U')} \ar[d]_{\str(U)} & & {\thr(U')}\ar[d]^{J(U')} \\
{\thr(U)}\ar[urr]^{\str(Q)} \ar[r]_{J(U)} & {[\cat{M}, \cat{C}]} \ar[r]_{Q^*} & {[\cat{M}', \cat{C}]}
}
\]
commute.
\end{lem}

\begin{proof}
In the diagram
\[
\xymatrix{
{\cat{A}} \ar[rr]^{\str(U')} \ar[dd]_{\str(U)}\ar[dr]^{\curryhi} & & {\thr(U')}\ar[dd]^{J(U')} \\
& {[\cat{B},\cat{C}]}\ar[d]_{U^*}\ar[dr]^{U'^*} & \\
{\thr(U)} \ar[r]_{J(U)} & {[\cat{M}, \cat{C}]} \ar[r]_{Q^*} & {[\cat{M'}, \cat{C}],}
} 
\]
the upper right triangle and lower left quadrilateral commute by definition of $\str$ (Definition~\ref{defn:str-objects}), and the lower right triangle commutes since $ U \of Q = U'$. Hence the outer square commutes.

The existence and uniqueness of $\str(Q) \from \thr(U) \to \thr(U')$ then follows from the fact that $\str(U)$ is bijective on objects and $J(U')$ is full and faithful, and the fact that the bijective-on-objects and full-and-faithful functors form a factorisation system on $\CAT$ (Lemma~\ref{lem:bo-ff-factorisation}).
\end{proof}

\begin{defn}
\label{defn:str-morphisms}
Let $\lpair -, - \rpair \from \cat{A} \times \cat{B} \to \cat{C}$ be an aritation and let $Q$ be a morphism from $ U' \from \cat{M}' \to \cat{B}$ to $U \from \cat{M} \to \cat{B}$ in $\catover{\cat{B}}$. Then $\str(Q) \from \thr(U) \to \thr(U')$ is defined as in Lemma~\ref{lem:str-well-defined}. Explicitly, $\str(Q)$ is the identity on objects, and sends a natural transformation $\gamma \from \lpair a , U- \rpair \to \lpair a', U - \rpair$ to the natural transformation
\[
\gamma Q \from \lpair a, U \of Q - \rpair = \lpair a, U' - \rpair \to \lpair a', U \of Q - \rpair = \lpair a', U' - \rpair.
\]
\end{defn}

\begin{prop}
Given an aritation $\lpair -, - \rpair \from \cat{A} \times \cat{B} \to \cat{C}$, the assignments denoted $\str$ in Definitions~\ref{defn:str-objects} and~\ref{defn:str-morphisms} together define a functor $\str \from \catover{\cat{B}} \to \proth(\cat{A})^{\op}$.
\end{prop}
\begin{proof}
We must show that if we have morphisms
\[
\xymatrix{
\cat{M}'' \ar[r]^{Q'}\ar[dr]_{U''} & \cat{M}'\ar[r]^Q\ar[d]_{U'} & \cat{M}\ar[dl]^U \\
& \cat{B} &
}
\]
\begin{comment}
\[
(U'' \from \cat{M}'' \to \cat{B}) \toby{Q'} (U' \from \cat{M}' \to \cat{B}) \toby{Q} (U \from \cat{M} \to \cat{B})
\]
\end{comment}
	in $\catover{\cat{B}}$, then $\str(Q \of Q') = \str(Q') \of \str(Q)$. The diagram
\[
\xymatrix{
{\cat{A}}\ar[rrr]^{\str(U'')}\ar[dr]^{\str(U')} \ar[ddd]_{\str(U)}  & & & {\thr(U'')}\ar[ddd]^{J(U'')} \\
& {\thr(U')}\ar[dr]_{J(U')}\ar[urr]_{\str(Q')} & & & \\
& & {[\cat{M}',\cat{C}]}\ar[dr]^{Q'^*} & \\
{\thr(U)}\ar[uur]_{\str(Q)}\ar[r]_{J(U)}  & {[\cat{M},\cat{C}]}\ar[rr]_{(Q \of Q')^*}\ar[ur]^{Q^*} & & {[\cat{M}'', \cat{C}]}
}
\]
commutes, and so $\str(Q') \of \str(Q)$ is a diagonal fill-in for the outer square. But so is $\str(Q \of Q')$ by definition, so by uniqueness, $\str(Q \of Q') = \str(Q') \of \str(Q)$.
\end{proof}

Let us examine the structure functor in more detail for the aritations defined in Examples~\ref{ex:canon-aritation} and~\ref{ex:fin-prod-aritation}.

\begin{ex}
Let $\cat{B}$ be a locally small category and consider the aritation from Example~\ref{ex:canon-aritation}. This aritation gives rise to a structure functor
\[
\str \from \catover{\cat{B}} \to \proth(\cat{B}^{\op})^{\op};
\]
let us examine what this functor does explicitly. Given $U \from \cat{M} \to \cat{B}$, the category $\thr(U)$ has the same objects as $\cat{B}$ and a morphism $b \to b'$ in $\thr(U)$ is a natural transformation
\[
\cat{B}(b, U-) \to \cat{B}(b', U-),
\]
and $\str(U) \from \cat{B}^{\op} \to \thr(U)$ sends $f \from b' \to b$ to $f^* \from \cat{B}(b, U-) \to \cat{B}(b', U-)$.
\end{ex}

\begin{ex}
Let $\cat{B}$ be a finite product category. Then the aritation from Example~\ref{ex:fin-prod-aritation} gives rise to a structure functor
\[
\str \from \catover{\cat{B}} \to \proth(\fin^{\op})^{\op}.
\]
Given $U \from \cat{M} \to \cat{B}$, the category $\thr(U)$ has the natural numbers as objects, and a morphism $n \to m$ in $\thr(U)$ is a natural transformation $U^n \to U^m$.
\end{ex}

\section{The structure--semantics adjunction}
\label{sec:str-sem-adj}

Throughout this section, fix an aritation $\lpair -, - \rpair \from \cat{A} \times \cat{B} \to \cat{C}$. We show that for any such aritation, $\str$ does indeed provide a left adjoint for $\sem$. We do this by establishing, for $L \in \proth(\cat{A})$ and $U \in \catover{\cat{B}}$, a bijection of hom-sets
\[
\catover{\cat{B}}(U, \sem(L)) \iso \proth(\cat{A})(L, \str(U))
\]
that is natural in $L$ and $U$.

\begin{lem}
Let
\[
\xymatrix{
\cat{M}\ar[rr]^R\ar[dr]_{U} && \mod(L)\ar[dl]^{\sem(L)} \\
& \cat{B} &
}
\]
be a morphism in $\catover{\cat{B}}$, where $(U \from \cat{M} \to \cat{B}) \in \catover{\cat{B}}$, and $(L \from \cat{A} \to \cat{L}) \in \proth(\cat{A})$. Then for every $l \from La' \to La$ in $\cat{L}$, there is a natural transformation
\[
\Gamma^{R(-)} (l) \from \lpair a', U - \rpair \to \lpair a, U- \rpair
\]
with components
\[
\Gamma^{Rm} (l) \from \Gamma^{Rm} (La') =  \lpair a', U m \rpair \to \Gamma^{Rm} (La) =  \lpair a, U m \rpair.
\]
\end{lem}
\begin{proof}
We must show that for each morphism $f \from m \to m'$ in $\cat{M}$, the square
\[
\xymatrix{
{\lpair a', Um \rpair} \ar[r]^{\lpair a', Uf \rpair} \ar[d]_{\Gamma^{Rm}(l)} & {\lpair a', Um' \rpair} \ar[d]^{\Gamma^{Rm'} (l)} \\
{\lpair a, Um \rpair} \ar[r]_{\lpair a, Uf \rpair} & {\lpair a, Um' \rpair}
}
\]
commutes. But since $\sem (L) \of R = U$, we know that $Uf$ is the underlying map of the $L$-model homomorphism $Rf \from Rm \to Rm'$. So the commutativity of this square follows from the definition of $L$-model homomorphism (Definition~\ref{defn:models-explicit}\bref{part:models-explicit-homomorphisms}).
\end{proof}

We now construct, for $U \in \catover{\cat{B}}$ and $L \in \proth(\cat{A})$, a function (in fact a bijection)
\[
\Psi_{U,L} \from \catover{\cat{B}}(U, \sem(L)) \to \proth(L, \str(U)).
\]

\begin{defn}
\label{defn:adj-bij}
Let
\[
\xymatrix{
\cat{M}\ar[rr]^R\ar[dr]_{U} && \mod(L)\ar[dl]^{\sem(L)} \\
& \cat{B} &
}
\]
be a morphism in $\catover{\cat{B}}$, where $(U \from \cat{M} \to \cat{B}) \in \catover{\cat{B}}$, and $(L \from \cat{A} \to \cat{L}) \in \proth(\cat{A})$. Then we define $\Psi_{U, L} (R) \from \cat{L} \to \thr(U)$ as follows.
\begin{description}
\item [On objects:] For an arbitrary object $La \in \cat{L}$, define $\Psi_{U, L} (R)(La) = a$ (Recalling that $L$ is bijective on objects so any object of $\cat{L}$ is of this form for a unique $a$, and the objects of $\thr(U)$ are exactly the objects of $\cat{A}$);
\item [On morphisms:] Given $l \from La' \to La$ in $\cat{L}$, define $\Psi_{U, L} (R)(l)$ to be the natural transformation $\Gamma^{R(-)} (l) \from \lpair a', U-\rpair \to \lpair a , U- \rpair$ from the previous lemma.
\end{description}
It is clear that this is a functor $\cat{L} \to \thr(U)$. We will usually omit explicit mention of $U$ and $L$ and write $\Psi_{U,L} = \Psi$.
\end{defn}

\begin{lem}
For $U$, $L$ and $R$ as above, we have $\Psi (R) \of L = \str(U)$.
\end{lem}
\begin{proof}
It is clear that the two functors are equal on objects. Given $f \from a' \to a$ in $\cat{A}$, the natural transformation $\str(U)(f)  \from \lpair a', U- \rpair \to \lpair a, U- \rpair$ has components
\[
\lpair f, Um \rpair \from \lpair a' , Um \rpair \to \lpair a, Um \rpair,
\]
for each $m \in \cat{M}$, whereas $\Psi (R) (Lf)$ has components $ \Gamma^{Rm} (Lf)$ for $m \in \cat{M}$. But by the definition of $L$-algebra (Definition~\ref{defn:models-explicit}\bref{part:models-explicit-models}), we have $\Gamma^{Rm} \of L = \lpair - , d^{Rm} \rpair = \lpair -, Um \rpair$. Hence $\str(U)(f) = \Psi(R)(Lf)$, as required.
\end{proof}

We have constructed a mapping $\Psi \from \catover{\cat{B}}(U, \sem(L)) \to \proth(\cat{A})(L, \str(U))$; to establish that $\str \ladj \sem$, we must show that $\Psi$ is a bijection and is natural in $U$ and $L$.

\begin{lem}
\label{lem:adj-nat}
The mapping $\Psi_{U,L}$ is natural in $U$ and $L$.
\end{lem}
\begin{proof}
First we show $\Psi$ is natural in $U$. Let
\[
\xymatrix{
\cat{M}'\ar[rr]^Q\ar[dr]_{U'} && \cat{M}\ar[dl]^{U} \\
& \cat{B} &
}
\]
and
\[
\xymatrix{
\cat{M} \ar[rr]^{R}\ar[dr]_U & & \mod(L)\ar[dl]^{\sem(L)} \\
& \cat{B} &
}
\]
be morphisms in $\catover{\cat{B}}$. We must show that $\Psi_{U', L}(R \of Q) = \str (Q) \of \Psi_{U,L} (R)$. It is clear that these two functors $\cat{L} \to \thr(U')$ are equal on objects. Let $l \from La' \to La$ in $\cat{L}$. Then both $\Psi_{U', L}(R \of Q) (l)$ and $\str (Q) \of \Psi_{U,L} (R) (l)$ are natural transformations $\lpair a', U' - \rpair \to \lpair a, U' - \rpair$, and, taking the component at $m' \in \cat{M}'$, we have
\begin{align*}
(\str (Q) \of \Psi_{U,L} (R))_{m'} &= \Psi_{U, L} (R)_{Q(m')} && \text{(by Definition~\ref{defn:str-morphisms})}\\
&= \Gamma^{R \of Q (m')} (l) && \text{(by Definition~\ref{defn:adj-bij})} \\
&= \Psi_{U',L} (R \of Q) (l)_{m'}  && \text{(by Definition~\ref{defn:adj-bij})} \\
\end{align*}
as required. Now let
\[
\xymatrix{
\cat{L}'\ar[rr]^P && \cat{L} \\
& \cat{A}\ar[ul]^{L'}\ar[ur]_{L} &
}
\]
be a morphism in $\proth(\cat{A})$ and let $R \from U \to \sem(L)$ in $\catover{\cat{B}}$ as before. We must show that $\Psi_{U,L}(R) \of P = \Psi_{U, L'}(\sem(P) \of R)$. As before, these functors are clearly equal on objects. Let $l' \from L' a' \to L' a$. Then, taking the component at $m \in \cat{M}$, we have
\begin{align*}
\Psi_{U,L} (R) (Pl')_m &= \Gamma^{Rm}(Pl') && \text{(by Definition~\ref{defn:adj-bij})} \\
& = \Gamma^{Rm} \of P (l') && \\
& = \Gamma^{(\sem(P) \of R) m} (l') && \text{(by Definition~ \ref{defn:sem-morphisms})} \\
& = \Psi_{U, L'}(\sem(P) \of R)(l')_m && \text{(by Definition~ \ref{defn:adj-bij})}
\end{align*}
as required.
\end{proof}

We now construct an inverse
\[
\Theta_{U,L} \from \proth(A)(L, \str(U)) \to \catover{\cat{B}}(U, \sem(L))
\]
to $\Psi_{U,L}$.

\begin{defn}
\label{defn:adj-bij-inv}
Let
\[
\xymatrix{
\cat{L}\ar[rr]^S & & \thr(U) \\
& \cat{A}\ar[ul]^L \ar[ur]_{\str(U)}
}
\]
be a morphism in $\proth(\cat{A})$, where $(U \from \cat{M} \to \cat{B})\in \catover{\cat{B}}$. Define
\[
\Theta_{U,L} (S) \from \cat{M} \to \mod(L)
\]
as follows.
\begin{description}
\item [On objects:] Given $m \in \cat{M}$, define $d^{\Theta_{U,L} (S)(m)} = Um$ and, for $l \from La' \to La$ in $\cat{L}$, define 
\[
\Gamma^{ \Theta_{U,L} (S)(m)} (l) = S(l)_m \from \lpair a', Um \rpair \to \lpair a, Um \rpair.
\]
\item [On morphisms:] Given a morphism $h \from m \to m'$, we define $\Theta_{U,L} (S) (h) \from \Theta_{U,L} (S)(m) \to \Theta_{U,L} (S)(m')$ to be the $L$-model homomorphism with underlying morphism $Uh \from Um \to Um'$ in $\cat{B}$.
\end{description}
We will omit mention of $U$ and $L$ and write $\Theta_{U,L} = \Theta$ when it is convenient and unambiguous to do so.
\end{defn}
We must check that this definition makes sense.
\begin{lem}
The functor $\Theta_{U,L} (S) \from \cat{M} \to \mod(L)$ described in Definition~\ref{defn:adj-bij-inv} is well-defined.
\end{lem}
\begin{proof}
We must check that for $m \in \cat{M}$ the proposed definition of $\Theta(S)(m)$ is indeed an $L$-model, and that for $h \from m \to m'$, the map $Uh$ does give a homomorphism $\Theta (S)(m) \to \Theta (S)(m')$.

First we must show that for $f \from a' \to a$ in $\cat{A}$, we have
\[
\Gamma^{ \Theta (S)(m)} (Lf) = \lpair f, Um \rpair \from \lpair a', Um \rpair \to \lpair a, Um \rpair.
\]
But
\[
\Gamma^{ \Theta (S)(m)} (Lf) = S ( Lf)_m = \str(U) (f)_m = \lpair f, Um \rpair,
\]
since $S \of L = \str(U)$.

To check that $Uh$ gives a homomorphism of $L$-models, we must check that, for each $l \from La' \to La$, the square
\[
\xymatrix{
{\lpair a', Um \rpair}\ar[r]^{\lpair a', Uh \rpair}\ar[d]_{\Gamma^{\Theta (S)(m)}(l) = S (l)_m} & {\lpair a', Um' \rpair }\ar[d]^{\Gamma^{\Theta(S)(m')}(l) = S (l)_{m'}} \\
{\lpair a, Um \rpair}\ar[r]_{\lpair a, Uh \rpair} & {\lpair a, Um' \rpair, }
}
\]
commutes, but this is simply a naturality square for $S(l) \from \lpair a', U- \rpair \to \lpair a , U -  \rpair$.
\end{proof}

\begin{lem}
\label{lem:adj-bij}
The mappings 
\[
\Psi \from \catover{\cat{B}} ( U, \sem(L)) \to \proth(\cat{A})(L, \str(U))
\]
and 
\[
\Theta \from  \proth(\cat{A})(L, \str(U)) \to  \catover{\cat{B}} ( U, \sem(L))
\]
are inverse bijections.
\end{lem}
\begin{proof}
Let
\[
\xymatrix{
\cat{M}\ar[rr]^{R}\ar[dr]_U && \mod(L)\ar[dl]^{\sem(L)} \\
& \cat{B} &
}
\]
be a morphism in $\catover{\cat{B}}$. We must show that $ \Theta \Psi (R) = R$. First we show they are equal on objects. Let $m \in \cat{M}$. Then
\[
d^{\Theta \Psi (R) (m)} = \sem(L) \of \Theta  \Psi (R) (m) = Um = \sem(L) \of R (m) = d^{Rm},
\]
so $\Theta  \Psi (R)(m)$ and $Rm$ have the same underlying object. Let $l \from La' \to La$ in $\cat{A}$. Then
\begin{align*}
\Gamma^{\Theta \Psi (R)(m)} (l) &= \Psi (R) (l)_m && \text{(by Definition~\ref{defn:adj-bij-inv})} \\
& = \Gamma^{Rm}(l) && \text{(by Definition~\ref{defn:adj-bij})}
\end{align*}
so $\Gamma^{\Theta \Psi (R)(m)} = \Gamma^{Rm}$.  Thus $\Theta \Psi (R)(m) = Rm$, as required.

Now, note that
\[
\sem(L) \of \Theta  \Psi (R) = U = \sem(L) \of R.
\]
Since $\sem(L)$ is faithful by construction, it follows that $\Theta \Psi (R)$ and $R$ are equal on morphisms, hence $\Theta \Psi (R) = R$.

Now let
\[
\xymatrix{
\cat{L}\ar[rr]^S && \thr(U) \\
& \cat{A}\ar[ul]^L\ar[ur]_{\str(U)} &
}
\]
be a morphism in $\proth(\cat{A})$. We must show that $ \Psi \Theta (S) = S$. Evidently they are equal on objects, since they are both morphisms $L \to \str(U)$ in $\proth(\cat{A})$. Suppose $l \from La' \to La$ in $\cat{L}$. Then, for $m \in \cat{M}$, we have
\begin{align*}
\Psi \Theta (S) (l) _m & = \Gamma^{\Theta(S)(m)} (l) && \text{(by Definition~\ref{defn:adj-bij})} \\
& = S(l)_m && \text{(by Definition~\ref{defn:adj-bij-inv})}
\end{align*}
as required. Hence $\Theta$ and $\Psi$ are mutually inverse.
\end{proof}

\begin{thm}
\label{thm:str-sem-adj}
We have an adjunction
\[
\xymatrix{
{\catover{\cat{B} }}\ar@<5pt>[r]_-{\perp}^-{\str}\ & {\proth(\cat{A})^{\op},}\ar@<5pt>[l]^-{\sem}
}
\]
called the \demph{structure--semantics adjunction} for the aritation $\lpair - , - \rpair$.
\end{thm}
\begin{proof}
This follows from Lemmas~\ref{lem:adj-nat} and~\ref{lem:adj-bij}.
\end{proof}

\section{Profunctor viewpoint}
\label{sec:profunctor}

In this section we explore another way of looking at proto-theories and their semantics for a given aritation.

\begin{defn}
Let $\cat{A}$ and $\cat{A}'$ be large categories. A \demph{profunctor} $M \from \cat{A} \proto \cat{A}'$ (also known as a \demph{module} or \demph{bimodule}) is a functor $M \from (\cat{A}')^{\op} \times \cat{A} \to \SET$. Given profunctors $M \from \cat{A} \proto \cat{A}'$ and $M' \from \cat{A}' \proto \cat{A}''$, the \demph{composite} profunctor $M' \tensor M \from \cat{A} \proto \cat{A}''$ is defined by the following coend:
\[
M' \tensor M (a'', a) = \int^{a' \in \cat{A}'} M' (a'', a') \times M (a', a)
\]
for $a'' \in \cat{A}''$ and $a \in \cat{A}$.
\end{defn}

\begin{prop}
There is a bicategory $\PROF$, with large categories as objects, profunctors as $1$-cells, and natural transformations as 2-cells. Given a large category $\cat{A}$, the identity profunctor on $\cat{A}$ is given by the hom-functor $\cat{A}(-,-) \from \cat{A}^{\op} \times \cat{A} \to \SET$.
\end{prop}
\begin{proof}
This is well-known. The bicategory $\PROF$ was first defined by B\'enabou in~\cite{benabou73}; see also Section~7.8 of~\cite{borceux94v1} for an overview.
\end{proof}

\begin{prop}
There is a canonical bicategory homomorphism $\mathcal{P} \from \CAT^{\op} \to \PROF$ that is the identity on objects, sends a functor $F \from \cat{C} \to \cat{D}$ to the profunctor $\mathcal{P}(F) \from \cat{D} \proto \cat{C}$ given by $\cat{D}(F-, -) \from \cat{C}^{\op} \times \cat{D} \to \SET$, and sends a natural transformation $\alpha \from F \to G$ to the natural transformation $\cat{D}(\alpha - , -) \from \cat{D}(F-,-) \to \cat{D}(G-, -)$. 
\end{prop}
\begin{proof}
This is essentially Proposition~7.8.5 of~\cite{borceux94v1}.
\end{proof}

Recall from Street~\cite{street72} that one can talk about monads in an arbitrary 2-category, or indeed \emph{bi}category, not just in $\CAT$. In particular we can consider monads in $\PROF$.
\begin{prop}
\label{prop:monad-in-prof}
The category $\proth(\cat{A})$ is equivalent to the category $\monad_{\PROF} (\cat{A})$ of monads on $\cat{A}$ in the bicategory $\PROF$.
\end{prop}
\begin{proof}
This follows from Corollary~3.8 in Cheng~\cite{cheng11}, which shows that monads in $\PROF$ can be identified with \emph{identity}-on-objects functors, yielding an isomorphism between $\monad_{\PROF} (\cat{A})$ and the full subcategory of $\proth(\cat{A})$ of identity-on-objects proto-theories. Since every proto-theory on $\cat{A}$ is isomorphic in $\proth(\cat{A})$ to one that is the identity on objects, this yields the desired equivalence.
\end{proof}

For the rest of this section, fix an aritation $\lpair -,- \rpair \from \cat{A} \times \cat{B} \to \cat{C}$.

\begin{lem}
\label{lem:prof-monadic}
Let $(L \from \cat{A} \to \cat{L}) \in \proth(\cat{A})$. Then the functor 
\[
\mathcal{P}(L)_* \from \PROF(\cat{C}, \cat{L}) \to \PROF(\cat{C},\cat{A})
\]
is monadic.
\end{lem}

\begin{proof}
We have $\PROF(\cat{C}, \cat{L}) = [\cat{L}^{\op} \times \cat{C}, \SET]$ and $\PROF(\cat{C}, \cat{A}) = [\cat{L}^{\op} \times \cat{A}, \SET]$, and viewed in this light the functor $\mathcal{P}(L)_*$ is
\[
(L^{\op} \times \id_{\cat{C}})^* \from [\cat{L}^{\op} \times \cat{C}, \SET] \to [\cat{A}^{\op} \times \cat{C}, \SET].
\]
Since $L^{\op} \times \id_{\cat{C}} \from \cat{A}^{\op} \times \cat{C} \to \cat{L}^{\op} \times \cat{C}$ is bijective on objects, it follows from Corollary~\ref{cor:rest-bo-monadic} that the functor above is monadic if and only if it has a left adjoint. But since $\cat{A}^{\op} \times \cat{C}$ is large and $\SET$ has large limits, it follows that left Kan extensions along $L^{\op} \times \id_{\cat{C}}$ exist, giving a left adjoint to $(L^{\op} \times \id_{\cat{C}})^*$.
\end{proof}

Corollary 8.1 of Street~\cite{street72} says that a 1-cell in a 2-category is monadic if and only if it is sent to a monadic functor by each covariant representable 2-functor. The corresponding result holds for bicategories as well, and it follows that $\mathcal{P}(L)\from \cat{L} \proto \cat{A}$ exhibits $\cat{L}$ as the Eilenberg--Moore object for the monad on $\cat{A}$ in $\PROF$ corresponding to $L$.

\begin{prop}
Let $(L \from \cat{A} \to \cat{L}) \in \proth(\cat{A})$. The square
\[
\xymatrix{
[\cat{L},\cat{C}] \ar[r]^-{\mathcal{P}}\ar[d]_{L^*} & \PROF(\cat{C}, \cat{L})\ar[d]^{\mathcal{P}(L)_*} \\
[\cat{A},\cat{C}] \ar[r]_-{\mathcal{P}} & \PROF(\cat{C},\cat{A}) 
}
\]
is a pullback.
\end{prop}
\begin{proof}
Recall that an object of the pullback of $\mathcal{P}(F)_*$ and $\mathcal{P}$ consists of a pair $(F, M)$ where $F \from \cat{A} \to \cat{C}$ and $M \from \cat{C} \proto \cat{L}$ such that $\mathcal{P}(F) = \mathcal{P}(L) \tensor M \from \cat{C} \proto \cat{A}$, and morphisms in the pullback are defined similarly. We will construct an isomorphism between this explicit description of the pullback and $[\cat{L}, \cat{C}]$, compatible with the functors to $\PROF(\cat{C}, \cat{L})$ and $[\cat{A}, \cat{C}]$. We will construct this isomorphism on objects; it is straightforward to extend it to morphisms.

Let $(F,M)$ be as above. This means that
\[
M(L-, -) = \cat{C}(F-, -) \from \cat{A}^{\op} \times \cat{C} \to \SET.
 \]
 We define a functor $G \from \cat{L} \to \cat{C}$ as follows. On objects, we set $G(La) = Fa$ (recalling that every object of $\cat{L}$ is of the form $La$ for a unique $a$). Suppose $l \from La \to La'$ in $\cat{L}$. This defines a natural transformation
 \[
 M(l, -) \from M(La', -) \to M(La, -)
 \]
 which is equivalently a natural transformation $\cat{C}(Fa', -) \to \cat{C}(Fa, -)$, which, by the Yoneda lemma, is given by a unique morphism $Fa \to Fa'$; we define $Gl$ to be this morphism. The uniqueness in the definition of each $Gl$ then guarantees that $G$ is a functor $\cat{L} \to \cat{C}$.
 
 We check that $G \of L = F$ and $\mathcal{P}(G) = M$ and that this characterises $G$ uniquely. By definition $G \of L = F$ on objects. If $f \from a \to a'$, then 
 \[
 M(Lf, -) = \cat{C}(Ff, -) \from \cat{C}(Fa', -) \to \cat{C}(Fa, -)
 \]
 by assumption, so $GLf = Ff$ as required. We have $\cat{P}(G) =M$ if and only if $M(-,-) = \cat{C}(G-, -)$, but this is clear from how $G$ is defined, and $G$ is clearly unique such that these to properties hold.
\end{proof}

\begin{prop}
\label{prop:sem-pullback-monadic}
Any functor of the form $\sem(L) \from \mod(L) \to \cat{B}$ for some aritation $\lpair -,-\rpair \from \cat{A} \times \cat{B} \to \cat{C}$ and proto-theory $L \from \cat{A} \to \cat{L}$ is a pullback of a monadic functor whose codomain is locally large.
\end{prop}
\begin{proof}
Consider the diagram
\[
\xymatrix{
\mod(L)\ar[r]^{J(L)}\ar[d]_{\sem(L)}\pullbackcorner & [\cat{L}, \cat{C}]\ar[d]_{L^*}\ar[r]^-{\mathcal{P}}\pullbackcorner & \PROF(\cat{C},\cat{L})\ar[d]^{\mathcal{P}(L)_*} \\
\cat{B}\ar[r]_{\currylo}& [\cat{A},\cat{C}]\ar[r]_-{\mathcal{P}} & \PROF(\cat{C},\cat{A}).
}
\]
The left-hand square is a pullback by definition of $\sem(L)$, and the right-hand square is a pullback by the previous proposition. Thus the outer rectangle is a pullback, and the morphism on its right-hand edge is monadic by Lemma~\ref{lem:prof-monadic}.

By definition, $\PROF(\cat{C},\cat{A}) = [\cat{A}^{\op} \times \cat{C}, \SET]$, and this is locally large since $\cat{A}^\op \times \cat{C}$ is large and $\SET$ is locally large.
\end{proof}

\begin{remark}
The pullback square appearing in the above proof gives a new perspective on models of a proto-theory. Recall that $\cat{L}$ is the Eilenberg--Moore object of the monad $\mnd{T} = (T, \eta, \mu)$ on $\cat{A}$ in $\PROF$ corresponding to $L \from \cat{A} \to \cat{L}$. But, as in any bicategory, morphisms into the Eilenberg--Moore object of a monad on $\cat{A}$ correspond to morphisms into $\cat{A}$ equipped with an action of the monad. Thus to equip an object $d \in \cat{B}$ with the structure of an $L$-model is to equip the profunctor $\mathcal{P}(\lpair -, d \rpair) \from \cat{C} \proto \cat{A}$ with an action of the monad $\mnd{T}$; that is, a morphism $T \tensor \mathcal{P}(\lpair -, d \rpair) \to \mathcal{P}(\lpair -, d \rpair)$ that is compatible with the unit and multiplication of the monad.
\end{remark}

We can refine Proposition~\ref{prop:sem-pullback-monadic} slightly. We saw that any functor of the form $\sem(L) \from \mod(L) \to \cat{B}$ for an aritation $\lpair -,-\rpair \from \cat{A} \times \cat{B} \to \cat{C}$ and proto-theory $L \from \cat{A} \to \cat{L}$ is a pullback of a monadic functor along
\[
\cat{B} \toby{\currylo} [\cat{A},\cat{C}] \toby{\mathcal{P}} \PROF(\cat{C},\cat{A}).
\]
But the category $\PROF(\cat{C},\cat{A})$ is huge, since its objects are arbitrary functors $\cat{A}^{\op} \times \cat{C} \to \SET$. Thus we might still wonder whether $\sem(L)$ can be expressed as  a pullback of a monadic functor whose codomain is only large. The answer is yes:

\begin{prop}
\label{prop:sem-pullback-monadic-large}
Let $\lpair -, - \rpair \from \cat{A} \times \cat{B} \to \cat{C}$ be an aritation and $L \in \proth(\cat{A})$. Then there is a large category $\cat{D}\in \CAT$, a monad $\mnd{T}$ on $\cat{D}$ and a functor $H \from \cat{B} \to \cat{D}$ such that we have a pullback square
\[
\xymatrix{
\mod(L)\ar[d]_{\sem(L)}\ar[r]\pullbackcorner & \cat{D}^{\mnd{T}}\ar[d]^{U^{\mnd{T}}} \\
\cat{B} \ar[r]_{H} & \cat{D}.
}
\]
\end{prop}
\begin{proof}
By the previous proposition, there is a locally large category $\cat{E}$, a monad $\mnd{S} = (S, \eta, \mu)$ on $\cat{E}$, a functor $K \from \cat{B} \to \cat{E}$ and a pullback square
\[
\xymatrix{
\mod(L)\ar[d]_{\sem(L)}\ar[r]\pullbackcorner & \cat{E}^{\mnd{S}}\ar[d]^{U^{\mnd{S}}} \\
\cat{B} \ar[r]_{K} & \cat{E}.
}
\]
Let $\cat{D}$ be the smallest full subcategory of $\cat{E}$ that contains the image of $K$ and such that $S$ restricts to an endofunctor of $\cat{D}$. Then $\cat{D}$ is large: it has a large set of objects since $\cat{B}$ does and only countably many iterates of $S$ are needed to close the image of $K$ under $S$, and it is locally large since $\cat{E}$ is. Write $H \from \cat{B} \to \cat{D}$ for the factorisation of $K$ through $\cat{D}$.

Clearly $\mnd{S}$ restricts to a monad $\mnd{T}$ on $\cat{D}$, and a $\mnd{T}$-algebra is just an $\mnd{S}$-algebra whose underlying object lies in $\cat{D}$. That is, we have a pullback square
\[
\xymatrix{
\cat{D}^{\mnd{T}}\ar[r]\ar[d]_{U^{\mnd{T}}}\pullbackcorner & \cat{E}^{\mnd{S}}\ar[d]^{U^{\mnd{S}}} \\
\cat{D}\ar[r] &\cat{E}.
}
\]
It follows that we have a commutative diagram
\[
\xymatrix{
\mod(L)\ar[d]_{\sem(L)}\ar[r] &\cat{D}^{\mnd{T}}\ar[r]\ar[d]_{U^{\mnd{T}}} & \cat{E}^{\mnd{S}}\ar[d]^{U^{\mnd{S}}} \\
\cat{B} \ar[r]_{H} & \cat{D}\ar[r]& \cat{E},
}
\]
in which the bottom composite is $K$. But the right-hand square is a pullback, and so is the outer rectangle. It follows that the left-hand square is also a pullback.
\end{proof}

We can use the viewpoint of proto-theories as monads in the bicategory $\PROF$ to deduce some useful properties of the category of proto-theories.

\begin{prop}
\label{prop:monoid-monadic}
If $(\cat{V}, \tensor, I)$ is a monoidal biclosed category with small colimits, then the forgetful functor $\monoid(\cat{V}) \to \cat{V}$ is monadic.
\end{prop}
\begin{proof}
The following is a well-known argument. We apply the monadicity theorem (Theorem~1 in~VI.7 of~\cite{maclane71}). It is straightforward to see that the forgetful functor creates the relevant coequalisers, so all that is necessary is to show that it has a left adjoint.

Theorem~23.4 of Kelly~\cite{kelly80} states that if $\cat{V}$ has countable coproducts and, for each $V \in \cat{V}$ both of the functors $V \tensor -$ and $- \tensor V$ preserve countable coproducts then this forgetful functor has a left adjoint. This is in particular the case when $\cat{V}$ is biclosed since then these functors are left adjoints, so preserve all colimits.
\end{proof}

\begin{prop}
\label{prop:proth-prof-monadic}
The functor $U \from \proth(\cat{A}) \to [\cat{A}^{\op} \times \cat{A}, \SET]$ sending $L \from \cat{A} \to \cat{L}$ to $\cat{L}(L-, L-) \from \cat{A}^{\op} \times \cat{A} \to \SET$ is weakly monadic.
\end{prop}
\begin{proof}
Recall that a bicategory with one object is precisely a monoidal category. In particular, composition of profunctors makes the functor category $[\cat{A}^{\op} \times \cat{A}, \SET]$ into a monoidal category. The category of monoids in this monoidal category is equivalent to $\proth(\cat{A})$, by Proposition~\ref{prop:monad-in-prof}, and under this equivalence, the forgetful functor $\monoid([\cat{A}^{\op} \times \cat{A}, \SET]) \to [\cat{A}^{\op} \times \cat{A}, \SET]$ corresponds to the functor described above.

But $[\cat{A}^{\op} \times \cat{A}, \SET] \iso \PROF(\cat{A}, \cat{A})$ is biclosed (This follows from Theorem~2.3.3 in B\'enabou~\cite{benabou73}) and cocomplete and so by the previous proposition, the forgetful functor
\[
\monoid([\cat{A}^{\op} \times \cat{A}, \SET]) \to [\cat{A}^{\op} \times \cat{A}, \SET]
\]
is monadic. Hence $U$ (as a composite of a monadic functor with an equivalence) is weakly monadic.
\end{proof}

\begin{cor}
The category $\proth(\cat{A})$ has all large limits.
\end{cor}
\begin{proof}
By the previous proposition, this category is monadic over $[\cat{A}^{\op} \times \cat{A}, \SET]$, which has all large limits since $\SET$ does. Since monadic functors create all limits, the result follows.
\end{proof}

\section{Example: monoids}
\label{sec:str-sem-adj-monoids}

As seen in Section~\ref{sec:notions-monoids}, monoids can be thought of as an extremely simple kind of algebraic theory, and as such they have their own structure--semantics adjunction, as in Proposition~\ref{prop:str-sem-adj-monoid}. We will explore how more complicated notions of algebraic theory arise from proto-theories in later chapters, but for now let us see how monoids fit into this framework. Throughout this section, fix a large category $\cat{B}$.

\begin{defn}
Let $\scat{1}$ denote the category with a single object, and just an identity morphism.
\end{defn}

\begin{lem}
\label{lem:monoid-theory-iso}
We have an isomorphism of categories
\[
\proth(\scat{1}) \iso \MONOID,
\]
where $\MONOID$ is the category of large monoids.
\end{lem}
\begin{proof}
A functor out of $\scat{1}$ just picks out an object of its codomain, and such a functor is bijective on objects if and only if its codomain has a single object, that is, it is a monoid. So the objects of $\proth(\scat{1})$ can be identified with the monoids. Furthermore, any functor between 1-object categories (i.e.\ any monoid homomorphism) makes the appropriate triangle commute, and so defines a morphism in $\proth(\scat{1})$.
\end{proof}

\begin{defn}
Define an aritation
\[
\lpair -, - \rpair \from \scat{1} \times \cat{B} \to \cat{B}
\]
to be the projection onto the second factor (note that this is an isomorphism of categories).
\end{defn}
\begin{remark}
The aritation defined above gives rise to an adjunction
\[
\xymatrix{
{\catover{\cat{B}}}\ar@<5pt>[r]_-{\perp}^-{\str}\ & {\proth(\scat{1})^{\op}}\ar@<5pt>[l]^-{\sem}
}
\]
as in Theorem~\ref{thm:str-sem-adj}.
\end{remark}

\begin{prop}
The adjunction in the above remark coincides under the isomorphism $\proth(\scat{1}) \iso \MONOID$ from Lemma~\ref{lem:monoid-theory-iso} with the adjunction
\[
\xymatrix{
{\catover{\cat{B}}}\ar@<5pt>[r]_-{\perp}^-{\str_{\MONOID}}\ & {\MONOID^{\op}}\ar@<5pt>[l]^-{\sem_{\MONOID}}
}
\]
from Proposition~\ref{prop:str-sem-adj-monoid}.
\end{prop}
\begin{proof}
Let $M$ be a monoid, and write $L \from \scat{1} \to M$ for the unique such functor. Then the category $\mod_{\MONOID} (M)$ of actions of $M$ in the category $\cat{B}$ can be identified with the functor category $[M, \cat{B}]$, and the forgetful functor with the functor $[M, \cat{B}] \to \cat{B}$ given by evaluation at the unique object of $M$.

Now $\sem(M) \from \mod(M) \to \cat{B}$ is defined by the pullback
\[
\xymatrix{
\mod(M)\ar[r]\ar[d]_{\sem(M)}\pullbackcorner & [M, \cat{B}]\ar[d]^{L^*} \\
\cat{B}\ar[r]_{\currylo} & [\scat{1},\cat{B}].
}
\]
But $\currylo$ is an isomorphism, so we can identify $\mod(M)$ with $[M,\cat{B}]$, and $\sem(M)$ with the composite $\currylo^{-1} \of L^*$, which is precisely the evaluation functor described above.
\end{proof}

\section{Chu spaces}
\label{sec:chu}

In this section we look at aritations and their structure--semantics adjunctions from the point of view of Chu spaces and the Chu construction for closed symmetric monoidal categories. These notions were first developed by Barr and Chu in~\cite{barr79}, and a historical overview can be found in~\cite{barr06}. In particular, the definition below first appeared in the appendix to~\cite{barr79}.

\begin{defn}
Let $\cat{V}$ be a closed symmetric monoidal category with tensor $\tensor$ and internal hom $[-,-]$, and let $C$ be an object of $\cat{V}$. Then the category $\chu(\cat{V}, C)$ of \demph{Chu spaces (in $\cat{V}$ over $C$)} is defined as follows.
\begin{description}
\item[Objects:] An object of $\chu(\cat{V}, C)$ consists of two objects $A$ and $B$ of $\cat{V}$ together with a morphism $\lpair - , - \rpair \from A \times B \to C$ in $\cat{V}$.
\item[Morphisms:] A morphism $(A, B, \lpair - , - \rpair) \to (A', B', \lpair - , - \rpair')$ consists of morphisms $f \from A \to A'$ and $g \from B' \to B$ in $\cat{V}$ such that
\[
\xymatrix{
A \tensor B' \ar[r]^{\id_{A} \tensor g}\ar[d]_{f \tensor \id_{B'}} & A \tensor B\ar[d]^{\lpair - , - \rpair} \\
A' \tensor B'\ar[r]_{\lpair-, - \rpair'} & C
}
\]
commutes.
\end{description}
For a Chu space $\lpair-,-\rpair \from A \tensor B \to C$ of $\chu(\cat{V}, C)$ in $\cat{V}$ we call $A$ the object of \demph{points}, $B$ the object of \demph{states}, $C$ the object of \demph{truth values} and $\lpair-,-\rpair$ the \demph{pairing}.
\end{defn}

Clearly an aritation $\lpair -, - \rpair \from \cat{A} \times \cat{B} \to \cat{C}$ is a Chu space in $\CAT$ over $\cat{C}$. Let us interpret the semantics and structure functors for such an aritation in terms of Chu spaces.

Let $L \from \cat{A} \to \cat{L}$ be a proto-theory. Recall that by definition we have a pullback square
\[
\xymatrix{
\mod(L) \pullbackcorner \ar[r]^{J(L)}\ar[d]_{\sem(L)} & [\cat{L},\cat{C}]\ar[d]^{L^*} \\
\cat{B} \ar[r]_{\currylo} &[\cat{A},\cat{C}].
}
\]
The functor $J(L) \from \mod(L) \to [\cat{L}, \cat{C}]$ corresponds to a functor $\lpair - , - \rpair' \from \cat{L} \times \mod(L) \to \cat{C}$ and we can think of $( \cat{L},\mod(L), \lpair-,-\rpair')$ as an object of $\chu(\CAT, \cat{C})$. The commutativity of the above pullback corresponds to the commutativity of
\[
\xymatrix{
\cat{A} \times \mod(L)\ar[r]^{L \times \id}\ar[d]_{\id_{\cat{A}} \times \sem(L)} & \cat{L} \times \mod(L)\ar[d]^{\lpair-,-\rpair'} \\
\cat{A} \times \cat{B} \ar[r]_{\lpair-,- \rpair}& \cat{C},
}
\]
which says that $(L, \sem(L))$ is a morphism $(\cat{A}, \cat{B}, \lpair-,-\rpair) \to (\cat{L}, \mod(L), \lpair-,-\rpair')$ in $\chu(\CAT, \cat{C})$. 

Let $\cat{B}' \in \CAT$, and  $\lpair-,-\rpair'' \from \cat{L} \times \cat{B}' \to \cat{C}$, so that $(\cat{L}, \cat{B}', \lpair-,-\rpair'') \in \chu(\CAT, \cat{C})$, and let $G \from \cat{B}' \to \cat{B}$ be such that $(L, G)$ is a Chu space morphism $  (\cat{A}, \cat{B}, \lpair -,- \rpair) \to (\cat{L}, \cat{B}', \lpair-,-\rpair'')$. Then
\[
\xymatrix{
\cat{B}'\ar[r]^K\ar[d]_{G} & [\cat{L}, \cat{C}]\ar[d]^{L^*} \\
\cat{B}\ar[r]_{\currylo} &[\cat{A}, \cat{L}]
}
\]
commutes, where $K$ is the transpose of $\lpair-,-\rpair''$. Thus, by the universal property of pullbacks, there is a unique functor $G' \from \cat{B}' \to \mod(L)$ such that $ \sem(L) \of G' = G$ and $J(L) \of G' = K$. Equivalently, $G'$ is unique such that $(\id_{\cat{L}}, G')$ is a morphism $( \cat{L}, \mod(L), \lpair -,- \rpair') \to (\cat{L}, \cat{B}', \lpair-,-\rpair'')$ in $\chu(\CAT, \cat{C})$ with
\[
(\id_{\cat{L}}, G') \of (L, \sem(L)) = (L,G).
\]
In other words, $(\cat{L},\mod(L), \lpair-,-\rpair')$ and $(L, \sem(L))$ provide the universal way of extending $\cat{L}$ to a Chu space and $L$ to a morphism of Chu spaces out of $(\cat{A}, \cat{B}, \lpair-,- \rpair)$.

Now let $M \in \CAT$ and $U \from \cat{M} \to \cat{B}$; we will give a similar universal property of $\str(U) \from \cat{A} \to \thr(U)$ in terms of Chu spaces. First let us fix some terminology: call a morphism
\[
(F,G) \from (\cat{A}, \cat{B}, \lpair-,-\rpair) \to (\cat{A}', \cat{B}', \lpair -,- \rpair')
\]
of Chu spaces in $\CAT$ \demph{bijective on objects} if $F \from \cat{A} \to \cat{A}'$ is a bijective-on-objects functor.

Recall that by definition, we have a commutative square
\[
\xymatrix{
\cat{A} \ar[r]^{\curryhi}\ar[d]_{\str(U)} & [\cat{B}, \cat{C}]\ar[d]^{U^*} \\
\thr(U)\ar[r]_{I(U)} & [\cat{M}, \cat{C}].
}
\]
Writing $\lpair-,-\rpair' \from \thr(U) \times \cat{M} \to \cat{C}$ for the transpose of $I(U)$, this corresponds to the commutativity of
\[
\xymatrix
@C=40pt{
\cat{A}\times\cat{M}\ar[r]^-{\str(U)\times \id}\ar[d]_{\id_{\cat{A}} \times U} & \thr(U)\times \cat{M}\ar[d]^{\lpair-,-\rpair'} \\
\cat{A} \times \cat{B} \ar[r]_{\lpair-,-\rpair} & \cat{C}
}
\]
which says that $(\str(U), U)$ is a Chu space morphism $(\cat{A}, \cat{B}, \lpair-,-\rpair) \to (\thr(U), \cat{M}, \lpair-,-\rpair') $, and it is bijective on objects since $\str(U)$ is.

Let $\cat{A} \in \CAT$, and $\lpair-,-\rpair'' \from \cat{A}' \times \cat{M} \to \cat{C}$, so that $(\cat{A}', \cat{M}, \lpair-,-\rpair'')$ is an object of $\chu(\CAT, \cat{C})$. Let $F \from \cat{A} \to \cat{A}'$ be a bijective-on-objects functor such that $(F, U)$ is a morphism $(\cat{A}, \cat{B}, \lpair-,-\rpair) \to (\cat{A}', \cat{M}, \lpair-,-\rpair'')$ in $\chu(\CAT, \cat{C})$. This means that the bottom-left triangle in
\[
\xymatrix{
\cat{A}\ar[dd]_{F}\ar[rr]^{\str(U)}\ar[dr]^{\curryhi} & & \thr(U)\ar[dd]^{I(U)} \\
& [\cat{B}, \cat{C}]\ar[dr]^{U^*} & \\
\cat{A}'\ar[rr]_{M} & & [\cat{M}, \cat{C}]
}
\]
commutes, where $M$ is the transpose of $\lpair-,-\rpair'' \from \cat{A}' \times \cat{M} \to \cat{C}$. The top-right triangle commutes by definition of $\str(U)$. Since $F$ is bijective on objects and $I(U)$ is full and faithful, there exists a unique $F' \from \cat{A}' \to \thr(U)$ making both triangles in
\[
\xymatrix{
\cat{A}\ar[r]^{\str(U)}\ar[d]_F & \thr(U)\ar[d]^{I(U)} \\
\cat{A}'\ar[ur]^{F'}\ar[r]_M & [\cat{M}, \cat{C}].
}
\]
commute. That is, $F'$ is the unique bijective-on-objects functor such that $(F', \id_\cat{M})$ is a Chu space morphism $(\cat{A}', \cat{M}, \lpair-,-\rpair'') \to (\str(U), \cat{M}, \lpair-,-\rpair')$ such that
\[
(F', \id_{\cat{M}}) \of (F, U) = (\str(U), U).
\]
Thus $(\thr(U),\cat{M}, \lpair-,-\rpair')$ and $(\str(U), U)$ provide the universal way of extending $\cat{M}$ to a Chu space and $U$ to a morphism of Chu spaces out of $(\cat{A}, \cat{B}, \lpair-,- \rpair)$.

\chapter{Monads and the canonical aritation}
\label{chap:canonical1}
In this chapter we define a canonical aritation associated with each locally small category, and show that the resulting structure--semantics adjunction generalises the structure--semantics adjunction for monads described in Section~\ref{sec:notions-monads} of Chapter~\ref{chap:notions}.

\begin{defn*}
Let $\cat{B}$ be a locally small category. The \demph{canonical aritation} on $\cat{B}$ is given by the hom-functor
\[
\lpair -, - \rpair = \cat{B}(-,-) \from \cat{B}^{\op} \times \cat{B} \to \Set.
\]
In particular, the category of arities for this aritation is $\cat{B}^{\op}$, the base category is $\cat{B}$, and it takes values in $\Set$.
 \end{defn*}
 This aritation gives rise to a structure-semantics adjunction of the form
 \[
\xymatrix{
{\catover{\cat{B} }}\ar@<5pt>[r]_-{\perp}^-{\str}\ & {\proth(\cat{B}^{\op})^{\op}.}\ar@<5pt>[l]^-{\sem}
}
\]
 
 In Section~\ref{sec:alt-model-algebra} we give an alternative definition of a model of a proto-theory for the canonical aritation. This alternative formulation is often more convenient to work with in practice, but is only available for the canonical aritation. In Section~\ref{sec:monads-are-proths} we show how we can think of monads as proto-theories, and how the canonical aritation provides an extension of the usual semantics of monads. Finally in Section~\ref{sec:canonical1-monads-arities} we describe a variant of the canonical aritation for when the base category is equipped with a specified dense subcategory and we show that the corresponding semantics generalises the semantics of monads with arities as discussed in Section~\ref{sec:notions-monads-arites}.

 \section{An alternative description of $L$-models}
 \label{sec:alt-model-algebra}
 
In the case of the canonical aritation, there is an alternative formulation of the definition of a model of a proto-theory with arities $\cat{B}$, which is often more convenient to work with. In this section we state this alternative definition and show that it is equivalent to Definition~\ref{defn:models-explicit}. In this section only, we shall refer to models in the alternative formulation as ``algebras''. Once we have shown that models and algebras are the same thing, we shall use the term ``model'' to refer to them both, relying on notation to distinguish between the two equivalent formulations.

For the rest of this section, let $\cat{B}$ be a fixed locally small category and let $(L \from \cat{B}^{\op} \to \cat{L}) \in \proth(\cat{B})$. Whenever we refer to structure and semantics functors, we mean those induced by the canonical aritation on $\cat{B}$.

\begin{defn}
\label{defn:L-alg}
An \demph{algebra} $x$ of $L$ consists of an object $d^x \in \cat{B}$ together with a collection of maps
\[
\alpha^x_b \from \cat{L}(Ld^x, Lb) \to \cat{B}(b, d^x)
\]
natural in $b \in \cat{B}$, such that
\begin{enumerate}
\item
\label{part:L-alg-id}
$\alpha^x_{d^x} (\id_{L(d^x)}) = \id_{d^x}$; and
\item
\label{part:L-alg-comp}
for all $l \from Ld^x \to Lb$ and $k \from Lb \to Lb'$, we have
\[
\alpha^x_{b'} (k \of l) = \alpha^x_{b'} (k \of L( \alpha^x_b(l))).
\]
\end{enumerate}
\end{defn}

\begin{defn}
\label{defn:alg-hom}
An \demph{algebra homomorphism} between $L$-algebras $x \to y$ consists of a morphism $h \from d^x \to d^y$ in $\cat{B}$ such that for all $b \in \cat{B}$ and $l \from Ld^x \to Lb$ in $\cat{L}$, we have
\[
h \of \alpha^x_b(l) = \alpha^y_b (l \of L(h) ).
\]
\end{defn}

Let us compare the definitions of $L$-algebras and $L$-models. An $L$-model structure on $d \in \cat{B}$ consists of a functor $\Gamma \from \cat{L} \to \cat{B}$ such that $\Gamma \of L = \cat{B}(-, d)$; in particular, for all $b, b' \in \cat{B}$, we have a map
\[
\cat{L}(Lb, Lb') \to \Set( \cat{B}(b, d), \cat{B}(b', d)).
\]
An element of $\cat{L}(Lb, Lb')$ is an operation of the proto-theory $L$ with arity $b$ and output shape $b'$; the $L$-model structure on $d$ gives a concrete interpretation of such an operation as a way of turning (generalised) elements of $d$ of shape $b$ into elements of shape $b'$.

An $L$-algebra structure on $d$ consists of, for each $b \in \cat{B}$, a map
\[
\cat{L}(Ld, Lb) \to \cat{B}(b, d);
\]
that is, a way of turning operations of arity $d$ and output shape $b$ into elements of $d$ of shape $b$. Thus the equivalence between the two, which we establish below, says that in order to interpret operations of $L$ of arbitrary arity in the object $d$, it is enough to give a way of turning $d$-ary operations into elements of $d$.

This may seem surprising, but the same phenomenon occurs with monads. For a monad $\mnd{T} = (T, \eta, \mu)$ on $\cat{B}$, we think of $Tb$ as the ``object of $b$-ary operations'' of $\mnd{T}$, for any object $\cat{B}$. Thus we might think of generalised elements of  $Tb$ of shape $b'$ (that is, maps $b' \to Tb$) as operations with arity $b$ and output shape $b'$. Reasoning \emph{a priori} then, we might think that a $\mnd{T}$-model structure on $d \in \cat{B}$ should provide for each morphism $b' \to Tb$ a way of turning elements of $d$ of shape $b$ into elements of shape $b'$. That is, we should have a map
\[
\cat{B}(b', Tb) \to  \Set(\cat{B}(b,d), \cat{B}(b',d)).
\]
We could define a $\mnd{T}$-model along these lines and end up with a definition equivalent to the usual notion of $\mnd{T}$-algebra. However we know, of course, that a $\mnd{T}$-algebra structure on $d$ can be described much more simply, by a map $Td \to d$. Since elements of $Td$ are $d$-ary operations, such a map gives a way of turning $d$-ary operations into elements of $d$.

This similarity between algebras of an arbitrary proto-theory and algebras for a monad is not a coincidence: we show below that the semantics of proto-theories for the canonical aritation generalises the semantics of monads.

The following simple consequence of the definition shall often be useful.

\begin{lem}
\label{lem:alg-func}
Let $x = (d^x, \alpha^x)$ be an $L$-algebra. Then for all $b \in \cat{B}$ and $f \from b \to d^x$, we have
\[
\alpha^x_b (Lf) =f.
\]
\end{lem}
\begin{proof}
We have
\[
\alpha^x_b(Lf) = \alpha^x_b ( Lf \of \id_{Ld^x}  ) = \alpha^x_{d^x} (\id_{Ld^x}) \of f = f,
\]
where the second equality is by naturality of $\alpha^x$, and the third is by Definition~\ref{defn:L-alg}.\bref{part:L-alg-id}.
\end{proof}

We now show that the notions of $L$-model and $L$-algebra coincide.

\begin{prop}
\begin{enumerate}
\item
\label{part:mod-to-alg}
Given an $L$-model $x = (d^x, \Gamma^x)$, we may define an $L$-algebra $(d^x, \alpha^x)$ with the same underlying object by defining
\begin{align*}
\alpha^x_b \from \cat{L}(Ld^x, Lb)& \to \cat{B}(b, d^x) \\
\hfill l &\mapsto \Gamma^x(l)(\id_{d^x})
\end{align*}
for each $b \in \cat{B}$, recalling that $\Gamma^x$ is a functor $\cat{L} \to \Set$ with $\Gamma^x \of L = \cat{B}(-, d^x)$, so $\Gamma^x (l)$ is a function $\cat{B}(d^x,d^x) \to \cat{B}(b,d^x)$.
\item
\label{part:alg-to-mod}
Given an $L$-algebra $x = (d^x, \alpha^x)$, we may define an $L$-model $(d^x, \Gamma^x)$ with the same underlying object by defining, for each $l \in \cat{L}(Lb', Lb)$,
\begin{align*}
\Gamma^x (l) \from \cat{B}(b', d^x) &\to \cat{B}(b, d^x) \\
\hfill f & \mapsto \alpha^x_{b'} (l \of L(f)).
\end{align*}
\item
\label{part:mod-alg-inv}
These two assignments, from model to algebra and vice versa, are inverse. 
\item
\label{part:mod-alg-hom}
Given $L$-models (or equivalently algebras) $x$ and $y$, a morphism $d^x \to d^y$ is an $L$-model homomorphism $(d^x, \Gamma^x) \to (d^y, \Gamma^y)$ if and only if it is an $L$-algebra homomorphism $(d^x, \alpha^x) \to (d^y, \alpha^y)$.
\end{enumerate}
\end{prop}
\begin{proof}
\bref{part:mod-to-alg}: First we check that $\alpha^x$ is natural. Given $g \from b' \to b$ in $\cat{B}$ and $l \from  Ld^x \to Lb $ in $\cat{L}$, we have
\begin{align*}
\alpha^x_{b'} ( Lg \of l) &= \Gamma^x (Lg \of l ) (\id_{d^x}) \\
& = \Gamma^x (Lg) \of \Gamma^x (l) (\id_{d^x}) \\
& = g^* \of \Gamma^x (l) (\id_{d^x}) \\
&= \alpha^x_b (l) \of g
\end{align*}
as required. Clearly
\[
\alpha^x_{d^x} (\id_{Ld^x}) = \Gamma^x(\id_{Ld^x})(\id_{d^x}) = \id_{d^x}.
\]
Now let $l \from Ld^x \to Lb$ and $k \from Lb \to Lb'$. Then
\begin{align*}
\alpha^x_{b'}(k \of l) &= \Gamma^x( k \of l) (\id_{d^x}) \\
&= \Gamma^x(k) ( \Gamma^x (l) (\id_{d^x})) \\
&= \Gamma^x (k) (\alpha^x_b (l)) \\
&= \Gamma^x (k) \of \alpha^x_b(l)^* (\id_{d^x}) \\
&= \Gamma^x(k) \of \Gamma^x (L \alpha^x_b(l)) (\id_{d^x}) \\
&= \Gamma^x ( k \of L \alpha^x_b(l) )  (\id_{d^x}) \\
&= \alpha^x_{b'} (k \of  L \alpha^x_b (l))
\end{align*}
as required. So $(d^x, \alpha^x)$ is indeed an $L$-algebra.

\bref{part:alg-to-mod}: We check that $\Gamma^x$ as defined is functorial. Certainly
\begin{align*}
\Gamma^x (\id_{Lb}) \from \cat{B}(b, d^x) &\to \cat{B}(b, d^x) \\
\hfill f & \mapsto \alpha^x_b (Lf) = f
\end{align*}
by Lemma~\ref{lem:alg-func}. Let $ l \from Lb \to Lb'$ and $k \from Lb' \to Lb''$. Then if $f \in \cat{B}(b, d^x)$, we have
\begin{align*}
\Gamma^x (k \of l) (f) &= \alpha^x_{b''} ( k \of l \of Lf) \\
&= \alpha^x_{b''}(k \of L\alpha^x_{b'}( l \of Lf ) ) \\
&= \Gamma^x (k) (\alpha^x_{b'}( l \of Lf)) \\
&= \Gamma^x (k) ( \Gamma^x(l)(f)) \\
&= \Gamma^x (k) \of \Gamma^x(l) (f)
\end{align*}
and hence $ \Gamma^x (k \of l) = \Gamma^x (k) \of \Gamma^x(l)$. In addition, if $f \from b \to d^x$ and $g \from b' \to b$ them
\[
\Gamma^x(Lg)(f) = \alpha^x_{b'} (Lg \of Lf) = \alpha^x_{b'} (L(f\of g)) = f \of g
 \]
 by Lemma~\ref{lem:alg-func}, so $\Gamma^x \of L = \cat{B}(-, d^x)$. Hence $(d^x, \Gamma^x)$ is an $L$-model.

\bref{part:mod-alg-inv}: In one direction we must show that for any $L$-model $(d^x, \Gamma^x)$, any $l \in \cat{L}(Lb , Lb')$ and any $f \in \cat{B}(b, d^x)$, we have
\[
\Gamma^x(l) (f) = \Gamma^x(l \of Lf ) (\id_{d^x}).
\] 
But
\[
\Gamma^x( l \of Lf )(\id_{d^x}) = \Gamma^x (l) \of \Gamma^x (Lf) (\id_{d^x}) = \Gamma^x(l) \of f^* (\id_{d^x}) = \Gamma^x(l) (f).
\]
In the other direction, we must show that for any $L$-algebra $(d^x, \alpha^x)$, any $b \in \cat{B}$ and any $l \in \cat{L}(Ld^x, Lb)$, we have 
\[
\alpha^x_b (l) = \alpha^x_b ( l \of L(\id_{d^x})).
\]
But this is immediate.

\bref{part:mod-alg-hom}: Let $h \from d^x \to d^y$ be a model homomorphism; that means that for all $l \in \cat{L} (Lb, Lb')$ and $f \in \cat{B}(b, d^x)$, we have
\begin{equation}
\label{eqn:mod-hom}
\Gamma^y(l)(h \of f) = h \of \Gamma^x(l) (f),
\end{equation}
or equivalently,
\begin{equation}
\label{eqn:alg-hom}
\alpha^y_{b'}( l \of Lf \of Lh) = h \of \alpha^x_{b'}( l \of Lf ).
\end{equation}
But in the case when $b = d^x$ and $f = \id_{d^x}$, this is precisely what is required for $h$ to be an algebra homomorphism. Conversely, if $h$ in an algebra homomorphism, then~\bref{eqn:alg-hom} holds as a special case of Definition~\ref{defn:alg-hom}, and hence~\bref{eqn:mod-hom} holds, and so $h$ is also a model homomorphism.
\end{proof}

From now on, when we write ``$x$ is a model of $L$'', this will be understood to mean we have both an $L$-model denoted $(d^x, \Gamma^x)$ and the corresponding $L$-algebra denoted $(d^x, \alpha^x)$, and we will freely use whichever manifestation of the structure is most convenient at the time. We will also use the term ``model'' to refer to either of these, relying on the difference in notation to indicate which is intended.

 \section{Monads} 
 \label{sec:monads-are-proths}
Throughout this section, let $\cat{B}$ be a locally small category unless stated otherwise. We will show that the structure--semantics adjunction that arises from the canonical aritation of $\cat{B}$ generalises the adjunction between right adjoints into $\cat{B}$ and monads on $\cat{B}$ that was described in Proposition~\ref{prop:sem-str-adj-monad}. More precisely, we will show that there is a canonical full and faithful functor $\monad(\cat{B}) \incl \proth(\cat{B}^{\op})$  such that
\[
\vcenter{
\xymatrix{
{\radjover{\cat{B}}}\ar[r]^{\str_{\monad}}\ar[d] & {\monad(\cat{B})^{\op}}\ar[d] \\
{\catover{\cat{B}}}\ar[r]_{\str} & {\proth(\cat{B}^{\op})^{\op}}
}}
\quad \quad \text{and} \quad \quad
\vcenter{
\xymatrix{
{\radjover{\cat{B}}}\ar[d] & {\monad(\cat{B})^{\op}}\ar[d]\ar[l]_{\sem_{\monad}} \\
{\catover{\cat{B}}} & {\proth(\cat{B}^{\op})^{\op}}\ar[l]^{\sem}
}}
\]
commute, where the left-hand vertical arrows are the obvious forgetful functor.

First we construct the embedding $\monad(\cat{B}) \incl \proth(\cat{B}^{\op})$.

\begin{defn}
\label{defn:kle-mor}
Given monads $\mnd{T}$ and $\mnd{T}'$ on $\cat{B}$, a \demph{Kleisli morphism} $\hat{P} \from \mnd{T} \to \mnd{T}'$ is an operation $\hat{P}$ that sends morphisms $f \from b \to Tb'$ in $\cat{B}$ to morphisms $\hat{P}(f) \from b \to T'b'$ such that 
\begin{enumerate}
\item
\label{part:kle-mor-unit}
If $k \from b \to b'$ then $\hat{P}(\eta_{b'} \of k) = \eta_{b'}' \of k$, and
\item
\label{part:kle-mor-mult}
If $f \from b \to Tb'$ and $g \from b' \to Tb''$, then $\hat{P}(\mu_{b''} \of Tg \of f) = \mu_{b''}' \of T'\hat{P}(g) \of \hat{P}(f)$.
\end{enumerate}
\end{defn}
Clearly a Kleisli morphism $\hat{P} \from \mnd{T} \to \mnd{T}'$ is precisely the same as a functor $P \from \cat{B}_{\mnd{T}} \to \cat{B}_{\mnd{T}'}$ such that $ P \of F_{\mnd{T}} = F_{\mnd{T}'}$, where $F_{\mnd{T}} \from \cat{B} \to \cat{B}_{\mnd{T}}$ is the canonical free functor from $\cat{B}$ to the Kleisli category of $\mnd{T}$.

The following lemma will be used in the proof of Proposition~\ref{prop:kle-ff}.

\begin{lem}
\label{lem:kleisli}
For any Kleisli morphism $\hat{P} \from \mnd{T} \to \mnd{T}'$ and for every $f \from b \to Tb'$ in $\cat{B}$, we have $\hat{P}(f) = \hat{P}(\id_{Tb'}) \of f$.
\end{lem}
\begin{proof}
We have
\begin{align*}
\hat{P}(f) &= \hat{P}(\mu_{b'} \of \eta_{Tb'} \of f) = \hat{P}(\mu_{b'} \of T(\id_{Tb'}) \of \eta_{Tb'} \of f) & \\
&= \mu'_{b'} \of T' \hat{P}(\id_{Tb'}) \of \hat{P}(\eta_{Tb'} \of f) & \text{(by Definition \ref{defn:kle-mor}.\bref{part:kle-mor-mult})} \\
&= \mu'_{b'} \of T' \hat{P}(\id_{Tb'}) \of \eta'_{Tb'} \of f & \text{(by \ref{defn:kle-mor}.\bref{part:kle-mor-unit})} \\
&= \mu'_{b'} \of T' \hat{P}(\id_{Tb'}) \of \hat{P}(\eta_{Tb'}) \of f & \text{(by \ref{defn:kle-mor}.\bref{part:kle-mor-unit})} \\
&= \hat{P} (\mu_{b'} \of T(\id_{Tb'}) \of \eta_{Tb'}) \of f & \text{(by \ref{defn:kle-mor}.\bref{part:kle-mor-mult})} \\
&= \hat{P}(\id_{Tb'}) \of f
\end{align*}
as required.
\end{proof}

\begin{defn}
\label{defn:kle-functor-objects}
Given a monad $\mnd{T}$ on $\cat{B}$, define the proto-theory $\kl{\mnd{T}} \in \proth(\cat{B}^{\op})$ to be $F_\mnd{T}^{\op} \from \cat{B}^{\op} \to \cat{B}_{\mnd{T}}^{\op}$.
\end{defn}

\begin{remark}
\label{rem:bo-no-op}
Clearly a bijective-on objects functor out of $\cat{B}^{\op}$ is essentially the same thing as a bijective-on-objects functor out of $\cat{B}$. In order to avoid a proliferation of $\op$'s, we can therefore identify $\kl{\mnd{T}}$ (which we have defined to be $F_{\mnd{T}}^{\op} \from \cat{B}^{\op} \to \cat{B}_{\mnd{T}}^{\op}$) with $F_{\mnd{T}} \from \cat{B} \to \cat{B}_{\mnd{T}}$. Similarly, when we define $\kle \from \monad(\cat{B}) \to \proth(\cat{B}^{\op})$ on monad morphisms $\phi \from \mnd{T} \to \mnd{T}'$ (as we are about to), we will define $\kl{\phi}$ as a functor $\cat{B}_{\mnd{T}} \to \cat{B}_{\mnd{T}'}$, even though strictly speaking it should be a functor $\cat{B}_{\mnd{T}}^{\op} \to \cat{B}_{\mnd{T}'}^{\op}$. This minor abuse of notation should not cause any confusion, but is worth keeping in mind.
\end{remark}

\begin{defn}
\label{defn:kle-functor-morphisms}
Given a monad morphism $\phi \from \mnd{T} \to \mnd{T}'$, define $\kl{\phi} \from\kl{\mnd{T}} \to \kl{\mnd{T}'}$ in $\proth(\cat{B}^{\op})$ to be the functor $\cat{B}_{\mnd{T}} \to \cat{B}_{\mnd{T}'}$ with corresponding Kleisli morphism $\widehat{\kl{\phi}} \from \mnd{T} \to \mnd{T}'$  given by 
\[
\widehat{\kl{\phi}}(f) = \phi_{b} \of f
\]
for $f \from b' \to Tb$.
\end{defn}

\begin{lem}
\label{lem:kle-functor}
Definitions~\ref{defn:kle-functor-objects} and~\ref{defn:kle-functor-morphisms} together give a well-defined functor $\kle \from \monad(\cat{B}) \to \proth(\cat{B}^{\op})$; that is, for each monad morphism $\phi \from \mnd{T} \to \mnd{T}'$, the given definition of $\kle{\phi}$ is indeed a Kleisli morphism.
\end{lem}
\begin{proof}
First we check Definition~\ref{defn:kle-mor}.\bref{part:kle-mor-unit}. Given $k \from b \to b'$,
\begin{align*}
\widehat{\kl{\phi}}(\eta_{b'} \of k) = \phi_{b'} \of \eta_{b'} \of k = \eta'_{b'} \of k
\end{align*}
as required. And if $f \from b \to Tb'$, $g \from b' \to Tb''$, then
\begin{align*}
\widehat{\kl{\phi}}(\mu_{b''} \of Tg \of f) & = \phi_{b''} \of \mu_{b''} \of Tg \of f \\
&= \mu'_{b''} \of T' \phi_{b''} \of \phi_{Tb''} \of Tg \of f \\
&= \mu'_{b''} \of T' \phi_{b''} \of T'g \of \phi_{b'} \of f \\
&= \mu'_{b''} \of T' \widehat{\kl{\phi}}(g) \of \widehat{\kl{\phi}}(f)
\end{align*}
as required.
\end{proof}

We have defined the proto-theory associated to a monad $\mnd{T}$ to be given by its Kleisli category. This makes precise the standard intuition that $Tb$ is ``the object of $b$-ary operations'' in the following way. Recall that, in an arbitrary category $\cat{B}$, we cannot talk about elements of an object $b \in \cat{B}$, but we can talk about ``generalised elements''; a generalised element of $b$ is simply a morphism with codomain $b$, and the domain of that morphism is sometimes called the ``shape'' of the generalised element. So, given a monad $\mnd{T} = (T, \eta, \mu)$ on $\cat{B}$, a generalised element of $Tb$ with shape $b'$ is a morphism $b' \to Tb$, which is precisely the same as an operation of $\kl{\mnd{T}}$ with arity $\cat{B}$ and shape $b'$ in the sense of proto-theories.

\begin{prop}
\label{prop:kle-ff}
The functor $\kle \from \monad(\cat{B}) \to \proth(\cat{B}^{\op})$ is full and faithful, and its essential image consists of the bijective-on-objects functors out of $\cat{B}$ with right adjoints.
\end{prop}
The first part of this proposition says that when we pass from monads to proto-theories we are not throwing away any \emph{structure} --- having an associated monad is merely a \emph{property} of a proto-theory.

The second part tells us exactly what this property is. Recall that an arbitrary functor $F \from \cat{B} \to \cat{D}$ has a right adjoint if and only if the presheaf $\cat{D}(F-, d)$ is representable for each $d \in \cat{D}$. Thus, a proto-theory $L \from \cat{B}^{\op} \to \cat{L}$ comes from a monad if and only if, for each $b \in \cat{B}$, the ``presheaf of $b$-ary operations'' $\cat{L}(Lb, L-) \from \cat{B}^{\op} \to \SET$ (with variable shape) is representable. In that case, the object that represents it is $Tb$, the ``object of $b$-ary operations'', for the corresponding monad.

\begin{proof}
First let us show that $\kle$ is faithful. Suppose $\phi, \theta \from \mnd{T} \to \mnd{T}'$ with $\phi \neq \theta$. Then there is some $b \in \cat{B}$ for which $ \phi_b \neq \theta_b$. Then 
\[
\widehat{\kl{\phi}}(\id_{Tb}) = \phi_b \neq \theta_b = \widehat{\kl{\theta}}(\id_{Tb}),
\]
so $\kl{\phi} \neq \kl{\theta}$ as required.

Next we check that $\kle$ is full. Suppose $P \from \cat{B}_{\mnd{T}} \to \cat{B}_{\mnd{T}'}$ with $P F_{\mnd{T}} = F_{\mnd{T}'}$. For each $b \in \cat{B}$, define
\[
\phi_b = \hat{P}(\id_{Tb}) \from Tb \to T'b.
\]
We must show that $\phi$ is natural. Given $f \from b \to b'$,
\begin{align*}
T'f \of \phi_b &= T'f \of \hat{P}(\id_{Tb}) = \mu'_{b'} \of T' \eta'_{b'} \of T'f \of \hat{P}(\id_{Tb}) & \\
&= \mu'_{b'} \of T' \hat{P}(\eta_{b'} \of f) \of \hat{P}(\id_{Tb}) & \text{(by \ref{defn:kle-mor}.\bref{part:kle-mor-unit})}\\
&= \hat{P}(\mu_{b'} \of T(\eta_{b'} \of f) \of \id_{Tb}) & \text{(by \ref{defn:kle-mor}.\bref{part:kle-mor-mult})}\\
& = \hat{P}(Tf) = \hat{P}(\id_{Tb'}) \of Tf & \text{(by Lemma~\ref{lem:kleisli})} \\
&= \phi_{b'} \of Tf
\end{align*}
as required. Next we check that $\phi$ is compatible with the units:
\[
\phi_b \of \eta_b = \hat{P}(\id_{Tb}) \of \eta_b = \hat{P}(\eta_b) = \eta'_b,
\]
using Lemma~\ref{lem:kleisli} and Definition~\ref{defn:kle-mor}.\bref{part:kle-mor-unit}. Finally we check compatibility with the multiplications:
\begin{align*}
\mu'_{b} \of T'\phi_b \of \phi_{Tb} & = \mu'_b \of T'\hat{P}(\id_{Tb}) \of \hat{P}(\id_{TTb}) & \\
& = \hat{P}(\mu_b \of T\id_{Tb} \of \id_{TTb}) = \hat{P}(\mu_b) & \text{(by \ref{defn:kle-mor}.\bref{part:kle-mor-mult})} \\
&= \hat{P}(\id_{Tb}) \of \mu_b = \phi_b \of \mu_b. & \text{(by Lemma~\ref{lem:kleisli})}
\end{align*}
Thus $\phi$ is a monad morphism $\mnd{T} \to \mnd{T}'$, and it is clear from Lemma~\ref{lem:kleisli} that $\kl{\phi} = P$, which concludes the proof that $\kle$ is full.

All that remains is to show that the essential image of $\kle$ consists of the bijective-on-objects functors with right adjoints. Certainly every object of $\proth(\cat{B}^{\op})$ in the image of $\kle$ \emph{does} have a right adjoint, since all functors of the form $F_{\mnd{T}}$ do. Let $L \from \cat{B} \to \cat{L}$ be a bijective on objects functor with a right adjoint. Then let $\mnd{T}$ be the monad on $\cat{B}$ generated by $L$ and its right adjoint. The comparison functor $K \from \cat{B}_{\mnd{T}} \to \cat{L}$ is always full and faithful. However it is also bijective on objects since both $F_{\mnd{T}}$ and $L$ are, and $K \of F_{\mnd{T}} = L$. Hence $K$ is an isomorphism, so $L$ is in the essential image of $\kle$.
\end{proof}

We have shown that $\kle$ exhibits $\monad(\cat{B})$ as a full subcategory of $\proth(\cat{B}^{\op})$.

\begin{prop}
\label{prop:codensity-str}
Let $(U \from \cat{M} \to \cat{B}) \in \catover{\cat{B}}$. Then $U$ has a pointwise codensity monad if and only if $\str(U) \in \proth(\cat{B}^{\op})$ lies in the essential image of $\kle \from \monad(\cat{B}) \incl \proth(\cat{B}^{\op})$, and then the codensity monad of $U$ is the essentially unique monad $\mnd{T}$ such that $\str(U) \iso \kle(\mnd{T})$.
\end{prop}

\begin{proof}
The functor $U$ has a pointwise codensity monad if and only if, for each $b \in \cat{B}$, the diagram
\[
(b \downarrow U) \to \cat{M} \toby{U} \cat{B}
\]
has a limit. But a cone on this diagram with vertex $b'$ is essentially the same as a natural transformation
\[
\cat{B}(b, U-) \to \cat{B}(b', U-).
\]
Thus $U$ has a pointwise codensity monad if and only if, for each $b \in \cat{B}$, the presheaf sending an arbitrary $b' \in \cat{B}$ to
\[
[\cat{M}, \Set](\cat{B}(b, U-), \cat{B}(b', U-))
\]
is representable. But this presheaf is precisely $\thr(U) ( b, \str(U) -)$, so it being representable for each $b$ is equivalent to $\str(U) \from \cat{B}^{\op} \to \thr(U)$ having a left adjoint. By the previous proposition this is the same as $\str(U)$ being in the essential image of $\kle$.
\end{proof}

 Before examining how the semantics of monads relates to the semantics of proto-theories, we note the following:

\begin{prop}
\label{prop:mnd-incl-limits}
Suppose $\cat{B}$ admits all small limits. Then the inclusion $\kle \from \monad(\cat{B}) \incl \proth(\cat{B}^{\op})$ preserves all small limits.
\end{prop}
\begin{proof}
Since $\cat{B}$ has all small limits, small limits of monads on $\cat{B}$ can be computed component-wise; that is, the forgetful functor $\monad(\cat{B}) \to [\cat{B},\cat{B}]$ creates and preserves limits.

Write
\[
\cat{Q} \from [\cat{B}, \cat{B}] \to [\cat{B}^{\op} \times \cat{B}, \SET]
\]
for the functor sending an endofunctor $F$ of $\cat{B}$ to $\cat{B}(-, F-) \from \cat{B}^{\op} \times \cat{B} \to \SET$. It is clear that $\cat{Q}$ preserves limits.

Now consider the square
\[
\xymatrix{
\monad(\cat{B})\ar[r]^{\kle}\ar[d] & \proth(\cat{B}^{\op})\ar[d] \\
[\cat{B},\cat{B}]\ar[r]_-{\cat{Q}} & [\cat{B}^{\op} \times \cat{B}, \SET],
}
\]
where the right-hand vertical arrow is the forgetful functor from Proposition~\ref{prop:proth-prof-monadic}. This diagram commutes, and the left and the bottom functors preserve small limits as already noted. But the right-hand functor is monadic by Proposition~\ref{prop:proth-prof-monadic} and so creates limits. It follows that $\kle$ preserves limits.
\end{proof}

Now we show that $\str \ladj \sem$ restricts to the adjunction $\str_{\monad} \ladj \sem_{\monad}$.

\begin{prop}
\label{prop:syn-rest}
The diagram
\[
\xymatrix{
{\radjover{\cat{B}}}\ar[r]^{\str_{\monad}}\ar[d] & {\monad(\cat{B})^{\op}}\ar[d]^{\kle^{\op}} \\
{\catover{\cat{B}}}\ar[r]_{\str} & {\proth(\cat{B}^{\op})^{\op}}
}
\]
commutes up to isomorphism.
\end{prop}
\begin{proof}
Let $(\cat{M}, U, F, \eta, \epsilon) \in \radjover{\cat{B}}$. Both $\thr (U)$ and $\cat{B}_{(UF, \eta, U\epsilon F)}$ have as objects the objects of $\cat{B}$, and we have
\begin{align*}
\thr(U) (b, b') &= [\cat{M}, \Set] ( \cat{B}(b, U-), \cat{B}(b', U-)) \\
& \iso [\cat{M},\Set](\cat{M}(Fb, -), \cat{M}(Fb', -)) \\
& \iso \cat{M}(Fb', Fb) \\
& \iso \cat{B}(b', UFb) \\
& = \cat{B}_{(UF, \eta, U\epsilon F)} (b', b).
\end{align*}
Tracing through the steps in this isomorphism, a natural transformation $\gamma \from \cat{B}(b, U-) \to \cat{B}(b', U-)$ is sent to $\gamma_{Fb}(\eta_b) \from b' \to UFb$. In particular, the identity natural transformation on $\cat{B}(b, U-)$ is sent to $ \eta_b$, so identities in $\thr(U)$ agree with those in $\cat{B}_{(UF, \eta, U\epsilon F)}^{\op}$. Now suppose $\gamma \from \cat{B}(b, U-) \to \cat{B}(b', U-)$ and $\delta \from \cat{B}(b', U-) \to \cat{B}(b'', U-)$. Then the composite of $\gamma_{Fb}(\eta_b)$ and $\delta_{Fb'}(\eta_{b'})$ in the Kleisli category is given by
\begin{align*}
U \epsilon_{Fb} \of UF \gamma_{Fb}(\eta_b) \of \delta_{Fb'}(\eta_{b'}) & = \delta_{Fb}(U\epsilon_{Fb} \of UF\gamma_{Fb}(\eta_b) \of \eta_{b'}) \\
&= \delta_{Fb}(U\epsilon_{Fb} \of \eta_{UFb}\of \gamma_{Fb}(\eta_b)) \\
&=\delta_{Fb}(\gamma_{Fb}(\eta_b)),
\end{align*}
which corresponds to the composite $\delta \of \gamma$. So composition in $\thr(U)$ agrees with composition in $\cat{B}_{(UF, \eta, U\epsilon F)}^{\op}$. Hence we have an isomorphism $\thr(U) \iso \cat{B}_{(UF, \eta, U\epsilon F)}^{\op}$. Furthermore, this isomorphism is compatible with the functors $\str(U) \from \cat{B}^{\op} \to \thr(U)$ and $F_{\str_{\monad}(U)} \from \cat{B} \to \cat{B}_{\str_{\monad}(U)}$: for $f \from b' \to b$ in $\cat{B}$, the natural transformation $\str(U)(f) \from \cat{B}(b, U-) \to \cat{B}(b', U-)$ is given by composition with $f$, and so corresponds to
\[
f^*_b (\eta_b) = \eta_b \of f \in \cat{B}(b', UFb),
\]
and this is precisely $F_{\str_{\monad}(U)} (f) \in \cat{B}_{\str_{\monad}(U)} (b' ,b)$.

We have shown that the two composites in the diagram in the proposition are equal on objects; we now show that they are equal on morphisms. Suppose $Q \from (\cat{M}, U,F,\eta,\epsilon) \to (\cat{M}', U',F',\eta',\epsilon')$ in $\radjover{\cat{B}}$. We must show that
\[
\xymatrix{
{\thr(U')}\ar[r]^{\iso}\ar[d]_{\str (Q)} & {\cat{B}_{\str_{\monad}(U')}^{\op}}\ar[d]^{(\kle \of \str_{\monad} (Q))^{\op}} \\
{\thr(U)}\ar[r]_{\iso} & {\cat{B}_{\str_{\monad}(U)}^{\op}}
}
\]
commutes. Let $\gamma \from \cat{B}(b, U'-) \to \cat{B}(b', U'-)$. Then the bottom composite sends this to
\[
\gamma_{QFb} (\eta_b) \from b' \to UFb = U' QF b
\]
(since $\eta_b \from b \to UFb = U' QF b$, it is valid to apply $\gamma_{QFb}$ to it). The top composite sends $\gamma$ to the composite
\[
b' \toby{\gamma_b(\eta'_b) } U'F'b \toby{U'F' \eta_b} U'F'UFb = U'F'U' QFb \toby{U'\epsilon'_{QFb}} U'QFb = UFb.
\]
But
\begin{align*}
U'\epsilon'_{QFb} \of U' F' \eta_b \of \gamma_b(\eta'_b) & = \gamma_{QFb} (U'\epsilon'_{QFb} \of U' F' \eta_b \of \eta'_b) \\
&= \gamma_{QFb} (U'\epsilon'_{QFb} \of \eta'_{UFb} \of \eta_b) \\
&= \gamma_{QFb}(U' \epsilon'_{QFb} \of \eta'_{U' QFb} \of \eta_b) \\
&= \gamma_{QFb}(\eta_b),
\end{align*}
as required.
\end{proof}

\begin{prop}
\label{prop:sem-rest}
The diagram
\[
\xymatrix{
{\radjover{\cat{B}}}\ar[d] & {\monad(\cat{B})^{\op}}\ar[d]^{\kle^{\op}}\ar[l]_{\sem_{\monad}} \\
{\catover{\cat{B}}} & {\proth(\cat{B}^{\op})^{\op}}\ar[l]^{\sem}
}
\]
commutes up to isomorphism.
\end{prop}
This follows from Theorem~14 of Street~\cite{street72}, however we prove it here for completeness.
\begin{proof}
Let $\mnd{T} = (T, \eta,\mu)$ be a monad on $\cat{B}$. First we show that a $\mnd{T}$-algebra structure on an object of $\cat{B}$ is the same as a $\kl{\mnd{T}}$-model structure, and the two notions of homomorphism coincide. First let us spell out explicitly what a $\kl{\mnd{T}}$-model is. It consists of an object $d^x$ together with a family of maps
\[
\alpha^x_b \from \cat{B}(b, Td^x) \to \cat{B}(b, d^x)
\]
that are natural in $b\in \cat{B}$, such that
\begin{equation}
\label{eqn:kle-mod-unit}
\alpha^x_{d^x} (\eta_{d^x}) = \id_{d^x},
\end{equation}
which is Definition~\ref{defn:L-alg}.\bref{part:L-alg-id}, and such that for any $l \from b' \to Tb$ and $k\from b \to Td^x$, we have
\[
\alpha^x_{b'} (\mu_{d^x} \of Tk \of l) = \alpha^x_{b'} (\mu_{d^x} \of T \eta_{d^x} \of T \alpha^x_b (k) \of l);
\]
the left-hand side here is $\alpha^x_{b'}$ applied to the composite of $k$ and $l$ in the Kleisli category of $\mnd{T}$, and the right-hand side is $\alpha^x_{b'}$ applied to the composite of $F_\mnd{T}(\alpha^x_b(k))$ and $l$ in the Kleisli category, so this is the appropriate instantiation of Definition~\ref{defn:L-alg}.\bref{part:L-alg-comp}. This last equation can be written equivalently as
\begin{equation}
\label{eqn:kle-mod-mult}
\alpha^x_{b'} (\mu_{d^x} \of Tk \of l) = \alpha^x_{b'} (T\alpha^x_b(k) \of l).
\end{equation}
By the Yoneda lemma, natural transformations $\alpha^x \from \cat{B}(-, Td^x) \to \cat{B}(-, d^x)$ correspond to morphisms $s^x = \alpha^x_{Td^x}(\id)\from Td^x \to d^x$, and then $\alpha^x = s^x_*$. Equation~\bref{eqn:kle-mod-unit} is satisfied if and only if $s^x \of \eta_{d^x} = \id_{d^x}$, which is the unit axiom for $s^x$ to be a $\mnd{T}$-algebra. If $s^x$ is a $\mnd{T}$-algebra structure, then for $k$ and $l$ as above,
\begin{align*}
\alpha^x_{b'} (\mu_{d^x} \of Tk \of l) & = s^x \of \mu_{d^x} \of Tk \of l \\
&= s^x \of Ts^x \of Tk \of l \\
&= \alpha^x_{b'}(T\alpha^x_b(k) \of l)
\end{align*}
so Equation~\bref{eqn:kle-mod-mult} is satisfied. Conversely if Equation~\bref{eqn:kle-mod-mult} holds for all $k$ and $l$, then in particular it holds when $k = \id_{Td^x}$ and $l = \id_{TTd^x}$, which gives
\[
s^x \of \mu_{d^x} = s^x \of Ts^x,
\]
so $s^x$ is a $\mnd{T}$-algebra structure. 

Let $x = (d^x, \alpha^x)$ and $y= (d^y, \alpha^y)$ be two $\kl{\mnd{T}}$-model structures with corresponding $\mnd{T}$-algebra structures $s^x$ and $s^y$. Then $h \from d^x \to d^y$ is a $\mnd{T}$-algebra homomorphism if and only if $s^y \of Th = h \of s^x$. But this is equivalent to the commutativity of
\[
\xymatrix{
{\cat{B}(-, Td^x)}\ar[d]_{\alpha^x = s^x_*}\ar[r]^{(Th)_*} & {\cat{B}(-, Td^y)}\ar[d]^{\alpha^y = s^y_*} \\
{\cat{B}(-,d^x)}\ar[r]_{h_*} & {\cat{B}(-,d^y),}
}
\]
which is what it means for $h$ to be a $\kle(\mnd{T})$-model homomorphism.

We have shown that there is an isomorphism $\mod(\kl{\mnd{T}}) \iso \cat{B}^{\mnd{T}}$. Now suppose $\phi \from \mnd{T} = (T, \eta, \mu) \to \mnd{T}' = (T', \eta', \mu')$ is a monad morphism. We must show that
\[
\xymatrix{
{\mod(\kl{\mnd{T}'})}\ar[r]^-{\iso} \ar[d]_{\sem(\kl{\phi})} & \cat{B}^{\mnd{T}'}\ar[d]^{\sem_{\monad}(\phi)} \\
{\mod(\kl{\mnd{T}})}\ar[r]_-{\iso} & \cat{B}^{\mnd{T}}
}
\]
commutes. Let $x = (d^x, \alpha^x)$ be a $\kl{\mnd{T}'}$-model, with corresponding $\mnd{T}'$-algebra structure $s^x \from T' d^x \to d^x$. Then the top composite sends this to the $\mnd{T}$-algebra
\[
Td^x \toby{\phi_{d^x}} T'd^x \toby{s^x} d^x.
\]
On the other hand, $\sem(\kl{\phi})$ sends $x$ to the $\mnd{T}$-model structure
\[
\cat{B}(-, Td^x) \toby{(\phi_{d^x})_*} \cat{B}(-,T'd^x) \toby{\alpha^x = s^x_*} \cat{B}(-,d^x),
\]
and this corresponds to the $\mnd{T}$-algebra structure
\[
s^x_* \of (\phi_{d^x})_* (\id_{Td^x}) = s^x \of \phi_{d^x},
\]
so the two $\mnd{T}$-algebra structures coincide. Hence the two composites in the diagram are equal on objects. They are equal on morphisms because they both commute with the forgetful functors to $\cat{B}$, which are all faithful.
\end{proof}

\begin{prop}
\label{prop:sem-str-counit-restr}
Writing
\[
\vcenter{
\xymatrix{
{\radjover{\cat{B}}}\ar[r]^{\str_{\monad}}\ar[d]& {\monad(\cat{B})^{\op}}\ar[d]^{\kle^{\op}} \\
{\catover{\cat{B}}}\ar[r]_{\str} & {\proth(\cat{B}^{\op})^{\op}}\ultwocell\omit{\zeta} 
}}
\quad \text{and} \quad
\vcenter{
\xymatrix{
\monad(\cat{B})^{\op} \ar[d]_{\kle^{\op}} \ar[r]^{\sem_{\monad}} & \radjover{\cat{B}}\ar[d] \\
\proth(\cat{B}^{\op})^{\op}\ar[r]_{\sem} & \catover{\cat{B}}\ultwocell\omit{\xi}
}}
\]
for the natural isomorphisms from Propositions~\ref{prop:syn-rest} and~\ref{prop:sem-rest} respectively, we have an equality
\[
\vcenter{
\xymatrix
@C=13pt{
\monad(\cat{B})^{\op} \ar@{=}[rr]\ar[dr]^{\sem_{\monad}}\ar[dd]_{\kle} & & \monad(\cat{B})^{\op}\ar[dd]^{\kle} \\
& \radjover{\cat{B}}\ar[dd]\ar[ur]^{\str_{\monad}} & \\
\proth(\cat{B}^{\op})^\op\ar[dr]_{\sem} & \ultwocell\omit{\xi} & \proth(\cat{B}^{\op})^{\op}\ultwocell\omit{\zeta} \\
& \catover{\cat{B}}\ar[ur]_{\str} &
}}
=
\vcenter{
\xymatrix
@C=13pt{
\monad(\cat{B})^{\op} \ar@{=}[rr] \ar[dd]_{\kle} & & \monad(\cat{B})^{\op}\ar[dd]^{\kle} \\
&  & \\
\proth(\cat{B}^{\op})^\op\ar[dr]_{\sem}\ar@{=}[rr]  & & \proth(\cat{B}^{\op})^{\op} \\
& \catover{\cat{B}}\ar[ur]_{\str}\ar@{}[u]|-{\labelstyle{E} \objectstyle\Uparrow } &
}}
\]
where $E \from \str \of \sem \to \id$ is the counit of the $\str \ladj \sem$ adjunction.
\end{prop}
Note that we do indeed have an \emph{equality} $\str_{\monad} \of \sem_{\monad} = \id_{\monad(\cat{B})}$.
\begin{proof}
First let us explicitly describe the natural transformations in question. First consider an arbitrary proto-theory $L \from \cat{B}^{\op} \to \cat{L}$. Then the component $E_L$ of the counit of the $\str \ladj \sem$ adjunction is as follows: given $b', b \in \cat{B}$, the map
\[
E_L \from \cat{L}(Lb',Lb) \to \thr(\sem(L))(b', b) = [\mod(L), \Set](\cat{B}(b', \sem(L)-), \cat{B}(b, \sem(L)-))
\]
sends $l \from Lb' \to Lb$ to the natural transformation whose component at $x=(d^x, \alpha^x)$ sends $f \from b' \to d^x$ to $\alpha^x_{b'}(l \of Lf)$. In particular, if $L$ is of the form $\kle(\mnd{T})$ for $\mnd{T} \in \monad(\cat{B})$ then, for
\[
g \in \kle(\mnd{T})(b', b) = \cat{B}(b, Tb')
\]
we have
\[
E_{\kle(\mnd{T})} (g)_x(f) = \alpha^x_b(\mu_{d^x} \of T(Tf \of \eta_{b'}) \of g) = \alpha_{b}^x(Tf \of g).
\]

Now let us consider $\zeta_{\sem_{\monad}(\mnd{T})}$. Note that $\mnd{T} = \str_{\monad} \of \sem_{\monad}(\mnd{T}) = \str_{\monad} ( U^{\mnd{T}})$ so we have
\[
\zeta_{U^{\mnd{T}}} \from \kle(\mnd{T}) = \kle(\str_{\monad} (U^{\mnd{T}})) \to \str(U^{\mnd{T}}).
\]
For $b, b'$ in $\cat{B}$, the map
\[
\zeta_{U^{\mnd{T}}} \from \kle(\str_{\monad}(U^{\mnd{T}})) ( b', b) = \cat{B}(b, Tb') \to \thr(U^{\mnd{T}})(b', b) = [\cat{M},\Set](\cat{B}(b', U^{\mnd{T}}-), \cat{B}(b, U^{\mnd{T}}-))
\]
sends
$g \from b \to Tb'$ to the natural transformation $\cat{B}(b', U^{\mnd{T}}-) \to \cat{B}(b, U^{\mnd{T}}-)$ whose component at $x = (d^x, s^x) \in \cat{B}^{\mnd{T}}$ sends
$f \from b' \to d^x$ to the composite
\[
b \toby{g} Tb' \toby{Tf} Td^x \toby{s^x} d^x.
\]

In addition,
\[
\xi_{\mnd{T}} \from \mod(\kle(\mnd{T})) \to \cat{B}^{\mnd{T}}
\]
sends a $\kle(\mnd{T})$-model $x = (d^x, \alpha^x)$ to $d^x$ equipped with the $\mnd{T}$-algebra structure
\[
s^x = \alpha^x_{Td^x} (\id_{Td^x}) \from Td^x \to d^x
\]
and is the identity on morphisms.

It follows that, for $b, b' \in \cat{B}$, and
\[
g \in \kle(\mnd{T})(b', b) = \cat{B}(b, Tb')
\]
the natural transformation
\begin{align*}
\str(\xi_{\mnd{T}}) \of\zeta_{U^{\mnd{T}}}(g) \from \cat{B}(b', \sem(\kle(\mnd{T}))-), \cat{B}(b, \sem(\kle(\mnd{T}))-))
\end{align*}
has component at $x = (d^x, \alpha^x) \in \mod(\kle(\mnd{T}))$ sending $f \from b' \to d^x$ to
\begin{align*}
\str(\xi_{\mnd{T}}) \of\zeta_{U^{\mnd{T}}}(g)_x(f ) &= \zeta_{U^{\mnd{T}}}(g)_{\xi_{\mnd{T}}}(f) \\
&= \alpha^x_{Td^x} (\id_{Td^x}) \of Tf \of g \\
&= \alpha^x_{b} (Tf \of g),
\end{align*}
and this is precisely the same as $E_{\kle(\mnd{T})} (g)_x (f)$ as required.
\end{proof}

Thus we have

\begin{thm}
\label{thm:str-sem-rest-monad}
The structure--semantics adjunction
\[
\xymatrix{
{\catover{\cat{B}} }\ar@<5pt>[r]_-{\perp}^-{\str}\ & {\proth(\cat{B}^{\op})^{\op}}\ar@<5pt>[l]^-{\sem}
}
\]
restricts along the inclusions 
\[
\radjover{\cat{B}} \incl \catover{\cat{B}} \quad \text{ and }\quad \kle \from \monad(\cat{B}) \incl \proth(\cat{B}^{\op})
\]
to the adjunction
\[
\xymatrix{
{\radjover{\cat{B}} }\ar@<5pt>[r]_-{\perp}^-{\str_{\monad}}\ & {\monad(\cat{B})^{\op}.}\ar@<5pt>[l]^-{\sem_{\monad}}
}
\]
\end{thm}
\begin{proof}
This is precisely Propositions~\ref{prop:syn-rest},~\ref{prop:sem-rest} and~\ref{prop:sem-str-counit-restr} together.
\end{proof}

We can now prove the converse of Proposition~\ref{prop:sem-pullback-monadic-large}.

\begin{prop}
Let $\cat{B}$ be any large category (not necessarily locally small) and let $(U \from \cat{M} \to \cat{B}) \in \catover{\cat{B}}$. Then $U$ is of the form $\sem(L) \from \mod(L) \to \cat{B}$ for some aritation $\lpair - , - \rpair \from \cat{A} \times \cat{B} \to \cat{C}$ in $\CAT$ and $L \in \proth(\cat{A})$ if and only if it is the pullback of a monadic functor with large codomain.
\end{prop}
\begin{proof}
In Proposition~\ref{prop:sem-pullback-monadic-large} we saw that a functor of the form $\sem(L)$ is a pullback of a monadic functor with large codomain; now we show the converse.

Suppose we have a monad $\mnd{T}$ on a large category $\cat{D}$ and a pullback square
\[
\xymatrix{
\cat{M}\pullbackcorner\ar[d]_U \ar[r] & \cat{D}^{\mnd{T}} \ar[d]^{U^{\mnd{T}}} \\
\cat{B}\ar[r]_H & \cat{D}.
}
\]

Let $\kappa$ be the supremum of the cardinalities of the hom-sets of $\cat{D}$ and write $\SET_{\kappa}$ for the large category of sets of cardinality at most $\kappa$; note that $\SET_{\kappa}$ is large. Then by Theorem~14 of Street~\cite{street72}, we have a pullback square
\[
\xymatrix{
\cat{D}^{\mnd{T}}\pullbackcorner\ar[d]_{U^{\mnd{T}}}\ar[r] & [\cat{D}_{\mnd{T}}^{\op}, \SET_{\kappa}]\ar[d]^{(F_{\mnd{T}}^{\op})^*} \\
\cat{D}\ar[r]_-{\currylo} & [\cat{D}^{\op},\SET_{\kappa}].
}
\]
Putting the two pullback squares together, the outer rectangle of 
\[
\xymatrix{
\cat{M}\ar[d]_U \ar[r] & \cat{D}^{\mnd{T}} \ar[d]^{U^{\mnd{T}}}\ar[r] & [\cat{D}_{\mnd{T}}^{\op}, \SET_{\kappa}]\ar[d]^{(F_{\mnd{T}}^{\op})^*}\\
\cat{B}\ar[r]_H & \cat{D}\ar[r]_-{\currylo} & [\cat{D}^{\op}, \SET_{\kappa}]
}
\]
is a pullback, and so we can identify $U \from \cat{M} \to \cat{B}$ with the functor $\sem(F_{\mnd{T}}^{\op}) \from \mod(F_{\mnd{T}}^{\op}) \to \cat{B}$ given by the aritation
\[
\cat{D}(-, H- ) \from \cat{D}^{\op} \times \cat{B} \to \SET_{\kappa}
\]
as required.
\end{proof}

\section{Monads with arities}
\label{sec:canonical1-monads-arities}
Recall from Section~\ref{sec:notions-monads-arites} that monads can be generalised to monads with arities. Berger, Melli\`es and Weber showed in~\cite{bergerMelliesWeber12} that monads with arities on $(\cat{B},\cat{A})$ can equivalently be described as certain bijective-on-objects functors out of $\cat{A}$, which they called theories with arities. In this section we show how the semantics of theories with arities (and thus monads with arities) arises from an aritation. More precisely, we define an aritation whose structure--semantics adjunction extends the structure--semantics adjunction for monads with arities described in Proposition~\ref{prop:monad-arities-adj}, in the same way that the structure--semantics adjunction for the canonical aritation extends that of ordinary monads. Throughout this section, let $(\cat{B},\cat{A})$ be a category with arities as defined in Definition~\ref{defn:cat-with-arities}, so that $\cat{B}$ is a large category, and $\cat{A}$ is a dense subcategory of $\cat{B}$. Assume in addition that $\cat{B}$ is locally small.

\begin{defn}
A \demph{$(\cat{B},\cat{A})$-theory} is a proto-theory $L \from \cat{A}^{\op} \to \cat{L}$ with arities $\cat{A}^{\op}$ such that the composite
\[
[\cat{A}^{\op},\SET] \toby{L_!} [\cat{L},\SET] \toby{L^*} [\cat{A}^{\op},\SET]
\]
restricts to an endofunctor on the essential image of $N_{\cat{A}} \from \cat{B} \to [\cat{A}^{\op} ,\SET]$, where $L_!$ denotes left Kan extension along $L$. The category of $(\cat{B},\cat{A})$-theories is the full subcategory of $\proth(\cat{A}^{\op})$ on such theories.
\end{defn}

\begin{remark}
This is Definition~3.1 of Berger, Melli\`es and Weber~\cite{bergerMelliesWeber12}, where such theories are called simply \emph{theories with arities}.
\end{remark}

Recall from Definition~\ref{defn:arities-algebras} that, for a monad with arities $\mnd{T} = (T, \eta, \mu)$ on a category with arities ($\cat{B}, \cat{A})$, the category $\Theta_{\mnd{T}}$ is defined to be the full subcategory of $\cat{B}^{\mnd{T}}$ consisting of the algebras of the form $(Ta, \mu_a)$ where $a \in \cat{A}$. Recall also that $j_{\mnd{T}} \from \cat{A} \to \Theta_{\mnd{T}}$ is $F^{\mnd{T}} \from \cat{B} \to \cat{B}^{\mnd{T}}$ with domain restricted to $\cat{A}$ and codomain restricted to $\Theta_{\mnd{T}}$. Note that $j_{\mnd{T}}$ is bijective on objects.

\begin{lem}
\label{lem:monad-arities-kle}
The assignment $\mnd{T} \mapsto (j_{\mnd{T}}^{\op} \from \cat{A}^{\op} \to \Theta_{\mnd{T}}^{\op})$ defines a full and faithful functor 
\[
\monad_{\ACAT}(\cat{B},\cat{A}) \to \proth(\cat{A}^{\op})
\]
whose essential image is the category of $(\cat{B},\cat{A})$-theories.
\end{lem}
\begin{proof}
This is a rewording of Theorem~3.4 of \cite{bergerMelliesWeber12}.
\end{proof}

In order to define the semantics of general proto-theories with arities $\cat{A}^{\op}$, and in particular $(\cat{B},\cat{A})$-theories, we need to be able to interpret the arities $\cat{A}^{\op}$ in $\cat{B}$. Since $\cat{A}$ is a subcategory of $\cat{B}$, we can do this with a suitable restriction of the canonical aritation: writing $I \from \cat{A} \incl \cat{B}$ for the inclusion, we have an aritation
\[
\cat{B}(I-, -) \from \cat{A}^{\op} \times \cat{B} \to \Set,
\]
and thus we obtain a structure--semantics adjunction
\[
\xymatrix{
\catover{\cat{B} }\ar@<5pt>[r]_-{\perp}^-{\str}\ & {\proth(\cat{A}^{\op})^{\op}.}\ar@<5pt>[l]^-{\sem}
}
\]
The following proposition relates this adjunction to the adjunction of Proposition~\ref{prop:monad-arities-adj}.
\begin{prop}
\label{prop:mnd-arity-adj-rest}
Both squares in
\[
\xymatrix{
{(\ACAT/(\cat{B},\cat{A}))_{\ra} }\ar@<5pt>[r]_-{\perp}^-{\str_{\Amonad}}\ar[d] & {\Amonad(\cat{B}, \cat{A})^{\op}}\ar@<5pt>[l]^-{\sem_{\Amonad}}\ar[d] \\
\catover{\cat{B} }\ar@<5pt>[r]_-{\perp}^-{\str}\ & {\proth(\cat{A}^{\op})^{\op}}\ar@<5pt>[l]^-{\sem}
}
\]
commute up to equivalence, where the left-hand vertical arrow is the obvious forgetful functor, and the right-hand vertical arrow is the inclusion from Lemma~\ref{lem:monad-arities-kle}.
\end{prop}
\begin{proof}
In \cite{bergerMelliesWeber12}, the category of models of a $(\cat{B},\cat{A})$-theory $(\Theta, j)$ is defined to be the full subcategory of $[\Theta^{\op}, \Set]$ consisting of those presheaves $\Gamma$ such that $\Gamma \of j^{\op}$ belongs to the essential image of $N_{\cat{A}} \from \cat{B} \to [\cat{A}^{\op}, \Set]$. Proposition~3.2 of~\cite{bergerMelliesWeber12} shows that the category of models so defined is equivalent to the category of algebras for the corresponding monad on $\cat{B}$. Thus if we can show that the category of models of $(\Theta, j)$ in the sense of~\cite{bergerMelliesWeber12} is equivalent to the category of models of $(\Theta, j)$ as a proto-theory, we will have shown that the square involving $\sem$ and $\sem_{\Amonad}$ commutes up to equivalence.

The category of models of $(\Theta, j)$ as a proto-theory is defined by the pullback
\[
\xymatrix{
\mod(\Theta, j)\ar[r]^{J(\Theta, j)}\ar[d]_{\sem(\Theta, j)}\pullbackcorner & [\Theta^{\op}, \Set]\ar[d]^{(j^{\op})^*}\\
\cat{B} \ar[r]_-{N_{\cat{A}}} & [\cat{A}^{\op}, \Set].
}
\]
Since $N_{\cat{A}}$ is full and faithful (because $\cat{A}$ is dense in $\cat{B}$), it follows that $J(\Theta, j)$ is also full and faithful. Therefore we can identify $\mod(\Theta, j)$, up to equivalence, with the full subcategory of $[\Theta^{\op}, \Set]$ consisting of those $\Gamma$ such that $\Gamma \of j^{\op} = N_{\cat{A}}(b)$ for some $b \in \cat{B}$. Thus to establish the required equivalence, it is sufficient to show that if $\Gamma \from \Theta^{\op} \to \Set$ is such that $\Gamma \of j^{\op} \iso N_{\cat{A}}(b)$ for some $\cat{B}$, then there is some $\Gamma' \from \Theta^{\op} \to \Set$ such that $\Gamma \iso \Gamma'$ and $\Gamma' \of j^{\op} = N_{\cat{A}}(b)$. But this follows from the fact that $(j^{\op})^*$ is an isofibration; see Lemma~\ref{lem:bo-restrict-isofib}.

We now show that the square involving $\str_{\Amonad}$ and $\str$ commutes. Let $U \from (\cat{M}, \cat{N}) \to (\cat{B},\cat{A})$ be an arity-respecting functor with arity-respecting left adjoint $F$. Write $\mnd{T} = (T, \eta, \mu)$ for the induced monad with arities on $(\cat{B},\cat{A})$. Then we have, for $a, a' \in \cat{A}$,
\begin{align*}
\Theta_{\mnd{T}} (j_{\mnd{T}} (a'), j_{\mnd{T}}(a)) \iso &\cat{B}^{\mnd{T}} ( (Ta', \mu_{a'}), (Ta, \mu_{a})) \\
\iso & \cat{B}(a', Ta) \\
\iso &\cat{B}(a', UFa) \\
\iso & \cat{M}(Fa', Fa) \\
\iso &  [\cat{M}, \SET] (\cat{M}(Fa, -), \cat{M}(Fa', -)) \\
\iso & [\cat{M}, \SET] (\cat{B}(a, U-), \cat{M}(a', U-)) \\
\iso & \thr(U) (a, a'),
\end{align*}
and this composite isomorphism is functorial, hence we have $\Theta_{\mnd{T}}^{\op} \iso \thr(U)$ in $\proth(\cat{A}^{\op})$, as required.
\end{proof}

\chapter{Proto-theories with structure}
\label{chap:structure}

In Chapter~\ref{chap:adjunction} we described a general notion of algebraic theory (namely a proto-theory) and the corresponding structure--semantics adjunction. In this chapter, we generalise further, allowing us to encompass all of the notions of theory described in Chapter~\ref{chap:notions}. A proto-theory will now live in an arbitrary 2-category rather than $\CAT$; this makes the theory considerably more abstract and less intuitive, however it is worth keeping in mind that, although we work in an arbitrary 2-category, in all the examples of interest the 2-category in question will in fact be a category of categories equipped with some extra structure. Thus the intuition developed in Chapter~\ref{chap:adjunction} will carry over at least for these examples, with some added caveats about compatibility with the extra structure.

In particular, the 2-categories that will be of most interest to us, besides $\CAT$, are $\finprod$, the 2-category of large finite product categories, $\symmon$, the 2-category of large symmetric monoidal categories, $\monCAT$, the 2-category of large monoidal categories and $\multi$, the 2-category of large multicategories.

We call a 2-category equipped with the relevant structure for interpreting proto-theories and aritations a \emph{setting}. In Section~\ref{sec:setting-proto-theory-aritation} we give the definition of a setting and the notions of proto-theory and aritation within a setting, and in Section~\ref{sec:sem-general} we define the semantics functor arising from an aritation in a general setting. Then in Section~\ref{sec:str-general} we define the corresponding structure functor and show that the two are adjoint to one another. In Section~\ref{sec:concrete-settings} we discuss a particular type of setting in which the 2-category in question has a suitable forgetful functor to $\CAT$. We show how concrete settings can arise from certain 2-monads on $\CAT$ in Section~\ref{sec:settings-2-monads}, and in Section~\ref{sec:structure-examples} we see how the remaining notions of algebraic theory from Chapter~\ref{chap:notions} and their semantics can be described in terms of proto-theories and aritations in various settings.

\section{\Setting s, proto-theories and aritations}
\label{sec:setting-proto-theory-aritation}

In this section we introduce settings, which are 2-categories equipped with the structure necessary to talk about proto-theories and aritations within them. We also define these notions for a general setting, generalising the definitions of Chapter~\ref{chap:adjunction}.

\begin{defn}
A \demph{\setting} consists of a locally large 2-category $\cat{X}$ equipped with:
\begin{itemize}
\item a factorisation system $(\ofsfont{E}, \ofsfont{N})$ on its underlying $1$-category, and
\item cotensors over $\CAT$; that is, there is a 2-functor $[-, -] \from \CAT^{\op} \times \cat{X} \to \cat{X}$ and a specified isomorphism of categories
\[
\CAT ( \cat{B}, \cat{X}(\cat{A}, \cat{C})) \iso \cat{X}(\cat{A}, [\cat{B}, \cat{C}])
\]
natural in $\cat{A}, \cat{C} \in \cat{X}$ and $\cat{B} \in \CAT$.
\end{itemize}
\end{defn}

\begin{defn}
Given a \setting\ $(\cat{X}, \ofsfont{E}, \ofsfont{N})$, a \demph{proto-theory} in $\cat{X}$ with arities $\cat{A} \in \cat{X}$ consists of a 1-cell
\[
L \from \cat{A} \to \cat{L}
\]
such that $L \in \ofsfont{E}$.

A morphism of proto-theories from $L' \from \cat{A} \to \cat{L}'$  to  $L \from \cat{A} \to \cat{L}$ is simply a 1-cell $P \from \cat{L}' \to \cat{L}$ in $\cat{X}$ such that $P \of L' = L$. We write $\proth(\cat{A})$ for the category of proto-theories in $\cat{X}$ with arities $\cat{A}$ and their morphisms. Thus $\proth(\cat{A})$ depends on the setting $(\cat{X}, \ofsfont{E}, \ofsfont{N})$, leading to potential ambiguity. However it is usually clear from the context which setting is intended, and so we omit making this dependence explicit in order to simplify the notation.
\end{defn}

\begin{defn}
Let $\cat{X}$ be a \setting, and suppose  $\cat{A}, \cat{C} \in \cat{X}$ and $\cat{B} \in \CAT$. Then an \demph{interpretation of arities from $\cat{A}$ in $\cat{B}$, with values in $\cat{C}$}, or just an \demph{aritation}, consists of a functor
\[
\currylo \from \cat{B} \to \cat{X}(\cat{A}, \cat{C}),
\]
or equivalently a $1$-cell
\[
\curryhi \from \cat{A} \to [\cat{B}, \cat{C}]
\]
in $\cat{X}$.
\end{defn}

\begin{remark}
These definitions generalise those of Chapter~\ref{chap:adjunction} as follows. Recall that the 2-category $\CAT$ has a factorisation system given by the bijective-on-objects and full-and-faithful functors, and it is also cotensored over itself, with the cotensor $[\cat{B}, \cat{C}]$ being given by the usual functor category. Thus $\CAT$ is a \setting, and the category of proto-theories with arities $\cat{A} \in \CAT$ in this setting is precisely the category of proto-theories as defined in Definition~\ref{defn:proto-theory}.

An aritation in this \setting, meanwhile, consists of a functor $\cat{B} \to [\cat{A}, \cat{C}]$, which corresponds to a functor $\lpair -, - \rpair \from \cat{A} \times \cat{B} \to \cat{C}$ as in Definition~\ref{defn:aritation}. Indeed, if $\cat{X}$ is also \emph{tensored} over $\CAT$, then a third way to define an aritation is as a 1-cell
\[
\cat{B} \tensor \cat{A} \to \cat{C}
\]
in $\cat{X}$.
\end{remark}

\section{Semantics for proto-theories in a general \setting}
\label{sec:sem-general}

Throughout this section, fix a setting $(\cat{X}, \ofsfont{E}, \ofsfont{N})$ and an aritation $\currylo \from \cat{B} \to \cat{X}(\cat{A},\cat{C})$ in $\cat{X}$. We explain how such an aritation gives rise to a semantics functor
\[
\sem \from \proth(\cat{A})^{\op} \to \catover{\cat{B}}.
\]

\begin{defn}
\label{defn:sem-factors}
We define three functors as follows.
\begin{itemize}
\item The functor $\iota \from \proth(\cat{A}) \to \cat{A}/ \cat{X}$ is simply the full inclusion;
\item the functor $G \from (\cat{A}/\cat{X})^{\op} \to \catover{\cat{X}(\cat{A}, \cat{C})}$ sends $(K \from \cat{A} \to \cat{K}) \in \cat{A}/\cat{X}$ to $K^* \from \cat{X}(\cat{K}, \cat{C}) \to \cat{X}(\cat{A}, \cat{C})$, and acts similarly on morphisms; and
\item the functor $\currylo^* \from \catover{\cat{X}(\cat{A}, \cat{C})} \to \catover{\cat{B}}$ is given by pullback along $\currylo \from \cat{B} \to \cat{X}(\cat{A}, \cat{C})$.
\end{itemize}
\end{defn}

\begin{defn}
\label{defn:sem-general}
We define $\sem \from \proth(\cat{A})^{\op} \to \catover{\cat{B}}$ to be the composite
\[
\proth(\cat{A})^{\op} \toby{\iota^{\op}} (\cat{A}/ \cat{X})^{\op} \toby{G} \catover{\cat{X}(\cat{A}, \cat{C})} \toby{\currylo^*} \catover{\cat{B}}.
\]

Explicitly, given $(L \from \cat{A} \to \cat{L}) \in \proth(\cat{A})$, the functor $\sem(L) \from \mod(L) \to \cat{B}$ is defined by the pullback
\[
\xymatrix{
\mod(L)\pullbackcorner\ar[r]^{J(L)}\ar[d]_{\sem(L)} & \cat{X}(\cat{L}, \cat{C} )\ar[d]^{L^*} \\
\cat{B}\ar[r]_{\currylo} & \cat{X}(\cat{A}, \cat{C}),
}
\]
and, given a morphism
\[
\xymatrix{
\cat{L}'\ar[rr]^P & & \cat{L} \\
& \cat{A}\ar[ul]^{L'} \ar[ur]_L
}
\]
in $\proth(\cat{A})$, the functor $\sem(P) \from \mod(L) \to \mod(L')$ is defined to be the unique functor making
\[
\xymatrix{
{\mod(L)}\ar[rrr]^{I(L)} \ar[dr]^{\sem(P)} \ar[dd]_{\sem(L)} & & & {\cat{X}(\cat{L}, \cat{C} )}\ar[dl]_{P^*}\ar[dd]^{L^*} \\
& {\mod(L')}\ar[r]^{I(L')}\ar[dl]^{\sem(L)} & {\cat{X}(\cat{L}', \cat{C})}\ar[dr]_{L'^*} & \\
 {\cat{B}}\ar[rrr]_{\currylo} & & &{\cat{X}(\cat{A}, \cat{C})} 
}
\]
commute.
\end{defn}

\begin{remark}
It is clear that in the \setting\ $\CAT$ with the bijective-on-objects/full-and-faithful factorisation system, this definition coincides with the definition of $\sem$ given in Definitions~\ref{defn:sem-objects} and \ref{defn:sem-morphisms}.
\end{remark}

Note that the definition of $\sem$ does not make use of the assumption that $\cat{X}$ has cotensors; this assumption will only be used when we construct a left adjoint to $\sem$. Thus, if we were not concerned with the existence of such a left adjoint we could drop this requirement and talk about the semantics of proto-theories in a wider range of contexts.

\section{The structure functor in a general \setting}
\label{sec:str-general}

In this section we define the structure functor arising from an aritation in a general setting, and show that it is left adjoint to the semantics functor. As in the previous section, we fix a setting $(\cat{X},\ofsfont{E},\ofsfont{N})$ and an aritation $\currylo \from \cat{B} \to \cat{X}(\cat{A},\cat{C})$ in $\cat{X}$.

Such a left adjoint exists if and only if, for each $(U \from \cat{M} \to \cat{B}) \in \CAT$, the comma category $(U \downarrow \sem)$ has an initial object. Explicitly, this means that there is a commutative square
\[
\xymatrix{
\cat{M} \ar[r]^P\ar[d]_U & \cat{X}(\thr(U), \cat{C})\ar[d]^{\str(U)^*} \\
\cat{B}\ar[r]_{\currylo} & \cat{X}(\cat{A}, \cat{C})
}
\]
where $(\str(U) \from \cat{A} \to \thr(U)) \in \proth(\cat{A})$, such that for every other square
\begin{equation}
\label{eq:str-desc-explicit}
\vcenter{
\xymatrix{
\cat{M} \ar[r]^R\ar[d]_U & \cat{X}(\cat{L}, \cat{C})\ar[d]^{L^*} \\
\cat{B}\ar[r]_{\currylo} & \cat{X}(\cat{A}, \cat{C})
}}
\end{equation}
where $(L \from \cat{A} \to \cat{L}) \in \proth(\cat{A})$, there exists a unique $Q \from L \to \str(U)$ in $\proth(\cat{A})$ such that
\[
\xymatrix{
\cat{M}\ar[rr]^P\ar[dr]^R\ar[dd]_U & & \cat{X}(\thr(U),\cat{C})\ar[dd]^{\str(U)^*}\ar[dl]_{Q^*} \\
& \cat{X}(\cat{L},\cat{C})\ar[dr]_{L^*} & \\
\cat{B}\ar[rr]_{\currylo} && \cat{X}(\cat{A},\cat{C})
}
\]
commutes (noting that squares of the from \bref{eq:str-desc-explicit} correspond to morphisms $U \to \sem(L)$ in $\catover{\cat{B}}$, by the universal property of the pullback defining $\sem(L)$).

We will see that such an adjoint exists for any aritation in any setting. The construction of this adjoint will make use the existence of cotensors which was part of the definition of a setting; recall that this assumption was not used in constructing the semantics functor itself. In the absence of cotensors, it could still be the case that the adjoint exists for particular choices of aritation, and in such cases the theory we develop in this thesis will still apply. The reason we require cotensors to exist as part of the definition of a setting is to have an explicit description of the left adjoint in terms of familiar constructions, and this requirement is satisfied in all known examples of interest.
\begin{defn}
\label{defn:str-factors}
We define three functors that will soon be seen to be adjoint to those defined in Definition~\ref{defn:sem-factors}.
\begin{itemize}
\item
We define a functor $\rho \from \cat{A}/ \cat{X} \to \proth(\cat{A})$ as follows. Recall that any $K \from \cat{A} \to \cat{K}$ in $\cat{X}$ has a distinguished (and essentially unique) factorisation
\[
\cat{A} \toby{E_K} \cat{L}_K \toby{N_K} \cat{K}
\]
where $E_K \in \ofsfont{E}$ and $N_K \in \ofsfont{N}$. On objects, $\rho$ sends $K$ to $E_K$. On morphisms, it sends a morphism $P$ from $K' \from \cat{A} \to \cat{K}'$ to $ K \from \cat{A} \to \cat{K}$ to the diagonal fill-in of
\[
\xymatrix{
\cat{A}\ar[rr]^{E_K}\ar[d]_{E_{K'}} & & \cat{L}_{K}\ar[d]^{N_K} \\
\cat{L}_{K'}\ar[r]_{N_{K'}} & \cat{K}' \ar[r]_{P} & \cat{K}
}
\]
which exists and is unique since $E_{K'} \in \ofsfont{E}$ and $N_{K} \in \ofsfont{N}$.
\item
We define $F \from \catover{\cat{X}(\cat{A}, \cat{C})} \to (\cat{A} / \cat{X})^{\op}$ on objects by sending
\[
V \from \cat{M} \to \cat{X}(\cat{A}, \cat{C})
\]
to the morphism
\[
\cat{A} \to [\cat{M}, \cat{C}]
\]
that corresponds to it under the universal property of cotensors. The action of $F$ on morphisms is given by the fact that $[\cat{M}, \cat{C}]$ is functorial and contravariant in the first argument.
\item
We define the functor $(\currylo)_! \from \catover{\cat{B}} \to \catover{\cat{X}(\cat{A}, \cat{C})}$ to be given by post-composition with $\currylo \from \cat{B} \to \cat{X}(\cat{A}, \cat{C})$.
\end{itemize}
\end{defn}

\begin{lem}
\label{lem:sem-factor-adjoints}
The functors defined in Definition~\ref{defn:str-factors} are adjoint to those defined in Definition~\ref{defn:sem-factors}. Specifically, $\rho$ is right adjoint to $\iota$, whereas $F$ is left adjoint to $G$ and $(\currylo)_!$ is left adjoint to $\currylo^*$.
\end{lem}

\begin{proof}
This is straightforward, therefore we only sketch the proof. First we show that $\iota \ladj \rho$. Let $ K \from \cat{A} \to \cat{K}$, and $L \from \cat{A} \to \cat{L}$, with $L \in \ofsfont{E}$. Then a morphism $\iota ( L) \to K$ consists of a 1-cell $\cat{L} \to \cat{K}$ making
\[
\xymatrix{
\cat{A}\ar[r]^{E_K}\ar[d]_{L}\ar[dr]^K & \cat{L}_{\cat{K}}\ar[d]^{N_K} \\
\cat{L}\ar[r] & \cat{K}
}
\]
commute, but since $L \in \ofsfont{E}$ and $N_K \in \ofsfont{N}$, these are in bijective correspondence with 1-cells $\cat{L} \to \cat{L}_K$ making
\[
\xymatrix{
\cat{A}\ar[r]^{L_K}\ar[d]_L & \cat{L}_K \\
\cat{L}\ar[ur]
}
\]
commute, which are precisely morphisms $L \to \rho(K)$.

Next we show that $F \ladj G$. Let $V \from \cat{P} \to \cat{X}(\cat{A}, \cat{C})$ and $K \from \cat{A} \to \cat{K}$. Then a morphism $V \to G(K)$ in $\catover{\cat{X}(\cat{A}, \cat{C})}$ is a functor $\cat{P} \to \cat{X}(\cat{K}, \cat{C})$ making
\[
\xymatrix{
\cat{P} \ar[r]\ar[dr]_{V} & \cat{X}(\cat{K}, \cat{C})\ar[d]^{K^*} \\
& \cat{X}(\cat{A}, \cat{C})
}
\]
commute. But by the universal property of cotensors, these are in bijective correspondence with 1-cells $\cat{K} \to [\cat{P}, \cat{C}]$ such that
\[
\xymatrix{
\cat{A} \ar[r]^{F(V)}\ar[d]_K & [\cat{P}, \cat{C}] \\
\cat{K} \ar[ur]
}
\]
commutes, but these are precisely morphisms $K \to F(V)$ in $ \cat{A}/\cat{X}$.

Now we show that $(\currylo)_! \ladj \currylo^*$. Let $U \from \cat{M} \to \cat{B}$ and $V \from \cat{P} \to \cat{X}(\cat{A}, \cat{C})$. Then a morphism $(\currylo)_! (U) \to V$ consists of a functor $\cat{M} \to \cat{P}$ such that
\[
\xymatrix{
\cat{M} \ar[r]\ar[d]_U & \cat{P}\ar[d]^V \\
\cat{B}\ar[r]_{\currylo} & \cat{X}(\cat{A}, \cat{C})
}
\]
commutes, and by the universal property of pullbacks, these are the same as functors from $\cat{M}$ to the pullback of $V$ along $\currylo$ making
\[
\xymatrix{
\cat{M}\ar[r]\ar[dr]_U & \bullet\ar[d]^{\currylo^*(V)} \\
& \cat{B}
}
\]
commute, that is, morphisms $U \to \currylo^*(V)$.
\end{proof}

Thus we have three composable adjunctions:
\[
\xymatrix{
{\catover{\cat{B} }}\ar@<5pt>[r]_-{\perp}^-{(\currylo)_!} & {\catover{\cat{X}(\cat{A},\cat{C})}}\ar@<5pt>[r]_-{\perp}^-F\ar@<5pt>[l]^-{\currylo^*} & {(\cat{A}/\cat{X})^{\op}}\ar@<5pt>[r]_-{\perp}^-{\rho^{\op}}\ar@<5pt>[l]^-G & {\proth(\cat{A})^{\op}.}\ar@<5pt>[l]^-{\iota^{\op}}
}
\]

\begin{defn}
\label{defn:str-general}
Write $\str \from \catover{\cat{B}} \to \proth(\cat{A})$ for the composite
\[
\catover{\cat{B}} \toby{(\currylo)_!} \catover{\cat{X}(\cat{A},\cat{C})} \toby{F} (\cat{A}/\cat{X})^{\op} \toby{\rho^{\op}} \proth(\cat{A})^{\op}.
\]
\end{defn}

\begin{thm}
\label{thm:str-sem-adj-general}
The functors defined in Definitions~\ref{defn:sem-general} and~\ref{defn:str-general} form an adjunction
\[
\xymatrix{
{\catover{\cat{B} }}\ar@<5pt>[r]_-{\perp}^-{\str}\ & {\proth(\cat{A})^{\op},}\ar@<5pt>[l]^-{\sem}
}
\]
called the \demph{structure--semantics adjunction} induced by the aritation $\currylo \from \cat{B} \to \cat{X}(\cat{A},\cat{C})$.
\end{thm}
\begin{proof}
This is immediate from Lemma~\ref{lem:sem-factor-adjoints}.
\end{proof}

\section{Concrete \setting s}
\label{sec:concrete-settings}

So far we have dealt with settings as completely abstract 2-categories. In this section we consider settings equipped with a forgetful 2-functor to $\CAT$, and consider how the theory of proto-theories in such settings relates to the theory of proto-theories in the setting $\CAT$ as developed in Chapter~\ref{chap:adjunction}.

\begin{defn}
\label{defn:concrete-setting}
A \demph{concrete \setting} consists of a setting $(\cat{X}, \ofsfont{E}, \ofsfont{N})$ together with a 2-functor $\und \from \cat{X} \to \CAT $ that preserves cotensors strictly, sends 1-cells in $\ofsfont{E}$ to bijective-on-objects functors, and sends 1-cells in $\ofsfont{N}$ to full and faithful functors. For brevity, we will also write $\cat{A}_0$, $F_0$ and $\gamma_0$ for the result of applying $\und$ to an object $\cat{A}$, 1-cell $F$ and 2-cell $\gamma$ in $\cat{X}$ respectively.
\end{defn}

For the rest of this section, fix a concrete setting $(\cat{X}, \ofsfont{E}, \ofsfont{N}, \und)$ and an aritation $\currylo \from \cat{B} \to \cat{X}(\cat{A},\cat{C})$ in $\cat{X}$.

\begin{defn}
We define $(H_0)_{\bullet}$ to be the composite
\[
\cat{B} \toby{\currylo} \cat{X}(\cat{A},\cat{C}) \toby{\und} [\cat{A}_0, \cat{C}_0]
\]
and write
\[
\lpair - , - \rpair_0 \from \cat{A}_0 \times \cat{B} \to \cat{C}_0
\]
and
\[
(H_0)^{\bullet} \from \cat{A}_0 \to [\cat{B}, \cat{C}_0]
 \]
 for the functors corresponding to $(H_0)_{\bullet}$ under the cartesian closed structure of $\CAT$. The aritation that is defined in any of these equivalent ways is called the \demph{underlying plain aritation} of $\currylo$.
\end{defn}

\begin{remark}
Since $\und$ is required to preserve cotensors, we can identify $[\cat{B}, \cat{C}_0] = [\cat{B}, \cat{C}]_0$, and it can be easily checked that we have
\[
(H_0)^{\bullet} = \und (\curryhi) \from \cat{A}_0 \to[\cat{B}, \cat{C}]_0 = [\cat{B}, \cat{C}_0].
\]
\end{remark}

\begin{defn}
Write $\sem_0$ and $\str_0$ for the semantics and structure functors for the underlying plain aritation of $\currylo$. Thus we have
\[
\xymatrix{
{\catover{\cat{B} }}\ar@<5pt>[r]_-{\perp}^-{\str_0}\ & {\proth(\cat{A}_0)^{\op}.}\ar@<5pt>[l]^-{\sem_0}
}
\]
\end{defn}

\begin{prop}
The triangle
\[
\xymatrix{
\catover{\cat{B}}\ar[r]^{\str}\ar[dr]_{\str_0} & \proth(\cat{A})^{\op}\ar[d]^{\und} \\
& \proth(\cat{A}_0)^{\op}
}
\]
commutes up to isomorphism, where the vertical arrow marked $\und$ is the evident functor induced by $\und \from \cat{X} \to \CAT$.
\end{prop}

\begin{proof}
Consider the diagram
\[
\xymatrix{
\catover{\cat{B}}\ar[r]^-{(\currylo)_!}\ar[dr]_{((H_0)_{\bullet})_!} & \catover{\cat{X}(\cat{A}, \cat{C})}\ar[d]^{\und_!}\ar[r]^-F & \cat{A} / \cat{X}\ar[r]^-{\rho}\ar[d]^{\und} & \proth(\cat{A})\ar[d]^{\und} \\
& \catover{[\cat{A}_0, \cat{C}_0]}\ar[r]_-{F} & \cat{A}_0 / \CAT\ar[r]_-{\rho} & \proth(\cat{A}_0),
}
\]
where $\und_! \from \catover{\cat{X}(\cat{A}, \cat{C})} \to \catover{[\cat{A}_0, \cat{C}_0]}$ is the functor given by post-composition with $\und \from \cat{X}(\cat{A}, \cat{C}) \to [\cat{A}_0, \cat{C}_0]$, and the other two vertical arrows are induced in the evident way by $\und \from \cat{X} \to \CAT$.

The top-right composite in this diagram is $\und \of \str$, whereas the bottom-left composite is $\str_0$. Thus if we can show that each of the cells in this diagram commutes, we will be done. But the left-hand triangle commutes since $(H_0)_{\bullet}$ is by definition the composite $\und \of \currylo$, the middle square commutes because $\und$ preserves cotensors, and the right-hand square commutes because $\und$ respects the factorisation systems on $\cat{X}$ and $\CAT$.
\end{proof}

\section{Settings arising from 2-monads}
\label{sec:settings-2-monads}

In this section we examine how certain 2-monads on $\CAT$ naturally give rise to concrete settings. Throughout the section, we fix a 2-monad $\mnd{T} = (T, \eta, \mu)$ on $\CAT$. Recall from Definition~\ref{defn:2-monad-algebras} that $\talg$ denotes the 2-category of strict $\mnd{T}$-algebras, pseudo-$\mnd{T}$-morphisms and $\mnd{T}$-transformations.

\begin{prop}
\label{prop:cotensors-monad-lift}
The category $\talg$ is cotensored over $\CAT$, and cotensors are preserved by the forgetful functor to $\CAT$.
\end{prop}

\begin{proof}
This is Proposition~2.5 in Blackwell, Kelly and Power~\cite{blackwellKellyPower89}.
\end{proof}

\begin{prop}
\label{prop:boff-monad-lift}
Suppose the 2-functor $T \from \CAT \to \CAT$ preserves bijective-on-objects functors. Then there is a factorisation system $(\ofsfont{E}, \ofsfont{N})$ on $\talg$, where $\ofsfont{E}$ and $\ofsfont{N}$ are the classes of pseudo-$\mnd{T}$-morphisms whose underlying functors are bijective-on-objects and full and faithful respectively. Furthermore, in the factorisation of any morphism of $\talg$ as a member of $\ofsfont{E}$ followed by a member of $\ofsfont{N}$, the first factor can be taken to be a \emph{strict} $\mnd{T}$-morphism.
\end{prop}
This result is known, however it does not appear in the existing literature as far as I know. The condition that $T$ preserve bijective-on-objects functor is well-known and important in the literature on 2-monads. In particular it is shown in Power~\cite{power89} that such 2-monads satisfy a strong coherence result, namely that every pseudo-algebra is equivalent to a strict algebra. In fact the proof of the main result (Theorem~3.4) of~\cite{power89} and the proof of Proposition~\ref{prop:boff-monad-lift} make use of the bijective-on-objects/full-and-faithful factorisation system in a similar way, using it to factor a certain 1-cell and equipping each factor with an algebra morphism structure.

\begin{proof}
It is clear that $\ofsfont{E}$ and $\ofsfont{N}$ both contain all the isomorphisms and are closed under composition. To show that they form a factorisation system (Definition~\ref{defn:fact-system}), we must show that every pseudo-$\mnd{T}$-morphism factors as a member of $\ofsfont{E}$ followed by a member of $\ofsfont{N}$, and that every member of $\ofsfont{E}$ is left orthogonal to every member of $\ofsfont{N}$.

We begin by showing that any pseudo-$\mnd{T}$-morphism $(H,h) \from (\cat{A}, W) \to (\cat{C},Y)$ factors as a member of $\ofsfont{E}$ that is in addition a \emph{strict} $\mnd{T}$-morphism, followed by a member of $\ofsfont{N}$. Let
\[
\cat{A} \toby{L} \cat{K} \toby{J} \cat{C}
\]
be the bijective-on-objects/full-and-faithful factorisation of the underlying functor $H$. We must equip $\cat{K}$ with a $\mnd{T}$-algebra structure $V \from T\cat{K} \to \cat{K}$  and $J$ with a pseudo-$\mnd{T}$-morphism structure $j \from Y \of TJ \to J \of V$ such that $L$ becomes a strict $\mnd{T}$-morphism $(\cat{A}, W) \to (\cat{K}, V)$ and $(H, h) = (J, j) \of L$.

We have a natural isomorphism
\[
\xymatrix{
T\cat{A} \ar[r]^{W}\ar[d]_{TL}\drrtwocell\omit{^h} & \cat{A}\ar[r]^L & \cat{K}\ar[d]^{J} \\
T\cat{K} \ar[r]_{TJ} & T\cat{C}\ar[r]_Y & \cat{C}
}
\]
and $TL$ is bijective on objects and $J$ is full and faithful. Therefore, since the bijective-on-objects and full and faithful functors form an enhanced factorisation system (Lemma~\ref{lem:bo-ff-factorisation-enhanced}) this square has a unique fill-in. That is, there is a unique functor $V \from T\cat{K} \to \cat{K}$ and an isomorphism $j \from Y \of TJ \to J \of V$ such that $V \of TL = L \of W$, and $ jTL = h \from Y \of TH = Y \of TJ \of TL \to J \of L \of W = H \of W$. Diagrammatically, we have the following equality of natural isomorphisms:
\begin{equation}
\label{eq:pseudo-T-factor}
\vcenter{
\xymatrix{
T\cat{A}\ar[r]^{W}\ar[d]_{TL} & \cat{A}\ar[r]^L & \cat{K}\ar[d]^{J} \\
T\cat{K}\ar[r]_{TJ}\ar[urr]^{V} & T\cat{C}\ar[r]_Y^*!/^5pt/{ \labelstyle j \objectstyle \Uparrow} & \cat{C}
}}
\quad = \quad
\vcenter{
\xymatrix{
T\cat{A} \ar[r]^{W}\ar[d]_{TL}\drrtwocell\omit{^h} & \cat{A}\ar[r]^L & \cat{K}\ar[d]^{J} \\
T\cat{K} \ar[r]_{TJ} & T\cat{C}\ar[r]_Y & \cat{C}.
}}
\end{equation}
We must check that $V$ is a $\mnd{T}$-algebra structure, that $L$ is a strict $\mnd{T}$-morphism and that $(J,j)$ is a pseudo-$\mnd{T}$-morphism. It is then clear from the above that $(J, j) \of L = (H, h)$.

It is evident that if $V$ is an algebra structure, then $L$ is a strict $\mnd{T}$-morphism since by definition $V \of TL = L \of W$. So we check that $V$ is an algebra structure, and simultaneously show that $(J, j)$ is a pseudo-$\mnd{T}$-morphism.

First we show that $V \of \eta_{\cat{K}} = \id_{\cat{K}}$. One of the conditions for $(H, h)$ to be a pseudo-$\mnd{T}$-morphism is that we have an equality of two-cells
\[
\vcenter{
\xymatrix{
\cat{A}\ar[rr]^{\eta_{\cat{A}}} \ar[d]_L & & T\cat{A}\ar[r]^{W}\ar[d]_{TH}\drrtwocell\omit{^h} & \cat{A} \ar[r]^L & \cat{K}\ar[d]^J \\
\cat{K}\ar[r]_{J} & \cat{C}\ar[r]_{\eta_{\cat{C}}} & T\cat{C}\ar[rr]_{Y} & & \cat{C}
}}
\quad = \quad
\vcenter{
\xymatrix{
\cat{A} \ar[r]^L \ar[d]_{L} & \cat{K}\ar[d]^J \\
\cat{K}\ar[r]_J & \cat{C},
}}
\]
noting that these two two-cells do have the same domain and codomain since $W \of \eta_{\cat{A}} = \id_{\cat{A}}$ and similarly for $Y$. But the identity is clearly a fill-in for the right-hand square, and the left-hand square is equal to
\[
\xymatrix{
\cat{A}\ar[rr]^{\eta_{\cat{A}}} \ar[dd]_L & & T\cat{A}\ar[r]^{W}\ar[d]_{TL} & \cat{A} \ar[r]^L & \cat{K}\ar[dd]^J \\
& & T{\cat{K}}\ar[d]^{TJ}\ar[urr]_{V}\drrtwocell\omit{^<-2>j} & & \\
\cat{K}\ar[r]_{J}\ar[urr]^{\eta_{\cat{K}}} & \cat{C}\ar[r]_{\eta_{\cat{C}}} & T\cat{C}\ar[rr]_{Y} & & \cat{C}
}
\]
 by Equation~\bref{eq:pseudo-T-factor}, so $(V \of \eta_{\cat{K}}, j \eta_{\cat{K}})$ is a fill-in. Hence, by uniqueness of fill-ins, we have $V \of \eta_{\cat{K}} = \id_{\cat{K}}$ as required, and in addition the two cell
\[
\xymatrix{
\cat{K}\ar[r]^{\eta_{\cat{K}}} \ar[d]_{J} & T\cat{K} \ar[r]^V\ar[d]_{TJ}\drtwocell\omit{^j} & \cat{K}\ar[d]^J \\
\cat{C}\ar[r]_{\eta_{\cat{C}}} & T\cat{C}\ar[r]_{Y} & \cat{C}
}
\]
is the identity on $J$, which is one of the axioms required for $j$ to be a pseudo-$\mnd{T}$-morphism. 

Now we show that $V \of TV = V \of \mu_{\cat{K}}$. One of the conditions that $(H, h)$ satisfies as a pseudo-$\mnd{T}$-morphism is the equality of two-cells
\[
\vcenter{
\xymatrix
@C=17pt{
TT\cat{A}\ar[rr]^{\mu_{\cat{A}}}\ar[dd]_{TTL} & & T\cat{A}\ar[r]^W\ar[d]_{TL}\ddrrtwocell\omit{^h} & \cat{A}\ar[r]^L & \cat{K}\ar[dd]^J \\
& & T\cat{K}\ar[d]_{TJ} & & \\
TT\cat{K}\ar[r]_{TTJ} & TT \cat{C}\ar[r]_{\mu_\cat{C}} & T \cat{C}\ar[rr]_{Y} & & \cat{C}
}}
\quad = \quad
\vcenter{
\xymatrix
@C=17pt{
TT\cat{A}\ar[rr]^{TW}\ar[dd]_{TTL}\ddrrtwocell\omit{^{Th\:\:\:}} & & T\cat{A}\ar[r]^W\ar[d]_{TL}\ddrrtwocell\omit{^h} & \cat{A}\ar[r]^L & \cat{K}\ar[dd]^J \\
& & T\cat{K}\ar[d]_{TJ} & & \\
TT\cat{K}\ar[r]_{TTJ} & TT \cat{C}\ar[r]_{TY} & T \cat{C}\ar[rr]_{Y} & & \cat{C}.
}}
\]
Note that these two two-cells do have the same domain and codomain, since $W \of \mu_{\cat{A}} = W \of TW$, and similarly for $Y$. But the left-hand two-cell is equal to
\[
\xymatrix{
TT\cat{A}\ar[rr]^{\mu_{\cat{A}}}\ar[dd]_{TTL} & & T\cat{A}\ar[r]^W\ar[d]_{TL} & \cat{A}\ar[r]^L & \cat{K}\ar[dd]^J \\
& & T\cat{K}\ar[d]_{TJ}\ar[urr]_{V}\drrtwocell\omit{^<-2>j} & & \\
TT\cat{K}\ar[r]_{TTJ}\ar[urr]^{\mu_{\cat{K}}} & TT \cat{C}\ar[r]_{\mu_\cat{C}} & T \cat{C}\ar[rr]_{Y} & & \cat{C},
}
\]
by Equation~\bref{eq:pseudo-T-factor} and so $(V \of \mu_{\cat{K}}, j \mu_{\cat{K}})$ is a fill-in. On the other hand, the right-hand square is equal to 
\[
\xymatrix{
TT\cat{A}\ar[rr]^{TW}\ar[dd]_{TTL} & & T\cat{A}\ar[r]^W\ar[d]_{TL} & \cat{A}\ar[r]^L & \cat{K}\ar[dd]^J \\
& & T\cat{K}\ar[d]^{TJ}\ar[urr]_{V}\drrtwocell\omit{^<-2>j} & & \\
TT\cat{K}\ar[r]_{TTJ}\ar[urr]^{TV} & TT \cat{C}\ar[r]_{TY}^*!/u5pt/{\labelstyle{Tj} \objectstyle\Uparrow \:\:\:} & T \cat{C}\ar[rr]_{Y} & & \cat{C},
}
\]
by Equation~\bref{eq:pseudo-T-factor} (and the result of applying $T$ applied to the same equation), so $(V \of TV,  j TV \of Y (Tj))$ is a fill-in. By the uniqueness of fill-ins, it follows that $V \of TV = V \of \mu_{\cat{K}}$, and in addition, we have an equality of two-cells
\[
\vcenter{
\xymatrix{
TT\cat{K} \ar[r]^{\mu_{\cat{K}}}\ar[d]_{TTJ} & T{\cat{K}}\ar[r]^{V}\ar[d]_{TJ}\drtwocell\omit{^j} & \cat{K}\ar[d]^J \\
TT\cat{C}\ar[r]_{\mu_{\cat{C}}} & T\cat{C}\ar[r]_Y & \cat{C}
}}
\quad = \quad
\vcenter{
\xymatrix{
TT\cat{K} \ar[r]^{TV}\ar[d]_{TTJ}\drtwocell\omit{^Tj\:\:\:} & T{\cat{K}}\ar[r]^{V}\ar[d]_{TJ}\drtwocell\omit{^j} & \cat{K}\ar[d]^J \\
TT\cat{C}\ar[r]_{TY} & T\cat{C}\ar[r]_Y & \cat{C},
}}
\]
which is one of the identities required for $(J, j)$ to be a pseudo-$\mnd{T}$-morphism. That completes the proof that $V$ is a $\mnd{T}$-algebra structure on $\cat{K}$, and that $(J, j)$ is a pseudo-$\mnd{T}$-morphism $(\cat{K}, V ) \to (\cat{C}, Y)$.

Now we must show that for any commutative square
\begin{equation}
\label{eq:pseudo-T-lift}
\xymatrix{
(\cat{A}, W) \ar[r]^{(F,f)}\ar[d]_{(E, e)} & (\cat{B}, X)\ar[d]^{(N,n)} \\
(\cat{C},Y)\ar[r]_{(G, g)} & (\cat{D}, Z)
}
\end{equation}
in $\talg$ where $E$ is bijective on objects  and $N$ is full and faithful, there is a unique $(H, h) \from (\cat{C}, Y) \to (\cat{B},X)$ such that $(H, h) \of (E,e) = (F, f)$ and $(N, n) \of (H, h) = (G,g)$. Since the corresponding square of underlying functors commutes, there is a unique functor $H \from \cat{C} \to \cat{B}$ such that both triangles in
\[
\xymatrix{
\cat{A}\ar[r]^F\ar[d]_E & \cat{B}\ar[d]^{N} \\
\cat{C}\ar[r]_G\ar[ur]^H & \cat{D}
}
\]
commute. Thus we just need to show that there is a unique pseudo-$\mnd{T}$-morphism structure $h$ on $H$ such that $(H, h) \of (E, e) = (F, f)$ and $(N, n) \of (H, h) = (G,g)$.

Consider the square
\[
\xymatrix{
T\cat{A}\ar[r]^W \ar[d]_{TE}\drrtwocell\omit{^e} & \cat{A}\ar[r]^E & \cat{C}\ar@{=}[d] \\
T\cat{C}\ar[rr]_Y & & \cat{C}.
}
\]
Let the unique fill-in of this natural isomorphism be given by $Y' \from T\cat{C} \to \cat{C}$ and $\kappa \from Y \to Y'$. Then we have
\[
\vcenter{
\xymatrix{
T\cat{A}\ar[r]^W\ar[d]_{TE}\drtwocell\omit{^e} & \cat{A}\ar[r]^F\ar[d]^E & \cat{B}\ar[d]^N \\
T\cat{C}\ar[r]_Y & \cat{C}\ar[r]_G & \cat{D}
}}
\quad = \quad
\vcenter{
\xymatrix{
T\cat{A}\ar[r]^W\ar[d]_{TE} & \cat{A}\ar[r]^F\ar[d]^E & \cat{B}\ar[d]^N \\
T\cat{C}\rtwocell^{Y'}_{Y}{^\kappa} & \cat{C}\ar[r]_G\ar[ur]_H & \cat{D},
}}
\]
that is, $(H \of Y', G\kappa)$ is a fill-in for the natural isomorphism displayed on the left.

The commutativity of the square~\bref{eq:pseudo-T-lift} implies an equality of natural transformations
\[
\vcenter{
\xymatrix{
T \cat{A} \ar[r]^W \ar[dr]_{TF} \ar[dd]_{TE}\rrtwocell\omit{^<3>f} & \cat{A} \ar[r]^F & \cat{B}\ar[dd]^N \\
& T\cat{B}\ar[ur]_X\ar[d]_{TN}\drtwocell\omit{^<-1>n} & \\
T\cat{C} \ar[r]_{TG} & T\cat{D} \ar[r]_Z& \cat{D}
}}
\quad = \quad
\vcenter{
\xymatrix{
T\cat{A} \ar[r]^W\ar[dd]_{TE}\drtwocell\omit{^<1>e} & \cat{A} \ar[d]^E\ar[r]^F & \cat{B}\ar[dd]^N \\
& \cat{C}\ar[dr]^G & \\
T\cat{C}\ar[r]_{TG}\ar[ur]^Y\rrtwocell\omit{^<-3>g} & T\cat{D} \ar[r]_Z & \cat{D}.
}}
\]
It follows from Lemma~\ref{lem:enhanced-2-dim} that there is a unique natural transformation $h' \from X\of TH \to H \of Y'$ such that
\begin{equation}
\label{eq:enhanced-lift-property-1}
\vcenter{
\xymatrix{
T\cat{A}\ar[r]^W\ar[dd]_{TE} & \cat{A}\ar[r]^F\ar[d]_E & \cat{B}\ar@{=}[dd] \\
 & \cat{C}\ar[dr]^H & \\
T\cat{C}\ar[ur]^{Y'}\ar[r]_{TH}\rrtwocell\omit{^<-3> h'} & T\cat{B}\ar[r]_X  & \cat{B}
}}
\quad = \quad
\vcenter{
\xymatrix{
T\cat{A}\ar[r]^W\ar[d]_{TE}\ar[dr]^{TF} & \cat{A}\ar[r]^F\drtwocell\omit{^f} & \cat{B}\ar@{=}[d] \\
T\cat{C}\ar[r]_{TH} & T \cat{B}\ar[r]_X & \cat{B}
}}
\end{equation}
and
\begin{equation}
\label{eq:enhanced-lift-property-2}
\vcenter{
\xymatrix{
& \cat{C}\ar[r]^H  & \cat{B}\ar[dd]^N \\
T\cat{C}\ar[ur]^{Y'}\ar[dr]_{TG}\ar[r]^{TH}\urrtwocell\omit{^h'}  & T\cat{B}\ar[d]_{TN}\ar[ur]_X\drtwocell\omit{^n} & \\
&  T\cat{D}\ar[r]_Z & \cat{D}
}}
\quad = \quad
\vcenter{
\xymatrix{
& \cat{C}\ar[r]^H\ar[ddr]^G & \cat{B}\ar[dd]^N \\
T\cat{C}\urtwocell^{Y'}_Y{^\kappa}\ar[dr]_{TG}\drrtwocell\omit{^g} & &  \\
& T\cat{D}\ar[r]_Z & \cat{D}.
}}
\end{equation}
The first of these equalities implies that $h'$ is an isomorphism, since $f$ is an isomorphism and $TE$ is bijective on objects.

Let $h \from X \of TH \to H \of Y$ be the composite natural isomorphism
\[
\xymatrix
@C=40pt
@R=40pt{
T\cat{C}\rtwocell^Y_{Y'}{^\kappa^{-1}\:\:\:}\ar[d]_{TH}\drtwocell\omit{^h'} & \cat{C}\ar[d]^H  \\
 T\cat{B}\ar[r]_X & \cat{B}.
}
\]
We will shortly show that $(H,h)$ is a pseudo-$\mnd{T}$-morphism $(\cat{C}, Y) \to (\cat{B}, X)$. First however, note that
\begin{align*}
&\vcenter{
\xymatrix{
T\cat{A} \ar[r]^{TE}\ar[d]_W\drtwocell\omit{e} & T\cat{C}\ar[r]^H\ar[d]_Y\drtwocell\omit{h} & T\cat{B}\ar[d]^X \\
\cat{A}\ar[r]_E & \cat{C}\ar[r]_H & \cat{B}
}}
\quad = \quad
\vcenter{
\xymatrix{
T\cat{A} \ar[r]^{TE}\ar[d]_W & T\cat{C}\ar[r]^H\dtwocell^Y_{Y'}{\kappa}\drtwocell\omit{h} & T\cat{B}\ar[d]^X \\
\cat{A}\ar[r]_E & \cat{C}\ar[r]_H & \cat{B}
}} \\
\quad = \quad &
\vcenter{
\xymatrix{
T\cat{A} \ar[r]^{TE}\ar[d]_W & T\cat{C}\ar[r]^H\ar[d]_{Y'}\drtwocell\omit{h'} & T\cat{B}\ar[d]^X \\
\cat{A}\ar[r]_E & \cat{C}\ar[r]_H & \cat{B}
}}
\quad = \quad
\vcenter{
\xymatrix{
T\cat{A}\ar[r]^{TF}\ar[d]_W\drtwocell\omit{f} & T\cat{B}\ar[d]^X \\
\cat{A}\ar[r]_F & \cat{B},
}}
\end{align*}
with the first equality following from the definition of $\kappa$, the second from the definition of $h$, and the third from Equation~\bref{eq:enhanced-lift-property-1}. Thus we will have $(H,h) \of (E, e) = (F,f)$, and the fact that $h'$ is unique satisfying Equation~\bref{eq:enhanced-lift-property-1} implies that $h$ is the unique pseudo-$\mnd{T}$-morphism structure on $H$ making this equality hold. Similarly we have
\begin{align*}
&\vcenter{
\xymatrix{
T\cat{C}\ar[r]^{TH}\ar[d]_Y\drtwocell\omit{h} & T\cat{B}\ar[r]^{TN}\ar[d]_X\drtwocell\omit{n} & T\cat{D}\ar[d]^Z \\
\cat{C}\ar[r]_H & \cat{B}\ar[r]_N & \cat{D}
}}
\quad = \quad
\vcenter{
\xymatrix{
T\cat{C}\ar[r]^{TH}\dtwocell^{Y'}_Y{\kappa^{-1}}\drtwocell\omit{h'} & T\cat{B}\ar[r]^{TN}\ar[d]_X\drtwocell\omit{n} & T\cat{D}\ar[d]^Z \\
\cat{C}\ar[r]_H & \cat{B}\ar[r]_N & \cat{D}
}} \\
& \quad = \quad
\vcenter{
\xymatrix
@R=40pt
@C=40pt{
T\cat{C}\duppertwocell^Y{\kappa}\dlowertwocell_{Y}{\kappa^{-1}}\drtwocell\omit{<-1>g}\ar[d]|(0.3){Y'}\ar[r]^{TG} & T\cat{D}\ar[d]^Z \\
\cat{C}\ar[r]_G & \cat{D}
}}
\quad = \quad
\vcenter{
\xymatrix{
T\cat{C}\ar[r]^{TG}\ar[d]_Y\drtwocell\omit{g} & T\cat{D}\ar[d]^Z \\
\cat{C}\ar[r]_G & \cat{D},
}}
\end{align*}
where the first equality follows from the definition of $h$, and the second from Equation~\bref{eq:enhanced-lift-property-2}. Thus we will have $(N, n) \of (H, h) = (G,g)$.

So all that remains is to show that $h$ is in fact a pseudo-$\mnd{T}$-morphism structure on $H$. First we show that the natural isomorphism
\[
\vcenter{
\xymatrix{
\cat{C}\ar[r]^{\eta_{\cat{C}}}\ar[d]_H & T\cat{C}\ar[r]^Y\ar[d]_{TH} \drtwocell\omit{^h} & \cat{C}\ar[d]^H \\
\cat{B}\ar[r]_{\eta_{\cat{B}}} & T\cat{B}\ar[r]_X & \cat{B}
}}
\]
is the identity. Since $E \from \cat{A} \to \cat{C}$ is bijective on objects, it is sufficient to show that it becomes the identity when whiskered with $E$. Note however, that the natural isomorphism
\[
\vcenter{
\xymatrix{
\cat{A}\ar[r]^{\eta_{\cat{A}}}\ar[d]_E & T\cat{A}\ar[r]^W\ar[d]_{TE} \drtwocell\omit{^e} & \cat{A}\ar[d]^E \\
\cat{C}\ar[r]_{\eta_{\cat{C}}} & T\cat{C}\ar[r]_Y & \cat{C}
}}
\]
is an identity since $(E, e)$ is a pseudo-$\mnd{T}$-morphism, so the natural isomorphism above is an identity if and only if it becomes the identity when pasted with this one. But we have
\[
\vcenter{
\xymatrix{
\cat{A}\ar[r]^{\eta_{\cat{A}}}\ar[d]_E & T\cat{A}\ar[r]^W\ar[d]_{TE} \drtwocell\omit{^e} & \cat{A}\ar[d]^E \\
\cat{C}\ar[r]^{\eta_{\cat{C}}}\ar[d]_H & T\cat{C}\ar[r]^Y\ar[d]_{TH} \drtwocell\omit{^h} & \cat{C}\ar[d]^H \\
\cat{B}\ar[r]_{\eta_{\cat{B}}} & T\cat{B}\ar[r]_X & \cat{B}
}}
\quad = \quad
\vcenter{
\xymatrix{
\cat{A}\ar[r]^{\eta_{\cat{A}}}\ar[d]_F & T\cat{A}\ar[r]^W\ar[d]_{TF} \drtwocell\omit{^f} & \cat{A}\ar[d]^F \\
\cat{B}\ar[r]_{\eta_{\cat{B}}} & T\cat{B}\ar[r]_X & \cat{B}
}}
\]
as established above. And the right-hand natural isomorphism is the identity as required since $(F,f)$ is a pseudo-$\mnd{T}$-morphism. Finally we show that
\begin{equation}
\label{eq:2-monad-lift-mult-condition}
\vcenter{
\xymatrix{
TT\cat{C}\ar[r]^{\mu_{\cat{C}}}\ar[d]_{TTH} & T\cat{C}\ar[r]^Y\ar[d]_{TH} \drtwocell\omit{^h} & \cat{C}\ar[d]^H \\
TT\cat{B}\ar[r]_{\mu_{\cat{B}}} & T\cat{B}\ar[r]_X & \cat{B}
}}
\quad = \quad
\vcenter{
\xymatrix{
TT\cat{C}\ar[r]^{TY}\ar[d]_{TTH}\drtwocell\omit{^Th\:\:\:} & T\cat{C}\ar[r]^Y\ar[d]_{TH} \drtwocell\omit{^h} & \cat{C}\ar[d]^H \\
TT\cat{B}\ar[r]_{TX} & T\cat{B}\ar[r]_X & \cat{B}.
}}
\end{equation}
As before, it is sufficient to show that these natural transformations become equal when whiskered with the bijective-on-objects $TTE \from TT\cat{A} \to TT\cat{C}$. In addition, the natural transformation
\[
\vcenter{
\xymatrix{
TT\cat{A}\ar[r]^{\mu_{\cat{A}}}\ar[d]_{TTE} & T\cat{A}\ar[r]^W\ar[d]_{TE} \drtwocell\omit{^e} & \cat{A}\ar[d]^E \\
TT\cat{C}\ar[r]_{\mu_{\cat{C}}} & T\cat{C}\ar[r]_Y & \cat{C}
}}
\]
is an isomorphism, so it is sufficient to check that they become equal when pasted with this. We have
{\allowdisplaybreaks
\begin{align*}
& \vcenter{
\xymatrix{
TT\cat{A}\ar[r]^{\mu_{\cat{A}}}\ar[d]_{TTE} & T\cat{A}\ar[r]^W\ar[d]_{TE} \drtwocell\omit{^e} & \cat{A}\ar[d]^E \\
TT\cat{C}\ar[r]^{\mu_{\cat{C}}}\ar[d]_{TTH} & T\cat{C}\ar[r]^Y\ar[d]_{TH} \drtwocell\omit{^h} & \cat{C}\ar[d]^H \\
TT\cat{B}\ar[r]_{\mu_{\cat{B}}} & T\cat{B}\ar[r]_X & \cat{B}
}}
\quad = \quad
\vcenter{
\xymatrix{
TT\cat{A}\ar[r]^{\mu_{\cat{A}}}\ar[d]_{TTF} & T\cat{A}\ar[r]^W\ar[d]_{TF} \drtwocell\omit{^f} & \cat{A}\ar[d]^F \\
TT\cat{B}\ar[r]_{\mu_{\cat{B}}} & T\cat{B}\ar[r]_X & \cat{B}
}} \\
= \quad &
\vcenter{
\xymatrix{
TT\cat{A}\ar[r]^{TW}\ar[d]_{TTF}\drtwocell\omit{^Tf\:\:\:} & T\cat{A}\ar[r]^W\ar[d]_{TF} \drtwocell\omit{^f} & \cat{A}\ar[d]^F \\
TT\cat{B}\ar[r]_{TX} & T\cat{B}\ar[r]_X & \cat{B}
}}
\quad = \quad
\vcenter{
\xymatrix{
TT\cat{A}\ar[r]^{TW}\ar[d]_{TTE}\drtwocell\omit{^Te\:\:\:} & T\cat{A}\ar[r]^W\ar[d]_{TE} \drtwocell\omit{^e} & \cat{A}\ar[d]^E \\
TT\cat{C}\ar[r]^{TY}\ar[d]_{TTH}\drtwocell\omit{^Th\:\:\:} & T\cat{C}\ar[r]^Y\ar[d]_{TH} \drtwocell\omit{^h} & \cat{C}\ar[d]^H \\
TT\cat{B}\ar[r]_{TX} & T\cat{B}\ar[r]_X & \cat{B}
}} \\
= \quad &
\vcenter{
\xymatrix
@R=10pt{
TT\cat{A}\ar[r]^{\mu_{\cat{A}}}\ar[dd]_{TTE} & T\cat{A}\ar[r]^W\ar[d]_{TE} \ddrtwocell\omit{^<-1>e} & \cat{A}\ar[dd]^E \\
& T\cat{C}\ar[dr]^Y & \\
TT\cat{C}\ar[ur]^{\mu_{\cat{C}}}\ar[dr]^{TY}\ar[dd]_{TTH} &  & \cat{C}\ar[dd]^H \\
\drtwocell\omit{^<-1>Th\:\:\:}& T \cat{C}\ar[ur]^Y\ar[d]_{TH} \drtwocell\omit{^<-1>h} & \\
TT\cat{B}\ar[r]_{\mu_{\cat{B}}} & T\cat{B}\ar[r]_X & \cat{B},
}}
\end{align*}
}
where the last equality comes from the fact that $(E,e)$ is a pseudo-$\mnd{T}$-morphism.
\end{proof}

\begin{cor}
\label{cor:proth-strict-equiv}
Suppose $T$ preserves bijective-on-objects functors and let $(E,e) \from (\cat{A}, W) \to (\cat{L}, U)$ be a bijective-on-objects pseudo-$\mnd{T}$-morphism. Then there is a bijective-on-objects \emph{strict} $\mnd{T}$-morphism $E' \from (\cat{A}, W) \to (\cat{L}', U')$ and a pseudo-$\mnd{T}$-isomorphism $(I, i) \from (\cat{L}', U') \to (\cat{L}, U)$ such that $(I, i) \of E' = (E, e)$.
\end{cor}
\begin{proof}
By the previous theorem, $(E,e)$ has a factorisation as a bijective-on-objects strict $\mnd{T}$-morphism $E' \from(\cat{A}, W) \to (\cat{L}', U')$ followed by a full and faithful pseudo-$\mnd{T}$-morphism $(I, i)$. But since $E$ and $E'$ are both bijective on objects, $I$ must be as well. Therefore $(I, i) \in \ofsfont{E} \cap \ofsfont{N}$, so is an isomorphism in $\talg$.
\end{proof}

\begin{thm}
\label{thm:monad-setting}
If $T$ preserves bijective-on-objects functors then $\talg$ is a concrete setting in a canonical way. In particular, for each $(\cat{A}, W), (\cat{C}, Y) \in \talg$, $\cat{B} \in \CAT$ and $\currylo \from \cat{B} \to \talg ((\cat{A}, W), (\cat{C}, Y) )$ there is a structure--semantics adjunction
\[
\xymatrix{
{\catover{\cat{B} }}\ar@<5pt>[r]_-{\perp}^-{\str}\ & {\proth(\cat{A}, W)^{\op}.}\ar@<5pt>[l]^-{\sem}
}
\]
\end{thm}
\begin{proof}
Let $\ofsfont{E}$ be the class of pseudo-$\mnd{T}$-morphisms whose underlying functors are bijective on objects, and let $\ofsfont{N}$ be the class of pseudo-$\mnd{T}$-morphisms whose underlying functors are full and faithful. Then by Proposition~\ref{prop:boff-monad-lift}, $(\ofsfont{E},\ofsfont{N})$ is a factorisation system on $\talg$, and it is evidently preserved by $\und$. By Proposition~\ref{prop:cotensors-monad-lift}, $\talg$ has cotensors over $\CAT$ and they are preserved by $\und$. Thus $(\talg, \ofsfont{E}, \ofsfont{N}, \und)$ is a concrete setting.
\end{proof}

\section{Examples}
\label{sec:structure-examples}

We saw how the structure--semantics adjunction for monoids arises from an aritation in Section~\ref{sec:str-sem-adj-monoids}, and monads (possibly with arities) were discussed in Chapter~\ref{chap:canonical1}. In this section we examine how the remaining notions of algebraic theory from Chapter~\ref{chap:notions} arise via proto-theories and aritations in general settings.

\begin{thm}
\label{thm:examples-monads-boff}
There are 2-monads $\mnd{T}_1$, $\mnd{T}_2$ and $\mnd{T}_3$ on $\CAT$ such that:
\begin{itemize}
\item For $\mnd{T}_1$, the strict algebras are the finite product categories, the pseudo-$\mnd{T}_1$-morphisms are finite product preserving functors, and the $\mnd{T}_1$-transformations are natural transformations;
\item For $\mnd{T}_2$, the strict algebras are the symmetric monoidal categories, the pseudo-$\mnd{T}_2$-morphisms are symmetric monoidal functors, and the $\mnd{T}_2$-transformations are symmetric monoidal natural transformations; and
\item For $\mnd{T}_3$, the strict algebras are the monoidal categories, the pseudo-$\mnd{T}_3$-morphisms are monoidal functors, and the $\mnd{T}_3$-transformations are monoidal natural transformations.
\end{itemize}
Furthermore, each of these 2-monads preserves bijective-on-objects functors.
\end{thm}

\begin{proof}
In~\cite{kelly72} and~\cite{kelly72_2}, Kelly shows that these three 2-categories arise via a special kind of 2-monad, arising from a structure known as a ``club". But it is clear from the way a 2-monad is constructed from a club (outlined in Section~3 of~\cite{kelly72_2}) that such 2-monads preserve bijective-on-objects functors.
\end{proof}

\subsection*{Lawvere theories}
It follows from Theorems~\ref{thm:monad-setting} and~\ref{thm:examples-monads-boff} that the 2-category $\finprod$ of finite product categories, product-preserving functors and natural transformations is a concrete setting. Note also, that for any finite product category $\cat{B}$ (writing $\cat{B}$ also for the underlying category of $\cat{B}$) we have an aritation in $\finprod$ determined by the functor
\[
\cat{B} \to \finprod ( \fin^{\op}, \cat{B} )
\]
sending $b \in \cat{B}$ to the functor $b^{(-)} \from \fin^{\op} \to \cat{B}$. This gives rise to a structure--semantics adjunction
\[
\xymatrix{
{\catover{\cat{B} }}\ar@<5pt>[r]_-{\perp}^-{\str}\ & {\proth(\fin^{\op})^{\op}.}\ar@<5pt>[l]^-{\sem}
}
\]
But the categories $\proth(\fin^{\op})$ and $\LAW$ are clearly isomorphic by definition, and comparing Definitions~\ref{defn:sem-general} and~\ref{defn:lawvere-sem-objects}, we obtain the following.

\begin{prop}
\label{prop:law-sem-from-proth}
The functors
\[
\sem \from \proth(\fin^{\op})^{\op} \to \catover{\cat{B}}
\]
and
\[
\sem_{\LAW} \from \LAW^{\op} \to \catover{\cat{B}}
\]
coincide. \hfill \qed
\end{prop}

In particular, Proposition~\ref{prop:str-sem-adj-law} now follows from Theorem~\ref{thm:str-sem-adj-general}, and the structure--semantics adjunction for Lawvere theories is an instance of a structure--semantics adjunction for an aritation in a concrete setting.

In this thesis we have considered a notion of Lawvere theory that is very close to Lawvere's original definition. However, Lawvere theories can be generalised beyond this basic notion; for example, a theory of Lawvere theories relative to an arbitrary (possibly enriched) locally finitely presentable category is developed by Nishizawa and Power in \cite{nishizawaPower09}. Let us briefly discuss how this theory relates to ours; first let us fix some notation. Let $\cat{V}$ be a locally finitely presentable symmetric monoidal category, and $\cat{A}$ a locally finitely presentable $\cat{V}$-category in the sense of \cite{kelly82}. Write $\iota \from \cat{A}_f \incl \cat{A}$ for the full $\cat{V}$-inclusion of the category of finitely presentable objects of $\cat{V}$; then we have a corresponding nerve functor
\[
N_{\iota} \from \cat{A} \incl [\cat{A}^{\op}, \cat{V}] \toby{(\iota^{\op})^*} [\cat{A}_f^{\op}, \cat{V}].
\]
The following two definitions are Definitions~2.1 and~2.2 in \cite{nishizawaPower09}. 

\begin{defn}
A Lawvere $\cat{A}$-theory consists of a $\cat{V}$-category $\cat{L}$ together with an identity-on-objects strict finite $\cat{V}$-limit preserving functor $L \from \cat{A}_f^{\op} \to \cat{L}$.
\end{defn}

\begin{defn}
Given a Lawvere $\cat{A}$-theory $L \from \cat{A}_f^{\op} \to \cat{L}$, we define its $\cat{V}$-category $\mod(L)$ of models via the pullback
\[
\xymatrix{
\mod(L)\ar[r]\ar[d]\pullbackcorner & [\cat{L}, \cat{V}]\ar[d]^{L^*} \\
\cat{A} \ar[r]_{N_{\iota}} & [\cat{A}_f^{\op}, \cat{V}]
}
\]
in $\cat{V}$-$\Cat$.
\end{defn}
This definition is evidently closely related to the semantics of proto-theories defined in Definition~\ref{defn:sem-objects}. It is likely that this general notion of Lawvere theory could be reconciled with the theory of proto-theories and aritations that we have developed. However to do so, we would need notions of proto-theory and aritation relative to $\cat{V}\hyph\CAT$ rather than $\CAT$. Such notions are available (indeed, in Section~\ref{sec:open-more-general} we shall observe that one can discuss proto-theories and aritations in any symmetric monoidal category), but developing them is beyond the scope of this thesis.

Note that a Lawvere $\cat{A}$-theory is defined to be the \emph{identity} on objects, whereas for us Lawvere theories and other notions of proto-theory are merely \emph{bijective} on objects. This is a minor distinction: for every notion of proto-theory in this thesis there is an appropriate notion of an identity-on-objects proto-theory, and every proto-theory is isomorphic to an identity-on-objects one. However, in the full generality of an arbitrary setting we want the proto-theories to be those 1-cells that come from the left class of a specified factorisation system. This is the case for bijective-on-objects functors but not identity-on-objects functors. For this reason we require Lawvere theories (and other proto-theories) to be merely bijective on objects.

Another difference between our approach and that of \cite{nishizawaPower09} is that Nishizawa and Power emphasise the role of $\cat{V}$-categories with finite cotensors (which are finite powers in the unenriched case), whereas we typically consider models of Lawvere theories only in categories with all finite products. This is not an essential limitation: all of the theory we have developed for Lawvere theories could be translated in a straightforward way to the setting of finite power categories. Similarly the results in the following subsections on PROPs and PROs could be translated into settings of ``categories with (possibly symmetric) tensor powers'' in an appropriate sense, rather than (possibly symmetric) monoidal categories. The reason we chosen to restrict our attention to finite product categories and (symmetric) monoidal categories is that these categories are more familiar and well documented in the literature.

\subsection*{PROPs}

Again by Theorems~\ref{thm:monad-setting} and~\ref{thm:examples-monads-boff}, the 2-category $\symmon$ of large symmetric monoidal categories, (strong) symmetric monoidal functors and symmetric monoidal transformations is a concrete setting.
\begin{defn}
Let $\finbij$ be the (non-full) sub-symmetric monoidal category of the finite product category $\fin^{\op}$ consisting of the same objects, but only the invertible morphisms. Thus $\finbij$ has the natural numbers as objects, and $\finbij(n,m)$ is empty unless $n = m$, and $\finbij(n,n) = S_n$, the symmetric group on $n$ elements.
\end{defn}

\begin{lem}
\label{lem:PROP-proth-equiv}
The category $\PROP$ of PROPs is equivalent to the category $\proth(\finbij)$ of proto-theories with arities $\finbij$ in the setting $\symmon$.
\end{lem}
\begin{proof}
By Corollary~\ref{cor:proth-strict-equiv}, every object of $\proth(\finbij)$ is isomorphic to a bijective-on-objects \emph{strict} monoidal functor out of $\finbij$, and so the full subcategory of such proto-theories is equivalent to $\proth(\finbij)$ itself. Thus it is sufficient to show that $\PROP$ is isomorphic to this full subcategory. Let $\cat{P}$ be a PROP; then there is a unique strict monoidal functor $\finbij \to \cat{P}$ that is the identity on objects, and sends a permutation $\sigma \in S_n = \finbij(n,n)$ to the corresponding automorphism of $n = 1^{\tensor n}$ in $\cat{P}$ induced by the symmetry. Conversely, given a bijective-on-objects strict monoidal functor $P \from \finbij \to \cat{P}$, we can identify the objects of $\cat{P}$ with the natural numbers via $P$, and then $\cat{P}$ becomes a PROP. It is clear that any strict monoidal functor $F \from \cat{P} \to \cat{P}'$ between PROPs makes the triangle
\[
\xymatrix{
\finbij \ar[r]^P\ar[dr]_{P'} & \cat{P}\ar[d]^F \\
& \cat{P'}
}
\]
commute.
\end{proof}

\begin{defn}
Let $(\cat{B}, \tensor, I)$ be a symmetric monoidal category. Write $\currylo \from \cat{B} \to \symmon (\finbij, \cat{B}) $ for the canonical functor that sends $b \in \cat{B}$ to the functor $b^{\tensor (-)}$ that sends $n \in \finbij$ to $b^{\tensor n}$, the distinguished $n$-th tensor power of $b$, and sends a permutation $\sigma \in S_n$ to the automorphism of $b^{\tensor n}$ induced by $\sigma$ via the symmetry of $\cat{B}$.
\end{defn}

\begin{remark}
\label{rem:sym-mon-proth-adj}
For every large symmetric monoidal category $\cat{B}$, we can regard $\currylo \from \cat{B} \to \symmon (\finbij, \cat{B})$ defined above as an aritation, inducing a structure--semantics adjunction
\[
\xymatrix{
{\catover{\cat{B} }}\ar@<5pt>[r]_-{\perp}^-{\str}\ & {\proth(\finbij)^{\op}.}\ar@<5pt>[l]^-{\sem}
}
\]
\end{remark}

\begin{lem}
Let $P \from \finbij \to \cat{P}$ be an identity-on-objects, strict symmetric monoidal functor. Then a symmetric monoidal functor $\Gamma \from \cat{P} \to \cat{B}$ is a model of $\cat{P}$ as a PROP if and only if $\Gamma \of P = \Gamma(1)^{\tensor(-)} \from \finbij \to \cat{B}$.
\end{lem}
\begin{proof}
This condition says precisely that $\Gamma$ preserves the distinguished tensor powers of $1$ not only up to isomorphism, but strictly. But this is exactly what is required for $\Gamma$ to be a model of the PROP $\cat{P}$, as per Definition~\ref{defn:PROP-model}.
\end{proof}

\begin{prop}
Under the equivalence of Lemma \ref{lem:PROP-proth-equiv}, the semantics functor $\sem \from \proth(\finbij)^{\op} \to \catover{\cat{B}}$ from Remark~\ref{rem:sym-mon-proth-adj} corresponds to the semantics functor $\sem_{\PROP} \from \PROP^{\op} \to \catover{\cat{B}}$ from Proposition~\ref{prop:PROP-sem-functor}.
\end{prop}
\begin{proof}
This is immediate from the preceding lemma.
\end{proof}

\subsection*{PROs}

We can repeat all of the previous subsection with PROs in place of PROPs, and monoidal categories in place of symmetric monoidal categories. Specifically the 2-category $\monCAT$ of large monoidal categories, monoidal functors and monoidal natural transformations is a concrete setting, and the following results hold, with proofs identical to (or simpler than) the PROP case. Write $\nat$ for the monoid of natural numbers with addition, regarded as a discrete strict monoidal category.

\begin{lem}
\label{lem:PRO-proth-equiv}
The category $\PRO$ of PROs is equivalent to the category $\proth(\nat)$ of proto-theories with arities $\nat$ in the setting $\monCAT$.
\end{lem}

\begin{defn}
Let $(\cat{B}, \tensor, I)$ be a monoidal category. Write $\currylo \from \cat{B} \to \monCAT (\nat, \cat{B}) $ for the canonical functor that sends $b \in \cat{B}$ to the functor $b^{\tensor (-)}$ that sends $n \in \nat$ to $b^{\tensor n}$, the distinguished $n$-th tensor power of $b$.
\end{defn}

\begin{remark}
\label{rem:mon-proth-adj}
For every large monoidal category $\cat{B}$, we can regard $\currylo \from \cat{B} \to \monCAT (\nat, \cat{B})$ defined above as an aritation, inducing a structure--semantics adjunction
\[
\xymatrix{
{\catover{\cat{B} }}\ar@<5pt>[r]_-{\perp}^-{\str}\ & {\proth(\nat)^{\op}.}\ar@<5pt>[l]^-{\sem}
}
\]
\end{remark}

\begin{lem}
Let $P \from \nat \to \cat{P}$ be an identity-on-objects, strict monoidal functor. Then a monoidal functor $\Gamma \from \cat{P} \to \cat{B}$ is a model of $\cat{P}$ as a PRO if and only if $\Gamma \of P = \Gamma(1)^{\tensor(-)} \from \nat \to \cat{\cat{B}}$.
\end{lem}

\begin{prop}
Under the equivalence of Lemma \ref{lem:PRO-proth-equiv}, the semantics functor $\sem \from \proth(\nat)^{\op} \to \catover{\cat{B}}$ from Remark~\ref{rem:mon-proth-adj} corresponds to the semantics functor $\sem_{\PRO} \from \PRO^{\op} \to \catover{\cat{B}}$ from Proposition~\ref{prop:PROP-sem-functor}.
\end{prop}

\subsection*{Operads}

Operads can also be described as proto-theories in a certain setting, and their semantics in an arbitrary multicategory, as described in Section~\ref{sec:notions-operads}, then arises naturally as part of a structure--semantics adjunction for a certain aritation. However, the setting in question does not arise from a 2-monad as for Lawvere theories, PROPs and PROs; we must construct it by hand.

\begin{defn}
Define a 2-functor $\und \from \multicat \to \CAT$ by sending a multicategory $\cat{C}$ to its category of unary morphisms; that is, the objects of $\und(\cat{C})$ are the objects of $\cat{C}$ and $\und(\cat{C}) (c, c') = \cat{C} (c; c')$. The 2-functor $\und$ behaves in the obvious way on 1-cells and 2-cells. As usual we use $(-)_0 \from \multicat \to \CAT$ synonymously with $\und$.
\end{defn}

\begin{remark}
Although we label this 2-functor $\und$, which stands for ``underlying'', and it is clearly in a sense ``forgetful'', this 2-functor is not faithful (i.e.\ injective on 1-cells). That is, $\und$ forgets not only structure and properties, but also ``stuff'', in the sense described in~2.4 of Baez and Shulman~\cite{baezShulman06}. In fact, $\und$ does not even reflect isomorphisms. This tells us that $\multicat$ is not of the form $\talg$ for any 2-monad on $\CAT$.
\end{remark}

\begin{lem}
\label{lem:multicat-cotensor}
The 2-category $\multicat$ is cotensored over $\CAT$, and cotensors are preserved by $\und$.
\end{lem}
\begin{proof}
Given a multicategory $\cat{C}$ and an ordinary category $\cat{B}$, we must define a multicategory $[\cat{B}, \cat{C}]$ whose category of unary morphisms is the functor category $[\cat{B}, \cat{C}_0]$. Thus we must define the objects of $[\cat{B}, \cat{C}]$ to be the functors $\cat{B} \to \cat{C}_0$.
Given functors $F_1, \ldots, F_n, G \from \cat{B} \to \cat{C}_0$, define a morphism
\[
F_1, \ldots, F_n \to G
\]
to be a family of morphisms
\[
\alpha_b \from F_1 b, \ldots, F_n b \to G b
\]
in $\cat{C}$ indexed by $b \in \cat{B}$ such that for every $f \from b \to b'$ in $\cat{B}$ we have
\[
\alpha_{b'} \of (F_1 f, \ldots, F_n f) = Gf \of \alpha_b.
\]
Composition and identities in $[\cat{B}, \cat{C}]$ are defined component-wise via the corresponding structure in $\cat{C}$; then the multicategory axioms for $[\cat{B}, \cat{C}]$ follow from those for $\cat{C}$.

It remains to make $[-, \cat{C}]$ into a 2-functor $\CAT \to \multicat^{\op}$, and show that it is left adjoint to $\multicat(-, \cat{C}) \from \multicat^{\op} \to \CAT$. This is straightforward, and we omit it here.
\end{proof}

\begin{defn}
Let $\ofsfont{E}$ be the class of morphisms of multicategories that are bijective on objects, and let $\ofsfont{N}$ be the class of morphism that are full and faithful in the appropriate multicategorical sense; that is $F \from \cat{C} \to \cat{D}$ is in $\ofsfont{N}$ if
\[
F \from \cat{C} (c_1, \ldots, c_n ; c') \to \cat{D}( Fc_1, \ldots, Fc_n ; Fc')
\]
is a bijection for all $c_1, \ldots, c_n, c' \in \cat{C}$.
\end{defn}

\begin{prop}
\label{prop:multicat-factor}
The classes $\ofsfont{E}$ and $\ofsfont{N}$ form a factorisation system on $\multicat$.
\end{prop}
\begin{proof}
It is clear that $\ofsfont{E}$ and $\ofsfont{N}$ are closed under composition and that every isomorphism of multicategories lies in both $\ofsfont{E}$ and $\ofsfont{N}$. Thus, to show that $(\ofsfont{E}, \ofsfont{N})$ is a factorisation system it is sufficient to show that every morphism factors as a member of $\ofsfont{E}$ followed by a member of $\ofsfont{N}$, and that every member of $\ofsfont{E}$ is left orthogonal to every member of $\ofsfont{N}$.

Let $F \from \cat{C} \to \cat{D}$ be a morphism of multicategories. Define a multicategory $\cat{K}$ as follows; the objects of $\cat{K}$ are the same as the objects of $\cat{C}$, and the hom-sets in $\cat{K}$ are defined by
\[
\cat{K}(c_1, \ldots, c_n ; c') = \cat{D}(Fc_1, \ldots, Fc_n; Fc').
\]
The identities and composition in $\cat{L}$ are inherited from $\cat{D}$. Then we have a morphism of multicategories $E \from \cat{C} \to \cat{K}$ that is the identity on objects and acts by $F$ on morphisms, and a multicategory morphism $N \from \cat{K} \to \cat{D}$ that acts by $F$ on objects and is the identity on morphisms. Clearly $E \in \ofsfont{E}$ and $N \in \ofsfont{N}$ and $N \of E = F$.

Now suppose that we have a commutative diagram
\[
\xymatrix{
\cat{A}\ar[r]^F\ar[d]_E & \cat{B}\ar[d]^N \\
\cat{C}\ar[r]_G & \cat{D}
}
\]
in $\multicat$ where $E$ is bijective on objects and $N$ is full and faithful. We must show that there is a unique $H \from \cat{C} \to \cat{B}$ such that $H \of E = F$ and $N \of H = G$.

Given $c \in \cat{C}$, there is a unique $a \in \cat{A}$ such that $c = Ea$; define $Hc = Fa$ in $\cat{B}$. Suppose $f \from c_1, \ldots, c_n \to c'$ in $\cat{C}$, and $a_1, \ldots, a_n, a' \in \cat{A}$ such that $Ea_i = c_i, Ea' = c'$. Then $N Fa_i = Gc_i$ and $NFa' = Gc'$, so
\[
Gf \from NFa_1, \ldots, NFa_n \to NFa'.
\]
Since $N$ is full and faithful, there is a unique $g \from Fa_1, \ldots, Fa_n \to Fa'$ such that $Ng = Gf$; we define $Hf = g$. The uniqueness property of $g$ ensures that $H$ as defined is a multicategory morphism, and it is clear that it is unique such that $H \of E = F$ and $N \of H = G$.
\end{proof}

\begin{prop}
The 2-category $\multicat$ together with $\und \from \multicat \to \CAT$ and $(\ofsfont{E}, \ofsfont{N})$ is a concrete setting.
\end{prop}
\begin{proof}
By Lemma~\ref{lem:multicat-cotensor}, $\multicat$ has cotensors and they are preserved by $\und$. By Proposition~\ref{prop:multicat-factor}, $(\ofsfont{E}, \ofsfont{N})$ is a factorisation system on $\multicat$ and it is clearly preserved by $\und$.
\end{proof}

\begin{defn}
Write $\cat{O}$ for the initial operad. Explicitly, $\cat{O}(n)$ is empty except when $n = 1$ and then $\cat{O}(n)$ contains only the identity.
\end{defn}

\begin{remark}
Recall that an operad is a multicategory with a single object. Thus $\cat{O} \in \multicat$, and we can talk about $\proth(\cat{O})$, the category of proto-theories with arities $\cat{O}$ in the concrete setting $\multicat$.
\end{remark}

\begin{prop}
\label{prop:iso-proth-operad}
We have an isomorphism of categories
\[
\proth(\cat{O}) \iso \operad.
\]
\end{prop}
\begin{proof}
An object of $\proth(\cat{O})$ consists of a multicategory $\cat{L}$ together with a bijective-on-objects multicategory morphism $\cat{O} \to \cat{L}$. But such a morphism exists for any given $\cat{L}$ if and only if $\cat{L}$ is an operad (i.e.\ has one object) and in that case is unique. So we can identify the objects of $\proth(\cat{O})$ with the operads. Given operads $\cat{L}$ and $\cat{L}'$, a morphism $\cat{L} \to \cat{L}'$ in $\proth(\cat{O})$ is a multicategory morphism (which is the same thing as an operad morphism) $F \from \cat{L} \to \cat{L}'$ such that
\[
\xymatrix{
\cat{O} \ar[r]\ar[dr] & \cat{L}\ar[d]^F \\
& \cat{L}'
}
\]
commutes. But since $\cat{O}$ is the initial operad, this diagram commutes for every $F$. Thus the morphisms in $\proth(\cat{O})$ can be identified with the operad morphisms.
\end{proof}

\begin{lem}
There is an isomorphism of 2-functors
\[
\multicat(\cat{O}, -) \iso \und \from \multicat \to \CAT.
\]
\end{lem}
\begin{proof}
Since $\cat{O}$ has one object and no non-trivial morphisms, a multicategory morphism $\cat{O} \to \cat{C}$ just picks out a single object of $\cat{C}$ and a multicategory transformation between such morphisms picks out a single unary morphism between the corresponding objects. It remains to check that the two 2-functors agree on 1-cells and 2-cells; this is straightforward and we omit the details.
\end{proof}

\begin{defn}
\label{defn:multicat-aritation}
Write
\[
\currylo \from \cat{B}_0 \to \multicat (\cat{O}, \cat{B})
\]
for the isomorphism from the previous lemma regarded as an aritation in the setting $\multicat$.
\end{defn}

\begin{remark}
The aritation defined above gives rise to a structure--semantics adjunction
\[
\xymatrix{
{\catover{\cat{B}_0 }}\ar@<5pt>[r]_-{\perp}^-{\str}\ & {\proth(\cat{O})^{\op} \iso \operad^{\op}.}\ar@<5pt>[l]^-{\sem}
}
\]
\end{remark}

\begin{prop}
The semantics functor arising from the aritation defined in Definition~\ref{defn:multicat-aritation} coincides with the semantics functor
\[
\sem_{\operad} \from \operad^{\op} \to \catover{\cat{B}_0}
\]
from Proposition~\ref{prop:operad-sem-functor}.
\end{prop}
\begin{proof}
Let $L \from \cat{O} \to \cat{L}$ be an operad, together with its unique morphism from $\cat{O}$. Then $\sem(L)$ is defined by the pullback
\[
\xymatrix{
\mod(L)\ar[r]\ar[d]_{\sem(L)}\pullbackcorner & \multicat(\cat{L}, \cat{B} )\ar[d]^{L^*} \\
\cat{B}_0 \ar[r]_-\currylo & \multicat(\cat{O}, \cat{B})
}
\]
in $\CAT$. But $\currylo$ is an isomorphism, so we can take $\mod(L)$ to be $\multicat(\cat{L}, \cat{B} )$ and $\sem(L)$ to be the composite $\currylo^{-1} \of L^*$, which sends a multicategory morphism $\Gamma \from \cat{L} \to \cat{B}$ to the object $\Gamma(*) \in \cat{B}$. This is exactly how $\sem_{\operad}$ was defined in \ref{defn:sem-operad-objects}.
\end{proof}

\chapter{Limits, colimits, and other properties of categories of models}
\label{chap:general}

In this chapter we explore some general properties of categories of models of proto-theories. First we consider the question of limits and colimits; it is well-known that monadic  functors create limits, and it is straightforward to see that forgetful functors from categories of models of Lawvere theories do as well. In Section~\ref{sec:lims-colims-models} we give a unified proof of these, and other, results.

The category of models for a proto-theory is defined by a certain pullback in $\CAT$. However, since $\CAT$ is a 2-category, we have some choice as to how strict this pullback should be: we have chosen the strictest possible version. In Section~\ref{sec:weak-pullbacks} we show that we would obtain an equivalent category of models using a weaker notion of pullback, at least in all of the examples of interest.

\section{Limits and colimits in categories of models}
\label{sec:lims-colims-models}

In this section we explore the conditions under which the category of models for a proto-theory inherits limits and colimits from the base category, or in other words, when a functor of the form $\sem(L)$ creates limits. In particular we give a unified proof that forgetful functors from categories of algebras for monads and categories of models for Lawvere theories create all limits. However, this unified proof comes at the cost of some quite messy details. Recall from Definition~\ref{defn:creation-of-limits} that we use the term ``creation of limits'' in the sense defined in~V.1 of~\cite{maclane71}, which is somewhat stricter than how some more recent authors use the term. 

\begin{lem}
\label{lem:pullback-limits}
Let
\[
\xymatrix{
\cat{A} \ar[r]^F \ar[d]_G\pullbackcorner & \cat{B}\ar[d]^H \\
\cat{C}\ar[r]_K & \cat{D}
}
\]
be a pullback square in $\CAT$, and let $\cat{I} \in \CAT$. Suppose that $\cat{C}$ has and $K$ preserves limits of shape $\cat{I}$, and $H$ creates limits of shape $\cat{I}$. Then $G$ creates limits of shape $\cat{I}$.
\end{lem}

\begin{proof}
Let $D \from \cat{I} \to \cat{A}$ be a diagram in $\cat{A}$ of shape $\cat{I}$, and let $c$ be a limit of $G \of D$ with limit cone $(\lambda_i \from c \to GDi)_{i\in \cat{I}}$. Then since $K$ preserves limits of shape $\cat{I}$, the cone $(K \lambda_i \from Kc \to KGDi)_{i \in \cat{I}}$ is a limit cone for $K \of G \of D = H \of F \of D$. Since $H$ creates limits of shape $\cat{I}$, there is a unique $b \in \cat{B}$ and cone $(\mu_i \from b \to FDi)_{i \in \cat{I}}$ such that $Hb = Kc$ and $H\mu_i = K\lambda_i$, and furthermore this is a limit cone for $F \of D$.

Since $\cat{A}$ is the pullback of $H$ and $K$, we have an object $(c,b) \in \cat{A}$ and a cone $((\lambda_i, \mu_i) \from (c,b) \to Di)_{i \in \cat{I}}$ on $D$. Morphisms $(c' , b') \to (c,b)$ correspond bijectively to pairs of morphisms $f \from c' \to c$ and $g \from b' \to b$ such that $Kf = Hg$, which correspond to pairs of cones $(\sigma, \tau)$ on $G \of D$ and $F \of D$ such that $K \sigma = H \tau$, which correspond to cones on $D$. Thus $(\lambda_i \mu_i)_{i \in \cat{I}}$ is a limit cone for $D$.

Since $G \from \cat{A} \to \cat{C}$ is just the projection, it is clear that $G(\lambda_i, \mu_i) = \lambda_i$. Furthermore, if $((\lambda_i, \mu'_i) \from (c, b') \to  Di)_{i \in \cat{I}}$ were some other cone on $D$ sent to $(\lambda_i)_{i \in \cat{I}}$ by $G$, then we would have
\[
H ( \mu'_i \from b' \to F \of Di )_{i \in \cat{I}} = K ( \lambda_i \from c \to G \of Di)_{i \in \cat{I}}.
\]
But $(\mu_i)_{i \in \cat{I}}$ was chosen to be unique with this property. Thus $(\lambda_i, \mu_i)$ is the unique cone on $D$ that is sent to $\lambda_i$ by $G$, as required.
\end{proof}

\begin{prop}
\label{prop:sem-lim-create}
Let $(\cat{X}, \ofsfont{E}, \ofsfont{N}, \und)$ be a concrete setting and let $\cat{I} \in \CAT$. Fix $\cat{C} \in \cat{X}$ such that $\cat{C}_0 \in \CAT$ has limits of shape $\cat{I}$. Suppose that for every $\cat{D} \in \cat{X}$, and every $d \in \cat{D}_0$, the functor
\[
\und \from \cat{X} (\cat{D}, \cat{C}) \to [\cat{D}_0, \cat{C}_0]
\]
creates limits of shape $\cat{I}$.

Let $\currylo \from \cat{B} \to \cat{X} (\cat{A}, \cat{C})$ be an aritation such that $\cat{B}$ has limits of shape $\cat{I}$, and they are preserved by each $\lpair a, - \rpair_0 \from \cat{B} \to \cat{C}_0$ for $a \in \cat{A}_0$. Then for each proto-theory $(L \from \cat{A} \to \cat{L}) \in \ofsfont{E}$, the functor $\sem(L) \from \mod(L) \to \cat{B}$ creates limits of shape $\cat{I}$.
\end{prop}

\begin{proof}
Recall that we have a pullback
\[
\xymatrix{
\mod(L)\ar[r]\ar[d]_{\sem(L)}\pullbackcorner & \cat{X}(\cat{L}, \cat{C}) \ar[d]^{L^*} \\
\cat{B}\ar[r]_{\currylo} & \cat{X}(\cat{A},\cat{C}).
}
\]
Thus, by Lemma~\ref{lem:pullback-limits}, it is enough to show that $\currylo \from \cat{B} \to \cat{X}(\cat{A}, \cat{C})$ preserves limits of shape $\cat{I}$, and $L^* \from \cat{X}(\cat{L}, \cat{C}) \to \cat{X}(\cat{A}, \cat{C})$ creates them.

First we show that  $\currylo \from \cat{B} \to \cat{X}(\cat{A}, \cat{C})$ preserves limits of shape $\cat{I}$. Suppose $D \from \cat{I} \to \cat{B}$; this has a limit since $\cat{B}$ has all limits of shape $\cat{I}$. The composite
\[
\cat{B} \toby{\currylo} \cat{X}(\cat{A}, \cat{C}) \toby{\und} [\cat{A}_0, \cat{C}_0]
\]
sends the limit cone $(\pi_j \from \lim_{i \in \cat{I}} Di \to Dj)_{j \in \cat{I}}$ to
\[
(\lpair -, \pi_j \rpair_0 \from \lpair -, \lim_{i \in \cat{I}} Di \rpair_0 \to \lpair -, Dj \rpair_0)_{j \in \cat{I}}.
\]
This cone is component-wise a limit cone, since by assumption each $\lpair a, - \rpair_0$ preserves limits of shape $\cat{I}$, and hence it is a limit cone. Thus $\und \of \currylo$ preserves limits of shape $\cat{I}$. Since $\und \from \cat{X}(\cat{A}, \cat{C}) \to [\cat{A}_0, \cat{C}_0]$ creates such limits, it follows that $\currylo$ preserves them.

Now we show that $L^* \from \cat{X}(\cat{L}, \cat{C}) \to \cat{X}(\cat{A}, \cat{C})$ creates limits of shape $\cat{I}$. Consider the commuting square
\[
\xymatrix{
\cat{X} (\cat{L} , \cat{C})\ar[r]^{L^*}\ar[d]_{\und} & \cat{X}(\cat{A}, \cat{C})\ar[d]^{\und} \\
[\cat{L}_0, \cat{C}_0] \ar[r]_{L_0^*} & [\cat{A}_0,\cat{C}_0].
}
\]
Since $[\cat{A}_0,\cat{C}_0]$ has all limits of shape $\cat{I}$ and $\und$ creates them, it follows that $\cat{X}(\cat{A}, \cat{C})$ has limits of shape $\cat{I}$ and $\und$ preserves them. Thus, if $\lambda$ is a limit cone on $L^* \of D$ for some $D \from \cat{I} \to \cat{X}(\cat{L},\cat{C})$, then $\und (\lambda)$ is a limit cone on $\und \of L^* \of D = L^*_0 \of \und \of D$. Since $L^*_0 \of \und$ creates limits of shape $\cat{I}$, there is a unique cone $\mu$ on $\cat{D}$ such that $L_0^* \of \und (\mu) = \und(\lambda)$, and this is a limit cone. Since $\und \from \cat{X}(\cat{A},\cat{C}) \to [\cat{A}_0, \cat{C}_0]$ creates limits of shape $\cat{I}$, it follows that $L^* (\mu) = \lambda$. If $\mu'$ were any other such cone we would have $L^*_0 \of \und (\mu) = L^*_0 \of \und (\mu')$ which implies that $\mu = \mu'$.
\begin{comment}
Both instances of $\und$ create limits of shape $\cat{I}$, as does $L_0^*$ by Lemma~\ref{lem:bo-restrict-create-lims} since $L_0$ is bijective on objects, and $\cat{C}_0$ has limits of shape $\cat{I}$. Since $[\cat{A}_0, \cat{C}_0]$ has all limits of shape $\cat{I}$ (computed pointwise) so does $\cat{X}(\cat{A},\cat{C})$ and they are preserved by $\und \from \cat{X}(\cat{A},\cat{C}) \to [\cat{A}_0, \cat{C}_0]$. Thus, if a cone on a diagram of shape $\cat{I}$ is sent to a limit cone by $L^*$ then it is also sent to a limit cone by $\und \of L^* = L_0^* \of \und$. But the latter composite creates limits of shape $\cat{I}$ so this cone was already a limit, as required.
\end{comment}
\end{proof}

\begin{ex}
Let $\cat{B}$ be a locally small category and consider the canonical aritation $\cat{B}(-,-) \from \cat{B}^{\op} \times \cat{B} \to \Set$. This aritation lives in the setting $\CAT$, which is trivially a concrete setting with $\und \from \CAT \to \CAT$ given by the identity. Identity functors create all limits, so the first hypothesis of Proposition~\ref{prop:sem-lim-create} is satisfied. For this aritation we have $\lpair b, - \rpair = \cat{B} (b, -)$ which preserves all limits, so the second hypothesis is satisfied. Thus by Proposition~\ref{prop:sem-lim-create}, for any proto-theory $L \from \cat{B}^{\op} \to \cat{L}$, the forgetful functor
\[
\sem(L) \from \mod(L) \to \cat{B}
\]
creates all limits. In particular, for any monad $\mnd{T}$, the forgetful functor $U^{\mnd{T}} \from \cat{B}^{\mnd{T}} \to \cat{B}$ creates all limits.
\end{ex}

\begin{ex}
Consider the setting $\finprod$ and the aritation from Section~\ref{sec:structure-examples} for a finite product category $\cat{B}$ giving rise to the structure--semantics adjunction for Lawvere theories. The functor
\[
\und \from \finprod(\fin^{\op}, \cat{B}) \to [\fin^{\op}, \cat{B}]
\]
is the inclusion of the full subcategory of finite product preserving functors. This subcategory is closed under all limits, since finite products commute with limits, so the functor above creates all limits. For the same reason, every $n \in \nat$, the functor $(-)^n \from \cat{B} \to \cat{B}$ preserves all limits. It follows from Proposition~\ref{prop:sem-lim-create} that for any Lawvere theory $L$, the forgetful functor
\[
\sem(L) \from \mod(L) \to \cat{B}
\]
creates all limits.
\end{ex}

\begin{ex}
Consider the case of PROPs, where we take $\cat{X} = \symmon$, and consider the aritation giving rise to the semantics of PROPs in a symmetric monoidal category $\cat{B}$ as in Section~\ref{sec:structure-examples}. In this case, the functors $(-)^{\tensor n} \from \cat{B} \to \cat{B}$ need not preserve limits: for example when $\cat{B} = \Vect$ with the usual tensor product, we have
\begin{align*}
(V \oplus W)^{\tensor 2} &\iso (V \otimes V) \oplus (W \otimes V) \oplus (V \otimes W) \oplus (W \otimes W) \\
& \ncong V^{\tensor 2} \oplus W^{\tensor 2}.
\end{align*}
Thus we do not necessarily expect $\sem(\cat{P}) \from \mod(\cat{P}) \to \cat{B}$ to create limits for a PROP $\cat{P}$.
\end{ex}

Let us now consider limits and colimits, not in the category of models for a particular proto-theory, but in the category of proto-theories itself.

\begin{prop}
\label{prop:lim-colim-proth}
Let $(\cat{X}, \ofsfont{E}, \ofsfont{N})$ be a setting and $\cat{A} \in \cat{X}$. If $\cat{X}$ has all large limits (respectively colimits) then so does $\proth(\cat{A})$.
\end{prop}
\begin{proof}
The inclusion $\proth(\cat{A}) \incl \cat{A} / \cat{X}$ has a right adjoint, namely the functor $\rho \from \cat{A} /\cat{X} \to \proth(\cat{A})$ defined in Definition~\ref{defn:str-factors}. Thus if $\cat{A} / \cat{X}$ has all large colimits, so does $\proth(\cat{A})$, since inclusions of coreflective subcategories are comonadic and in particular create all colimits. Furthermore, if $\cat{A} / \cat{X}$ has all limits, then so does $\proth(\cat{A})$, with the limit of a diagram in $\proth(\cat{A})$ being computed by taking its limit in $\cat{A} / \cat{X}$ and then applying the right adjoint $\rho$.

So it is sufficient to show that $\cat{A} / \cat{X}$ has limits or colimits respectively if $\cat{X}$ does. But the forgetful functor $\cat{A} / \cat{X} \to \cat{X}$ creates all limits (Lemma in Section V.6 of \cite{maclane71}). If $\cat{X}$ has all colimits, then colimits in $\cat{A}/\cat{X}$ can be computed as follows: (large) coproducts in $\cat{A}/\cat{X}$ are given by wide pushouts in $\cat{X}$, and coequalisers in $\cat{A}/\cat{X}$ are the same as in $\cat{X}$.
\end{proof}

\section{Weak pullbacks and isofibrations}
\label{sec:weak-pullbacks}

Recall from Definition~\ref{defn:sem-general} that for an aritation $\currylo \from \cat{B} \to \cat{X}(\cat{A},\cat{C})$ in a general setting $\cat{X}$, the semantics of a proto-theory $L \from \cat{A} \to \cat{L}$ is defined by the pullback
\[
\xymatrix{
\mod(L)\ar[r]^{J(L)}\ar[d]_{\sem(L)}\pullbackcorner & \cat{X}(\cat{L},\cat{C})\ar[d]^{L^*} \\
\cat{B}\ar[r]_{\currylo} & \cat{X}(\cat{A}, \cat{C}).
}
\]
In particular, a model of $L$ consists of an object $d \in \cat{B}$ together with a 1-cell $\Gamma \from \cat{L} \to \cat{C}$ in $\cat{X}$ such that $\Gamma \of L = \currylo (d)$. One may wonder why we require an \emph{equality} here; it may seem more natural to only require a specified isomorphism between $\currylo(d)$ and $\Gamma \of L$. This would amount to replacing the (strict) pullback above with a \emph{weak} pullback in the 2-category $\CAT$, in the sense defined below. In this section, we show that under certain conditions which are satisfied in all the cases of interest, the strict pullback above is also a weak pullback, at least up to equivalence (which is all we can hope for --- weak pullbacks are only unique up to unique-up-to-isomorphism equivalence).

\begin{defn}
Let $F \from \cat{A} \to \cat{C}$ and $G \from \cat{B} \to \cat{C}$ be functors between large categories. The \demph{weak pullback} of $G$ and $F$ is the category whose objects are of the form $(a, b, \phi)$ where $a \in \cat{A}$, $b \in \cat{B}$ and $\phi \from Fa \to Gb$ is an isomorphism, and whose morphisms $(a,b, \phi) \to (a',b', \phi')$ are pairs $(f,g)$ where $f \from a \to a'$ in $\cat{A}$, $g \from b \to b'$ in $\cat{B}$ and $\phi' \of Ff = Gg \of \phi$.
\end{defn}
\begin{remark}
This definition is taken from Joyal and Street~\cite{joyalStreet93} in which it is called a \emph{pseudo-pullback}. Sometimes this is called an \emph{iso-comma object} rather than a weak pullback or pseudo-pullback, with the latter terms being used for the category whose objects consist of $a \in \cat{A}$, $b \in \cat{B}$ and $c \in \cat{C}$ and isomorphisms $Fa \to c$ and $Gb \to c$. However, these two categories are always equivalent (indeed, iso-comma objects and weak pullbacks are equivalent in any 2-category).
\end{remark}

In general the strict and weak pullbacks of a pair of functors need not coincide, even up to equivalence. However, if one of the functors is an isofibration as defined in Definition~\ref{defn:isofibration} then they do.

\begin{prop}
\label{prop:isofib-weak-pullback-equiv}
Let $F \from \cat{A} \to \cat{C}$ and $\cat{B} \to \cat{C}$ be functors, with $F$ an isofibration. Then the strict pullback of $F$ and $G$ is equivalent to the weak pullback.
\end{prop}
\begin{proof}
This is Theorem~1 in Joyal and Street~\cite{joyalStreet93}.
\end{proof}

The following result generalises Lemma~\ref{lem:bo-restrict-isofib}.

\begin{lem}
\label{lem:isofib-2-monad}
Let $\mnd{T} = (T, \eta, \mu)$ be a 2-monad on $\CAT$ and let $(L,l) \from (\cat{A}, W) \to (\cat{L}, U)$ be a bijective-on-objects pseudo-$\mnd{T}$-morphism between $\mnd{T}$-algebras. Then
\[
(L, l)^* \from \talg((\cat{L}, U), (\cat{C},Y)) \to \talg((\cat{A}, W), (\cat{C},Y))
\]
is an amnestic isofibration for every $(\cat{C}, Y) \in \talg$.
\end{lem}
\begin{proof}
We will verify the condition from Lemma~\ref{lem:amnestic-isofib}. Let $(F, f) \from (\cat{L}, U) \to (\cat{C}, Y)$ and $(G,g) \from (\cat{A}, W) \to (\cat{C},Y)$ be pseudo-$\mnd{T}$-morphisms and let $\phi \from (F,f) \of (L, l) \to (G,g)$ be an isomorphism.

By Lemma~\ref{lem:bo-restrict-isofib}, the functor $L^* \from [\cat{L}, \cat{C}] \to [\cat{A},\cat{C}]$ is an amnestic isofibration, so there is a unique functor $G' \from \cat{L} \to \cat{C}$ and natural isomorphism $\theta \from F \to G'$ such that $L^*(G') = G$ and $L^* (\theta) = \phi$.

It remains to be seen that $G'$ can be given a unique pseudo-$\mnd{T}$-morphism structure $g'$ in such a way that $\theta$ becomes a $\mnd{T}$-transformation $(F,f) \to (G' , g')$ and $(G', g') \of (L ,l) = (G, g)$. We define $g' \from Y \of TG' \to G' \of U$ to be the natural isomorphism
\[
\vcenter{
\xymatrix
@C=40pt
@R=40pt{
T\cat{L}\ar[r]^{TG'} \ar[d]_{U}\drtwocell\omit{g'} & T\cat{C}\ar[d]^Y \\
\cat{L}\ar[r]_{G'} & \cat{C}
}}
\quad = \quad
\vcenter{
\xymatrix
@C=40pt
@R=40pt{
T\cat{L} \ar[d]_{U}\drtwocell\omit{f}\rtwocell^{TG'}_{TF}{\:\:\:\:\:\:\:\:T\theta^{-1}} & T\cat{C}\ar[d]^Y \\
\cat{L}\rtwocell_{G'}^{F}{\theta} & \cat{C}.
}}
\]
The equations that must be satisfied in order for $g'$ to be a pseudo-$\mnd{T}$-morphism structure follow from those for $f$ together with 2-naturality of $\eta$ and $\mu$, and $g'$ is clearly unique such that 
\[
\vcenter{
\xymatrix
@C=40pt
@R=40pt{
T\cat{L} \ar[d]_{U}\drtwocell\omit{g'}\rtwocell^{TF}_{TG'}{\:\:\:\:\:T\theta} & T\cat{C}\ar[d]^Y \\
\cat{L}\ar[r]_{G'} & \cat{C}
}}
\quad = \quad
\vcenter{
\xymatrix
@C=40pt
@R=40pt{
T\cat{L} \ar[d]_{U}\ar[r]^{TF}\drtwocell\omit{f} & T\cat{C}\ar[d]^Y\\
\cat{L}\rtwocell_{G'}^{F}{\theta} & \cat{C},
}}
\]
that is, such that $\theta$ is a $\mnd{T}$-transformation $(F,f) \to (G', g')$. Finally, we have
\begin{align*}
& \vcenter{
\xymatrix
@C=40pt
@R=40pt{
T\cat{A}\ar[r]^{TL}\ar[d]_W \drtwocell\omit{l} & T\cat{L}\ar[r]^{TG'} \ar[d]_{U}\drtwocell\omit{g'} & T\cat{C}\ar[d]^Y \\
\cat{A}\ar[r]_L &\cat{L}\ar[r]_{G'} & \cat{C}
}}
\quad  = \quad
\vcenter{
\xymatrix
@C=40pt
@R=40pt{
T\cat{A}\ar[r]^{TL}\ar[d]_W \drtwocell\omit{l} &T\cat{L} \ar[d]_{U}\drtwocell\omit{f}\rtwocell^{TG'}_{TF}{\:\:\:\:\:\:\:\:T\theta^{-1}} & T\cat{C}\ar[d]^Y \\
\cat{A}\ar[r]_L &\cat{L}\rtwocell_{G'}^{F}{\theta} & \cat{C}
}} \\
=&\quad
\vcenter{
\xymatrix
@C=40pt
@R=20pt{
T\cat{A}\ar[rr]^{TG}_*!/d7pt/{\labelstyle{\!\!\!\!\!T\phi^{-1}} \!\!\!\objectstyle\Downarrow }\ar[dr]_{TL}\ar[ddd]_W & & T\cat{C}\ar[ddd]^Y \\
\drtwocell\omit{l}& T\cat{L}\ar[ur]_{TF}\ar[d]^U\drtwocell\omit{f} & \\
& \cat{L}\ar[dr]^F & \\
\cat{A}\ar[ur]^L\ar[rr]_G^*!/u7pt/{\labelstyle{\phi} \!\! \objectstyle\Downarrow } & & \cat{C}
}}
\quad = \quad
\vcenter{
\xymatrix
@C=40pt
@R=40pt{
T\cat{A}\ar[r]^{TG}\ar[d]_W\drtwocell\omit{g} & T\cat{C}\ar[d]^Y \\
\cat{A}\ar[r]_G & \cat{C},
}}
\end{align*}
so $(G', g') \of (L ,l) = (G, g)$ as required.
\end{proof}

\begin{lem}
\label{lem:multicat-isofibration}
Let $\cat{A}, \cat{L}$ and $\cat{C}$ be multicategories, and suppose $L \from \cat{A} \to \cat{L}$ is a bijective-on-objects multicategory morphism. Then
\[
L^* \from \multicat (\cat{L}, \cat{C}) \to \multicat(\cat{A},\cat{C})
\]
is an amnestic isofibration.
\end{lem}
\begin{proof}
Let $F \from \cat{L} \to \cat{C}$ and $G \from \cat{A} \to \cat{C}$ be multicategory morphisms and let $\phi \from F \of L \to G$ be an isomorphism. Define a multicategory morphism $G' \from \cat{L} \to \cat{C}$ as follows.

Recall that every object of $\cat{L}$ is of the form $La$ for a unique $a \in \cat{A}$. Define $G' La = Ga$. Given a morphism $l \from La_1, \ldots, La_n \to La'$ in $\cat{L}$, define $G'l$ to be the composite
\[
\phi_{a'} \of Fl \of (\phi_{a_1}^{-1}, \ldots, \phi_{a_n}^{-1}) \from Ga_1, \ldots , Ga_n \to Ga'.
\]
This clearly does define a multicategory morphism $G' \from \cat{L} \to \cat{C}$ with $G' \of L = G$, and setting $\theta_{La} = \phi_a$ defines a multicategory transformation $\theta \from F \to G'$ such that $\theta L = \phi$. Furthermore, $G'$ and $\theta$ are the unique such.
\end{proof}

\begin{remark}
\label{remark:aritation-setting-list}
Recall the settings and aritations that we have considered so far in this thesis:
\begin{itemize}
\item the canonical aritation $\cat{B}(-,-) \from \cat{B}^{\op} \times \cat{B} \to \Set$ for a locally small category $\cat{B}$ in the setting $\CAT$, whose semantics generalises the semantics of monads;
\item the aritation $\scat{1} \times \cat{B} \toby{\iso} \cat{B}$ in the setting $\CAT$ giving rise to the semantics for monoids;
\item the aritation $ \cat{B} \to \finprod(\fin^{\op}, \cat{B})$ for a finite product category $\cat{B}$ sending $b \mapsto b^{(-)}$ in the setting $\finprod$, giving rise to the semantics of Lawvere theories;
\item the aritation $ \cat{B} \to \symmon(\finbij, \cat{B})$ for a symmetric monoidal category $\cat{B}$ sending $ b \mapsto b^{\tensor (-)}$ in the setting $\symmon$, giving rise to the semantics of PROPs;
\item the aritation  $ \cat{B} \to \monCAT(\nat, \cat{B})$ for a monoidal category $\cat{B}$ sending $b \mapsto b^{\tensor (-)}$ in the setting $\monCAT$ giving rise to the semantics for PROs; and
\item the aritation $\cat{B}_0 \toby{\iso} \multi(\cat{O}, \cat{B})$ for a multicategory $\cat{B}$ in the setting $\multi$, giving rise to the semantics of operads.
\end{itemize}
\end{remark}

\begin{prop}
For all of the aritations listed above (indeed for any aritation in any of the settings listed), the pullback defining the semantics of a proto-theory is equivalent to the corresponding weak pullback.
\end{prop}
\begin{proof}
This will follow from Proposition~\ref{prop:isofib-weak-pullback-equiv} if we can show that for any proto-theory $L \from \cat{A} \to \cat{C}$ in any of the settings $\cat{X}$ discussed, the functor $L^* \from \cat{X}(\cat{L}, \cat{C}) \to \cat{X}(\cat{A},\cat{C})$ is an isofibration. For $\cat{X} = \multicat$, this follows from Lemma~\ref{lem:multicat-isofibration}. All of the other settings are of the form $\talg$ for a 2-monad on $\CAT$ preserving bijective-on-objects functors, so in these cases the result follows from Lemma~\ref{lem:isofib-2-monad}.
\end{proof}

\begin{lem}
\label{lem-isofib-stable-pullback}
Amnestic isofibrations are stable under pullback in $\CAT$.
\end{lem}
This is presumably well-known, however I was unable to find a reference.
\begin{proof}
Let
\[
\xymatrix{
\cat{A}\ar[r]^F\ar[d]_G\pullbackcorner & \cat{B}\ar[d]^H \\
\cat{C}\ar[r]_K & \cat{D}
}
\]
be a pullback in $\CAT$ in which $H$ is an amnestic isofibration. We identify the objects of $\cat{A}$ with pairs $(c,b)$ where $c \in \cat{C}$, $b \in \cat{B}$ and $Hb = Kc$ and similarly with morphisms in $\cat{A}$. Suppose we have $(c,b) \in \cat{A}$ and an isomorphism $\phi \from c \to c'$ in $\cat{C}$. Then $K\phi$ is an isomorphism $Kc = Hb \to Kc'$ in $\cat{D}$, so since $H$ is an amnestic isofibration, there is a unique $b' \in \cat{B}$ and isomorphism $\theta \from b \to b'$ such that $Hb' = Kc'$ and $H\theta = K\phi$. So then we have an isomorphism $(\phi, \theta) \from (c, b) \to (c', b')$ in $\cat{A}$ with $G(\phi, \theta) = \phi$, and it is unique by the uniqueness of $b'$ and $\theta$.
\end{proof}

\begin{prop}
For any of the aritations listed in Remark~\ref{remark:aritation-setting-list} (or any other aritation in these settings), the forgetful functor from the category of models of a proto-theory to the base category is an amnestic isofibration.
\end{prop}
\begin{proof}
This follows from Lemma~\ref{lem-isofib-stable-pullback} together with Lemma~\ref{lem:multicat-isofibration} in the case of operads, or Lemma~\ref{lem:isofib-2-monad}.
\end{proof}

Recall from Remarks~\ref{rem:law-model-non-standard} and~\ref{rem:prop-model-non-standard} that we have adopted a slightly non-standard definition of algebra for Lawvere theories, PROPs and PROs. This is essential in the above result; if one defines a model of a Lawvere theory $L \from \fin^{\op} \to \cat{L}$ simply as a finite product preserving functor out of $\cat{L}$, then the forgetful functor from the category of $L$-models in some finite product category would not be an amnestic isofibration. The situation is similar with PROPs and PROs.

This can be seen as a point in favour of our non-standard definition: intuitively, there should be a unique way of transferring algebraic structure along an isomorphism.

\chapter{The structure--semantics monad for the canonical aritation}
\label{chap:canonical2}

Recall from the Introduction that one of our goals in this thesis is to find a ``convenient category of monads''. More precisely, we would like a category that contains the category of monads on a given category $\cat{B}$ as a full subcategory, and an extension of the semantics functor $\sem_{\monad} \from \monad(\cat{B})^{\op} \to \catover{\cat{B}}$ to this larger category, such that this extended semantics functor has a left adjoint defined on the whole of $\catover{\cat{B}}$ (rather than just $(\catover{\cat{B}})_{\ra}$). We saw in Chapter~\ref{chap:canonical1} that proto-theories on $\cat{B}^{\op}$ with the semantics provided by the canonical aritation provide one such extension.

However, as we will show in the Section~\ref{sec:profinite}, in passing from monads to more general proto-theories, we lose a desirable property of the semantics of monads. Namely, unlike the semantics functor for monads, the functor $\sem \from \proth(\cat{B}^{\op})^{\op} \to \catover{\cat{B}}$ need not be full and faithful. We demonstrate this by showing that, in the case $\cat{B} = \finset$ the monad on $\proth(\finset^{\op})$ induced by the structure--semantics adjunction extends the profinite completion monad on the category of groups. This monad is known not to be idempotent, from which it follows that the structure--semantics adjunction is not idempotent and in particular the semantics functor is not full and faithful.

Having a full and faithful semantics functor is a desirable feature of a notion of algebraic theory, because it can be thought of as a kind of completeness theorem, as explained in Remark~\ref{rem:ff-completeness-interpret}. We would therefore like an extension of $\monad(\cat{B})$ and an extension of the semantics functor that both has a left adjoint and is full and faithful. We will pursue this goal in later chapters of this thesis by developing the analogy between proto-theories and groups that is suggested above. In Section~\ref{sec:str-sem-as-codensity}, we begin to explore this analogy by giving a characterisation of the structure--semantics monad on $\proth(\cat{B}^{\op})$ as a codensity monad, mirroring a similar characterisation of the profinite completion monad on $\Gp$.

\section{Relation to profinite groups}
\label{sec:profinite}

In this section we specialise to the case where $\cat{B} = \finset$, and consider a special type of proto-theory with arities $\finset^{\op}$. Such a proto-theory is a bijective-on-objects functor out of $\finset^{\op}$, however in this section we will identify such functors with bijective-on-objects functors out of $\finset$ for the sake of notational convenience, as we did in Remark~\ref{rem:bo-no-op}. Whenever we refer to a structure or semantics functor in this section, we mean those induced by the canonical aritation on $\finset$.

\begin{defn}
Let $M$ be a small monoid, with unit $e_M \in M$ and multiplication $\mu_M \from M \times M \to M$. Recall that $M$ gives rise to a monad on $\SET$ whose algebras are sets equipped with an action of $M$. We write $\SET_M$ for the Kleisli category of this monad and $F_M \from \SET \to \SET_M$ for the corresponding free functor.
\end{defn}

More explicitly, the objects of $\SET_M$ are sets, and if $S, S'$ are sets, a morphism $S \to S'$ in $\SET_M$ is a function $S \to M \times S'$. Given $f \from S \to M \times S'$ and $g \from S' \to M \times S''$, their composite in $\SET_M$ is the composite function
\[
S \toby{f} M \times S' \toby{\id_M \times g} M \times M \times S'' \toby{\mu_M \times \id_{S''}} M \times S''.
\]
The identity morphism on $S$ in $\SET_M$ is the function $\eta_{M,S}\from  S \from \to M \times S$ sending $s \mapsto (e_M, s)$.

\begin{defn}
\label{defn:monoid-theory-include}
We define a functor $E \from \monoid \to \proth(\finset^{\op})$ as follows. Given a monoid $M$, consider the composite functor
\[
\finset \incl \SET \toby{F_M} \SET_M.
\]
Let
\[
\finset \toby{E(M)} \cat{E} (M) \toby{N(M)} \SET_M
\]
be the bijective-on-objects/full and faithful factorisation of this composite.

A monoid homomorphism $h \from M \to M'$ induces a functor $h_* \from \SET_M \to \SET_{M'}$ that is the identity on objects and sends $f \from S \to M \times S'$ to
\[
S \toby{f} M \times S' \toby{h \times \id_{S'}} M' \times S';
\]
then we have $h_* \of F_M = F_{M'}$. We define $E(h) \from \cat{E}(M) \to \cat{E}(M')$ to be the unique functor such that
\[
\xymatrix{
\SET \ar[r]^{E(M)}\ar[dr]_{E(M')} & \cat{E}(M)\ar[r]^{N(M)}\ar[d]^{E(h)} & \SET_M\ar[d]^{h_*} \\
& \cat{E}(M')\ar[r]_{N(M')} & \SET_{M'} 
}
\]
commutes. Functoriality of $E \from \monoid \to \proth(\finset^{\op})$ then follows from functoriality of the assignments $M \mapsto \SET_M$ and $h \mapsto h_*$.
\end{defn}

\begin{lem}
\label{lem:monoid-finite-coproducts}
Let $M, M'$ be small monoids, and let $K$ be any morphism $E(M) \to E(M')$ in $\proth(\finset^{\op})$. Then $E(M) \from \finset \to \cat{E}(M)$ and $K \from \cat{E}(M) \to \cat{E}(M')$ both preserve finite coproducts.
\end{lem}
\begin{proof}
First note that since the composite $\finset \incl \SET \toby{F_M} \SET_M$ preserves finite coproducts (since each factor does), and $N(M) \from \cat{E}(M) \incl \SET_M$ reflects them (since it is full and faithful), the functor $E(M) \from \finset \to \cat{E}(M)$ preserves finite coproducts. Likewise $E(M') \from \finset \to \cat{E}(M')$ preserves finite coproducts.

Thus, given any finite family of objects of $\cat{E}(M)$ their coproduct is the image under $E(M)$ of the corresponding coproduct in $\finset$, and in particular the coprojections are the images under $E(M)$ of the corresponding coprojections in $\finset$. The same is true for the coprojections for the corresponding coproduct in $\cat{E}(M')$. But since $K \of E(M) = E(M')$, this means that $K$ must send the coprojections in $\cat{E}(M)$ to the coprojections in $\cat{E}(M')$, that is, it preserves the coproduct.
\end{proof}

\begin{prop}
\label{prop:mon-in-proth-ff}
The functor $E \from \monoid \to \proth(\finset^{\op})$ is full and faithful.
\end{prop}
\begin{proof}
An endomorphism of $1 = \{*\}$ in $\cat{E}(M)$ is a function
\[
1 \to M \times 1
\]
and so can be identified with an element of $M$; composition of such endomorphisms then corresponds to the multiplication in $M$. Thus, any morphism $E(M) \to E(M')$ in $\proth(\finset^{\op})$ induces a monoid homomorphism
\[
\cat{E}(M)(1,1) \iso M \to \cat{E}(M')(1,1) \iso M'
\]
and it is clear that the monoid homomorphism induced in this way by $E(h)$ for $h \from M \to M'$ is $h$ itself. In particular, if $h \neq h' \from M \to M'$ then $E(h) \neq E(h')$, so $E$ is faithful.

Let $K \from E(M) \to E(M')$ in $\proth(\finset^{\op})$, and let $h\from M \to M'$ be the monoid homomorphism induced by the action of $K$ on endomorphisms of $1$ in the way described above. We must show that $K = E(h)$.

By Lemma~\ref{lem:monoid-finite-coproducts}, every object of $\cat{E}(M)$ is a finite copower of $1$, and these copowers are preserved by both $K$ and $E(h)$. Thus in order to check that $K = E(h)$, it is sufficient to check that they agree on hom-sets of the form $\cat{E}(M) (1, S)$. An element of such a hom-set is a function
\[
1 \to M \times S
\]
and so can be identified with a pair $(m, s) \in M \times S$. The functor $E(h)$ sends this pair to $(h(m), s) \in M' \times S$; we must show that $K$ does the same.

This morphism can be decomposed as
\[
m \from 1 \to M \times 1
\]
followed by
\[
(e_M, s) \from 1 \to M \times S.
\]
The former is sent to $h(m) \from 1 \to M' \times 1$ by definition of $h$. The latter is $E(M)$ applied to the morphism $1 \to S$ in $\finset$ that picks out the element $s$. Since $K \of E(M) = E(M')$, we must have
\[
K (e_M, s) = E(M')(s) = (e_{M'}, s) \from 1 \to M' \times S
\]
in $\cat{E}(M')$. Thus functoriality of $K$ means that it sends $(m, s)$ to the composite of $h(m) \from 1 \to M' \times 1$ and $(e_{M'}, s) \from 1 \to M' \times S$ in $\cat{E}(M')$, which is $(h(m), s)$. Thus $K = E(h)$, as required.
\end{proof}

\begin{lem}
\label{lem:act-semantics}
Let $M$ be a small monoid. Then $\mod (E(M))$ is isomorphic to the category of finite sets equipped with an action by $M$ and $M$-equivariant maps, and $\sem (E (M))$ is the usual forgetful functor.
\end{lem}
\begin{proof}
Let $X$ be a finite set. Write $I \from \finset \incl \Set$ for the inclusion. To define an $E(M)$-model structure on $X$ is to give a natural transformation $\alpha \from \cat{E}(M) ( E (M) - , X) \to \finset(-, X) \iso \Set(I-, X)$ satisfying the conditions of Definition~\ref{defn:L-alg}. But by definition, 
\begin{align*}
\cat{E}(M) ( E (M) - , X) &\iso \Set_M ( F_M I- , F_M X ) \\ 
& \iso \Set (I-, M \times X).
\end{align*}
Since $\finset$ is dense in $\Set$, such a natural transformation is given by a morphism $a \from M \times X \to X$. The identity morphism $\id_{E(M)(X)} \in \cat{E}(M)(X, X)$ corresponds to the map $\eta_{M, X} \from X \to M \times X$, so condition~\ref{defn:L-alg}.\bref{part:L-alg-id} corresponds to the commutativity of 
\[
\xymatrix{
X \ar[r]^{\eta_{M, X}} \ar@{=}[dr] & M \times X \ar[d]^{a} \\
& X.
}
\]
Condition~\ref{defn:L-alg}.\bref{part:L-alg-comp} corresponds to, for arbitrary $f \from S' \to M \times S$ and $ g \from S \to M \times X$, the commutativity of
\[
\xymatrix{
S' \ar[r]^-{f} & M \times S\ar[r]^-{\id_M \times g} & M \times M \times X \ar[r]^-{\mu_M \times \id_X} \ar[d]_{\id_M \times a} & M \times X \ar[d]^{a} \\
& & M \times X \ar[r]_-{a} & X.
}
\]
(The top-right composite corresponds to $\alpha_{S'} ( g \of f)$, regarding $g$ and $f$ as morphisms in $\cat{E}(M)$, and the bottom-left composite corresponds to $\alpha_{S'} ( E(M)(\alpha_S (g)) \of f)$.) Certainly the associativity axiom for an $M$-action implies that this diagram commutes. Conversely we recover the associativity axiom by taking $g$ and $f$ to be identities. An $E(M)$-model homomorphism between models $(X, \alpha) \to (X', \alpha')$, with corresponding $M$-actions $a$ and $a'$ is a morphism $h \from X \to X'$ such that
\[
\xymatrix{
M \times X \ar[r]^{\id_M \times h}\ar[d]_{a} & M \times X'\ar[d]^{a'}\\
X\ar[r]_{h} &X'
}
\]
commutes --- that is, an $M$-equivariant map.
\end{proof}

The next proposition characterises the proto-theories that arise from monoids in this way.

\begin{prop}
\label{prop:act-image}
Let $L \from \finset \to \cat{L}$ be a proto-theory with arities in $\finset^{\op}$. Then $L$ is in the essential image of $E \from \monoid \to \proth(\finset^{\op})$ if and only if:
\begin{enumerate}
\item
\label{part:act-image-ls}
$\cat{L}$ is locally small,
\item
\label{part:act-image-coproducts}
$L$ preserves finite coproducts, and 
\item
\label{part:act-image-unary}
for every $l \from L1 \to LS$ in $\cat{L}$, there is a unique $m \from L1 \to L1$ in $\cat{L}$ and $s \from 1 \to S$ in $\finset$ such that $l = Ls \of m$.
\end{enumerate}
When these conditions hold, $L$ is isomorphic to $E(M)$, where $M$ is the monoid $\cat{L}(L1, L1)$ (with composition as multiplication).
\end{prop}
\begin{proof}
First let us check that every proto-theory with arities $\finset^{\op}$ of the form $E(M)$ for a small monoid $M$ satisfies these properties.
\begin{enumerate}
\item It is clear from the definition that $\cat{E}(M)$ is locally small, given that $M$ is a small monoid.
\item This follows from Lemma~\ref{lem:monoid-finite-coproducts}.  
\item Let $f \in \cat{E}(M)(1, S)$, so $f$ is a function $1 \to M \times S$ in $\Set$. Write $f(*) = (m, s)$. Then $f$ is equal to the composite
\[
1 \toby{\widebar{m}} M \times 1 \toby{E(M)(s)} M \times M \times S \toby{\mu_M \times \id_S} M \times S
\]
which is the composite $E(M)(s) \of \widebar{m}$ in $\cat{E}(M)$, and $s$ and $m$ are unique such.
\end{enumerate}
Thus $E(M)$ does satisfy these properties. Furthermore, a morphism $1 \to 1$ in $\cat{E}(M)$ is a function $1 \to M \times 1$ in $\Set$, which clearly corresponds to an element of $M$, and it is easy to see that composition in $\cat{E}(M)$ corresponds to multiplication in $M$.

All that remains is to check that if $L \from \finset \to \cat{L}$ satisfies the stated properties then there is an isomorphism $L \iso \cat{E}(M)$ where $M = \cat{L}(L1,L1)$. Firstly condition~\bref{part:act-image-ls} ensures that $M$ is indeed a small monoid.

Let us define a functor $P \from \cat{L} \to \cat{E}(M)$. On objects, define $PLS = S$. Let $l \in \cat{L}(L1, LS')$. Then by condition~\bref{part:act-image-unary}, $l$ factors uniquely as $Ls' \of m$ with $m \in \cat{L}(L1,L1) = M$ and $s' \from 1 \to S'$. We must define $Pl \in \cat{E}(M)(1, S') = \Set (1, M \times S')$; define $Pl$ to be the map sending $*$ to $(m, s')$. We can then extend this definition to morphisms with an arbitrary domain, by noting that both $L$ and $E(M)$ preserve finite coproducts, and every object of $\finset$ is a finite copower of $1$. Thus a morphism $l \from LS \to LS'$ corresponds to a family $(l_s \from L1 \to LS')_{s \in S}$, and we can define $Pl \in \cat{E}(M)(S, S')$ to be the morphism corresponding to the family $(Pl_s \from 1 \to M \times S')_{s \in S}$.

Let us check that $P$ is functorial. It is sufficient to check that for $l \from L1 \to LS' $ and $k \from LS' \to LS''$ we have $P(k \of l) = P(k) \of P(l)$.  Suppose $l$ factors as 
\[
L1 \toby{m} L1 \toby{Ls'} LS'
\]
and that the composite
\[
L1 \toby{Ls'}  LS' \toby{k} LS''
\]
factors as
\[
L1 \toby{m'} L1 \toby{Ls''} LS''.
\]
Then
\[
\xymatrix{
L1\ar[r]^{m} \ar[dr]_{m' \cdot m}\ar@/^1.5pc/[rr]^{l} & L1\ar[r]^{Ls'}\ar[d]^{m'} & LS'\ar[d]^{k} \\
& L1 \ar[r]_{Ls''} & LS''
}
\]
commutes, so $P(k \of l) = (m' \cdot m, s'')$. On the other hand, $Pl = (m, s') \from 1 \to M \times S'$, and $P(k) \from S' \to M \times S''$ sends $s'$ to $(m', s'')$, so their composite in $\cat{E}(M)$ is also $(m' \cdot m, s'')$, as required. Thus $P$ is functorial.

Finally $P$ is a bijection on each hom-set of the form $\cat{L}(L1, LS')$ by condition~\bref{part:act-image-unary}, and it follows that it is a bijection on arbitrary hom-sets by~\bref{part:act-image-coproducts}.
\end{proof}

\begin{prop}
\label{prop:monoid-restrict}
There is a functor $T \from \monoid \to \monoid$ that is unique up to isomorphism such that
\[
\xymatrix{
\monoid\ar[r]^-E\ar[dd]_T & \proth(\finset^{\op})\ar[d]^{\sem} \\
& (\catover{\finset})^{\op}\ar[d]^{\str} \\
\monoid\ar[r]_-E & \proth(\finset^{\op})
}
\]
commutes up to isomorphism.
\end{prop}
\begin{proof}
It is sufficient to check that for a small monoid $M$, the proto-theory $\str \of \sem \of E (M)$ lies in the essential image of $E$; since $E$ is full and faithful, it then follows that such a functor $T$ exists. The uniqueness of $T$ follows from the fact that $E$ is full and faithful (Proposition~\ref{prop:mon-in-proth-ff}) and in particular reflects isomorphisms.

Let $M$ be a small monoid, and write $U \from \cat{M} \to \finset$ for $\sem (E(M)) \from \mod(E(M)) \to \finset$ for brevity --- recall (Lemma~\ref{lem:act-semantics}) that $\cat{M}$ is the category of finite $M$-sets, and $U$ is the usual forgetful functor. A morphism $S \to S'$ in $\thr (U)$ is a natural transformation
\[
\finset(S, U - ) \to \finset(S', U-)
\] 
or equivalently $U^S \to U^{S'}$.

We show that $\thr(U)$ satisfies the conditions of Proposition~\ref{prop:act-image}. Firstly $\cat{M}$ is isomorphic to the functor category $[M, \finset]$, regarding $M$ as a one-object category. Since $M$ and $\finset$ are small, so is $\cat{M}$, and so the functor category $[\cat{M}, \Set]$ is locally small. But $\thr(U)$ is equivalent to a full subcategory of this category, so it is also locally small.

Recall that by definition of $\str(U)$, we have a commutative square
\[
\xymatrix{
\finset^{\op}\ar[r]^{\curryhi}\ar[d]_{\str(U)} & [\finset, \Set]\ar[d]^{U^*} \\
\thr(U)\ar[r]_{J(U)} & [\cat{M}, \Set]
}
\]
in which $J(U)$ is full and faithful (and therefore reflects limits and colimits). Since $\curryhi$ and $U^*$ preserve finite products, and $J(U)$ reflects them, it follows that $\str(U)$ preserves finite products, or equivalently $\str(U)^{\op} \from \finset \to \thr(U)^{\op}$ preserves finite coproducts.

Let $S \in \finset$, and $\gamma \from U^S \to U$. We must show that $\gamma$ factors uniquely as a projection $\pi_s \from U^S \to U$ for some $s \in S$ followed by a natural transformation $\delta \from U \to U$.

Write $I \from \finset \to \cat{M}$ for the functor sending a finite set to the same set equipped with the trivial $M$-action. Then $U \of I$ is the identity on $\finset$, so whiskering $\gamma$ with $I$ gives a natural transformation
\[
\gamma I \from \id_{\finset} ^ S \iso \finset (S, -) \to \id_{\finset} \iso \finset(1, -),
\]
which, by the Yoneda lemma, is the same thing as an element $s \in \id_{\finset} (S) = S$. Then $\gamma I = \pi_s \from \id_{\finset}^S \to \id_{\finset}$, so on trivial $M$-sets, $\gamma$ is just the $s$-th projection.

Let $X$ be a finite $M$-set. By an \emph{orbit} of $X$, we mean an equivalence class for the smallest equivalence relation $\sim$ on $X$ such that for all $m \in M$ and $x \in X$, we have $m \cdot x \sim x$. Write $X_0, X_1, \ldots,  X_{n-1}$  for the orbits of $X$. Writing $n$ for the set $\{0, \ldots , n-1 \}$, define a map $h \from X \to I(n)$ by sending each $x \in X$ to the unique $i$ such that $x \in X_i$ --- this is clearly an $M$-set homomorphism. Hence, by naturality of $\gamma$,
\[
\xymatrix{
X^S \ar[r]^{h^S} \ar[d]_{\gamma_X} & I(n)^S\ar[d]^{\gamma_{I(n)} = \pi_s} \\
X\ar[r]_h & I(n)
}
\]
commutes. But $\pi_s \of h^S$ is also equal to $h \of \pi_s$. So it follows that if $\vect{x} = (x_s)_{s \in S}$ is an arbitrary element of $X^S$, then $\gamma_X (\vect{x})$ lies in the same orbit at $x_s$. Next we will show that in fact $\gamma_X(\vect{x})$ only depends on $x_s$.

Let $\vect{x} = (x_t)_{t \in S}$ and $\vect{x}'= (x'_t)_{t \in S}$ be elements of $X^S$ with $x_s = x'_{s}$. We will show that $\gamma_X(\vect{x}) = \gamma_X(\vect{x}')$, but first we construct slightly modified versions of $X, \vect{x}$ and $\vect{x}'$ as an intermediate step.

Let $Z = S \times X$, with $M$ acting on each copy of $X$ separately (so $m \cdot ( t, x) = (t, m \cdot x)$). Define $\vect{z}, \vect{z}' \in Z^S$ with components $z_t = (t, x_t)$ and $z'_t = (t, x'_t)$ respectively, for $t \in S$. Let $k \from Z = S \times X \to X$ be the projection --- this is a homomorphism. Clearly $k^S(\vect{z}) = \vect{x}$ and $k^S(\vect{z}') = \vect{x}'$, so the commutativity of 
\[
\xymatrix{
Z^S \ar[r]^{k^S} \ar[d]_{\gamma_Z} & X^S \ar[d]^{\gamma_X} \\
Z\ar[r]_{k} & X
}
\]
implies that if we can show that $\gamma_Z(\vect{z}) = \gamma_Z(\vect{z}')$, it follows that $\gamma_X(\vect{x}) = \gamma_X(\vect{x})$.

Suppose without loss of generality that $x_s = x'_s$ lies in the orbit $X_0$. Let $Y = X_0 \sqcup \{* \}$, with $M$ acting on $X_0$ as in $X$, and acting trivially on $*$. The orbits of $Z$ are precisely the sets of the form $\{t\} \times X_i$, where $t \in S$ and $i \in n$. Define $l \from Z \to Y$ by sending every orbit except $\{s \} \times X_0$ to $*$, and mapping $\{s \} \times X_0$ to $X_0$ by the projection. This makes $l$ into an $M$-set homomorphism, and so
\[
\xymatrix{
Z^S\ar[r]^{l^S}\ar[d]_{\gamma_Z} & Y^S\ar[d]^{\gamma_Y} \\
Z\ar[r]_{l} &Y
}
\]
commutes. Now, the $s$-components of $\vect{z}$ and $\vect{z}'$ are equal, and all other components lie in some orbit of $Z$ other that $\{s \} \times X_0$, so $l^S (\vect{z}) = l^S (\vect{z}')$. Hence
\[
l (\gamma_Z (\vect{z})) = \gamma_Y(l^S (\vect{z})) = \gamma_Y(l^S (\vect{z}')) = l (\gamma_Z (\vect{z}')).
\]
But by the above, $\gamma_Z(\vect{z})$ and $\gamma_Z(\vect{z}')$ lie in the orbit of $z_s = z'_s$, which is $\{s \} \times X_0$. And $l$ is injective when restricted to this orbit, so it follows that $\gamma_Z (\vect{z}) = \gamma_Z (\vect{z}')$, and so $\gamma_X (\vect{x}) = \gamma_X (\vect{x}')$.

Thus $\gamma_X(\vect{x})$ depends only on the $s$-component of $\vect{x}$. Let $\Delta \from U \to U^S$ denote the diagonal. Then 
\[
\gamma \of \Delta \of \pi_s = \gamma \from U^S \to U
\]
since for any $M$-set $X$ and any element $\vect{x} \in X^S$, both $\vect{x}$ and $\Delta \of \pi_s (\vect{x})$ have the same $s$-component. Thus if we define $\delta \from U \to U$ to be the composite $\gamma \of \Delta$, we have $\gamma =  \delta \of \pi_s$.

Furthermore, $s$ and $\delta$ are unique; suppose $\gamma = \delta' \of \pi_{s'}$. Note that $\delta' I$ is an endomorphism of the identity functor on $\finset$, and therefore must be the identity by the Yoneda lemma, since the identity functor on $\finset$ is represented by the terminal object, which has no non-trivial endomorphisms. Thus we must have $\pi_{s'} = \gamma I = \pi_{s}$, so $s' = s$. It follows that $\delta' = \delta$, since $\pi_s$ is an epimorphism, being split by $\Delta$.
\end{proof}

\begin{cor}
\label{cor:str-sem-monad-restrict-monoid}
The structure--semantics monad on $\proth(\finset^{\op})$ restricts to a monad $(T, \eta^{\mnd{T}}, \mu^{\mnd{T}})$ on the full subcategory $\monoid$ of small monoids.
\end{cor}
\begin{proof}
In order for a monad to restrict to a monad on a full subcategory, all that is required is that the endofunctor part of the monad restricts to an endofunctor of the subcategory. In this case, this is precisely the previous proposition.
\end{proof}

We now specialise to those proto-theories that correspond to groups, rather than all monoids. Our main reason for doing so is that we will be considering the codensity monad of the inclusion of finite groups into all groups, which is somewhat more well-behaved than the codensity monad of the inclusion of finite monoids into monoids. Algebras for this monad on the category of groups can be identified with profinite groups --- that is, topological groups whose underlying topological space is profinite. In particular, we will make use of the fact that, for any group $G$, the unit of this monad has dense image. I do not know whether the corresponding results hold for monoids.

\begin{defn}
Write $(\widehat{\ }, \eta, \mu)$ for the codensity monad of the inclusion $\FinGp \incl \Gp$ of finite groups into groups; recall from Proposition~\ref{prop:profinite-codensity-set-top-gp} that this is the profinite completion monad, whose category of algebras is the category of profinite groups and continuous group homomorphisms. For a group $G$, the group $\widehat{G}$ is given by the end
\[
\widehat{G} = \int_{H \in \FinGp} H^{\Gp(G, H)}.
\]
Explicitly, the elements of $\widehat{G}$ are families
\[
\xi = (\xi_h)_{h \from G \to H}
\]
of elements $\xi_h \in H$ with $h$ ranging over all group homomorphisms from $G$ to finite groups $H$, such that, for any group homomorphism $k \from H \to H'$ between finite groups, $\xi_{k \of h} = k(\xi_h)$. The identity, multiplication and inverses in $\widehat{G}$ are defined component-wise. If $l \from G \to G'$ is a group homomorphism, then $\widehat{l} \from \widehat{G} \to \widehat{G'}$ sends $\xi = (\xi_h \in H)_{h \from G \to H} \in \widehat{G}$ to the element $\widehat{l}(\xi) \in \widehat{G'}$ whose $(h'\from G' \to H)$-th component is $\xi_{h' \of l} \in H$.

The unit of the monad has components $\eta_G \from G \to \widehat{G}$ sending $g \in G$ to the family $(h(g) \in H)_{h \from G \to H}$. The multiplication has components $\mu_G \from \widehat{\widehat{G}} \to \widehat{G}$ sending $(\zeta_{h'} \in H)_{h' \from \widehat{G} \to H}$ to the element $\mu_G(\zeta) \in \widehat{G}$ with $(h \from G \to H)$-th component
\[
\zeta_{\pi_h} \in H,
\]
where $\pi_h \from \widehat{G} \to H$ is projection onto the $h$-th factor, sending $\xi \in \widehat{G}$ to $\xi_h \in H$.
\end{defn}

\begin{prop}
\label{prop:group-restrict-procomp}
The monad $(T, \eta^{\mnd{T}}, \mu^{\mnd{T}})$ on $\monoid$ from Corollary~\ref{cor:str-sem-monad-restrict-monoid} restricts to a monad on the full subcategory $\Gp \incl \monoid$ that is isomorphic to the profinite completion monad $(\widehat{\ }, \eta, \mu)$.
\end{prop}
\begin{proof}
First we must show that the square
\[
\xymatrix{
\Gp\ar[r]\ar[d]_{\widehat{\ }} & \monoid\ar[d]^T \\
\Gp\ar[r] & \monoid
}
\]
commutes up to isomorphism.

We construct an isomorphism $ \Phi \from \procomp{G} \to [\Gfinset, \finset ] (U, U) = TG$, where $U\from \Gfinset \to \finset$ is the forgetful functor. Let $\xi \in \procomp{G}$. A finite $G$-set $X$ with underlying set $X_0$ is determined by a group homomorphism $\rho_X \from G \to \sym (X_0)$, where $\sym(X_0)$ is the group of automorphisms of $X_0$ in $\finset$. Since $\sym(X_0)$ is a finite group, we may define $\Phi (\xi) _X = \xi_{\rho_X} \from X_0 \to X_0$. We must check that $\Phi (\xi)$ thus defined actually is a natural transformation $U \to U$, that is, for any $G$-set homomorphism $k \from X \to Y$, that
\[
\xymatrix{
X \ar[r]^k \ar[d]_{\Phi(\xi)_X} & Y\ar[d]^{\Phi(\xi)_Y} \\
X \ar[r]_k & Y
}
\]
commutes. Note that this square commutes for each $\xi \in \procomp{G}$ if and only if the right-hand diamond in
\[
\xymatrix{
& & {\sym{X_0}}\ar[dr]^{k_*} & \\
{G}\ar[urr]^{\rho_X}\ar[r]\ar[drr]_{\rho_Y} & {\procomp{G}}\ar[ur]_{\Phi(-)_X}\ar[dr]^{\Phi(-)_Y} & & {\finset(X_0,Y_0)} \\
& & {\sym{Y_0}}\ar[ur]_{k^*} &
}
\]
commutes. But the outer diamond in this diagram commutes, since $k$ is a $G$-set homomorphism. Also the two left-hand triangles commute by definition of $\Phi$. Suppose we equip $\procomp{G}$ with its canonical profinite topology, and all the other sets in this diagram with the discrete topology. Then all the maps in this diagram are continuous. Furthermore, since the image of $G$ is dense in $\procomp{G}$ (by, for example, Lemma~3.2.1 of~\cite{ribesZalesskii10}), and all the spaces involved are Hausdorff, it follows that the right-hand diamond commutes, since maps with dense image are epic in the category of Hausdorff spaces.

Hence each $\Phi (\xi)$ is indeed a natural transformation $U \to U$. In addition, 
\[
\Phi \from \procomp{G} \to [\Gfinset, \finset](U,U)
\]
 is a monoid homomorphism since each $\Phi(-)_X$ is by construction. Now we check that it is natural in $G$. Let $f \from G \to G'$ be a group homomorphism, and let us first describe the monoid homomorphism
 \[
 Tf \from  TG = [\Gfinset, \finset](U,U) \to TG' = [G'\text{-}\fcat{FinSet}, \finset](U',U')
 \]
 (where $U' \from G'\text{-}\fcat{FinSet} \to \finset$ is the forgetful functor). Let $\gamma \from U \to U$ be a natural transformation and $X$ be a $G'$-set, with action determined by a group homomorphism $\rho'_X \from G' \to \sym(X_0)$, where $X_0$ is the underlying set of $X$. Then $Tf (\gamma)$ is the natural transformation $U' \to U'$ whose component at such an $X$ is the component of $\gamma$ at the $G$-set with the same underlying set as $X$ and $G$-action determined by the group homomorphism
 \[
 G \toby{f} G' \toby{\rho_X'} \sym(X_0).	
 \]
 Now let $\xi \in \procomp{G}$. Then $Tf \of \Phi (\xi) \from U' \to U'$ is the natural transformation whose component at a $G'$-set $X$ corresponding to $\rho'_X \from G' \to \sym(X_0)$ is $\xi_{\rho'_X \of f}$.
 
 On the other hand, $\procomp{f} (\xi) \in \procomp{G}$ has, for $h' \from G' \to H$ with $H$ finite,
 \[
 (\procomp{f} (\xi))_{h'} = \xi_{h' \of f}.
 \]
 In particular,
 \[
 (\Phi \of \procomp{f} (\xi))_X = (\procomp{f} (\xi))_{\rho_X'} = \xi_{\rho_X' \of f}.
 \]
 Thus $Tf \of \Phi = \Phi \of \procomp{f}$, and so $\Phi$ is natural.

Now we construct an inverse $\Xi$ for $\Phi$. Let $\gamma \from U \to U$; we wish to construct an element $\Xi(\gamma) \in \procomp{G}$. Given a finite group $H$ and a group homomorphism $h \from G \to H$, we obtain a $G$-set $H_h$ with underlying set $H$, and with $g \in G$ acting by multiplication on the left by $h(g)$. Thus we have $\gamma_{H_h} \from H \to H$. Define $\Xi (\gamma)_h = \gamma_{H_h}(e_H)$, where $e_H$ denotes the group identity of $H$.

We must check that $\Xi (\gamma)$ so defined is indeed an element of $\procomp{G}$, that is, that if $k \from H \to H'$ is a homomorphism between finite groups, that $\Xi (\gamma)_{k \of h} = k ( \Xi(\gamma)_{h})$. But such a group homomorphism $k$ is also a $G$-set-homomorphism $H_h \to H'_{k \of h}$, so
\[
\xymatrix{
H \ar[r]^k \ar[d]_{\gamma_{H_h}} & H' \ar[d]^{\gamma_{H'_{k \of h}}} \\
H \ar[r]_k & H'
}
\]
commutes. Thus,
\[
\Xi(\gamma)_{k \of h} = \gamma_{H'_{k \of h}} (e_{H'}) = \gamma_{H'_{k \of h}} ( k ( e_H)) = k ( \gamma_{H_h}(e_H)) = k ( \Xi(\gamma)_h)
\]
as required.

Now we show that $\Xi$ is inverse to $\Phi$. Let $\xi \in \procomp{G}$. Then for any finite $G$-set $X$ and $x \in X$, we have $\Phi (\xi)_X (X) = \xi_{\rho_X} (x)$. Thus for $h \from G \to H$ with $H$ finite,
\[
\Xi \Phi (\xi)_h = \Phi (\xi)_{H_h} (e_H) = \xi_{\rho_{H_f}} (e_H).
\]
So we need to show that $\xi_{\rho_{H_h}} (e_H) = \xi_h$. Define $i \from H \to \sym (H_0)$ by sending $m \in H$ to left multiplication by $m$, where $H_0$ is the underlying set of $H$. Then $i$ is a group homomorphism and
\[
\xymatrix{
G \ar[r]^h \ar[dr]_{\rho_{H_h}} & H\ar[d]^i \\
& \sym (H_0)
}
\]
commutes. Thus we have
\[
\xi_{\rho_{H_h}} (e_H) = \xi_{i \of h} (e_H) = i (\xi_h) (e_H) =  \xi_h
\]
as required. So $\Xi \of \Phi = \id_{\procomp{G}}$.

Now let $\gamma \from U \to U$. For any finite $G$-set $X$ and $x \in X$, we have
\[
\Phi \Xi (\gamma)_X (x) = \Xi(\gamma)_{\rho_X} (x) = \gamma_{(\sym X_0)_{\rho_X}} (\id_{X_0})(x).
\]
Note that we have a $G$-set homomorphism $\ev_x \from (\sym X_0)_{\rho_X} \to X$, since, given $\sigma \in \sym X_0$ and $g \in G$,
\[
\ev_x (g \cdot \sigma) = \rho_X (g) \of \sigma (x) = g \cdot \sigma (x) = g \cdot \ev_x (\sigma).
\]
Hence
\[
\xymatrix{
{\sym X_0} \ar[r]^{\ev_x}\ar[d]_{\gamma_{(\sym X_0)_{\rho_X}}} & X\ar[d]^{\gamma_X} \\
{\sym X_0} \ar[r]_{\ev_x} & X
}
\]
commutes, and so
\begin{align*}
\gamma_{(\sym X_0)_{\rho_X}} (\id_X)(x) &= \ev_x \of \gamma_{(\sym X_0)_{\rho_X}} (\id_X) \\
&= \gamma_X \of \ev_x (\id_X) \\
&= \gamma_X (x).
\end{align*}
Hence $\Phi \Xi (\gamma) = \gamma$, and we have shown that $\Phi$ and $\Xi$ are inverses. Hence we have a natural monoid (and therefore group) isomorphism
\[
\procomp{G} \iso [\Gfinset, \finset](U,U),
\]
as claimed.

Now we must show that $\Phi$ is in fact an isomorphism of monads. For a given group $G$, consider the diagram
\[
\xymatrix{
G \ar[r]^{\eta_G}\ar[dr]_{\eta^{\mnd{T}}_G} & \procomp{G}\ar[d]^{\Phi} \\
& [\Gfinset, \finset](U, U).
}
\]
The map $\eta^{\mnd{T}}_G$ sends $g \in G$ to the natural transformation $U \to U$ whose component at a $G$-set $X$ sends $x$ to $g \cdot x$. On the other hand,
\[
(\Phi \of \eta_G (g))_X (x) = (\eta_G (g) )_{\rho_X} (x) = \rho_X (g) (x) = g \cdot x,
\]
so this diagram commutes. Now consider
\[
\xymatrix{
\procomp{\procomp{G}}\ar[r]^{\procomp{\Phi_{G}}} \ar[d]_{\mu_G} & \procomp{TG}\ar[r]^{\Phi_{TG}} & TTG\ar[d]^{\mu^{\mnd{T}}_G} \\
\procomp{G}\ar[rr]_{\Phi_G} && TG.
}
\]
Let us describe the map $\mu_G^{\mnd{T}}$. Write $V \from TG\text{-}\fcat{FinSet} \to \finset$ for the forgetful functor. Every $G$-set $X$ can be canonically made into a $TG$ set: let $\gamma \in TG$, so that $\gamma$ is a natural transformation $U \to U$. Then $\gamma$ acts on $X$ via $\gamma_X \from X \to X$. Thus, given a natural transformation $\delta \from V \to V$, we may consider $\delta_X \from X \to X$, with $X$ regarded as a $TG$-set as above. This defines the components $\delta_X = (\mu_G^{\mnd{T}}(\delta))_X$ of the natural transformation $\mu_G^{\mnd{T}}(\delta) \from U \to U$.

Now let us describe the composite $\Phi_{TG} \of \procomp{\Phi_G}$. Let $\zeta \in \procomp{\procomp{G}}$ so that $\zeta_h \in H$ for each group homomorphism $h \from \procomp{G} \to H$ with $H$ finite. Then $\procomp{\Phi_G} (\zeta) \in \procomp{TG}$ has
\[
(\procomp{\Phi_{G}} (\zeta))_h = \zeta_{h \of \Phi_G} \in H
\]
for each $h \from TG \to H$ with $H$ finite. And $\Phi_{TG} \of \procomp{\Phi_G} (\zeta)$ is the natural transformation $V \to V$ with components as follows: given a finite $TG$-set $Y$ defined by $\lambda \from TG \to \sym (Y_0)$, and $y \in Y$ we have
\[
(\Phi_{TG} \of \procomp{\Phi_G} (\zeta))_Y (y) = \procomp{\Phi_G} (\zeta)_{\lambda} (y) =  \zeta_{\lambda \of \Phi_G} (y).
\]
Thus the top-right composite sends $\zeta \in \procomp{\procomp{G}}$ to the natural transformation $U \to U$ whose component at a $G$-set $X$ defined by $\rho \from G \to \sym (X_0)$ sends $x \in X$ to
\[
\zeta_{\ev_X \of \Phi_G} (x) = \zeta_{\ev_{\rho_X}} (x)
\]
where $\ev_X \from TG \to \sym(X_0)$ sends $\gamma \mapsto \gamma_X$ and $\ev_{\rho_X} \from \procomp{G} \to \sym(X_0)$ sends $\xi \mapsto \xi_{\rho_X}$.

On the other hand, the bottom-left composite sends $\zeta$ to the natural transformation whose component at $X$ sends $x$ to
\[
\Phi_G \of \mu_G (\zeta) = \mu_G (\zeta)_{\rho_X} = \zeta_{\ev_{\rho_X}}.
\]
Since these agree, the square commutes, and so $\Phi$ does define an isomorphism of monads ($\procomp{\ }, \eta, \mu) \to (T, \eta^{\mnd{T}}, \mu^{\mnd{T}})$.
\end{proof}

\begin{cor}
The structure--semantics monad on $\proth(\finset^{\op})$ restricts to the profinite completion monad on the full subcategory $\Gp \incl \proth(\finset^{\op})$.
\end{cor}
\begin{proof}
This is immediate from Corollary~\ref{cor:str-sem-monad-restrict-monoid} and Proposition~\ref{prop:group-restrict-procomp}.
\end{proof}

This result suggests that in order to gain a better understanding of the theory of structure--semantics adjunctions in general, it may be profitable to compare it to the theory of profinite completions of groups. This shall be pursued in the chapters that follow, but for now we have the following immediate consequence.

\begin{cor}
The structure--semantics adjunction for the canonical aritation on $\finset$ is not idempotent, and in particular 
\[
\sem \from \proth(\finset^{\op})^{\op} \to \catover{\finset}
\]
is not full and faithful.
\end{cor}
\begin{proof}
Recall that an idempotent adjunction induces an idempotent monad and comonad respectively on the two categories involved in the adjunction. By Proposition~\ref{prop:group-restrict-procomp}, the monad on $\proth(\finset^{\op})$ induced by the structure--semantics adjunction restricts to profinite completion monad on $\Gp$. This monad is not idempotent; to see this it is sufficient to find a profinite group that is not the profinite completion of its underlying discrete group. Any infinite power of a non-trivial finite group is such a profinite group, as shown in Example~4.2.12 of~\cite{ribesZalesskii10}. It follows that the structure--semantics adjunction is not idempotent.
\end{proof}

\begin{remark}
\label{rem:motivate-completeness+str}
Recall from Remark~\ref{rem:ff-completeness-interpret} that if the semantics functor for a given notion of algebraic theory is full and faithful, this can be interpreted as a kind of completeness theorem: faithfulness says roughly that two theories are isomorphic if and only if they have the same models, and fullness says that a theory really does describe \emph{all} of the algebraic structure possessed by its models.

Recall from Proposition~\ref{prop:sem-mnd-ff} that such a completeness theorem holds for the usual semantics of monads; that is, the functor
\[
\sem_{\monad} \from \monad(\cat{B})^{\op} \to \radjover{\cat{B}}
\]
is full and faithful for any $\cat{B}$. One part of our motivation for passing from monads to more general proto-theories (with semantics given by the canonical aritation) was that it allowed us to define a proto-theory $\str(U)$ for \emph{all} functors $U$ with codomain $\cat{B}$. However in doing so, we have had to sacrifice the completeness theorem.

One of our goals for the remainder of this thesis is to develop a notion of algebraic theories generalising that of monads that (at least for certain well-behaved categories) maintains both of these desirable features: a semantics functor that is full and faithful \emph{and} that has a left adjoint defined on the whole of $\catover{\cat{B}}$. More precisely, we would like a convenient category of monads in the following sense.
\end{remark}

\begin{defn}
A \demph{convenient category of monads} on a category $\cat{B}$ consists of a category $\convt$ together with a full and faithful functor $\inc \from \monad(\cat{B}) \incl \convt$ and a functor $\sem \from \convt^{\op} \to \catover{\cat{B}}$ such that:
\begin{enumerate}
\item the diagram
\[
\xymatrix{
\convt^{\op}\ar[r]^{\sem} & \catover{\cat{B}} \\
\monad(\cat{B})^{\op}\ar[ur]_{\sem_{\monad}}\ar[u]^{\inc} &
}
\]
commutes;
\item the functor $\sem$ has a left adjoint; and
\item the functor $\sem$ is full and faithful.
\end{enumerate}
\end{defn}

The category $\monad(\cat{B})$ itself satisfies the first (trivially) and third of these criteria but not the second, whereas $\proth(\cat{B}^{\op})^{\op}$ satisfies the first and second but not the third.

In our search for a convenient category of monads, we  will take inspiration from the analogy we found in this section between groups and proto-theories. We continue to expand on this analogy in the following section.

\section{The structure--semantics monad as a codensity monad}
\label{sec:str-sem-as-codensity}

Throughout this section, fix a locally small category $\cat{B}$. We will give another characterisation of the structure--semantics monad on $\proth(\cat{B})$ induced by the adjunction
 \[
\xymatrix{
{\catover{\cat{B} }}\ar@<5pt>[r]_-{\perp}^-{\str}\ & {\proth(\cat{B}^{\op})^{\op}}\ar@<5pt>[l]^-{\sem}
}
\]
for the canonical aritation $\cat{B}(-,-) \from \cat{B}^{\op} \times \cat{B} \to \Set$. 

In the previous section we saw that, when $\cat{B} = \finset$, this monad restricts to the profinite completion monad on $\Gp$. The profinite completion monad can also be characterised as the codensity monad of the inclusion $\FinGp \incl \Gp$. We may wonder whether there is a similar characterisation of the structure--semantics monad for $\cat{B}$ as the codensity monad of some subcategory of $\proth(\cat{B}^{\op})$. How can we find a candidate for such a subcategory?

Let $G$ be a \emph{finite} group. Recall from Definition~\ref{defn:monoid-theory-include} that $E(G) \in \proth(\finset^{\op})$ was defined in terms of the Kleisli category of the monad $G \times -$ on $\SET$. But for any finite set $S$, the free $G$-set $G \times S$ on $S$ is finite since $G$ is. Hence $G \times -$ restricts to a monad on $\finset$, and the Kleisli category of this monad is precisely the full subcategory of the Kleisli category for the monad on $\SET$ consisting of the finite sets. In other words, the proto-theory $E(G)$ is isomorphic to $\kle(G \times -)$ (recall from Lemma~\ref{lem:kle-functor} that $\kle $ is a functor $\monad(\finset) \to \proth(\finset^{\op})$).

Moreover, for a group $G$, the proto-theory $E(G)$ is given by a monad if \emph{and only if} $G$ is finite, since if $G$ is not finite then $G \times S$ is not finite for a non-empty finite set $S$. In other words, we have a pullback
\[
\xymatrix{
{\FinGp}\ar[r]\ar[d] \pullbackcorner & {\monad(\finset)}\ar[d]^{\kle} \\
{\Gp}\ar[r]_-{E} & {\proth(\finset^{\op})}.
}
\]
Thus, in some sense at least, monads on $\cat{B}$ stand in the same relation to general proto-theories with arities $\cat{B}^{\op}$ as finite groups do to groups. We might therefore wonder whether the structure--semantics monad can be defined as the codensity monad of $\kle \from \monad(\cat{B}) \incl \proth(\cat{B}^{\op})$. Indeed, this is the case, as we shall now show.

As in Remark~\ref{rem:bo-no-op}, we identify objects of $\proth(\cat{B}^{\op})$ with bijective-on-objects functors out of $\cat{B}$, rather than $\cat{B}^{\op}$, for the sake of notational convenience.

\begin{prop}
\label{prop:str-sem-mnd-codensity}
Suppose $\cat{B}$ admits pointwise codensity monads of all finite diagrams. Then the structure--semantics adjunction for the canonical aritation is the codensity monad of
\[
\kle \from \monad(\cat{B}) \incl \proth(\cat{B}^{\op}).
\]
\end{prop}
The condition that $\cat{B}$ admits pointwise codensity monads of finite diagrams says that, for functors $D \from \scat{I} \to \cat{B}$ with $\cat{I}$ finite, and $b \in \cat{B}$, the composite
\[
(b \downarrow D) \to \scat{I} \toby{D} \cat{B}
\]
has a limit. Since $\cat{B}$ is locally small the comma category $(b \downarrow D)$ is always small, so in particular this condition holds when $\cat{B}$ admits all small limits. But it also holds when $\cat{B}$ admits only \emph{finite} limits, provided $\cat{B}$ is locally \emph{finite}, since then $(b \downarrow D)$ is finite. This is the case for $\finset$ for example, despite $\finset$ not admitting arbitrary small limits.

\begin{proof}
Let $L \from \cat{B} \to \cat{L}$ and $L' \from \cat{B} \to \cat{L}'$ be proto-theories with arities $\cat{B}^{\op}$. We will establish bijection between morphisms
\[
\sem (L) \to \sem (L')
\]
in $\catover{\cat{B}}$ and natural transformations
\[
\proth(\cat{B}^{\op})(L, \kle - ) \to \proth(\cat{B}^{\op})(L', \kle - )
\]
satisfying the conditions of Lemma~\ref{lem:codensity-iso-kleisli}.

First let $Q \from \sem (L) \to \sem (L')$ in $\catover{\cat{B}}$, that is, $Q$ is a functor $\mod (L) \to \mod (L')$ such that $\sem (L') \of Q = \sem (L)$. We construct a natural transformation
\[
\Sigma (Q) \from \proth(\cat{B}^{\op})(L, \kle - ) \to \proth(\cat{B}^{\op})(L', \kle - ).
\]
Let $\mnd{T} = (T, \eta, \mu)$ be a monad on $\cat{B}$ and $S \from L \to \kle (\mnd{T})$ in $\proth(\cat{B}^{\op})$. Recall that for any monad $\mnd{T}$ we have
\[
\kle (\mnd{T}) \iso \str ( U^{\mnd{T}} \from \cat{B}^{\mnd{T}} \to \cat{B}) \iso \str (\sem_{\monad} (\mnd{T})).
\]
Hence $S$ corresponds to a morphism $\Theta (S) \from \sem_{\monad} (\mnd{T}) \to \sem (L)$ in $\catover {\cat{B}}$, where $\Theta$ is the bijection from Definition~\ref{defn:adj-bij-inv} defining the structure--semantics adjunction. Then we can compose this with $Q$, giving a morphism
\[
Q \of \Theta (S) \from \sem_{\monad} (\mnd{T}) \to \sem (L) \to \sem (L'),
\]
to which we can apply the inverse bijection $\Psi$ (defined in Definition~\ref{defn:adj-bij}), giving a morphism
\[
L' \to \str (\sem_{\monad} (\mnd{T})) \iso \kle (\mnd{T})
\]
in $\proth(\cat{B}^{\op})$. Thus we define the natural transformation $\Sigma (Q)$ component-wise by
\[
\Sigma (Q)_{\mnd{T}} (S) = \Psi ( Q \of \Theta (S)).
\]
This is evidently natural in $\mnd{T}$ by the naturality of $\Psi$ and $\Theta$.

Let us show that $\Sigma$ is compatible with composition. It is immediate from the definition that $\Sigma$ preserves identities. Suppose $Q \from \sem (L) \to \sem (L')$ and $Q' \from \sem(L') \to \sem(L'')$ in $\catover{\cat{B}}$, and $S\from L \to \kle (\mnd{T})$ in $\proth(\cat{B}^{\op})$. Then
\begin{align*}
(\Sigma (Q') \of \Sigma (Q) )_{\mnd{T}} (S) & = \Psi ( Q' \of \Theta ( \Sigma (Q)_{\mnd{T}} (S))) && \text{(Definition of $\Sigma(Q')$)} \\
&= \Psi(Q' \of \Theta( \Psi ( Q \of \Theta (S)))) && \text{(Definition of $\Sigma(Q)$)} \\
&= \Psi (Q' \of Q \of \Theta(S)) && \text{($\Psi$ and $\Theta$ are inverses)} \\
&= \Sigma (Q' \of Q)_{\mnd{T}} (S) && \text{(Definition of $\Sigma(Q' \of Q)$)}
\end{align*}
as required.

We should also check that for $P \from L' \to L$ in $\proth(\cat{B}^{\op})$, we have 
\[
\Sigma (\sem (P)) = P^* \from \proth(\cat{B}^{\op})(L, \kle - ) \to \proth(\cat{B}^{\op})(L', \kle - ).
\]
But if $S \from L \to \kle (\mnd{T})$, then
\[
\Sigma (\sem(P))_{\mnd{T}} (S) = \Psi (\sem (P) \of \Theta (Q)) = \Psi ( \Theta (Q)) \of P = Q \of P,
\]
by naturality of $\Psi$.

Now we construct an inverse $\Pi$ of $\Sigma$. Let $\chi \from \proth(\cat{B}^{\op})(L, \kle - ) \to \proth(\cat{B}^{\op}(L', \kle - )$. Let $b$ be an object of $\cat{B}$ and write $\const{b} \from \scat{1} \to \cat{B}$ for the functor that just picks out the object $\cat{B}$ (where $\scat{1}$ here denotes the terminal category). Then an $L$-model $x$ with underlying object $d^x = b$ is precisely a morphism $\const{x} \from \const{b} \to \sem (L)$ in $\catover{\cat{B}}$. Such a morphism corresponds to $\Psi (\const{x}) \from L \to \str (\const{b})$ in $ \proth(\cat{B}^{\op})$. But since $\scat{1}$ is finite and $\cat{B}$ and $\cat{B}$ admits pointwise codensity monads of finite diagrams, the pointwise codensity monad $\mnd{T}^{\const{b}}$ of $\const{b}$ exists, and so, by Proposition~\ref{prop:codensity-str}, the proto-theory $\str (\const{b}) \iso \kle (\mnd{T}^{\const{b}})$ lies in the essential image of $\kle$. Thus we may apply $\chi_{\mnd{T}^{\const{b}}}$ to $\Psi (\const{x})$ to obtain
\[
\chi_{\mnd{T}^{\const{b}}} (\Psi (\const{x})) \from L' \to \str(\const{b}).
\]
Now we can apply $\Theta$ to obtain
\[
\Theta (\chi_{\mnd{T}^{\const{b}}} (\Psi (\const{x})) ) \from \const{b} \to \sem(L')
\]
in $\catover{\cat{B}}$, which corresponds to an $L'$-model $\Pi (\chi) ( x)$ with the same underlying object as $x$. This defines the functor $\Pi ( \chi) \from \mod (L) \to \mod (L')$ on objects.

In order to define $\Pi (\chi)$ on morphisms, consider the category $\scat{2}$ with two objects $0, 1$ and a single non-identity morphism $u \from 0 \to 1$. Let $h \from x \to y$ be a homomorphism in $\mod(L)$ with underlying morphism $h_0 \from d^x \to d^y$ in $\cat{B}$. Define $\const{h_0} \from \scat{2} \to \cat{B}$ to be the functor sending $u$ to $h_0$ and $\const{h} \from \scat{2} \to \mod(L)$ to be the functor sending $u$ to $h$. Then $\const{h}$ is a morphism $\const{h_0} \to \sem (L)$ in $\catover{\cat{B}}$. Furthermore, pre-composing $\const{h}$ with $\const{0} \from \scat{1} \to \scat{2}$ or $\const{1} \from \scat{1} \to \scat{2}$ gives $\const{x}$ or $\const{y}$ respectively. That is, we have a commuting diagram
\[
\xymatrix{
{\const{d^x}} \ar[d]_{\const{0}}\ar[dr]^{\const{x}} & \\
{\const{h_0}}\ar[r]^-{\const{h}} & {\sem(L)} \\
{\const{d^y}}\ar[u]^{\const{1}}\ar[ur]_{\const{y}} &
}
\]
in $\catover{\cat{B}}$. Thus naturality of $\Psi$ implies we have a commuting diagram
\[
\xymatrix{
& \str(\const{d^x}) \\
L\ar[ur]^{\Psi (\const{x})} \ar[r]^{\Psi(\const{h})} \ar[dr]_{\Psi(\const{y})} & \str(\const{h_0})\ar[d]^{\str(\const{1})}\ar[u]_{\str(\const{0})} \\
& {\str(\const{d^y}).}
}
\]
Since $\scat{2}$ is finite, $\const{h_0}$ has a pointwise codensity monad $\mnd{T}^{\const{h_0}}$, and $\str(\const{h_0}) \iso \kle (\mnd{T}^{\const{h_0}})$. So the right-hand side of this diagram lies entirely in the essential image of the full and faithful $\kle \from \monad(\cat{B}) \to \proth(\cat{B}^{\op})$. So, applying $\chi$ yields a commutative diagram
\[
\xymatrix
@C=50pt
@R=50pt{
& \str(\const{d^x}) \\
L'\ar[ur]^{\chi_{\mnd{T}^{\const{d^x}}}(\Psi (\const{x}))} \ar[r]^{\chi_{\mnd{T}^{\const{h_0}}}(\Psi(\const{h}))} \ar[dr]_{\chi_{\mnd{T}^{\const{d^y}}}(\Psi(\const{y}))} & \str(\const{h_0})\ar[u]_{\str(\const{0})}\ar[d]^{\str(\const{1})} \\
& {\str(\const{d^y}),}
}
\]
and applying $\Theta$ gives
\[
\xymatrix@C=50pt
@R=50pt{
{\const{d^x}} \ar[d]_{\const{0}}\ar[dr]^{\Theta(\chi_{\mnd{T}^{\const{d^x}}}(\Psi (\const{x})))} & \\
{\const{h_0}}\ar[r]^-{\Theta\chi_{\mnd{T}^{\const{h_0}}}(\Psi(\const{h})))} & {\sem(L')} \\
{\const{d^y}.}\ar[u]^{\const{1}}\ar[ur]_{\Theta(\chi_{\mnd{T}^{\const{d^y}}}(\Psi(\const{y})))} &
}
\]
We define $\Pi(\chi)(h)$ to be the morphism in $\mod(L')$ corresponding to $\Theta\chi_{\mnd{T}^{\const{h_0}}}(\Psi(\const{h})))\from \scat{2} \to \mod(L')$ --- the commutativity of this last diagram implies that this has the appropriate domain and codomain. Furthermore, note that since
\[
\xymatrix
@C=50pt{
\scat{2}\ar[r]^{\Theta(\chi_{\mnd{T}^{\const{h_0}}}(\Psi(\const{h})))}\ar[dr]_{\const{h_0}} & \mod(L')\ar[d]^{\sem(L')} \\
& \cat{B}
}
\]
commutes, $\Pi(\chi)(h)$ has the same underlying morphism $h_0$ as $h$.

It remains to check that $\Sigma (\chi) \from \mod(L) \to \mod(L')$ is functorial. Let $h \from x \to y$ and $k \from y \to z$ in $\mod(L)$ with underlying morphisms $h_0 \from d^x \to d^y$ and $k_0 \from d^y \to d^z$ respectively. We have already observed that $\Pi(\chi)$ preserves both underlying objects and underlying morphisms. So $\Pi(\chi)(k \of h)$ has the same underlying morphism as $k \of h$, namely $k_0 \of h_0$. But since composition in $\mod(L')$ is defined by composing underlying morphisms in $\cat{B}$, this is also the underlying morphism of the composite $\Pi ( \chi)(k) \of \Pi (\chi)(h)$. The forgetful functor $\sem (L') \from \mod(L') \to \cat{B}$ is faithful, so this implies $\Pi(\chi)(k \of h) = \Pi ( \chi)(k) \of \Pi (\chi)(h)$.

Similarly, for $x \in \mod(L)$, the morphism $\Pi(\chi)(\id_x)$ is a morphism $\Pi(\chi)(x) \to \Pi(\chi)(x)$ with underlying morphism $\id_{d^x}$, and $\id_{\Pi(\chi)(x)}$ is the unique such morphism, so they are equal.

That completes the definition of $\Pi$; it remains to show that it is inverse to $\Sigma$.

Let $\chi \from \proth(\cat{B}^{\op})(L, \kle - ) \to \proth(\cat{B}^{\op})(L', \kle - )$ and $S \from L \to \kle (\mnd{T})$. We wish to show that $\Sigma (\Pi (\chi))_{\mnd{T}}(S) = \chi_{\mnd{T}}(S)$. Equivalently, applying the bijection $\Theta$, we will show that
\[
\Theta ( \Sigma (\Pi (\chi))_{\mnd{T}}(S)) = \Theta (\chi_{\mnd{T}}(S)).
\]
Now, by definition of $\Sigma$,
\[
\Theta ( \Sigma (\Pi (\chi))_{\mnd{T}}(S)) = \Theta \Psi (\Pi (\chi) \of \Theta (S)) = \Pi (\chi) \of \Theta (S) \from \sem_{\monad} (\mnd{T}) \to \sem (L').
\]
To check that this is equal to $\Theta (\chi_{\mnd{T}} (S)) \from \sem_{\monad}(\mnd{T}) \to \sem(L')$, it is sufficient to check that they are equal on an arbitrary object of $\sem_{\monad}(\mnd{T})$ --- equality on morphisms then follows from the fact that they are both morphisms in $\catover{\cat{B}}$ and that $\sem(L')$ is faithful. Hence we need to show that for an arbitrary object $x \in \sem_{\monad}(\mnd{T})$ with underlying object $d^x$, the two morphisms
\begin{equation}
\label{eqn:codensity-bijection-1}
\const{d^x}\toby{\const{x}} \sem_{\monad}{\mnd{T}} \toby{\Theta(S)} \sem(L) \toby{\Pi(\chi)} \sem(L')
\end{equation}
and
\begin{equation}
\label{eqn:codensity-bijection-2}
\const{d^x}\toby{\const{x}} \sem_{\monad}{\mnd{T}} \toby{\Theta (\chi_{\mnd{T}} (S))} \sem(L')
\end{equation}
in $\catover{\cat{B}}$ are equal.

Consider the composite $\Theta (S) \of \const{x} \from \const{d^x} \to \sem (L)$. By naturality of $\Theta$, this is the result of applying $\Theta$ to the composite
\[
L \toby{S} \kle ( \mnd{T}) \iso \str (\sem_{\monad} (\mnd{T})) \toby{\str(\const{x})} \str(\const{d^x}).
\]
And by definition of $\Pi$, the result of composing $\Theta ( \str(\const{x}) \of S)$ with $\Pi (\chi)$  (that is, the composite displayed in Diagram~\bref{eqn:codensity-bijection-1}) is
\[
\Theta ( \chi_{\mnd{T}^{\const{d^x}}} ( \Psi ( \Theta ( \str(\const{x}) \of S)))) = \Theta ( \chi_{\mnd{T}^{\const{d^x}}} ( \str(\const{x}) \of S)),
\]
since $\Theta$ and $\Psi$ are inverses. But now note that 
\[
\str (\const{x}) \from \str(\sem_{\monad}(\mnd{T})) \iso \kle (\mnd{T}) \to \str(\const{d^x}) \iso \kle (\mnd{T}^{\const{d^x}})
\]
lies in the essential image of $\kle \from \monad(\cat{B}) \to \proth(\cat{B}^{\op})$, and so applying naturality of $\chi$, we may further rewrite the composite from Diagram~\bref{eqn:codensity-bijection-1} as
\[
\Theta(\str(\const{x}) \of \chi_{\mnd{T}} (S)),
\]
and this, by naturality of $\Theta$, is equal to
\[
\Theta ( \chi_{\mnd{T}} (S)) \of \const{x},
\]
which is the composite displayed in Diagram~\bref{eqn:codensity-bijection-2}. This completes the proof that $\Sigma \of \Pi = \id$.

Now we prove that $\Pi \of \Sigma = \id$. Let $Q \from \sem(L) \to \sem (L')$  in $\catover{\cat{B}}$. We need to show that $\Pi (\Sigma (Q)) = Q \from \mod (L) \to \mod(L')$. Again, it is sufficient to check that these functors are equal on objects, since they have equal composites with the faithful $\sem (L')$. So let $x \in \mod(L)$. Then by definition of $\Pi$,
\[
\xymatrix{
\scat{1}\ar[d]_{\const{x}}\ar[dr]^{\Theta(\Sigma(Q)_{\mnd{T}^{\const{d^x}}}(\Psi (\const{x})))} & \\
\mod(L)\ar[r]_{\Pi (\Sigma (Q))} & {\mod(L')}
}
\]
commutes. But by definition of $\Sigma$, we have
\[
\Theta(\Sigma(Q)_{\mnd{T}^{\const{d^x}}}(\Psi (\const{x}))) = \Theta (\Psi (Q \of \Theta ( \Psi ( \const{x})))) = Q \of \const{x},
\]
since $\Theta$ and $\Psi$ are inverses. Thus $\Pi (\Sigma (Q)) = Q$ as required, so $\Pi$ and $\Sigma$ are inverses.
\end{proof}

We summarise the analogy between the theory of groups and the theory of proto-theories in Table~\ref{tab:analogy}. We shall add further rows to this table as more aspects of the analogy are developed in the next two chapters.
\begin{table}[h]
\centering
\begin{tabular}{p{7cm}p{7cm}}
\toprule
Theory of groups & Theory of proto-theories \\
\midrule
$\Gp$ & $\proth(\cat{B}^{\op})$ \\
$\FinGp \incl \Gp$ & $\kle \from \monad(\cat{B}) \incl  \proth(\cat{B}^{\op})$ \\
The functor $\Gp^{\op} \to \catover{\finset}$ sending $G$ to the category $\Gfinset$ of finite $G$-sets & The functor $\sem \from \proth(\cat{B}^{\op})^{\op} \to \catover{\cat{B}}$ for the canonical aritation \\
The profinite completion monad on $\Gp$ & The structure--semantics monad on $\proth(\cat{B}^{\op})$\\
The profinite completion monad is the codensity monad of $\FinGp \incl \Gp$. & The structure--semantics monad is the codensity monad of $\kle \from \monad(\cat{B}) \incl \proth(\cat{B}^{\op})$. \\
\bottomrule

\end{tabular}
\caption{Some aspects of the analogy between groups and proto-theories on a locally small category $\cat{B}$ with pointwise codensity monads of finite diagrams.}
\label{tab:analogy}
\end{table}

\chapter{Topological proto-theories}
\label{chap:topology}

In this chapter we continue our search for a convenient category of monads. Recall from Remark~\ref{rem:motivate-completeness+str} that we hope to find a notion of algebraic theory generalising that of monads on a locally small category $\cat{B}$, for which the semantics functor is full and faithful, and has a left adjoint defined on the whole of $\catover{\cat{B}}$. To do so we make use of the analogy between proto-theories and groups developed in the previous chapter.

Under this analogy, the structure--semantics monad on $\proth(\cat{B}^{\op})$ corresponds to the profinite completion monad on $\Gp$. They are both the codensity monads of inclusions of full subcategories; the categories of monads and finite groups respectively. Algebras for the profinite completion monad on $\Gp$ can be identified with profinite topological groups, but these can also be described as algebras for a different codensity monad, namely the codensity monad of the inclusion of finite discrete groups into all topological groups.

Crucially, \emph{this} codensity monad is idempotent: the inclusion of profinite groups into topological groups has a left adjoint. This suggests that to find a notion of algebraic theory generalising monads for which the structure--semantics monad is idempotent (which is a first step towards the semantics functor being full and faithful), it might be useful to consider a topological notion of proto-theory.

In Section~\ref{sec:top-str-sem-adj}, we define such a notion and prove some of its basic properties, and in Section~\ref{sec:monads-as-top-proths} we show that the semantics of topological proto-theories generalises the semantics of monads. We then move on to describing the conditions under which the structure--semantics adjunction for topological proto-theories is idempotent: we define the relevant conditions in Section~\ref{sec:cats-w-enough-subobjects}, and prove that the adjunction is idempotent under these condition in Section~\ref{sec:top-str-sem-idem}. Finally, in Section~\ref{sec:top-proth-relation-profinite-groups} we show that the topological structure-semantics monad on the category of topological proto-theories extends the topological profinite completion monad on the category of topological groups.

\section{The topological structure--semantics adjunction}
\label{sec:top-str-sem-adj}

In this section we define a topological notion of proto-theory that plays a role relative to ordinary proto-theories that is analogous to the role played by topological groups relative to discrete groups.

\begin{defn}
Write $\TOPCAT$ for the 2-category of large categories enriched in the category $\TOP$ of large topological spaces (with the cartesian product). We make $\TOPCAT$ into a concrete setting as follows: let $\ofsfont{E}$ be the class of $\TOP$-functors that are bijective on objects, and let $\ofsfont{N}$ be the class of $\TOP$-functors that are homeomorphisms on each hom-space; that is, a $\TOP$-functor $F \from \cat{A} \to \cat{B}$ between $\TOP$-categories is in $\ofsfont{N}$ if for each $a, a' \in \cat{A}$, the map
\[
F \from \cat{A}(a, a') \to \cat{B}(Fa,Fa')
\]
is a homeomorphism. Define $\und \from \TOPCAT \to \CAT$ to be the evident forgetful 2-functor, that forgets the topology on each hom-space of each $\TOP$-category.
\end{defn}

\begin{lem}
The 2-category $\TOPCAT$ together with $(\ofsfont{E},\ofsfont{N})$ and $\und \from \TOPCAT \to \CAT$ is a concrete setting.
\end{lem}
\begin{proof}
The only part of Definition \ref{defn:concrete-setting} that is not obvious is that $(\ofsfont{E},\ofsfont{N})$ is a factorisation system on $\TOPCAT$. This follows from the more general fact that for any monoidal category $\cat{V}$, the category $\cat{V}\hyph\CAT$ of $\cat{V}$-enriched categories has a bijective-on-objects/full-and-faithful factorisation system. This result is well-known and straightforward to prove, so we omit it.
\end{proof}

\begin{defn}
Given a $\TOP$-category $\cat{A}$, write $\proth_t (\cat{A})$ for the category of proto-theories with arities $\cat{A}$ in the setting $\TOPCAT$.
\end{defn}

\begin{defn}
Let $\cat{C}$ and $\cat{D}$ be two $\TOP$-categories. We write $[\cat{C},\cat{D}]_t $ for the category of $\TOP$-functors from $\cat{C} \to \cat{D}$.
\end{defn}

\begin{defn}
Write $\Set_t$ for the category of small sets regarded as a $\TOP$-category in the following way: given sets $X$ and $Y$, we define a topology on $\Set(X,Y)$ as the $X$-fold power of the discrete space $Y$; that is the smallest topology such that for each element $x \in X$, the map $\ev_x \from \Set(X, Y) \to Y$ is continuous, where $Y$ is given the discrete topology.
\end{defn}

\begin{lem}
\label{lem:set-top-cat}
The category $\Set_t$ defined above \emph{is} a well-defined $\TOP$-category.
\end{lem}
\begin{proof}
We must check that for all sets $X$, $Y$ and $Z$, the composition map
\[
\Set(X,Y) \times \Set (Y, Z) \to \Set(X, Z)
\]
is continuous. It is sufficient to check that its composite with each $\ev_x \from \Set(X, Z) \to Z$ is continuous with $Z$ discrete, by definition of the topology on $\Set(X, Z)$. This composite sends
\[
(f,g) \in \Set(X,Y) \times \Set (Y, Z)
\]
to $gf (x)$. We must show that the preimage of each element of $Z$ under this map is open. The preimage of $z \in Z$ under this map is
\[
U_{x,z} = \{ (f, g) \such gf(x) = z \}.
\]
Let $(f_0, g_0) \in U_{x,z}$; we will find an open neighbourhood of $(f_0, g_0)$ that is contained in $U_{x,z}$. Let
\[
V = \{ f \in \Set(X, Y) \such f (x) = f_0 (x) \}.
\]
This is open in $\Set(X,Y)$ since it is the preimage of the point $f_0(x)$ under $\ev_x$. Similarly
\[
W = \{g \in \Set(Y, Z) \such g (f_0 (x)) = z \}
\]
is open in $\Set(Y,Z)$. Therefore $V \times W$ is open in $\Set(X, Y) \times \Set(Y, Z)$, and
\[
(f_0, g_0) \in V \times W \subseteq U_{x, z}
\]
as required.
\end{proof}

\begin{lem}
\label{lem:disc-und-adjunction}
The 2-functor $\und \from \TOPCAT \to \CAT$ has a left 2-adjoint $\disc \from \CAT \to \TOPCAT$.
\end{lem}
\begin{proof}
This is immediate since the forgetful functor $\TOP \to \SET$ has a left adjoint sending a set to the corresponding discrete space. The 2-functor $\disc$ therefore sends an ordinary category to the same category regarded as a $\TOP$-category in which every hom-space is discrete.
\end{proof}

For the rest of this section, we fix a locally small (ordinary, not $\TOP$-enriched) category $\cat{B}$.

\begin{cor}
We have an isomorphism of categories
\[
[\cat{B}^{\op}, \Set] \iso [\disc (\cat{B}^{\op}), \Set_t]_t.
\]
\end{cor}
\begin{proof}
This is immediate from Lemma~\ref{lem:disc-und-adjunction} since $\und(\Set_t) = \Set$.
\end{proof}

\begin{remark}
\label{rem:top-canon-arit}
It follows that we can view the canonical aritation on $\cat{B}$ as an aritation
\[
\currylo \from \cat{B} \to [\cat{B}^{\op}, \Set] \iso [\disc(\cat{B}^{\op}), \Set_t]_t
\]
in $\TOPCAT$, giving rise to a structure--semantics adjunction
 \[
\xymatrix{
{\catover{\cat{B} }}\ar@<5pt>[r]_-{\perp}^-{\str_t}\ & {\proth_t(\disc(\cat{B}^{\op}))^{\op}.}\ar@<5pt>[l]^-{\sem_t}
}
\]
From now on we will usually identify $\disc(\cat{B}^{\op})$ with $\cat{B}^{\op}$ itself (and likewise for other ordinary categories), and so we write this adjunction as
 \[
\xymatrix{
{\catover{\cat{B} }}\ar@<5pt>[r]_-{\perp}^-{\str_t}\ & {\proth_t(\cat{B}^{\op})^{\op}.}\ar@<5pt>[l]^-{\sem_t}
}
\]
This adjunction is our focus for the rest of this chapter.
\end{remark}

\begin{defn}
We call a proto-theory with arities $\disc(\cat{B}^{\op})$ in the setting $\TOPCAT$ a \demph{topological proto-theory with arities $\cat{B}^{\op}$}.

If $L \from \cat{B}^{\op} \to \cat{L}$ is a topological proto-theory, we have two possible notions of $L$-model. There are the models of $L$ arising from the topological structure--semantics adjunction of Remark~\ref{rem:top-canon-arit}, which we call \demph{topological $L$-models}, and models of the underlying discrete proto-theory of $L$ arising via the ordinary structure--semantics adjunction, which we call \demph{discrete $L$-models}.
\end{defn}

\begin{lem}
\label{lem:top-model-algebra}
A discrete model $x = (d^x, \alpha^x)$ of a topological proto-theory $L \from \cat{B}^{\op} \to \cat{L}$ is a topological $L$-model if and only if each
\[
\alpha^x_b \from \cat{L}(Ld^x, Lb) \to \cat{B}(b, d^x)
\]
is continuous, where the codomain is given the discrete topology.
\end{lem}

\begin{proof}
Recall that a model $x$ of $L$ as a discrete proto-theory may be described equivalently either in terms of a functor $\Gamma^x \from \cat{L} \to \Set$ or in terms of a natural transformation $\alpha^x \from \cat{L}( Ld^x, L-) \to \cat{B}(-, d^x)$ satisfying the conditions set out in Definition~\ref{defn:L-alg}, and the two descriptions are related by the following diagram, which commutes for each $f \from b \to d^x$: 

\[
\xymatrix{
\cat{L}(Lb, Lb')\ar[r]^{(Lf)^*}\ar[d]_{\Gamma^x} & \cat{L}( Ld^x, Lb')\ar[d]^{\alpha^x_{b'}} \\
\Set(\cat{B}(b, d^x), \cat{B}(b', d^x))\ar[r]_-{\ev_f} & \cat{B}(b', d^x).
}
\]
Now, the map along the top is always continuous, so if $\alpha^x_{b'}$ is continuous then so is the bottom left composite. But the topology on
\[
\Set_t(\cat{B}(b, d^x), \cat{B}(b', d^x))
\]
is generated by the maps $\ev_f$ for $f \in \cat{B}(b, d^x)$, so it follows that $\Gamma^x$ is continuous. Conversely, if we take $b = d^x$ and $f = \id_{d^x}$, then the map along the top becomes the identity, so $\alpha^x_{b'}$ is the composite $\ev_{\id_{d^x}} \of \Gamma^x$, which is continuous if $\Gamma^x$ is.
\end{proof}

\begin{defn}
We reuse the notation $\disc$ for the functor $\proth(\cat{B}^{\op}) \to \proth_t(\cat{B}^{\op})$ that sends a proto-theory $L \from \cat{B}^{\op} \to \cat{L}$ to $L$ regarded as a topological proto-theory on $\cat{B}^{\op}$ in which the hom-sets of $\cat{L}$ are equipped with the discrete topology.
\end{defn}

Note that the term ``discrete topological category'' is potentially ambiguous: there are unrelated notions of discreteness for both topological spaces and for categories. When we refer to discrete topological categories, we mean topological categories in which every hom-space is discrete, rather than in which the only morphisms are identities. Likewise, a discrete topological proto-theory is a proto-theory $L \from \cat{B}^{\op} \to \cat{L}$ for which $\cat{L}$ is a discrete topological category.

\begin{lem}
\label{lem:top-sem-commute}
The triangle
\[
\xymatrix{
\catover{\cat{B}} & \proth_t(\cat{B}^{\op})^{\op}\ar[l]_-{\sem_t} \\
& \proth(\cat{B}^{\op})^{\op}\ar[ul]^{\sem}\ar[u]_{\disc}
}
\]
commutes.
\end{lem}

\begin{proof}
A topological model of a topological proto-theory $L \from \cat{B}^{\op} \to \cat{L}$ consists of a model $x = (d^x, \Gamma^x)$ of the underlying ordinary proto-theory such that
\[
\Gamma^x \from \cat{L} \to \Set_t
\]
is a continuous functor. But if the hom-spaces of $\cat{L}$ are all discrete, then every functor out of $\cat{L}$ is continuous, so the notions of $L$-model and topological $L$-model coincide.
\end{proof}

\begin{lem}
\label{lem:top-sem-limit-create}
Let $L \from \cat{B}^{\op} \to \cat{L}$ be a topological proto-theory with arities $\cat{B}^{\op}$, and suppose that $\cat{B}$ has limits of shape $\cat{I}$ for some finite category $\cat{I}$. Then the functor $\sem_t \from \mod_t(L) \to \cat{B}$ creates limits of shape $\cat{I}$.
\end{lem}
\begin{proof}
By Proposition~\ref{prop:sem-lim-create}, it is sufficient to show that each $\lpair b, - \rpair_0 \from \cat{B} \to (\Set_t)_0$ preserves such limits, and
\[
\und \from [\cat{A}, \Set_t]_t \to [\cat{A}_0, \Set]
\]
creates them for each $\cat{A} \in \TOPCAT$. The first of these is trivial, since $\lpair b, - \rpair_0 = \cat{B}(b, -) \from \cat{B} \to (\Set_t)_0 = \Set$ and representables preserve all limits.

Let us show that 
\[
\und \from [\cat{A}, \Set_t]_t \to [\cat{A}_0, \Set]
\]
creates finite limits. This amounts to showing that a finite limit of continuous functors into $\Set_t$ is continuous. Let $D \from \cat{I} \to [\cat{A}, \Set_t]_t$ be a functor; write $D^i$ for $D(i)$. We must check that for each $a, a' \in \cat{A}$, the map
\[
\lim_{i \in \cat{I}} D^i \from \cat{A}(a, a') \to \Set_t( \lim_{i \in \cat{I}} D^i a, \lim_{i \in \cat{I}} D^i a')
\]
is continuous. By the definition of the topology on hom-sets in $\Set_t$ it is sufficient to check that it becomes continuous when composed with each $\ev_x \from \Set_t( \lim_{i \in \cat{I}} D^i a, \lim_{i \in \cat{I}} D^i a') \to \lim_{i \in \cat{I}} D^i a'$ for $x \in \lim_{i \in \cat{I}} D^i a$, where the codomain is discrete. Recall that an element $x \in \lim_{i \in \cat{I}} D^i a$ consists of a family $(x_i)_{i \in \cat{I}}$ indexed by $i \in \cat{I}$ where $x_i \in D^i a$ and for any $f \from i \to j $ in $\cat{I}$, we have $D^f x_i = x_j$.

Since $\cat{I}$ is finite, and a finite limit of discrete spaces is discrete, the topology on $\lim_{i \in \cat{I}} D^i a'$ as a limit of the discrete spaces $D^i a'$ is also discrete. Thus in order to check that $\ev_x \of \lim_{i \in \cat{I}} D^i$ is continuous, it is sufficient to check that it is continuous when composed with each $\pi_j \from \lim_{i \in \cat{I}} D^i a' \to D^j a'$. Now let $x = (x_i)_{i \in \cat{I}} \in \lim_{i \in \cat{I}} D^i a$ and $j \in \cat{I}$, and consider the diagram
\[
\xymatrix{
\cat{A}(a, a') \ar[r]^-{\lim_{i \in \cat{I}} D^i} \ar[dd]_{D^j} & \Set_t(\lim_{i \in \cat{I}} D^i a, \lim_{i \in \cat{I}} D^i a')\ar[d]^{\ev_x} \\
& \lim_{i \in \cat{I}} D^i a' \ar[d]^{\pi_j} \\
\Set_t(D^j a, D^j a') \ar[r]_{\ev_{x_j}} & D^j a'.
}
\]
This diagram commutes: both legs send $g \from a \to a'$ to $D^j g (x_j)$. The top-right composite is the map we wish to show is continuous. But the bottom-left composite is continuous, since $D^j$ is a continuous functor and by definition of the topology on hom-sets in $\Set_t$. Thus $\lim_{i \in \cat{I}}D^i$ is continuous, as required.
\end{proof}

\section{Monads as topological proto-theories}
\label{sec:monads-as-top-proths}

Throughout this section, fix a locally small category $\cat{B}$. Recall from Section~\ref{sec:monads-are-proths} that we can view monads on $\cat{B}$ as proto-theories with arities in $\cat{B}^{\op}$, and then the usual semantics for monads is recovered via the canonical aritation. In this section we show that the same is true when we replace proto-theories with topological proto-theories.
\begin{lem}
\label{lem:radj-str-discrete}
Let $(U \from \cat{M} \to \cat{B}) \in \catover{\cat{B}}$ have a left adjoint $F$, with unit $\eta$ and counit $\epsilon$. Then $\str_t (U) \in \proth_t(\cat{B}^{\op})$ is discrete.
\end{lem}
\begin{proof}
For $b, b' \in \cat{B}$, we must show that
\[
\thr_t(U) (b , b') = [\cat{M}, \Set_t] (\cat{B}(b, U- ) , \cat{B}(b', U- )),
\]
equipped with its canonical topology as a limit of the discrete spaces $\cat{B}(b', Um)$ for $m \in \cat{M}$, is discrete. Let $\gamma \from \cat{B}(b, U- ) \to \cat{B}(b', U- )$; we will show that $\{ \gamma \}$ is open in $\thr_t(U)(b, b')$.

The map
\[
\ev_{\eta_b} \from [\cat{M}, \Set] (\cat{B}(b, U- ) , \cat{B}(b', U- )) \to \cat{B}(b', U F b)
\]
that sends $\delta \from \cat{B}(b, U - ) \to \cat{B}(b', U-)$ to $\delta_{Fb} (\eta_b)$ is continuous, since it is the composite
\[
 [\cat{M}, \Set] (\cat{B}(b, U- ), \cat{B}(b', U -)) \toby{(-)_{F b}} \Set_t (\cat{B}(b, U F b),\cat{B}(b', U F b)) \toby{\ev_{\eta_b}} \cat{B}(b', U F b)
\]
and both of these factors are continuous by definition. Thus, the preimage of $\{ \gamma_{Fb} (\eta_b) \}$ under this map is an open set. Let $\delta \from \cat{B}(b, U  -) \to \cat{B}(b', U  -)$ be an element of this preimage. Then $\ev_{\eta_b} ( \gamma) = \ev_{\eta_b}(\delta)$, that is, $\gamma_{F b} (\eta_b) = \delta_{F b } (\eta_b)$. For any $m \in \cat{M}$ and $f \from b \to Um$, we have
\begin{align*}
\gamma_x (f) & = \gamma_x ( U\epsilon_m \of \eta_{Um} \of f) && \text{(Triangle identity)} \\
&= \gamma_x (U \epsilon_m \of UF f \of \eta_b ) && \text{(Naturality of $\eta$)} \\
&= U \epsilon_m \of UFf \of \gamma_{Fb} (\eta_b) && \text{(Naturality of $\gamma$)}
\end{align*}
Similarly,
\[
\delta_x (f) =U\epsilon_m \of UFf \of \delta_{Fb} (\eta_b),
\]
but $\delta_{Fb} (\eta_b) = \gamma_{Fb} (\eta_b)$, and so it follows that $\delta_x (f) =\gamma_x (f)$. Hence $\delta = \gamma$ and the open set
\[
\ev_{\eta_b} ^{-1} ( \{ \gamma_{Fb} (\eta_b) \} )
\]
is in fact $\{ \gamma \}$. So the space is discrete, as claimed.
\end{proof}

\begin{defn}
\label{defn:kle-top}
Write $\kle_t \from \monad(\cat{B}) \to \proth_t(\cat{B}^{\op})$ for the composite
\[
\monad(\cat{B}) \toby{\kle} \proth(\cat{B}^{\op}) \toby{\disc} \proth_t(\cat{B}^{\op}).
\]
\end{defn}

\begin{prop}
\label{prop:top-monad-adj-restrict}
Both squares in the diagram
\[
\xymatrix{
{\radjover{\cat{B} }}\ar@<5pt>[r]_-{\perp}^-{\str_{\monad}}\ar[d] & {\monad(\cat{B})^{\op}}\ar@<5pt>[l]^-{\sem_{\monad}}\ar[d]^{\kle_t^{\op}} \\
{\catover{\cat{B} }}\ar@<5pt>[r]_-{\perp}^-{\str_t}\ & {\proth_t(\cat{B}^{\op})^{\op}}\ar@<5pt>[l]^-{\sem_t}
}
\]
commute up to isomorphism, and these isomorphisms are compatible with the adjunction structure as in Proposition~\ref{prop:sem-str-counit-restr}.
\end{prop}
\begin{proof}

The square involving $\str_{\monad}$ and $\str_t$ commutes since, by Theorem~\ref{thm:str-sem-rest-monad}
\[
\xymatrix{
\radjover{\cat{B}}\ar[r]^{\str_{\monad}} \ar[d] & \monad(\cat{B})^{\op}\ar[d]^{\kle^{\op}} \\
\catover{\cat{B}} \ar[r]_{\str} & \proth(\cat{B}^{\op})^{\op}
}
\]
commutes, and by Lemma~\ref{lem:radj-str-discrete}, for any right adjoint $U \from \cat{M} \to \cat{B}$, we have $\str_t (U) = \disc^{\op} \of \str (U)$. The commutativity of the square involving $\alg$ and $\sem_t$ follows from the fact that
\[
\xymatrix{
\radjover{\cat{B}} \ar[d] & \monad(\cat{B})^{\op}\ar[d]^{\kle^{\op}}\ar[l]_{\sem_{\monad}} \\
\catover{\cat{B}} & \proth(\cat{B}^{\op})^{\op}\ar[l]^{\sem}
}
\]
commutes (Theorem~\ref{thm:str-sem-rest-monad}), and the fact that $\sem = \sem_t \of \disc^{\op}$ (Lemma~\ref{lem:top-sem-commute}). The proof that the counits of the adjunctions are compatible is identical to that of Proposition~\ref{prop:sem-str-counit-restr}.
\end{proof}

\section{Categories with enough subobjects}
\label{sec:cats-w-enough-subobjects}

We would like to show that the topological structure--semantics adjunction is idempotent, in an effort to find a notion of algebraic theory for which the completeness theorem holds. However, this is unlikely to be the case in an arbitrary category $\cat{B}$; in this section we define a technical condition that a category may satisfy, and in the next section we show that this guarantees idempotency. This condition appears to be very restrictive, however, it holds in the most important example, namely the category of small sets, as well as the category of vector spaces over any field.

\begin{defn}
We say that a category $\cat{B}$ \demph{has enough subobjects} if every presheaf $P \from \cat{B}^{\op} \to \SET$ that
\begin{itemize}
\item preserves all small products that exist in $\cat{B}^{\op}$, and
\item is a sub-presheaf of a representable presheaf
\end{itemize}
is itself representable.
\end{defn}

\begin{lem}
The category $\Set$ has enough subobjects.
\end{lem}
\begin{proof}
Let $P \from \Set^{\op} \to \SET$ be a functor that preserves small products, and let $m \from P \to \Set(-, X)$ be a monomorphism, that is, component-wise injective. As usual write $1 = \{*\}$ for an arbitrary one-element set. Then $P(1)$ is a small set since it admits an injection to $\Set(1,X) \iso X \in \Set$. We show that
\[
P \iso \Set(-, P(1)).
\]
Let $Y \in \Set$. Then $Y$ can be regarded as the coproduct of $Y$ copies of $1$. Hence, since $P$ sends small coproducts to products, we have
\[
P(Y) \iso P(1)^Y \iso \Set(1, P(1)^Y) \iso \Set(Y, P(1))
\]
and each of these isomorphisms is natural in $Y$.
\end{proof}

\begin{lem}
The category $\finset$ has enough subobjects.
\end{lem}
\begin{proof}
The proof is identical to that of the previous lemma.
\end{proof}

\begin{lem}
For any small field $k$, the category $\Vect_k$ of small vector spaces over $k$ has enough subobjects.
\end{lem}

\begin{proof}
Let $P \from \Vect_k ^{\op} \to \SET$ be a functor that preserves small products and let $m \from P \to \Vect_k(-,V)$ be a monomorphism. We can view $k$ as a 1-dimensional vector space over itself, and so we have $m_k \from P(k) \to \Vect_k(k, V)$. Define
\[
W = \{v \in V \such \exists p \in P(k) \text{ such that } m_k (p)(1_k) = v \} \subseteq V.
\]
First we show that $W$ is a subspace of $V$. Write $\vect{0}$ for the trivial vector space over $k$. Then since $\vect{0}$ is the initial object (i.e.\ empty coproduct) in $\Vect_k$, and $P$ preserves products, $P(\vect{0})$ must be the terminal object of $\SET$. There is a unique map $0 \from k \to \vect{0}$, and applying $P$ we obtain a map $P(0) \from 1 \iso P(\vect{0}) \to P(k)$; let $q \in P(k)$ be the unique value of this map. Then the commutativity of
\[
\xymatrix{
P(\vect{0}) \ar[r]^-{m_{\vect{0}}}\ar[d]_{P(0)} & \Vect_k(\vect{0}, V)\ar[d]^{0^*} \\
P(k)\ar[r]_-{m_k} & \Vect_k(k, V)
}
\]
implies that $m_k(q)$ is the zero map $k \to V$, so in particular $m_k(q)(1_k) = 0_V$, so $0_V \in W$.

Let $w = m_k(p)(1_k) \in W$ and let $c \in k$. We show that $cw \in W$. Consider the map $c \cdot - \from k \to k$ given by multiplication by $c$. The diagram
\[
\xymatrix{
P(k) \ar[r]^-{m_k} \ar[d]_{P(c \cdot -)} & \Vect_k(k,V) \ar[d]^{(c \cdot -)^*} \\
P(k) \ar[r]^-{m_k}  & \Vect_k(k,V)
}
\]
commutes, and hence the map $m_k (p) \of (c \cdot -) \from k \to V$ lies in the image of $m_k$. But this map sends $1_k$ to $m_k(p) (c \cdot 1_k) = c \cdot m_k (p)(1_k) = c w$, so $cw \in W$.

Now let $w, w' \in W$, say $w = m_k (p) (1_k)$ and $w' = m_k(p')(1_k)$. Consider the diagram
\[
\xymatrix{
P(k) \times P(k) \ar[r]^-{m_k \times m_k}\ar[d]_{\iso} & \Vect_k(k, V) \times \Vect_k(k, V)\ar[d]^{\iso} \\
P(k \oplus k)\ar[r]^-{m_{k \oplus k}}\ar[d]_{P(\Delta)} & \Vect_k(k \oplus k, V)\ar[d]^{\Delta^*} \\
P(k)\ar[r]_{m_k} & \Vect_k(k, V),
}
\]
where $\Delta \from k \to k \oplus k$ sends $x \in k$ to $(x,x) \in k \oplus k$. The top square commutes since both $P$ and $\Vect_k(-, V)$ preserve small products, and the bottom square commutes by naturality of $m$. We have $(p, p') \in P(k) \times P(k)$, and the top-right composite sends this to $m_k(p) + m_k(p')$, so in particular this element of $\Vect_k(k, V)$ lies in the image of $m_k$. But
\[
(m_k(p) + m_k(p') )(1_k) = m_k(p)(1_k) + m_k(p')(1_k) = w + w',
\]
so $w + w' \in W$.

Now we have established that $W$ is a subspace of $V$, and in particular that it \emph{is} a vector space over $k$, we show that we have a factorisation
\[
\xymatrix{
P \ar[rr]^{m}\ar[dr]_{n} & & \Vect_k(-, V). \\
& \Vect_k(-, W)\ar[ur] &
}
\]
Let $U$ be an arbitrary small $k$-vector space, and let $p \in P(U)$. We show that $m_U (p) \from U \to V$ takes values in $W$. Let $u \in U$. Then we have a unique $f \from k \to U$ such that $f(1_k) = u$, and
\[
m_U (p) (u) =  m_U(p) \of f (1_k) = m_k (P(f) (p))(1_k) \in W,
\]
by definition. Hence we have such a factorisation $n \from P \to \Vect_k(-, W)$.

Now we show that $n$ is an isomorphism. First note that $n$ is monic, since $m$ is; thus we only need to show that for each vector space $U$, the map
\[
n_U \from P(U) \to \Vect_k(U, W)
\]
is a surjection. Let $f \from U \to V$ be a map taking values in $W \subseteq V$; we must show that there exists $p \in P(U)$ such that $m_u (p) = f$. Recall that any vector space, and in particular $U$, has a basis (and this basis is small since $U$ is). This means that for some small set $S$ we may choose a family of maps $\iota_s \from k \to U$ indexed by $s \in S$ that exhibit $U$ as an $S$-fold copower of $k$ in $\Vect_k$. For each $s \in S$, consider the composite
\[
k \toby{\iota_s} U \toby{f} V.
\]
Since $f \of \iota_s (1_k) \in W$, there exists $p_s \in P(k)$ such that
\[
f \of \iota_s (1_k) = m_k (p_s) (1_k),
\]
and since a map out of $k$ is determined by its value on $1_k$ we must have $f \of \iota_s = m_k(p_s)$. Since $U$ is the $S$-th copower of $k$, and $P$ preserves small products, we have a commutative diagram
\begin{equation}
\label{eqn:vect-sub}
\xymatrix{
P(k)^S \ar[r]^{m_k^S}\ar[d]_{\iso} & \Vect_k(k, V)^S \ar[d]^{\iso} \\
P(U) \ar[r]_{m_U} & \Vect_k(U, V).
}
\end{equation}
We claim that the element $(p_s)_{s \in S} \in P(k)^S$ is mapped by the upper-right composite to $f \in \Vect_k(U,V)$. Since the isomorphism $\Vect_k(k, V)^S \iso \Vect_k(U, V)$ is such that
\[
\xymatrix{
\Vect_k(k, V)^{S} \ar[r]^{\iso}\ar[dr]_{\pi_s} & \Vect_k(U, V)\ar[d]^{\iota_s^*}\\
& \Vect_k(k, V)
}
\]
commutes for each $s \in S$, this is equivalent to the condition that $f \of \iota_s = m_k (p_s)$ for each $s \in S$, which was established above. Thus, by the commutativity of Diagram~\bref{eqn:vect-sub}, there is some $p \in P(U)$ (namely the image of $(p_s)_{s \in S}$ under the isomorphism $P(k)^S \iso P(U)$) that is mapped to $f$ by $m_U$, so $m_U$ is surjective.

We have shown that $n \from P \to \Vect_k(-, W)$ is an isomorphism, so $P$ is representable, as required.
\end{proof}

\begin{lem}
The category $\finvect_k$, of finite-dimensional vector spaces over a small field $k$ has enough subobjects.
\end{lem}
\begin{proof}
The proof is identical to that of the previous lemma.
\end{proof}

\begin{lem}
Let $Q$ be a small poset with arbitrary joins. Then $Q$, regarded as a category, has enough subobjects.
\end{lem}
\begin{proof}
For any $q \in Q$, a sub-presheaf of the representable $Q(-, q)$ can be identified with a subset $P$ of $Q$ that is downwards closed, and with $p \leq q$ for every $p \in P$. The condition that the presheaf preserves products corresponds to the condition that the join of any subset of $P$ is in $P$. In particular, the join of $P$ itself lies in $P$, so $P$ has a largest element, and so is representable.
\end{proof}

\section{Idempotency of the topological structure--semantics adjunction}
\label{sec:top-str-sem-idem}

In this section we show (Theorem~\ref{thm:str-sem-top-idempotent}) that, for a category with finite products and enough subobjects, the topological structure--semantics adjunction is idempotent. Fix a locally small category $\cat{B}$ throughout.

\begin{defn}
Let $L \from \cat{B}^{\op} \to \cat{L}$ be an object of $\proth_t(\cat{B}^{\op})$. Write
\[
E_L \from L \to \str_t(\sem_t(L))
\]
for the $L$-component of the counit of the $\str_t \ladj \sem_t$ adjunction. Explicitly, $E_L$ sends a morphism $l \from Lb \to Lb'$ to the natural transformation
\[
\Gamma^{(-)}(l) \from \cat{B}(b, \sem_t(-) ) \to \cat{B}(b', \sem_t(-)).
\]
\end{defn}

\begin{defn}
We say that a morphism 
\[
\xymatrix{
\cat{L}'\ar[rr]^{P} &  &\cat{L} \\
& \cat{B}^{\op}\ar[ul]^{L'}\ar[ur]_{L}
}
\]
in $\proth_t(\cat{B}^{\op})$ is \demph{topologically dense} if, for all $b, b' \in \cat{B}$, the continuous map
\[
P \from \cat{L}' (L' b', L'b) \to \cat{L} (Lb', Lb)
\]
has dense image.
\end{defn}

\begin{lem}
\label{lem:dense-iso}
Let $L \from \cat{B}^{\op} \to \cat{L}$ be a topological proto-theory and suppose $E_L$ is topologically dense. Then
\[
\sem_t (E_L) \from \sem_t \of \str_t \of \sem_t (L) \to \sem_t (L)
\]
is an isomorphism in $\catover{\cat{B}}$.
\end{lem}
\begin{proof}
We know that $\sem_t(E_L)$ is split epic by one of the triangle identities for the $\str_t \ladj \sem_t$ adjunction. So it is enough to show that it is monic --- that is, faithful and injective on objects. But it commutes with the forgetful functors to $\cat{B}$ which are both faithful, thus it is itself faithful.

Suppose $E_L$ is topologically dense, and that $x = (d^x, \Gamma^x) $ and $y = (d^y, \Gamma^y)$ are two objects of $\mod_t(\str_t(\sem_t(L)))$ such that $\sem_t (E_L)(x) = \sem_t(E_L)(y)$. Then certainly $d^x = d^y$ since $\sem_t (L) \of \sem_t(E_L) = \sem_t (\str_t(\sem_t(L)))$, so $x$ and $y$ have the same underlying objects. Note that $\Gamma^{\sem_t(E_L)(x)}$ is given by the composite
\[
\cat{L} \toby{E_L} \thr_t (\sem_t (L))\toby{\Gamma^x} \Set_t,
\]
and similarly for $y$. Since these two functors are equal, for each $b, b' \in \cat{B}$ we have a fork
\[
\xymatrix{
\cat{L}(Lb', Lb) \ar[r]^-{E_L}& \thr_t(\sem_t(L))(b', b)\ar@<2.5pt>[r]^-{\Gamma^x}\ar@<-2.5pt>[r]_-{\Gamma^y}& \Set_t(\cat{B}(b', d^x), \cat{B}(b, d^x))
}
\]
in $\TOP$. If a parallel pair of continuous maps have Hausdorff codomain and agree on a dense subset of their domain, then they are equal. But $ \Set_t(\cat{B}(b, d^x), \cat{B}(b, d^x))$ is Hausdorff, as a limit of discrete spaces, and $E_L$ has dense image by assumption, so $\Gamma^x = \Gamma^y$. Thus $x = y$ and $\sem_t(E_L)$ is injective on objects, as required.
\end{proof}

Recall that the surjections and injections form an orthogonal factorisation system on $\SET$, and it lifts to a factorisation system on any category of presheaves $[\cat{B}^{\op}, \SET]$, in which the left and right classes consist of the natural transformations that are component-wise surjections and injections respectively.

\begin{lem}
\label{lem:product-factor}
Let $P,Q, R \in [\cat{B}^{\op}, \SET]$, and $\sigma \from Q \to P$, $\tau \from P \to R$ with $\sigma$ component-wise surjective and $\tau$ component-wise injective. If $Q$ and $R$ preserve small products then so does $P$.
\end{lem}
\begin{proof}
Let $(b_s \in \cat{B})_{s \in S}$ be a family of objects of $\cat{B}$ indexed by some small set $S$. Then we have a commutative diagram
\[
\xymatrix{
Q(\sum_{s \in S} b_s) \ar[r]^{\sigma_{\sum b_s}} \ar[d]_{\iso} & P(\sum_{s \in S} b_s)\ar[r]^{\tau_{\sum b_s}} \ar[d] & R(\sum_{s \in S} b_s)\ar[d]^{\iso} \\
\prod_{s\in S} Q(b_s) \ar[r]_{\prod \sigma_{b_s}} & \prod_{s\in S} P(b_s) \ar[r]_{\prod \tau_{b_s}} & \prod_{s\in S} R(b_s)
}
\]
in which the vertical morphisms are the canonical comparison maps, and the left and right ones are isomorphisms since $Q$ and $R$ preserve small products. Now, 
\[
\prod_{s \in S} \sigma_{b_s} \from \prod_{s \in S} Q(b_s) \to \prod_{s \in S} P(b_s)
\]
is clearly surjective by the construction of products in $\SET$, since each $\sigma_{b_s}$ is. Similarly, 
\[
\prod_{s \in S} \tau_{b_s} \from \prod_{s \in S} P(b_s) \to \prod_{s \in S} R(b_s)
\]
is injective. But then by the uniqueness of epi-mono factorisations in $\SET$, the middle vertical map must be an isomorphism, so $P$ preserves small products.
\end{proof}

\begin{lem}
Suppose $\cat{B}$ has finite products. Let $L \from \cat{B}^{\op} \to \cat{L}$ be a topological proto-theory, and consider, for $b, b' \in \cat{B}$, the topological space $\thr_t(\sem_t(L)) (b ,b') = [\mod_t(L), \Set] (\cat{B}(b, \sem_t(L)-), \cat{B}(b', \sem_t(L)-))$. This space has a basis for its topology consisting of the sets
\[
U_{f,g} = \{ \delta \from \cat{B}(b, \sem_t(L)-) \to \cat{B}(b', \sem_t(L)- ) \such \delta_x (f) = g \},
\]
indexed by $x \in \mod_t(L), f \from b \to d^x$ and $g \from b' \to d^x$.
\end{lem}

\begin{proof}
Recall that the topology on $[\mod_t(L), \Set] (\cat{B}(b, \sem_t(L)-), \cat{B}(b', \sem_t(L)-))$ is generated by the maps
\[
(-)_x \from [\mod_t(L), \Set] (\cat{B}(b, \sem_t(L)-), \cat{B}(b', \sem_t(L)-)) \to \Set_t ((\cat{B}(b, d^x), \cat{B}(b', d^x))
\]
for $x \in \mod_t (L)$. But the topology on $\Set_t (\cat{B}(b, d^x), \cat{B}(b', d^x))$ is generated by the maps
\[
\ev_f \from \Set_t (\cat{B}(b, d^x), \cat{B}(b', d^x)) \to \cat{B}(b', d^x)
\]
for $f \from b \to d^x$. Hence the topology on $[\mod_t(L), \Set] (\cat{B}(b, \sem_t(L)-), \cat{B}(b', \sem_t(L)-))$ is generated by the maps
\[
\ev_f \from [\mod_t(L), \Set] (\cat{B}(b, \sem_t(L)-), \cat{B}(b', \sem_t(L)-)) \to \cat{B}(b', d^x)
\]
for $x \in \mod_t(L)$ and $f \from b \to d^x$, sending $\gamma$ to $\gamma_x (f)$. Thus 
\[
[\mod_t(L), \Set] (\cat{B}(b, \sem_t(L)-), \cat{B}(b', \sem_t(L)-))
\]
has a basis of open sets consisting of finite intersections of preimages of points under such maps. That is, the sets of the form
\begin{equation}
\label{eqn:basic-open}
\{\delta \from \cat{B}(b, \sem_t(L) - ) \to \cat{B}(b', \sem_t(L) - ) \such \delta_{x_i} (f_i) = g_i \text{ for } i= 1, \ldots, n \}
\end{equation}
form a basis, where $x_i \in \mod_t(L), f_i \from b \to d^{x_i}$, and $g_i \from b' \to d^{x_i}$ for $i = 1, \ldots, n$. In particular, the sets described in the lemma statement are of this form, by taking $n = 1$. Now, recall that $\sem_t (L)$ creates finite limits by Lemma~\ref{lem:top-sem-limit-create}, and in particular finite products, hence the product $x_1 \times \cdots \times x_n$ exists in $\mod_t(L)$ and $d^{x_1 \times \cdots \times x_n} = d^{x_1} \times \cdots \times d^{x_n}$. This means that, for any $\delta \from \cat{B}(b, \sem_t (L) -) \to \cat{B}(b', \sem_t (L) -)$, we have a commutative diagram
\[
\xymatrix
@C=40pt{
\cat{B}(b, d^{x_1} \times \cdots \times d^{x_n})\ar[r]^{\delta_{x_1 \times \cdots \times x_n}} \ar[d]^{\iso} & \cat{B}(b', d^{x_1} \times \cdots \times d^{x_n}) \ar[d]^{\iso} \\
\cat{B}(b, d^{x_1}) \times \cdots \times \cat{B}(b, d^{x_n}) \ar[r]_{\delta_{x_1} \times \cdots \times \delta_{x_n}} & \cat{B}(b', d^{x_1}) \times \cdots \times \cat{B}(b', d^{x_n}).
}
\]
It follows that $\delta_{x_i}(f_i) = g_i$ for $i = 1, \ldots, n$ if and only if 
\[
\delta_{x_1 \times \cdots \times x_n} ( \lpair f_1, \ldots, f_n \rpair ) = \lpair g_1, \ldots, g_n \rpair,
\]
where $\lpair f_1, \ldots, f_n \rpair \from b \to d^{x_1} \times \cdots \times d^{x_n}$ is the unique map such that $\pi_i \of \lpair f_1, \ldots, f_n \rpair = f_i$, and similarly for $\lpair g_1, \ldots, g_n \rpair$. Thus, the set displayed in \bref{eqn:basic-open} is equal to the set
\[
\{ \delta \from \cat{B}(b, \sem_t(L)-) \to \cat{B}(b', \sem_t(L)- ) \such \delta_{x_1 \times \cdots \times x_n} (\lpair f_1, \ldots , f_n \rpair ) = \lpair g_1, \ldots , g_n \rpair \}
\]
which is of the form given in the lemma statement.
\end{proof}

\begin{prop}
\label{prop:subobjects-dense}
Suppose $\cat{B}$ has finite products and enough subobjects. Then, for any $L \from \cat{B}^{\op} \to \cat{L}$ in $\proth_t(\cat{B}^{\op})$ such that $L$ preserves small products, $E_L$ is topologically dense.
\end{prop}
\begin{proof}
We must show that for all $b, b' \in \cat{B}$ the continuous map
\[
E_L \from \cat{L}(Lb, Lb') \to [\mod_t(L), \Set] (\cat{B}(b, \sem_t(L) -), \cat{B}(b', \sem_t(L) - ) )
\]
that sends a morphism $l \from Lb \to Lb'$ to the natural transformation $E_L(l)$ with components
\begin{align*}
E_L (l)_x \from &\cat{B}(b, d^x) \to \cat{B}(b', d^x) \\
& f \mapsto \alpha^x_{b'} (l \of Lf)
\end{align*}
has dense image. It is sufficient to show that for every $\gamma \from \cat{B}(b, \sem_t(L) - ) \to \cat{B}(b', \sem_t(L)-)$, every basic open neighbourhood of $\gamma$ contains $E_L(l)$ for some $l \from Lb \to Lb'$.  But by the previous lemma, a basic open neighbourhood of $\gamma$ is of the form
\[
\{\delta \from \cat{B}(b, \sem_t(L)- ) \to \cat{B}(b', \sem_t(L)- ) \such \gamma_x(f) = \delta_x(f) \}
\]
for some $x \in \mod_t(L)$ and $f \from b \to d^x$. Thus we must show that, for any such $x$ and $f$, there exists some $l \from Lb' \to Lb$ such that
\begin{equation}
\label{eqn:dense-condition}
\alpha^x_{b'} (l \of Lf) = \gamma_x(f).
\end{equation}
Define a presheaf $P \from \cat{B}^{\op} \to \SET$ and natural transformations $\sigma \from \cat{L}(Lb, L-) \to P$ and $\tau \from P \to \cat{B}(-, d^x)$ via the epi--mono factorisation of the composite
\[
\cat{L}(Lb, L-) \toby{Lf^*} \cat{L}(Ld^x, L-) \toby{\alpha^x} \cat{B}(-, d^x).
\]
Now, $L\from \cat{B}^{\op} \to \cat{L}$ preserves small products, and the representable $\cat{L}(Lb, -) \from \cat{L} \to \SET$ preserves small products, so it follows that their composite $\cat{L}(Lb, L-) \from \cat{B}^{\op} \to \SET$ preserves products. Also, the representable $\cat{B}(-, d^x)$ preserves small products. It follows from Lemma~\ref{lem:product-factor} that $P$ also preserves small products. In addition, since $\tau \from P \to \cat{B}(-, d^x)$ is monic by definition, $P$ is a sub-presheaf of a representable, and so is itself representable since $\cat{B}$ has enough subobjects. Since epi-mono factorisations are only defined up to isomorphism, we may assume that $P$ is actually \emph{equal} to the presheaf $\cat{B}(-, d^y)$ for some object $d^y \in \cat{B}$. By the Yoneda lemma, there is a unique $t \from d^y \to d^x$ such that 
\[
\tau = t_* \from P = \cat{B}(-, d^y) \to \cat{B}(-, d^x)
\]
and there is a unique $s \from b \to d^y$ such that $s_* \from \cat{B}(-, b) \to \cat{B}(-, d^y)$ is equal to the composite
\[
\cat{B}(-, b) \toby{L} \cat{L}(Lb, L-) \toby{\sigma} \cat{B}(-, d^y) = P.
\]
Note that since $\tau$ is monic, so is $t$. Also, note that the composite
\[
\cat{B}(-, b) \toby{s_*} \cat{B}(-, d^y) \toby{t_*} \cat{B}(-, d^x)
\]
sends an arbitrary morphism $g \from c \to b$ in $\cat{B}$ to
\[
 \tau_c \of\sigma_c (Lg) = \alpha_c^x ( Lg \of L f) = f \of g,
\]
and so $t \of s = f$.  

We now equip $d^y$ with an $L$-model structure $\alpha^y \from \cat{L}(Ld^y, L-) \to \cat{B}(-, d^y)$. Since $\sigma \from \cat{L}(Lb, L-) \to \cat{B}(-,d^y)$ is component-wise surjective, we may choose $k \in \cat{L}(Lb, Ld^y)$ such that $ \sigma_{d^y} (k) = \id_{d^y}$. Given such a $k$, define $\alpha^y$ to be the composite
\[
\cat{L}(Ld^y, L-) \toby{k^*} \cat{L}(Lb, L-) \toby{\sigma} \cat{B}(-, d^y).
\]
First let us show that $\alpha^y$ does not depend on the choice of $k$. For any $c \in \cat{B}$ and $l \from  Ld^y \to Lc$ in $\cat{L}$, we have
\begin{align}
\label{eqn:factor-model}
t \of \alpha^y_c (l) & = t \of \sigma_c ( l \of k) && \text{(Definition of $\alpha^y$)} \nonumber \\
&= \tau_c (\sigma_c (l \of k)) && \text{(Definition of $t$)} \nonumber\\
&= \alpha^x_c ( l \of k \of Lf) && \text{(Definition of $\tau, \sigma$)} \nonumber \\
&= \alpha^x_c ( l \of  L \alpha_{d^y}^x (k \of Lf ) ) && \text{(Definition~\ref{defn:L-alg}\bref{part:L-alg-comp})} \nonumber \\
&= \alpha^x_c (l \of L (\tau_{d^y} \of \sigma_{d^y} (k)) ) && \text{(Definition of $\tau, \sigma$)} \nonumber\\
&= \alpha^x_c (l \of L (t \of \sigma_{d^y} (k))  ) && \text{(Definition of $t$)} \nonumber\\
&= \alpha^x_c (l \of L t) && \text{(Choice of $k$)}
\end{align}
which clearly does not depend on the choice of $k$. But since $t$ is monic, this implies that $\alpha^y_c (l)$ does not depend on the choice of $k$ either. 

Now we show that $\alpha^y$ does make $y = (d^y, \alpha^y)$ into an $L$-model. Certainly
\[
\alpha^y_{d^y} (\id_{Ld^y}) = \sigma_{d^y} ( \id_{Ld^y} \of k) = \sigma_{d^y} (k) = \id_{d^y}
\]
by choice of $k$, so condition~\ref{defn:L-alg}.\bref{part:L-alg-id} holds. Let $l \from  Ld^y \to Lc$ and $l' \from Lc \to Lc'$ for $c, c' \in \cat{B}$. Then we have
\begin{align*}
t \of \alpha^y_{c'} (l' \of l) &= \alpha^x_{c'} (l' \of l \of Lt) && \text{(By \bref{eqn:factor-model})} \\
&= \alpha_{c'}^x ( l' \of L \alpha_c^x (l \of Lt)) && \text{(Definition~\ref{defn:L-alg}\bref{part:L-alg-comp})} \\
&= \alpha_{c'}^x (l' \of L (t \of \alpha_c^y (l) )  ) && \text{(By \bref{eqn:factor-model})} \\
&= \alpha_{c'}^x (l'  \of L \alpha_c^y (l) \of Lt ) && \text{(Functoriality of $L$)} \\
&= t \of \alpha_{c'}^y (l' \of L\alpha_c^y (l) ) && \text{(By \bref{eqn:factor-model})}
\end{align*}
and so, since $t$ is monic, we have $ \alpha_{c'}^y (l' \of l ) = \alpha_{c'}^y (l' \of L \alpha_c^y (l))$, so $\alpha^y$ satisfies condition~\ref{defn:L-alg}.\bref{part:L-alg-comp}. So $y = (d^y, \alpha^y)$ is an $L$-model. 

Now we check that $y$ satisfies the condition in Lemma~\ref{lem:top-model-algebra}, that is, $y$ is a \emph{topological} $L$-model. We must show that
\[
\alpha_c^y \from \cat{L}(Ld^y, Lc) \to \cat{B}(c, d^y)
\]
is continuous for each $c$, where the codomain is given the discrete topology. But the diagram
\[
\xymatrix{
\cat{L}(Ld^y, Lc) \ar[r]^{Lt^*} \ar[d]_{\alpha^y_c} & \cat{L}(Ld^x, Lc) \ar[d]^{\alpha^x_c} \\
\cat{B}(c, d^y)\ar[r]_{t_*} & \cat{B}(c, d^x)
}
\]
commutes by \bref{eqn:factor-model}. The continuity of $Lt^*$ follows from the fact that $\cat{L}$ is enriched in $\TOP$, and $ \alpha^x_c$ is continuous, since $x$ is a topological $L$-model. Hence $t_* \of \alpha_c^y$ is continuous. But since $t_*$ is the inclusion of one discrete space in another, it follows that $\alpha_c^y$ is itself continuous.  

Hence $y = (d^y, \alpha^y)$ is a topological $L$-model, and Equation~\bref{eqn:factor-model} says that $t \from d^y \to d^x$ is an $L$-model homomorphism $y \to x$. 

Now we check that
\[
\xymatrix{
\cat{L}(Lb, L-)\ar[dr]_{\sigma} \ar[r]^{Ls^*} & \cat{L}(Ld^y, L-)\ar[d]^{\alpha^y} \\
& \cat{B}(-, d^y)
}
\]
commutes. For any $c \in \cat{B}$ and $l \from Lb \to Lc$, we have 
\begin{align*}
t \of \alpha_c^y ( l \of Ls ) & = \alpha_c^x (l \of Ls \of Lt) && \text{(By \bref{eqn:factor-model})} \\
& = \alpha_c^x (l \of Lf ) && \text{(Since $t \of s = f$)} \\
&= \tau_c \of \sigma_c (l) && \text{(Definition of $\tau, \sigma$)} \\
&= t \of \sigma_c (l) && \text{(Definition of $t$)} \\
\end{align*}
and so, since $t$ is monic, we have $ \alpha_c^y (l \of Ls) = \sigma_c(l)$.

Now we are ready to complete the proof. Since $\sigma_{b'}$ is surjective, we may choose $l \in \cat{L}(Lb, Lb')$ such that $\sigma_{b'} (l) = \gamma_y (s)$. Then,
\begin{align*}
\alpha_{b'}^x (l \of Lf ) & = \tau_{b'} \of \sigma_{b'} (l) && \text{(Definition of $\tau, \sigma$)} \\
&= t \of \sigma_{b'} (l) && \text{(Definition of $t$)} \\
&= t \of \gamma_y (s) && \text{(Choice of $l$)} \\
&= \gamma_x (t \of s) && \text{(Since $t$ is a homomorphism $y \to x$ )} \\
&= \gamma_x (f) && \text{(Since $t \of s = f$)}
\end{align*}
as required.
\end{proof}
 
\begin{thm}
\label{thm:str-sem-top-idempotent}
Suppose $\cat{B}$ has finite products and enough subobjects. Then the topological structure--semantics adjunction
\[
\xymatrix{
{\catover{\cat{B} }}\ar@<5pt>[r]_-{\perp}^-{\str_t}\ & {\proth_t(\cat{B}^{\op})^{\op}}\ar@<5pt>[l]^-{\sem_t}
}
\]
is idempotent.  
\end{thm}
\begin{proof}
Let $U \from \cat{M} \to \cat{B}$ be an object of $\catover{\cat{B}}$. Then recall that, by definition of $\str_t$, we have a commutative square
\[
\xymatrix{
\cat{B}^{\op} \ar[r]^{\curryhi}\ar[d]_{\str_t(U)} & [\cat{B}, \Set] \ar[d]^{U^*} \\
\thr_t(U)\ar[r]_{J(U)} & [\cat{M}, \Set]
}
\]
where $J(U)$ is full and faithful. But $\curryhi$ preserves all limits, and $U^*$ preserves small limits (since small limits in each of the functor categories are computed pointwise). Hence the top-right composite in the square preserves small limits. But $J(U)$ is full and faithful and so reflects limits, so $\str_t(U)$ preserves small limits, and in particular, small products. 

It follows that $\str_t(U) \from \cat{B} \to \thr_t(U)$ satisfies the hypotheses of Proposition~\ref{prop:subobjects-dense}, and so
\[
E_{\str_t(U)} \from \thr_t(U) \to \thr_t ( \sem_t (\str_t(U)))
\]
is topologically dense. Hence by Lemma~\ref{lem:dense-iso}, the morphism
\[
\sem_t (E_{\str_t(U)}) \from \sem_t \of \str_t \of \sem_t \of \str_t (U) \to \sem_t \of \str_t (U)
\]
in $\catover{\cat{B}}$ is an isomorphism. This says that the monad on $\catover{\cat{B}}$ induced by the structure--semantics adjunction is idempotent. But this is one of the equivalent conditions for the adjunction itself to be idempotent, as in Lemma~\ref{lem:idem-adj-conditions}.
\end{proof}

\begin{cor}
If $\cat{B}$ has finite products and enough subobjects, then the topological structure--semantics adjunction for $\cat{B}$ factors via a category that embeds as a reflective subcategory of $\catover{\cat{B}}$ and as a replete, coreflective subcategory of $\proth_t(\cat{B})^{\op}$.
\end{cor}
\begin{proof}
Every idempotent adjunction admits such a factorisation by Lemma~\ref{lem:idem-adj-factor}.
\end{proof}

\begin{defn}
\label{defn:top-adj-factor}
Denote the factorisation from the previous corollary as
\[
\xymatrix{
{\catover{\cat{B} }}\ar@<5pt>[r]_-{\perp}^-{\str_{ct}}
& {\proth_{ct}(\cat{B}^{\op})^{\op}}\ar@<5pt>[l]^-{\sem_{ct}}\ar@<5pt>[r]_-{\perp}^-{\inc^{\op}}
& {\proth_t(\cat{B}^{\op})^{\op}} \ar@<5pt>[l]^-{\cplt^{\op}}.
}
\]
Explicitly, 
\begin{itemize}
\item $\proth_{ct}(\cat{B}^{\op})$ is the full subcategory of $\proth_{t}(\cat{B}^{\op})$ consisting of those topological proto-theories $L$ for which $E_L$ is an isomorphism. We call such a proto-theory a \demph{complete topological proto-theory with arities in $\cat{B}^{\op}$} (and the subscript $ct$ stands for ``complete topological'').
\item $\str_{ct}$ is obtained by restricting the codomain of $\str_{t}$ from $\proth_t(\cat{B}^{\op})$ to $\proth_{ct}(\cat{B}^{\op})$.
\item $\sem_{ct}$ is obtained by restricting the domain of $\sem_t$ from $\proth_t(\cat{B}^{\op})$ to $\proth_{ct}(\cat{B}^{\op})$; it is full and faithful.
\item $\inc$ is the full inclusion.
\item $\cplt$ is defined to be the composite $\str_{ct} \of \sem_t$. For a topological proto-theory $L$, we call $\cplt (L)$ the \demph{completion} of $L$.
\end{itemize}
\end{defn}

\begin{lem}
We have $\inc^{\op} \of \str_{ct} \iso \str_t$ and $\sem_{ct} \of \cplt^{\op} \iso \sem_t$.
\end{lem}
\begin{proof}
This is immediate from properties of idempotent adjunctions and reflective subcategories.
\end{proof}

\begin{remark}
Recall that we hoped to find a notion of algebraic theory that
\begin{itemize}
\item generalises monads and their semantics,
\item has a full and faithful semantics functor (that is, it satisfies the completeness theorem), and
\item has a structure functor defined on the whole of $\catover{\cat{B}}$.
\end{itemize}
The first of these properties will be shown in the next chapter (Proposition~\ref{prop:kle-top-factor}), and the second and third are immediate from Definition  \ref{defn:top-adj-factor}, and so, when $\cat{B}$ has finite products and enough subobjects, complete topological proto-theories provide such a notion. In fact they have many additional good properties, which shall be explored in the next chapter.
\end{remark}

\section{Relation to profinite groups}
\label{sec:top-proth-relation-profinite-groups}
Recall from Section~\ref{sec:profinite} that we may regard $\Gp$ as a full subcategory of $\proth(\finset^{\op})$ and that the structure--semantics monad on $\proth(\finset^{\op})$  restricts to the profinite completion monad on $\Gp$. Similarly, in this section we show that $\TopGp$ can be regarded as a full subcategory of $\proth_t(\finset^{\op})$ and that the topological structure--semantics adjunction restricts to the profinite completion monad on $\TopGp$. 

As we did in Section~\ref{sec:profinite}, we identify objects of $\proth(\finset^{\op})$ with bijective-on-objects functors out of $\finset$ (rather than $\finset^{\op}$), and similarly with objects of $\proth_t(\finset^{\op})$.

\begin{defn}
\label{defn:act-cat-top}
Let $M$ be a small topological monoid; we will define a $\TOP$-category $\cat{E}_t(M)$ as follows. Recall from Definition~\ref{defn:monoid-theory-include} that we have $E(M) \from \finset \to \cat{E}(M)$, where the objects of $\cat{E}(M)$ are the finite sets, and $\cat{E}(S, S') = \Set(S, M \times S') \iso (M \times S')^S$. We equip such a hom-set with a topology by regarding $M \times S'$ as the product in $\Top$ of $M$ with the discrete space $S'$, and $(M \times S')^S$ as a finite power of this space. We write $\cat{E}_t(S, S')$ for $\cat{E}(S, S')$ topologised in this way.

\end{defn}

\begin{lem}
For a topological monoid $M$, the definition above does give a well-defined $\TOP$-category $\cat{E}_t (M)$.
\end{lem}
\begin{proof}
We must show that composition in $\cat{E}_t(M)$ is continuous. Let $S, S'$ and $S''$ be finite sets. Then composition $\cat{E}_t(M) (S, S') \times \cat{E}_t(M)(S' ,S'') \to \cat{E}_t(M)(S, S'')$ can be written as the composite 
\[
\xymatrix{
(M \times S')^S \times (M \times S'')^{S'} \iso M^S \times (S')^S \times (M \times S'')^{S'}\ar[d]_{(\id_{M^S}) \times c} \\
M^S \times (M \times S'')^S \iso (M \times M \times S'')^S \ar[d]_{(\mu_M \times \id_{S''})^S} \\
(M^S \times S'')^S
}
\]
where $c \from (S')^S \times (M \times S'')^{S'} \to (M \times S'')^S$ is composition of functions, and $\mu_M \from M \times M \to M$ is the multiplication for $M$. All of the maps involved in the above are clearly continuous except \emph{a priori} for $c$; hence it is sufficient to check that $c$ is continuous. First note that since $S'$ is discrete and $S$ is finite, $(S')^S$ is also discrete, so $ (S')^S \times (M \times S'')^{S'}$ is the $(S')^S$-th copower of $(M \times S'')^{S'}$ in $\Top$. Hence it is sufficient to check that each of the composites
\[
(M \times S'')^{S'} \toby{\iota_f} (S')^S \times (M \times S'')^{S'} \toby{c} (M \times S'')^{S}
\]
is continuous, where, for $f \from S \to S'$, the map $\iota_f$ sends an element $f' \in (M \times S'')^{S'}$ to $ (f, f')$. And since $(M \times S'')^S$ is a power in $\Top$, it is sufficient to check that this is continuous when composed with the projection $\pi_s \from (M \times S'')^S \to M \times S''$ for each $s \in S$. But the composite
\[
(M \times S'')^{S'} \toby{\iota_f} (S')^S \times (M \times S'')^{S'} \toby{c} (M \times S'')^{S} \toby{\pi_s} M \times S''
\]
is just the projection $\pi_{f (s)} \from (M \times S'')^{S'} \to M \times S''$, which is continuous. 
\end{proof}

\begin{lem}
Given a continuous monoid homomorphism $h \from M \to M'$ between small topological monoids, the functor $E(h) \from \cat{E}(M) \to \cat{E}(M')$ is continuous as a functor $\cat{E}_t(M) \to \cat{E}_t(M')$.
\end{lem}
\begin{proof}
Recall that a hom-space in $\cat{E}_t(M)$ is of the form $(M \times S')^S$, and on each such hom-space, $E(h)$ is the map
\[
(h \times \id_{S'})^S \from (M \times S')^S \to (M' \times S')^S
\]
which is evidently continuous since $h \from M \to M'$ is.
\end{proof}

\begin{defn}
\label{defn:act-fun-top}
Let $h \from M \to M'$ be a continuous monoid homomorphism between small topological monoids. Write $E_t(h)$ for the functor $E(h) \from \cat{E}(M) \to \cat{E}(M')$, regarded as a $\TOP$-functor $\cat{E}_t(M) \to \cat{E}_t(M')$.
\end{defn}

\begin{defn}
Write $E_t \from \TopMonoid \to \proth_t(\finset^{\op})$ for the functor that sends $M$ to $E_t (M) \from \finset \to \cat{E}_t(M)$ as defined in Definition~\ref{defn:act-cat-top} and sends a continuous homomorphism $h \from M \to M'$ to the continuous functor $E_t(h) \from \cat{E}_t(M)  \to \cat{E}_t(M')$ defined in Definition~\ref{defn:act-fun-top}.
\end{defn}

\begin{lem}
\label{lem:mon-in-proth-ff-top}
The functor $E_t \from \TopMonoid \to \proth_t(\finset^{\op})$ is full and faithful.
\end{lem}
\begin{proof}
Since $E \from \monoid \to \proth(\finset^{\op})$ is full and faithful by Proposition~\ref{prop:mon-in-proth-ff} it is sufficient to show that, for small topological monoids $M, M'$, a monoid homomorphism $h \from M \to M'$ is continuous if and only if $E(h)$ is continuous as a functor $\cat{E}_t(M) \to \cat{E}_t(M')$.

Given $S, S' \in \finset$, the action of $E(h)$ on the hom-space $\cat{E}_t(M)(S, S')$ is the map
\[
E(h) = (h \times \id_{S'})^S \from \cat{E}_t(M)(S, S') = (M\times S')^S \to \cat{E}_t(M)(S, S') = (M'\times S')^S
\]
which is continuous if $h$ is.

Conversely, if $E(h)$ is continuous, then in particular its action on the hom-space $\cat{E}_t(M) (1,1)$ is. But this is given by
\[
h \from M \iso (M \times 1)^1 \to M' \iso (M' \times 1)^1,
\]
and so $h$ is continuous.
\end{proof}
 
\begin{lem}
The composite
\[
\TopMonoid^{\op} \toby{E_t} \proth_t(\finset^{\op}) \toby{\sem_t} \catover{\finset}
\]
sends a topological monoid $M$ to the category of finite, continuous $M$-sets, with its forgetful functor to $\finset$.
\end{lem}
By a finite continuous $M$-set, we mean a finite set $X$ together with an action $M \times X \to X$ that is continuous when $X$ is regarded as a discrete space and $M \times X$ is given the usual product topology.  
\begin{proof}
Recall from Lemma~\ref{lem:act-semantics} that for a small monoid $M$, a model of $E(M)$ can be identified with a finite $M$-set, and a homomorphism of $E(M)$-models is a $M$-equivariant map. We must show that if $M$ is a topological monoid then a model is continuous if and only if the corresponding $M$-set is.

Recall that an $E (M)$-model structure on a finite set $X$ consists of a natural transformation 
\[
\alpha \from \cat{E}(M)(E (M) - , X) \iso (M \times X)^{(-)} \to \finset (-, X) \iso X^{(-)},
\]
satisfying the conditions of Definition~\ref{defn:L-alg}, which by the Yoneda lemma, is of the form $a^{(-)} \from  (M \times X)^{(-)} \to  X^{(-)}$ for a unique $a \from M \times X \to X$, and this $a$ defines the corresponding action of $M$ on $X$. Clearly if $a$ is continuous, then so is 
\[
\alpha_S = a^S \from (M \times X)^S \to X^S
\]
for each $S$. On the other hand, if $\alpha_S$ is continuous for each $S$ (that is, if we have a continuous model structure) then by taking $S = 1$, we see that
\[
a \from M \times X \iso (M \times X)^1 \toby{\alpha_1} X^1 \iso X
\]
is continuous. Hence the $E(M)$-model structure on $X$ defines a continuous $E_t(M)$-model structure if and only if the corresponding $M$-action is continuous.
\end{proof}

Let $L \from \finset \to \cat{L}$ be a topological proto-theory on $\finset$, and suppose the underlying discrete proto-theory of $L$ is in the essential image of $E \from \monoid \to \proth(\finset^{\op})$, that is, it satisfies the conditions of Proposition~\ref{prop:act-image}. Then $L$ preserves finite coproducts, and so we have a bijection
\begin{equation}
\label{eq:top-proto-from-mon-coprod-bij}
\cat{L}(LS, LS') \iso \cat{L}(L1, LS')^S
\end{equation}
for all finite sets $S$ and $S'$. Similarly, by Proposition~\ref{prop:act-image}.\bref{part:act-image-unary}, we have a bijection
\begin{equation}
\label{eq:top-proto-from-mon-other-bij}
S \times \cat{L}(L1, L1) \iso \cat{L}(L1, LS)
\end{equation}
for all finite sets $S$.

\begin{lem}
\label{lem:act-top-image}

Let $L \from \finset \to \cat{L}$ be topological proto-theory. Then $L$ is in the essential image of $E_t \from \TopMonoid \to \proth_t(\finset^{\op})$ if and only if its underlying proto-theory is in the essential image of $E \from \monoid \to \proth(\finset^{\op})$ (that is, it satisfies the conditions of Proposition~\ref{prop:act-image}), and in addition, the bijections~\bref{eq:top-proto-from-mon-coprod-bij} and \bref{eq:top-proto-from-mon-other-bij} above are in fact homeomorphisms. When these conditions hold, $L$ is isomorphic to $E_t (M)$, where $M$ is the topological monoid $\cat{L}(L1, L1)$. 
\end{lem}

\begin{proof}
It is clear from the definitions that if a topological proto-theory is in the essential image of $E_t$, then its underlying discrete proto-theory is in the essential image of $E$. So we just need to check that a topological proto-theory $L \from \finset \to \cat{L}$ whose underlying theory is $E(M)$ for some monoid $M$ arises from a topology on $M$ if and only if these bijections are homeomorphisms. But these conditions imply that the topology on each hom-set is determined by that on $\cat{L}(L1, L1)$. Since the underlying monoid of $\cat{L}(L1, L1)$ is $M$, it is sufficient to check that if $M$ is equipped with a topological monoid structure then $E_t(M)$ does satisfy the above conditions. But this is clear since by definition,
\[
\cat{E}_t (M)(S, S') \iso (M \times S' )^S \iso (\cat{E}_t(M)(1, 1) \times S')^S.
\]
\end{proof}

\begin{prop}
\label{prop:top-monoid-restrict}
There is a functor $T \from \TopMonoid \to \TopMonoid$ that is unique up to isomorphism such that
\[
\xymatrix{
\TopMonoid\ar[r]^-{E_t}\ar[dd]_T & \proth_t(\finset^{\op})\ar[d]^{\sem_t^{\op}} \\
& (\catover{\finset})^{\op}\ar[d]^{\str_t^{\op}} \\
\TopMonoid\ar[r]_-{E_t} & \proth_t(\finset^{\op})
}
\]
commutes up to isomorphism.
\end{prop}
\begin{proof}
Similarly to Proposition~\ref{prop:monoid-restrict}, it is sufficient to check that for $M \in \TopMonoid$ the topological proto-theory $\str_t \of \sem_t \of E_t (M)$ lies in the essential image of $E_t$. Since $E_t$ is full and faithful by Lemma~\ref{lem:mon-in-proth-ff-top}, the existence of such a $T$ is then guaranteed, and uniqueness follows from the fact that full and faithful functors reflect isomorphisms.

Write $U\from \cat{M} \to \finset$ for the forgetful functor from the category of finite, continuous $M$-sets to $\finset$. Then, for finite sets $S$ and $S'$, the hom-space
\[
\thr_t (\sem_t(E_t(M))) (S, S')
\]
can be identified with the set $[\cat{M}, \finset](U^S, U^{S'})$ of natural transformations $U^{S} \to U^{S'}$, with the smallest topology such that for any finite continuous $M$-set $X$ and $f \in X^S$, the map sending a natural transformation $\gamma \from U^{S} \to U^{S'}$ to $\gamma_X (f)$ is continuous.

We show that $\str_t(\sem_t(E_t(M)))$ satisfies the conditions of Lemma~\ref{lem:act-top-image}. The proof that the underlying discrete proto-theory is in the essential image of $E \from \monoid \to \proth(\finset^{\op})$ is exactly the same as in Proposition~\ref{prop:monoid-restrict}, noting that if the $M$-set $X$ from the proof of Proposition~\ref{prop:monoid-restrict} is assumed to be continuous, then all the $M$-sets constructed subsequently are easily seen to be continuous.

Now we show that $\str_t(\sem_t(E_t(M)))$ satisfies the additional conditions from Lemma~\ref{lem:act-top-image}.

First we must show that, for finite sets $S$ and $S'$, the bijection
\[
[\cat{M}, \finset] (U^S, U^{S'}) \iso [\cat{M}, \finset](U^S, U)^{S'}
\]
is a homeomorphism. First we check it is continuous in the forwards direction. It is sufficient to check that it is continuous when composed with the projection $\pi_{s'} \from [\cat{M}, \finset](U^S, U)^{S'} \to [\cat{M}, \finset](U^S, U)$ for $s' \in S'$. But this composite can be identified with
\[
(\pi_{s'})_* \from[\cat{M}, \finset] (U^S, U^{S'}) \to [\cat{M}, \finset] (U^S, U)
\]
which is continuous, since $\thr_t(\sem_t(E_t(M)))$ is a $\TOP$-category. In the backwards direction, it is sufficient to show that the bijection is continuous when composed with the map
\[
\ev_f \from [\cat{M}, \finset] (U^S, U^{S'}) \to U^{S'}
\]
sending $\gamma$ to $\gamma_X (f)$ for each finite, continuous $M$-set $X$ and map $f \from S \to X$. But this composite is
\[
(\ev_f)^{S'} \from [\cat{M}, \finset](U^S, U)^{S'} \to U^{S'}
\]
which is continuous.

Now we must show that for each finite set $S$, the bijection
\[
S \times [\cat{M}, \finset](U, U) \to [\cat{M}m \finset] (U^S, U)
\]
that sends $(s, \gamma) $ to $\gamma \of \pi_s$ is a homeomorphism. It is continuous in the forwards direction, since it can be written as the composite
\[
\xymatrix{
S \times [\cat{M}, \finset](U, U) \iso \finset (1, S) \times [\cat{M}, \finset](U, U) \ar[d]_{\str_t (U) \times \id} \\
[\cat{M}, \finset](U^S, U) \times [\cat{M}, \finset](U, U) \ar[d] \\
[\cat{M}, \finset](U^S , U),
}
\]
where the last map is composition, and each of these maps is continuous. Now note that $[\cat{M}, \finset](U, U)$ is a profinite space, as a limit of finite spaces, and in particular is compact and Hausdorff. Thus $S \times [\cat{M}, \finset](U, U)$ is also compact and Hausdorff, as a finite copower of such spaces. Similarly $ [\cat{M}, \finset](U^S, U)$ is compact and Hausdorff. But any continuous bijection between compact Hausdorff spaces, and in particular the above map, is a homeomorphism. 
\end{proof}

\begin{cor}
\label{cor:str-sem-monad-restrict-monoid-top}
The topological structure--semantics monad on $\proth_t(\finset^{\op})$ restricts to a monad $(T, \eta^{\mnd{T}}, \mu^{\mnd{T}})$ on the full subcategory $\TopMonoid$ of small topological monoids.
\end{cor}
\begin{proof}
This is immediate from the previous proposition, since $E_t \from \TopMonoid \incl \proth_t(\finset^{\op})$ is full and faithful (Lemma~\ref{lem:mon-in-proth-ff-top}).
\end{proof}

Recall that, for a topological group $G$, the profinite completion of $G$ is the topological group $\pc{G}$ defined as follows. The elements of $\pc{G}$ are families $\xi = (\xi_h \in H)_h$ indexed by discrete finite groups $H$ and continuous group homomorphisms $h \from G \to H$, such that for any homomorphism $k \from H \to H'$ between finite groups, we have
\[
k (\xi_h ) = \xi_{k \of h}.
\]
The topology on $\pc{G}$ is the smallest topology such that each map
\begin{align*}
\ev_h \from & \pc{G} \to H \\
& \xi \mapsto \xi_h
\end{align*}
is continuous, where $H$ is finite and $h \from G \to H $ is continuous.

\begin{prop}
\label{prop:top-act-prof}
The monad $(T, \eta^{\mnd{T}}, \mu^{\mnd{T}})$ on $\TopMonoid$ from Corollary~\ref{cor:str-sem-monad-restrict-monoid-top} restricts to a monad on the full subcategory $\TopGp \incl \TopMonoid$ that is isomorphic to the profinite completion monad $(\widehat{\ }, \eta, \mu)$.
\end{prop}
\begin{proof}
The proof is essentially identical to that of Proposition~\ref{prop:group-restrict-procomp} but with the word ``continuous'' inserted as appropriate. We shall construct an isomorphism $\Phi \from \procomp{G} \to TG = [\Gfinset, \finset ] (U, U)$, where $G$ is a topological group and $U \from \cat{M} \to \finset$ is the forgetful functor from the category of finite continuous $G$-sets. We omit the proofs that this isomorphism is natural, and that it is compatible with the monad structure.

A finite continuous $G$-set $X$ with underlying set $X_0$ is determined by a continuous group homomorphism $\rho_X \from G \to \sym (X_0)$, where $\sym(X_0)$ is the group of automorphisms of $X_0$ in $\finset$. Since $\sym(X_0)$ is a finite group, we may define $\Phi (\xi) _X = \xi_{\rho_X} \from X_0 \to X_0$. We must check that $\Phi (\xi)$ thus defined actually is a natural transformation $U \to U$, that is, for any $G$-set homomorphism $k \from X \to Y$ between finite continuous $G$-sets, that
\[
\xymatrix{
X \ar[r]^k \ar[d]_{\Phi(\xi)_X} & Y\ar[d]^{\Phi(\xi)_Y} \\
X \ar[r]_k & Y
}
\]
commutes. Note that this square commutes for each $\xi \in \procomp{G}$ if and only if the right-hand diamond in
\[
\xymatrix{
& & {\sym(X_0)}\ar[dr]^{k_*} & \\
{G}\ar[urr]^{\rho_X}\ar[r]\ar[drr]_{\rho_Y} & {\procomp{G}}\ar[ur]_{\Phi(-)_X}\ar[dr]^{\Phi(-)_Y} & & {\finset(X_0,Y_0)} \\
& & {\sym(Y_0)}\ar[ur]_{k^*} &
}
\]
commutes. But the outer diamond in this diagram commutes, since $k$ is a $G$-set homomorphism. Also the two left-hand triangles commute by definition of $\Phi$. All the maps in this diagram are continuous, where $\sym X_0$, $\sym Y_0$ and $\finset(X_0, Y_0)$ are discrete. Recall that if a parallel pair of continuous maps have Hausdorff codomain and agree on a dense subset of their domain, then they are equal. The canonical map $G \to \procomp{G}$ has dense image; this is shown as part of the proof of Theorem~3.1 of Deleanu~\cite{deleanu83}. Since $\finset(X_0,Y_0)$ is Hausdorff (since it is discrete), it therefore follows that the right-hand diamond does indeed commute.

Hence each $\Phi (\xi)$ is indeed a natural transformation $U \to U$. In addition, 
\[
\Phi \from \procomp{G} \to [\Gfinset, \finset](U,U)
\]
 is a monoid homomorphism since each $\Phi(-)_X$ is by construction. It is also continuous: for any finite, continuous $G$-set $X$ and $x \in X$, the composite
 \[
 \procomp{G} \toby{\Phi} [\Gfinset, \finset](U,U) \toby{\ev_x} X
 \]
 sends $\xi$ to $\Phi (\xi)_X (x) = \xi_{\rho_X} (x)$, so it can also be written as the composite
 \[
 \pc{G} \toby{\ev_{\rho_X}} \sym (X_0) \toby{\ev_x} X,
 \]
and both of these maps are continuous. Both $\pc{G}$ and $[\Gfinset, \finset](U,U)$ are profinite and in particular compact Hausdorff, so if we can show that $\Phi$ is a bijection it will follow from continuity that it is a homeomorphism. We construct an inverse $\Xi$ to $\Phi$.

Let $\gamma \from U \to U$; we wish to construct an element $\Xi(\gamma) \in \procomp{G}$. Given a finite group $H$ and a continuous group homomorphism $h \from G \to H$, we obtain a continuous $G$-set $H_h$ with underlying set $H$, and with $g \in G$ acting by multiplication on the left by $h(g)$. Thus we have $\gamma_{H_h} \from H \to H$. Define $\Xi (\gamma)_h = \gamma_{H_h}(e_H)$, where $e_H$ denotes the group identity of $H$.

We must check that $\Xi (\gamma)$ so defined is indeed an element of $\procomp{G}$, that is, that if $k \from H \to H'$ is a homomorphism between finite groups, that $\Xi (\gamma)_{k \of h} = k ( \Xi(\gamma)_{h})$. But such a group homomorphism $k$ is also a $G$-set-homomorphism $H_h \to H'_{k \of h}$, so
\[
\xymatrix{
H \ar[r]^k \ar[d]_{\gamma_{H_h}} & H' \ar[d]^{\gamma_{H'_{k \of h}}} \\
H \ar[r]_k & H'
}
\]
commutes. Thus,
\[
\Xi(\gamma)_{k \of h} = \gamma_{H'_{k \of h}} (e_{H'}) = \gamma_{H'_{k \of h}} ( k ( e_H)) = k ( \gamma_{H_h}(e_H)) = k ( \Xi(\gamma)_h)
\]
as required.

Now we show that $\Xi$ is inverse to $\Phi$. Let $\xi \in \procomp{G}$. Then for any finite continuous $G$-set $X$ and $x \in X$, we have $\Phi (\xi)_X (X) = \xi_{\rho_X} (x)$. Thus for continuous $h \from G \to H$ with $H$ finite,
\[
\Xi \Phi (\xi)_h = \Phi (\xi)_{H_h} (e_H) = \xi_{\rho_{H_f}} (e_H).
\]
So we need to show that $\xi_{\rho_{H_h}} (e_H) = \xi_h$. Define $i \from H \to \sym (H_0)$ by sending $m \in H$ to left multiplication by $m$, where $H_0$ is the underlying set of $H$. Then $i$ is a group homomorphism and
\[
\xymatrix{
G \ar[r]^h \ar[dr]_{\rho_{H_h}} & H\ar[d]^i \\
& \sym H_0
}
\]
commutes. Thus we have
\[
\xi_{\rho_{H_h}} (e_H) = \xi_{i \of h} (e_H) = i (\xi_h) (e_H) =  \xi_h
\]
as required. So $\Xi \of \Phi = \id_{\procomp{G}}$.

Now Let $\gamma \from U \to U$. For any finite continuous $G$-set $X$ and $x \in X$, we have
\[
\Phi \Xi (\gamma)_X (x) = \Xi(\gamma)_{\rho_X} (x) = \gamma_{(\sym X_0)_{\rho_X}} (\id_X)(x).
\]
Note that we have a $G$-set homomorphism $\ev_x \from (\sym X_0)_{\rho_X} \to X$, since, given $\sigma \in \sym X_0$ and $g \in G$,
\[
\ev_x (g \cdot \sigma) = \rho_X (g) \of \sigma (x) = g \cdot \sigma (x) = g \cdot \ev_x (\sigma).
\]
Hence
\[
\xymatrix{
{\sym X_0} \ar[r]^{\ev_x}\ar[d]_{\gamma_{(\sym X_0)_{\rho_X}}} & X\ar[d]^{\gamma_X} \\
{\sym X_0} \ar[r]_{\ev_x} & X
}
\]
commutes, and so
\begin{align*}
\gamma_{(\sym X_0)_{\rho_X}} (\id_X)(x) &= \ev_x \of \gamma_{(\sym X_0)_{\rho_X}} (\id_X) \\
&= \gamma_X \of \ev_x (\id_X) \\
&= \gamma_X (x).
\end{align*}
Hence $\Phi \Xi (\gamma) = \gamma$, and we have shown that $\Phi$ and $\Xi$ are inverses. Hence we have a topological monoid isomorphism
\[
\procomp{G} \iso [\Gfinset, \finset](U,U),
\]
as claimed.
\end{proof}

\begin{cor}
The topological structure--semantics monad on $\proth_t(\finset^{\op})$ restricts to the profinite completion monad on the full subcategory $\TopGp \incl \proth_t(\finset^{\op})$.
\end{cor}
\begin{proof}
This is immediate from Corollary~\ref{cor:str-sem-monad-restrict-monoid-top} and Proposition~\ref{prop:top-act-prof}.
\end{proof}
In Table~\ref{tab:analogy-top} we extend Table~\ref{tab:analogy} to include the topological aspects of the analogy between groups and proto-theories that have been developed in this chapter.
\begin{table}[h]
\centering
\begin{tabular}{p{7cm}p{7cm}}
\toprule
Theory of groups & Theory of proto-theories \\
\midrule
$\Gp$ & $\proth(\cat{B}^{\op})$ \\
$\FinGp \incl \Gp$ & $\kle \from \monad(\cat{B}) \incl  \proth(\cat{B}^{\op})$ \\
The functor $\Gp^{\op} \to \catover{\finset}$ sending $G$ to the category $\Gfinset$ of finite $G$-sets & The functor $\sem \from \proth(\cat{B}^{\op})^{\op} \to \catover{\cat{B}}$ for the canonical aritation \\
The profinite completion monad on $\Gp$ & The structure--semantics monad on $\proth(\cat{B}^{\op})$\\
The profinite completion monad is the codensity monad of $\FinGp \incl \Gp$. & The structure--semantics monad is the codensity monad of $\kle \from \monad(\cat{B}) \incl \proth(\cat{B}^{\op})$. \\
$\TopGp$ & $\proth_t(\cat{B}^{\op})$ \\
The subcategory $\profGp \incl \TopGp$ of profinite groups & $\proth_{ct}(\cat{B}^{\op}) \incl \proth_t(\cat{B}^{\op})$ \\
The functor $\TopGp^{\op} \to \catover{\finset}$ sending a topological group $G$ to the category of finite continuous $G$-sets & The functor $\sem_t \from \proth_t(\cat{B}^{\op})^{\op} \to \catover{\cat{B}}$ \\
The profinite completion monad on $\TopGp$ & The topological structure--semantics monad on $\proth_t(\cat{B}^{\op})$ \\
The profinite completion monad on $\TopGp$ is idempotent, corresponding to the reflective subcategory $\profGp$ & The topological structure--semantics monad on $\proth_t(\cat{B}^{\op})$ is idempotent, corresponding to the reflective subcategory $\proth_{ct}(\cat{B}^{\op})$ \\
\bottomrule

\end{tabular}
\caption{Some further aspects of the analogy between groups and proto-theories, for a locally small category $\cat{B}$ with enough subobjects and pointwise codensity monads of finite diagrams.}
\label{tab:analogy-top}
\end{table}

\chapter{Complete topological proto-theories}
\label{chap:complete}

In this chapter we explore the properties of the category $\proth_{ct}(\cat{B}^{\op})$ of complete topological proto-theories with arities $\cat{B}^{\op}$, for a suitable category $\cat{B}$. Recall from Definition~\ref{defn:top-adj-factor} that we can decompose the topological structure--semantics adjunction
 \[
\xymatrix{
{\catover{\cat{B} }}\ar@<5pt>[r]_-{\perp}^-{\str_t}\ & {\proth_t(\cat{B}^{\op})^{\op}}\ar@<5pt>[l]^-{\sem_t}
}
\]
as a pair of adjunctions
\[
\xymatrix{
{\catover{\cat{B} }}\ar@<5pt>[r]_-{\perp}^-{\str_{ct}}
& {\proth_{ct}(\cat{B}^{\op})^{\op}}\ar@<5pt>[l]^-{\sem_{ct}}\ar@<5pt>[r]_-{\perp}^-{\inc^{\op}} & {\proth_t(\cat{B}^{\op})^{\op}} \ar@<5pt>[l]^-{\cplt^{\op}}.
}
\]
where $\sem_{ct}$ and $\inc$ are full and faithful.

 In Section~\ref{sec:top-str-sem-codensity} we show that the structure--semantics monad on $\proth_t(\cat{B}^{\op})$  arises as the codensity monad of the full subcategory of monads. In Section~\ref{sec:lims-colims-in-complete-proth} we turn to the question of what limits and colimits exist in $\proth_{ct}(\cat{B}^{\op})$ and how they relate to those in $\monad(\cat{B})$ and $\proth_t(\cat{B}^{\op})$. Finally in Section~\ref{sec:quationally-presentable} we see how an important class of categories that are not described by monads but are in a sense algebraic can be described by complete topological proto-theories.

\section{The topological structure--semantics monad as a codensity monad}
\label{sec:top-str-sem-codensity}

In this section we will show that when a locally small category $\cat{B}$ has pointwise codensity monads of finite diagrams and enough subobjects, the topological structure--semantics monad on $\proth_t(\cat{B}^{\op})$ arises as the codensity monad of the inclusion $\kle_t \from \monad(\cat{B}) \incl \proth_t(\cat{B}^{\op})$. We could presumably show this by an argument almost identical to that of Proposition~\ref{prop:str-sem-mnd-codensity} (and indeed, such a proof would work even if $\cat{B}$ did not have enough subobjects). However, when $\cat{B}$ has enough subobjects an alternative, shorter proof is available, which proves some other interesting and useful results along the way.

Throughout this section we fix a locally small category $\cat{B}$ with finite products and enough subobjects.

\begin{prop}
\label{prop:kle-top-factor}
The functor $\kle_t = \disc \of \kle \from \monad(\cat{B}) \to \proth_t(\cat{B}^{\op})$ factors through $\inc \from \proth_{ct}(\cat{B}^{\op}) \incl \proth_t(\cat{B}^{\op})$. 
\end{prop}
\begin{proof}
Recall that $\proth_{ct}(\cat{B}^{\op})$ is defined as the full subcategory of $\proth_{t}(\cat{B}^{\op})$ consisting of those topological proto-theories $L \from \cat{B}^{\op} \to \cat{L}$ such that
\[
E_L \from \cat{L} \to \str_t (\sem_t (L))
\]
is an isomorphism. Recall from Proposition~\ref{prop:sem-mnd-ff} that for any monad $\mnd{T}$ on $\cat{B}$, we have $\mnd{T} \iso \str_{\monad} (\sem_{\monad}(\mnd{T}))$, that is, the monad generated by the Eilenberg--Moore adjunction for $\mnd{T}$ is isomorphic to $\mnd{T}$. By Proposition~\ref{prop:top-monad-adj-restrict}, this implies that
\[
\kle_t (\mnd{T}) \iso \kle_t (\str_{\monad} (\sem_{\monad}(\mnd{T}))) \iso \str_t (\sem_t (\kle_t (\mnd{T}))),
\]
and this isomorphism is indeed given by $E_{\kle_t(\mnd{T})}$.
\end{proof}

\begin{defn}
Write $\kle_{ct} \from \monad(\cat{B}) \to \proth_{ct}(\cat{B}^{\op})$ for the factorisation of $\kle_t \from \monad(\cat{B}) \to \proth_t(\cat{B}^{\op})$ through $\inc \from \proth_{ct}(\cat{B}^{\op}) \incl \proth_t(\cat{B}^{\op})$.
\end{defn}

\begin{prop}
\label{prop:dense-slice}
Let $I \from \cat{D} \incl \cat{C}$ be a dense and full and faithful functor, and let $b \in \cat{C}$. Then the induced inclusion
\[
J \from (I\downarrow b) \to \cat{C}/b
\]
is also dense.
\end{prop}
\begin{proof}

We must show that, for every $(f \from c \to b) \in \cat{C}/b$, the canonical cocone on
\[
(J \downarrow f) \to (I\downarrow b) \toby{J} \cat{C}/b
\]
with vertex $f$ is a colimit cocone. Consider the commutative diagram
\[
\xymatrix{
(J \downarrow f)\ar[r]\ar[d] & (I\downarrow b) \ar[r]^J\ar[d] & \cat{C}/b\ar[d] \\
(I \downarrow c) \ar[r] & \cat{D} \ar[r]_I & \cat{C}
}
\]
where the vertical arrows are the evident forgetful functors. The left-most vertical arrow is in fact an isomorphism: an object of $(J \downarrow f)$ consists of an object $d \in \cat{D}$, and morphisms $f' \from Id \to b$ and $g \from Id \to c$ such that $f \of g = f'$, but such an object is determined by $g \from Id \to c$ alone, which is an object of $(I \downarrow c)$. Thus we can identify cocones on diagrams of shape $(J \downarrow f)$ with cocones on diagrams of shape $(I \downarrow c)$.

The forgetful functor $\cat{C}/b \to \cat{C}$ sends the canonical cocone on
\[
(J \downarrow f) \to (I\downarrow b) \toby{J} \cat{C}/b
\]
with vertex $f$ to the canonical cocone on
\[
(I \downarrow c) \to \cat{D} \toby{I} \cat{C}
\]
with vertex $c$. But this latter cocone is a colimit cocone since $I$ is dense. Furthermore, the forgetful functor $\cat{C}/b \to \cat{C}$ creates all colimits (Lemma in Section V.6 of \cite{maclane71}) so the former cocone is also a colimit cocone. Hence $J$ is dense, as claimed.
\end{proof}

\begin{prop}
\label{prop:dense-reflect}
Let $I \from \cat{D} \incl \cat{C}$ be a dense and full and faithful functor, and let $G \from \cat{E} \incl \cat{C}$ be a full and faithful functor with left adjoint $F$. Then the full image of $FI$ is dense in $\cat{E}$.
\end{prop}
\begin{proof}
Write $J \from \cat{F} \incl \cat{E}$ for the full image of $FI$. Let $\alpha \from \cat{E}(J-, e) \to \cat{E}(J-,e')$. We define a natural transformation $\beta \from \cat{C}(I-, Ge) \to \cat{C}(I-, Ge')$ as follows. Let $d \in \cat{D}$ and $g \from Id \to Ge$. This corresponds under the adjunction $F \ladj G$ to $\bar{g} \from FId \to e$ in $\cat{E}$, and since $FId \in \cat{F}$, we have $\alpha_{FId} (\bar{g}) \from FId \to e'$. Define $\beta_d (g)$ to be the composite
\[
Id \toby{\eta_{Id}} GFId \toby{G\alpha_{FId}(\bar{g})} Ge',
\]
where $\eta \from \id_{\cat{C}} \to GF$ is the unit of the adjunction.

We now show that $\beta$ is natural. Let $f \from d' \to d$ in $\cat{D}$. Then
\[
\alpha_{FId'}(\widebar{g \of If}) = \alpha_{FId'}(\bar{g} \of FIf) = \alpha_{FId}(\bar{g}) \of FIf,
\]
and so the right-hand triangle in
\[
\xymatrix{
{Id'} \ar[r]^{\eta_{Id'}} \ar[d]_{If} & {GFId'}\ar[dr]^{G\alpha_{FId'}(\widebar{g \of If})} \ar[d]_{GFIf} & \\
{Id}\ar[r]_{\eta_{Id}} & {GFId}\ar[r]_{G\alpha_{FId}(\bar{g})} & Ge'
}
\]
commutes, and the left-hand square commutes by naturality of $\eta$. It follows that $\beta_{d'}(g \of If) = \beta_{d}(g) \of If$, so $\beta$ is a natural transformation $ \cat{C}(I-, Ge) \to \cat{C}(I-, Ge')$.

Since $I$ is dense, there exists a unique $h' \from Ge \to G'e$ such that $\beta = h'_*$. Since $G$ is full and faithful, there is a unique $h \from e \to e'$ such that $h' = Gh$. We will show that $\alpha = h_*$. Any object of $\cat{F}$ is of the form $FId$ for some $d \in \cat{D}$. Let $f \from FId \to e$. We wish to show that $h \of f = \alpha_{FId}(f)$. This equality holds if and only if $G(h \of f) \of \eta_{Id} = G \alpha_{FId}(f) \of \eta_{Id}$. But if we set $g = Gf \of \eta_{Id}$, then $f = \bar{g}$, and so
\[
Gh \of Gf \of \eta_{Id} = h' \of g = \beta_d(g) = G\alpha_{FId}(\bar{g}) \of \eta_{Id} = G \alpha_{FId}(f)\of \eta_{Id}
\]
as required.

Finally we show that $h \from e \to e'$ is \emph{unique} such that $\alpha = h_*$. Let $k \from e \to e'$, and suppose $k$ has the property that for all $d \in \cat{D}$ and $f \from FId \to e$, we have $k \of f = h \of f$. Then for any $d \in \cat{D}$ and  $g \from Id \to Ge$, we have
\[
Gh \of g = Gh \of G \bar{g} \of \eta_{Id} = Gk \of G\bar{g} \of \eta_{Id} = Gk \of g,
\]
where we have taken $f = \bar{g}$. Since $I$ is dense, it follows that $Gh = Gk$, so $h= k$, since $G$ is faithful.

Thus $J$ is dense, as required.
\end{proof}

\begin{prop}
\label{prop:monad-ctheory-codense}
If $\cat{B}$ admits pointwise codensity monads of finite diagrams, then the inclusion $\kle_{ct} \from \monad(\cat{B}) \incl \proth_{ct}(\cat{B}^{\op})$ is codense.
\end{prop}
\begin{proof}
The full subcategory $\Delta$ of non-zero finite ordinals is dense in $\CAT$; this is well-known. Hence, by Proposition~\ref{prop:dense-slice}, the induced inclusion
\[
\Delta/\cat{B} \incl \catover{\cat{B}}
\]
is also dense. Now we apply Proposition~\ref{prop:dense-reflect}, with $\cat{C} = \catover{\cat{B}}$, and $\cat{D} = \Delta / \cat{B}$, and $I$ the inclusion, and with $\cat{E} = \proth_{ct}(\cat{B}^{\op})^{\op}$, $G = \sem_{ct}$ and $F = \str_{ct}$. It follows that the image of $\Delta /\cat{B}$ under $\str_{ct}$ in $\proth_{ct}(\cat{B}^{\op})^{\op}$ is dense, so its opposite is codense in $\proth_{ct}(\cat{B}^{\op})$. But $\cat{B}$ admits pointwise codensity monads of finite diagrams and all of the categories in $\Delta \subseteq \CAT$ are finite. Thus every functor into $\cat{B}$ with domain in $\Delta$ has a pointwise codensity monad, and so the image of $\Delta/\cat{B}$ is contained in the essential image of $\kle_{ct} \from \monad(\cat{B}) \incl \proth_{ct}(\cat{B}^{\op})$ by Proposition~\ref{prop:codensity-str}. But a subcategory of $\proth_{ct}(\cat{B}^{\op})$ that contains a subcategory codense in $\proth_{ct}(\cat{B}^{\op})$ is itself codense in $\proth_{ct}(\cat{B}^{\op})$. Hence $\kle_{ct} \from \monad(\cat{B}) \incl \proth_{ct}(\cat{B}^{\op})$ is codense.
\end{proof}

\begin{thm}
\label{thm:str-sem-monad-codensity-top}
If $\cat{B}$ has pointwise codensity monads of finite diagrams, then the structure--semantics monad on $\proth_t(\cat{B}^{\op})$ is the codensity monad of $\kle_t \from \monad(\cat{B}) \to \proth_t(\cat{B}^{\op})$.
\end{thm}

\begin{proof}
Under these conditions, $\kle_{ct} \from \monad(\cat{B}) \to \proth_{ct}(\cat{B}^{\op})$ is codense by Proposition~\ref{prop:monad-ctheory-codense}, and so by Lemma~\ref{lem:codense-radj-monad}, the codensity monad of the composite
\[
\monad(\cat{B}) \toby{\kle_{ct}} \proth_{ct}(\cat{B}^{\op}) \toby{\inc} \proth_{t}(\cat{B}^{\op})
\]
is the monad induced by the adjunction $\cplt \ladj \inc$. But this composite is $\kle_t \from \monad(\cat{B}) \to \proth_{t}(\cat{B}^{\op})$ by definition, and this monad is the topological structure--semantics monad.
\end{proof}

\section{Limits and colimits in $\proth_{ct}(\cat{B}^{\op})$}
\label{sec:lims-colims-in-complete-proth}

In this section, we examine whether limits and colimits exist in $\proth_{ct}(\cat{B}^{\op})$, and whether the inclusion $\kle_{ct} \from \monad(\cat{B}) \incl \proth_{ct}(\cat{B}^{\op})$ preserves them; since $\proth_{ct}(\cat{B}^{\op})$ is supposed to be in some sense a ``convenient category of monads'', it is desirable for it to have limits and colimits and for them to coincide with those in $\monad(\cat{B})$ at least reasonably often. These results lead to two further characterisations of the subcategory $\proth_{ct}(\cat{B}^{\op}) \incl \proth_{t}(\cat{B}^{\op})$ (Theorems~\ref{thm:ctpts-are-lims-of-monads} and~\ref{thm:ctpts-reflexive-hull-monads}) which do not involve the structure--semantics adjunction, and which deepen the analogy between proto-theories and groups.

Throughout this section we fix a locally small category $\cat{B}$ with finite products and enough subobjects.

\begin{prop}
Both $\proth_t(\cat{B}^{\op})$ and $\proth_{ct}(\cat{B}^{\op})$ have all large limits and colimits.
\end{prop}
\begin{proof}
If a monoidal category $\cat{V}$ has large limits then so does $\cat{V}\hyph\CAT$, and they are easy to compute: given a diagram $D \from \cat{I} \to \cat{V}\hyph\CAT$, the set of objects of its limit is given by the limit
\[
\ob (\lim D) = \lim_{i \in \cat{I}} \ob (Di)
\]
in $\SET$, and, for $(x_i)_{i \in \cat{I}}$ and $(y_i)_{i \in \cat{I}}$ in $ \lim_{i \in \cat{I}} \ob (Di)$, the hom-object is given by the limit
\[
\lim D ( (x_i)_{i \in \cat{I}}, (y_i)_{i \in \cat{I}}) = \lim_{i \in \cat{I}} Di (x_i, y_i)
\]
in $\cat{V}$. If $\cat{V}$ has large colimits then so does $\cat{V}\hyph\CAT$, although they are less straightforward --- this is the main result (Corollary~2.14) of Wolff~\cite{wolff74}. In particular, since $\TOP$ has large limits and colimits, so does $\TOPCAT$. Thus by Proposition~\ref{prop:lim-colim-proth}, it follows that $\proth_t(\cat{B}^{\op})$ has large limits and colimits. Since $\proth_{ct}(\cat{B}^{\op})$ is a reflective subcategory of $\proth_t(\cat{B}^{\op})$, it does too.
\end{proof}

\begin{lem}
\label{lem:limits-top-proth-pointwise}
For all $b, b' \in \cat{B}$, the functor $\ev_{b, b'} \from \proth_t(\cat{B}^{\op}) \to \TOP$ that sends $L \from \cat{B}^{\op} \to \cat{L}$ to $\cat{L}(Lb, Lb')$ preserves large limits.
\end{lem}
\begin{proof}
Limits in $\proth_t(\cat{B}^{\op})$ are computed as follows: given $D \from \cat{I} \to \proth_t(\cat{B}^{\op})$, write $D_i \from \cat{B}^{\op} \to \cat{D}_i$ for $D(i)$. Then the $\cat{D}_i$ form a diagram in $\TOPCAT$ and the $D_i$ give a cone on this diagram with vertex $\cat{B}^{\op}$. Thus there is an induced $\TOP$-functor from $\cat{B}^{\op}$ to the limit of this diagram in $\TOPCAT$. The limit of $D$ is the bijective-on-objects part of this induced $\TOP$-functor.

Writing $L \from \cat{B}^{\op} \to \cat{L}$ for this limit, it follows from the general form of limits in $\TOPCAT$ that 
\[
\cat{L}(Lb, Lb') =  \lim_{i \in \cat{I}} \cat{D}_i (D_i b, D_i b')
\]
as claimed.
\end{proof}

The following result is well-known but I could not find it in the existing literature.
\begin{lem}
\label{lem:codense-ff-cocont}
Any codense, full and faithful functor $I \from \cat{A} \to \cat{D}$ preserves all colimits that exist in $\cat{A}$.
\end{lem}
\begin{proof}
Since $I$ is full and faithful, it follows that the Yoneda embedding $(\curryhi)^{\op} \from \cat{A} \to [\cat{A}, \SET]^{\op}$ is isomorphic to the composite
\[
\cat{A} \toby{I} \cat{D} \toby{N^G} [\cat{A}, \SET]^{\op}
\]
where $N^G$ sends $d \in \cat{D}$ to $\cat{D}(d, I-)$. Since $(\curryhi)^{\op}$ preserves all colimits, so does this composite. But codensity of $I$ means that $N^G$ is full and faithful, and in particular it reflects all colimits, and hence $I$ preserves all colimits.
\end{proof}

\begin{cor}
If $\cat{B}$ has pointwise codensity monads of finite diagrams, then the inclusion $\kle_t \from \monad(\cat{B}) \incl \proth_{ct}(\cat{B}^{\op})$ preserves all colimits that exist in $\monad(\cat{B})$.
\end{cor}
\begin{proof}
Under these conditions, $\kle_{ct} \from \monad(\cat{B}) \to \proth_{ct}(\cat{B}^{\op})$ is codense by Proposition~\ref{prop:monad-ctheory-codense}. And $\kle_{ct}$ is always full and faithful, so by Lemma~\ref{lem:codense-ff-cocont}, it preserves all colimits that exist in $\monad(\cat{B})$.
\end{proof}

\begin{remark}
In contrast to this corollary, and to Proposition~\ref{prop:mnd-incl-limits}, it is not necessarily the case that the inclusion $\kle_{ct}$ preserves large limits, or even all small limits. For example, let $\mnd{T}$ be any monad on $\cat{B}$. Then if $\cat{B}$ has small limits, the countable power $\mnd{T}^{\nat}$ of $\mnd{T}$ exists in $\monad(\cat{B})$. Also, the countable power $\kle_t(\mnd{T})^{\nat}$ exists in $\proth_t(\cat{B}^{\op})$, and by Lemma~\ref{lem:limits-top-proth-pointwise}, the hom-spaces in this proto-theory are obtained by taking the $\nat$-th power of the corresponding hom-spaces of $\kle_t(\mnd{T})$. Recall that $\kle_t(\mnd{T})$ and $\kle_t(\mnd{T}^{\nat})$ are both topologically discrete (since $\kle_t$ of any monad is), but an infinite power of a discrete space (with more that one element) is not discrete. Thus we cannot have $\kle_t(\mnd{T}^{\nat}) \iso \kle_t(\mnd{T})^{\nat}$.

However, we do have the following:
\end{remark}

\begin{prop}
The inclusion $\kle_{ct} \from \monad(\cat{B}) \incl \proth_{ct}(\cat{B}^{\op})$ preserves finite limits.
\end{prop}
\begin{proof}
It is sufficient to show that the inclusion $\kle_t \from \monad(\cat{B}) \incl \proth_t(\cat{B}^{\op})$ preserves finite limits, since limits in $\proth_{ct}(\cat{B}^{\op})$ are the same as in $\proth_{t}(\cat{B}^{\op})$.

We show that $\disc \from \proth(\cat{B}^{\op}) \to \proth_t(\cat{B}^{\op})$ preserves finite limits; this is sufficient, since $\kle_t$ is the composite
\[
\monad(\cat{B})\toby{\kle} \proth(\cat{B}^{\op}) \toby{\disc} \proth_t(\cat{B}^{\op})
\]
and $\kle$ preserves finite limits by Proposition~\ref{prop:mnd-incl-limits}.

By Lemma~\ref{lem:limits-top-proth-pointwise}, limits in $\proth_t(\cat{B}^{\op})$ are computed by taking limits in $\TOP$ on each hom-set individually, so to show that $\disc \from \proth(\cat{B}^{\op}) \to \proth_t(\cat{B}^{\op})$ preserves finite limits, it is sufficient to show that its composite with each $\ev_{b,b'} \from \proth_t(\cat{B}^{\op}) \to \TOP$ does. But the diagram
\[
\xymatrix{
\proth(\cat{B}^{\op})\ar[d]_{\ev_{b,b'}}\ar[r]^{\disc} &\proth_t(\cat{B}^{\op})\ar[d]^{\ev_{b, b'}} \\
\SET\ar[r]_{\disc} & \TOP
}
\]
commutes, and both $\disc \from \SET \to \TOP$ and $\ev_{b, b'} \from \proth(\cat{B}^{\op}) \to \SET$ preserve finite limits.
\end{proof}

It is now straightforward to give further characterisations of complete topological proto-theories in terms of monads, without mentioning the structure--semantics adjunction.

\begin{thm}
\label{thm:ctpts-are-lims-of-monads}
Suppose $\cat{B}$ has pointwise codensity monads of finite diagrams. Then a topological proto-theory with arities $\cat{B}^{\op}$ is complete (that is, it lies in the full subcategory $\proth_{ct}(\cat{B}^{\op})$ of $\proth_{t}(\cat{B}^{\op})$) if and only if it can be expressed as a large limit of monads in $\proth_t(\cat{B}^{\op})$.
\end{thm}
\begin{proof}
Since $\kle_{ct} \from \monad(\cat{B}) \incl \proth_{ct}(\cat{B}^{\op})$ is codense by Proposition~\ref{prop:monad-ctheory-codense}, every object of $\proth_{ct}(\cat{B}^{\op})$ is a limit of monads in $\proth_{ct}(\cat{B}^{\op})$ itself, and the inclusion into $\proth_t(\cat{B}^{\op})$ preserves limits.

Conversely, the inclusion $\proth_{ct}(\cat{B}^{\op})\incl \proth_{t}(\cat{B}^{\op})$ creates limits, that is, $\proth_{ct}(\cat{B}^{\op})$ is closed under limits in $\proth_{t}(\cat{B}^{\op})$. Since every monad is a complete topological proto-theory, it follows that every limit of monads is.
\end{proof}

\begin{thm}
\label{thm:ctpts-reflexive-hull-monads}
Suppose $\cat{B}$ has pointwise codensity monads of finite diagrams. Then $\proth_{ct}(\cat{B}^{\op})$ is the smallest replete reflective subcategory of $\proth_{t}(\cat{B}^{\op})$ containing $\monad(\cat{B})$.
\end{thm}
\begin{proof}
First note that $\proth_{ct}(\cat{B}^{\op})$ \emph{is} a replete reflective subcategory of $\proth_{t}(\cat{B}^{\op})$ containing the monads. Any replete reflective subcategory of $\proth_t(\cat{B}^{\op})$ is closed under limits in $\proth_t(\cat{B}^{\op})$, and so if such a subcategory contains all the monads, then it contains all limits of monads. Thus, by the previous theorem, it contains $\proth_{ct}(\cat{B}^{\op})$.
\end{proof}

Finally, we once again summarise in Table~\ref{tab:analogy-comp} the analogy between groups and proto-theories, extending Tables~\ref{tab:analogy} and~\ref{tab:analogy-top}, including the characterisations of complete topological proto-theories in terms of monads found in this chapter.
\begin{table}[h!] 
\centering
\begin{tabular}{p{7cm}p{7cm}}
\toprule
Theory of groups & Theory of proto-theories \\
\midrule
$\Gp$ & $\proth(\cat{B}^{\op})$ \\
$\FinGp \incl \Gp$ & $\kle \from \monad(\cat{B}) \incl  \proth(\cat{B}^{\op})$ \\
The functor $\Gp^{\op} \to \catover{\finset}$ sending $G$ to the category $\Gfinset$ of finite $G$-sets & The functor $\sem \from \proth(\cat{B}^{\op})^{\op} \to \catover{\cat{B}}$ for the canonical aritation \\
The profinite completion monad on $\Gp$ & The structure--semantics monad on $\proth(\cat{B}^{\op})$\\
The profinite completion monad is the codensity monad of $\FinGp \incl \Gp$. & The structure--semantics monad is the codensity monad of $\kle \from \monad(\cat{B}) \incl \proth(\cat{B}^{\op})$. \\
$\TopGp$ & $\proth_t(\cat{B}^{\op})$ \\
The subcategory $\profGp \incl \TopGp$ of profinite groups & $\proth_{ct}(\cat{B}^{\op}) \incl \proth_t(\cat{B}^{\op})$ \\
The functor $\TopGp^{\op} \to \catover{\finset}$ sending a topological group $G$ to the category of finite continuous $G$-sets & The functor $\sem_t \from \proth_t(\cat{B}^{\op})^{\op} \to \catover{\cat{B}}$ \\
The profinite completion monad on $\TopGp$ & The topological structure--semantics monad on $\proth_t(\cat{B}^{\op})$ \\
The profinite completion monad on $\TopGp$ is the codensity monad on $\FinGp \incl \TopGp$. & The topological structure--semantics monad is the codensity monad of $\kle_t \from \monad(\cat{B}) \incl \proth_t(\cat{B}^{\op})$.\\
The profinite completion monad on $\TopGp$ is idempotent, corresponding to the reflective subcategory $\profGp$. & The topological structure--semantics monad on $\proth_t(\cat{B}^{\op})$ is idempotent, corresponding to the reflective subcategory $\proth_{ct}(\cat{B}^{\op})$. \\
$\FinGp \incl \profGp$ is codense. & $\kle_{ct} \from \monad(\cat{B}) \incl \proth_t(\cat{B}^{\op})$ is codense. \\
$\profGp$ is the smallest replete reflective subcategory of $\TopGp$ containing $\FinGp$. & $\proth_{ct}(\cat{B}^{\op})$ is the smallest replete reflective subcategory of $\proth_t(\cat{B}^{\op})$ containing $\monad(\cat{B})$. \\
The profinite groups are precisely the small limits of discrete finite groups in $\TopGp$. & The complete topological proto-theories are precisely the large limits of discrete monads in $\proth_t(\cat{B}^{\op})$. \\
\bottomrule
\end{tabular}
\caption{Some further aspects of the analogy between groups and proto-theories, for a locally small category $\cat{B}$ with enough subobjects and pointwise codensity monads of finite diagrams.}
\label{tab:analogy-comp}
\end{table}

\section{Equationally presentable categories}
\label{sec:quationally-presentable}

In this section we examine the relationship between proto-theories with arities in $\Set$ and the notion of an \emph{equational presentation}, as defined by Manes in \cite{manes76}. Roughly speaking, equational presentations are to arbitrary monads on $\Set$ as classical finitary algebraic theories are to finitary monads. There is, however, an important point at which this analogy breaks down. Every finitary monad gives rise to a finitary algebraic theory and vice versa. Similarly, every monad on $\Set$ can be described by an equational presentation, however not every equational presentation gives rise to a monad. This is because an equational presentation may have a large set of operations of a given arity, meaning that the forgetful functor from its category of models to $\Set$ will not have a left adjoint and so cannot be monadic.

The following definitions are taken from~5.34 in Chapter~1 of~\cite{manes76}.

\begin{defn}
An \demph{operator domain} $\Omega$ consists of, for each small cardinal $n$, a set $\Omega_n$.
\end{defn}
\begin{defn}
\label{defn:operator-domain-struct}
Given an operator domain $\Omega$, an \demph{$\Omega$-model} $(X, \delta)$ consists of a small set $X$ together with, for each small cardinal $n$ and each $\omega \in \Omega_n$, a function $\delta_{\omega} \from X^n \to X$.

If $(X, \delta)$ and $(X', \delta')$ are $\Omega$-models, then an \demph{$\Omega$-model homomorphism} $(X, \delta) \to (X', \delta')$ is a function $h \from X \to X'$ such that for each $n$ and each $\omega \in \Omega_n$, the diagram
\[
\xymatrix{
X^n \ar[r]^{h^n}\ar[d]_{\delta_{\omega}} & X'^n\ar[d]^{\delta'_{\omega}} \\
X \ar[r]_h & X'
}
\]
commutes.

We write $\cat{M}_{\Omega}$ for the category of $\Omega$-models and $U_{\Omega} \from \mod(\Omega) \to \Set$ for the evident forgetful functor.
\end{defn}

\begin{defn}
Given an operator domain $\Omega$, and a small cardinal $n$, an \demph{$n$-ary operation of $\Omega$} is a natural transformation $\gamma \from U_{\Omega}^n \to U_{\Omega}$, where $U_{\Omega}^n$ denotes the $n$-th power of $U_{\Omega}$ defined pointwise. An \demph{$n$-ary $\Omega$-equation} is just a pair $(\gamma, \gamma')$ of $n$-ary operations of $\Omega$, and an \demph{$\Omega$-equation} is an $n$-ary equation for some $n$.
\end{defn}

\begin{defn}
An \demph{equational presentation} consists of an operator domain $\Omega$ together with a possibly large set $E$ of $\Omega$-equations. The category $\cat{M}_{(\Omega,E)}$ of \demph{models} of $(\Omega, E)$ is the full subcategory of $\cat{M}_{\Omega}$ consisting of those $\Omega$-models $(X, \delta)$ such that whenever $(\gamma, \gamma')$ is an $n$-ary equation in $E$, we have an equality of functions
\[
\gamma_{(X, \delta)} = \gamma'_{(X, \delta)} \from X^n \to X.
\]
We write $U_{(\Omega, E)} \from \cat{M}_{(\Omega, E)} \to \Set$ for the restriction of $U_{\cat{M}}$ to this subcategory.
\end{defn}

Note that in~\cite{manes76}, Manes uses the word ``algebra'' instead of ``model'' in the definitions above. We have made this change of terminology to avoid confusion with our use of ``algebra'' in Section~\ref{sec:alt-model-algebra}.

Manes shows in Theorem~5.40 in Chapter~1 of~\cite{manes76} that every monad on $\Set$ can equivalently be described by an equational presentation. However, not every functor of the form $U_{(\Omega, E)}$ for an equational presentation $(\Omega, E)$ is monadic: in~5.46 and~5.48 in Chapter~1 of \cite{manes76}, Manes gives equational presentations for the categories of small complete lattices and small complete Boolean algebras respectively, while showing that these categories are not monadic over $\Set$.

Thus monads are not sufficiently powerful to describe equationally presentable categories. We now show that they can be described in a natural way by complete topological proto-theories. For the rest of this section, fix an equational presentation $(\Omega, E)$.

\begin{defn}
Write $K$ for the set of all small cardinals.
\end{defn}

\begin{defn}
Define a functor $\sem_0 \from (\SET^K)^{\op}\to \catover{\Set}$ as follows. We may regard the objects of $\SET^K$ as operator domains. Given $\Omega \in \SET^K$ we define $\sem_0 (\Omega)$ to be 
\[
U_{\Omega} \from \cat{M}_\Omega \to \Set
\]
as defined in Definition~\ref{defn:operator-domain-struct}. Let $f$ be a morphism $\Omega \to \Omega'$ in $\SET^K$, that is, a family of maps $f_n \from \Omega_n \to \Omega'_n$. Given an $\Omega'$-model $(X, \delta)$, we define $\sem_0(f) (X,\delta)$ to be the $\Omega$-model with underlying set $X$, and for which the interpretation of $\omega \in \Omega_n$ is
\[
\delta_{f_n(\omega)} \from X^n \to X.
\]
We define $\sem_0(f)$ to be the identity on morphisms.
\end{defn}

\begin{defn}
Define a functor $\str_0 \from \catover{\Set} \to (\SET^K)^{\op}$ as follows. Given $U \from \cat{M} \to \Set$, let $\str_0(U)$ be the operator domain with
\[
\str_0(U)_n = [\cat{M}, \Set](U^n, U)
\]
for any small cardinal $n$. Given
\[
\xymatrix{
\cat{M}'\ar[rr]^Q \ar[dr]_{U'} &&\cat{M}\ar[dl]^{U}\\
& \Set &
}
\]
in $\catover{\Set}$, define $\str_0(Q) \from \str_0(U) \to \str_0(U')$ by sending $\gamma \from U^n \to U$ to $\gamma Q \from U'^n \to U'$.
\end{defn}

\begin{lem}
\label{lem:str-sem-adj-opdom}
We have an adjunction
\[
\xymatrix{
\catover{\Set }\ar@<5pt>[r]_-{\perp}^-{\str_0}\ & {(\SET^K)^{\op}.}\ar@<5pt>[l]^-{\sem_0}
}
\]
\end{lem}
\begin{proof}
Let $\Omega \in \SET^K$ and $U \from \cat{M} \to \Set$; we will establish a bijection
\[
\catover{\Set}(U, \sem_0(\Omega)) \iso \SET^K (\Omega, \str_0(U)).
\]
Let $f \from \Omega \to \str_0(U)$. Then $f$ consists of, for every $n \in K$ and $\omega \in \Omega_n$, a natural transformation $f_n(\omega) \from U^n \to U$. Equivalently, to specify such an $f$ is to specify, for each $n \in K$, $\omega \in \Omega_n$ and $m \in \cat{M}$, a function
\[
f_n(\omega)_m \from (Um)^n \to Um,
\]
such that for each $g \from m \to m'$ in $\cat{M}$, the diagram
\[
\xymatrix{
(Um)^n\ar[r]^{(Ug)^n}\ar[d]_{f_n(\omega)_m} & (Um')^n\ar[d]^{f_n(\omega)_{m'}} \\
Um \ar[r]_{Ug} & Um'
}
\]
commutes. But this is clearly the same as specifying, for each $m \in \cat{M}$, an $\Omega$-model structure (given by all the $f_n(\omega)_m$ as $n$ and $\omega$ vary) such that for each $g \from m \to m'$, the map $Ug$ becomes an $\Omega$-model homomorphism. But this is precisely a morphism
\[
U \to \sem_0(\Omega)
\]
in $\catover{\Set}$. Thus we have such a bijection; it is straightforward to check that this bijection is natural in $\Omega$ and $U$.
\end{proof}

\begin{defn}
Define a functor $S \from \proth(\Set^{\op}) \to \SET^K$, sending $L \from \Set^{\op} \to \cat{L}$ to the family $S(L) \in \SET^K$ with
\[
S(L)_n = \cat{L}(Ln, L1)
\]
for $n \in K$, and sending a morphism $P \from L' \to L$ to the map sending $l \in \cat{L}'(L'n, L'1)$ to $P(l) \in \cat{L}(Ln, L1)$.
\end{defn}

\begin{lem}
\label{lem:opdom-str-commute}
The diagram
\[
\xymatrix{
\catover{\Set}\ar[r]^{\str}\ar[dr]_{\str_0} & \proth(\Set^{\op})\ar[d]^S \\
& \SET^K
}
\]
commutes up to isomorphism.
\end{lem}
\begin{proof}
Recall that for $U \from \cat{M} \to \Set$ and small sets $X$ and $Y$, we have
\[
\thr(U)(X, Y) = [\cat{M}, \Set](\Set(Y, U-), \Set(X, U-))
\]
and in particular
\[
S(\str(U))_n = \thr(U)(n, 1) =  [\cat{M}, \Set](\Set(n, U-), \Set(1, U-)) \iso [\cat{M}, \Set](U^n, U) \iso \str_0(U).
\]
It is straightforward to check that the two functors agree on morphisms.
\end{proof}

\begin{lem}
\label{lem:proth-to-opdom-adj}
The functor $S \from \proth(\Set^{\op}) \to \SET^K$ has a left adjoint.
\end{lem}
\begin{proof}
Regarding $K$ as a discrete category (i.e.\ a category with only identity morphisms), define a functor $F \from K \to \Set^{\op} \times \Set$ sending $n \in K$ to $(n, 1)$. Then we can write $S$ as the composite
\[
\proth(\Set^{\op}) \to [\Set^{\op} \times \Set, \SET] \toby{F^*} [K, \SET] \iso \SET^K,
\]
where the first factor sends $L \from \Set^{\op} \to \cat{L}$ to $\cat{L}(L-,L-) \from \Set^{\op} \times \Set \to \SET$. This first factor is monadic by Proposition~\ref{prop:proth-prof-monadic}, and in particular has a left adjoint, and $F^*$ has a left adjoint given by left Kan extensions along $F$. Thus the composite $S$ has a left adjoint.
\begin{comment}
We apply the General Adjoint Functor Theorem (Theorem~2 in~V.6 of~\cite{maclane71}). Let $L \from \Set^{\op} \to \cat{L}$ be an object of $\proth(\Set^{\op})$. Since $\proth(\Set^{\op})$ has all \emph{large} limits by Proposition~\ref{prop:lim-colim-proth}, it is sufficient to find a large solution set of $L$. Let $\lambda$ be the supremum of the cardinalities of all the hom-sets of $\cat{L}$; since $\cat{L}$ is a large (rather than huge) category this supremum exists, although it may be large.

Now, $\SET^K$ inherits image factorisations from $\SET$, so any morphism out of $S(L)$ can be factorised through a morphism out of $S(L)$ that is surjective on each cardinality. But if $f \from S(L) \to \Omega$ is such that each $f_n \from S(L)_n \to \Omega_n$ is surjective, then the cardinality of $\Omega_n$ is at most $\lambda$. But there is only a set of isomorphism classes of operator domains $\Omega$ for which $|\Omega_n | \leq \lambda$ for each $n \in K$, and taking a representative of each such isomorphism class gives a solution set for $L$.
\end{comment}
\end{proof}

\begin{defn}
\label{defn:proth-to-opdom-adj}
Write $T \from \SET^K \to \proth(\Set^{\op})$ for the left adjoint of $S$.
\end{defn}

\begin{prop}
The diagram
\[
\xymatrix{
\catover{\Set} & \proth(\Set^{\op})^{\op}\ar[l]_{\sem} \\
& (\SET^K)^{\op}\ar[u]_T\ar[ul]^{\sem_0}
}
\]
commutes up to isomorphism.
\end{prop}
\begin{proof}
This follows from Lemmas~\ref{lem:str-sem-adj-opdom}, \ref{lem:opdom-str-commute}, Definition~\ref{defn:proth-to-opdom-adj} and the uniqueness of adjoints.
\end{proof}

\begin{thm}
Every equationally presentable category over $\Set$ is of the form $\sem_{ct}(L)$ for some complete topological proto-theory $L$.
\end{thm}
\begin{proof}
Let $(\Omega, E)$ be an equational presentation. Let $\Omega'$ be the operator domain defined by setting $\Omega'_n$ to be the set of $n$-ary equations in $E$ for each $n \in K$. Consider the operator domain $\str_0 (\sem_0(\Omega))$; the elements of $\str_0 (\sem_0(\Omega))_n$ are the natural transformations
\[
U_{\Omega}^n \to U_{\Omega}
\]
for $n \in K$. Thus, $E$ defines two canonical morphisms $f_1, f_2 \from \Omega' \to \str_0 \of \sem_0 (\Omega)$; given $n \in K$ and $e = (\gamma_1, \gamma_2) \in \Omega'_n$, the map $(f_i)_n$ sends $e$ to $\gamma_i$ for $i= 1,2$.

Since $\str_0 \radj \sem_0$, these correspond to a pair of morphisms $F_1, F_2 \from \sem_0(\Omega) \rightrightarrows \sem_0(\Omega')$ in $\catover{\Set}$. An object of $\cat{M}_{\Omega'}$ is a set $X$ equipped with an $n$-ary operation for each $n$-ary equation of $E$, satisfying no equations. Given $(X, \delta) \in \cat{M}_{\Omega}$, the functor $F_i$ sends $(X, \delta)$ to the set $X$ equipped with, for each $n$-ary equation $e = (\gamma_1, \gamma_2) \in E$, the $n$-ary operation
\[
(\gamma_i)_{(X, \delta)} \from X^n \to X.
\]
Thus $(X, \delta)$ is a model of $(\Omega, E)$ if and only if $F_1 (X, \delta) = F_2 (X, \delta)$, so the category of $(\Omega, E)$ models is the equaliser of $F_1$ and $F_2$ in $\catover{\Set}$.

Since $\sem_{ct} \from \proth_{ct}(\Set^{\op})^{\op} \to \catover{\Set}$ is full and faithful, the full subcategory of $\catover{\Set}$ consisting of those functors isomorphic to one of the from $\sem_{ct}(L)$ for some $L \in \proth_{ct}(\Set^{\op})$ is reflective, and in particular is closed under equalisers. Thus if we can show that a functor of the form $\sem_0(\Omega)$ for an operator domain $\Omega$ lies in this subcategory, the result will follow. But each triangle in
\[
\xymatrix{
& \proth_{ct}(\Set^{\op})^{\op}\ar[dl]_{\sem_{ct}} \\
\catover{\Set}& \proth_{t}(\Set^{\op})^{\op}\ar[u]_{\cplt^{\op}}\ar[l]_-{\sem_t} \\
& \proth(\Set^{\op})^{\op}\ar[ul]_{\sem}\ar[u]_{\disc^{\op}} \\
& (\SET^K)^{\op}\ar[u]_{T^{\op}}\ar[uul]^{\sem_0}
}
\]
commutes up to isomorphism, and so $\sem_0$ takes values in this subcategory as required.
\end{proof}

\chapter{Open questions}
\label{chap:open}
In this thesis we have developed the notions of proto-theory and aritation, leading to a general structure--semantics adjunction that generalises the semantics of many existing category-theoretic notions of algebraic theory. We have also seen that complete topological proto-theories provide a generalisation of monads for which the semantics functor has a left adjoint, while maintaining the desirable property that the semantics functor is full and faithful. However, there are many questions that remain to be answered both about proto-theories in general, and about complete topological proto-theories in particular. In this chapter we discuss some of these questions.

In Section~\ref{sec:open-morphisms} we consider the question of what the most appropriate notion of morphism is between proto-theories in different settings and between aritations. Then in Section~\ref{sec:open-more-general} we discuss whether even more general notions of proto-theory and aritation, that appear to make sense formally, have any practical relevance. Finally in Section~\ref{sec:open-profinite} we ask whether we can further extend the analogy we have developed between complete topological proto-theories and profinite groups.

\section{Relating different aritations and settings}
\label{sec:open-morphisms}

In category theory, whenever one encounters a new type of mathematical object, it is natural to ask what the appropriate morphisms between those objects are. The two main new concepts introduced in this thesis are proto-theories and aritations, so we should ask what the corresponding morphisms are. Doing so may also help us further understand the relationships between the different notions of algebraic theory that are generalised by proto-theories.

If one picks any two of the notions of algebraic theory from Chapter~\ref{chap:notions} at random there are likely to be canonical functors in either direction between the corresponding categories of theories, and these functors will be compatible with the semantics to various extents. In addition there are various notions of morphisms of monads, not just between monads on the same category, but between monads on different categories as described in Street~\cite{street72}, which again are compatible with the semantics to different extents. One may hope that these different ways of relating algebraic theories of different types can be understood in terms of morphisms of proto-theories and aritations.

Given a setting $\cat{X}$, and an object $\cat{A}$ of that setting we have considered the category $\proth(\cat{A})$ of proto-theories with arities $\cat{A}$ from the beginning. However we may still ask whether there is a sensible notion of morphism between proto-theories with different objects of arities in $\cat{X}$, or even between proto-theories in different settings.

Here is one possible approach. Let $\cat{X}$ and $\cat{X}'$ be settings, let $\cat{A} \in \cat{X}$ and $\cat{A}' \in  \cat{X}'$, and let $L \from \cat{A} \to \cat{L}$ and $L' \from \cat{A}' \to \cat{L}'$ be proto-theories in $\cat{X}$ and $\cat{X}'$ respectively. Then one could define a morphism $L \to L'$ to consist of a morphism of settings $F \from \cat{X} \to \cat{X}'$ (meaning a 2-functor that preserves cotensors and is compatible with the factorisation systems on $\cat{X}$ and $\cat{X}'$), together with 1-cells $G \from F\cat{A} \to \cat{A}'$ and $H \from F\cat{L} \to \cat{L}'$ such that
\[
\xymatrix{
F\cat{A}\ar[r]^{FL}\ar[d]_G & F\cat{L}\ar[d]^{H} \\
\cat{A}'\ar[r]_{L'} & \cat{L}'
}
\]
commutes up to equality, up to a specified isomorphism, or up to a specified 2-cell in either direction in $\cat{X}'$. Alternatively one could take $G$ and $H$ to be in the opposite directions and again ask for the resulting square to commute in one of the senses above. This yields 8 possible definitions, and no doubt there are yet more variations one could think of.

\begin{question}
What is the most appropriate notion of morphism between general proto-theories in different settings?
\end{question}

Before one can talk about the semantics of a proto-theory, one needs an interpretation of the arities of that proto-theory. Thus, if we want to relate the semantics of two proto-theories, we will probably need not only a morphism of proto-theories, but also a morphism between the aritations we are using to define their semantics.

Recall from Section~\ref{sec:chu} that an aritation $\lpair-,-\rpair \from \cat{A} \times \cat{B} \to \cat{C}$ in the setting $\CAT$ can be viewed as a Chu space in $\CAT$ and so morphisms of Chu spaces provide a notion of morphisms of aritations. An important class of aritations in $\CAT$ are given by the canonical aritation associated with an arbitrary locally small category $\cat{B}$, given by the hom-functor
\[
\cat{B}(-,-) \from \cat{B}^{\op}\times \cat{B} \to \Set.
\]
If $\cat{B}'$ is another locally small category then a Chu space morphism between the corresponding Chu spaces consists of functors $F \from \cat{B} \to \cat{B}'$ and $G \from \cat{B}' \to \cat{B}$ such that
\[
\xymatrix{
\cat{B}^{\op} \times \cat{B}'\ar[r]^{\id \times G}\ar[d]_{F \times \id} & \cat{B}^{\op} \times \cat{B}\ar[d]^{\cat{B}(-,-)} \\
\cat{B}'^{\op} \times \cat{B}' \ar[r]_{\cat{B}'(-,-)} & \Set
}
\]
commutes. However, asking for this square to commute strictly is too strong; it is more appropriate to consider $F$ and $G$ as above such that this square commutes up to specified isomorphism. A morphism of aritations in this sense consists of $F \from \cat{B} \to \cat{B}'$, $G \from \cat{B}' \to \cat{B}$ and specified isomorphisms
\[
\cat{B}'(Fb, b') \iso \cat{B}(b, Gb')
\]
natural in $b \in \cat{B}$, $b' \in \cat{B}'$. In other words, morphisms of aritations in this sense are precisely adjunctions.

Chu space morphisms provide a good candidate for a notion of morphism of aritations. However, the interpretation of aritations as Chu spaces is only available in the setting $\CAT$, and Chu space morphisms only make sense between aritations that take values in the same category $\cat{C}$. Therefore if we wish to use morphisms of aritations to compare proto-theories in different settings, we will need a more general notion.

Recall that an aritation in a setting $\cat{X}$ consists of a functor of the form
\[
\currylo \from \cat{B} \to \cat{X}(\cat{A} ,\cat{C})
\]
where $\cat{B}$ is a category and $\cat{A}, \cat{C} \in \cat{X}$. Thus we might define a morphism from this aritation to another, say
\[
\currylo' \from \cat{B}' \to \cat{X}(\cat{A}' ,\cat{C}')
\]
to consist of a functor $F \from \cat{B} \to \cat{B}'$ and 1-cells $G \from \cat{A}' \to \cat{A}$ and $K \from \cat{C} \to \cat{C'}$ in $\cat{X}$ such that
\[
\xymatrix
@R=15pt{
\cat{B}\ar[r]^-{\currylo}\ar[dd]_{F} & \cat{X}(\cat{A},\cat{C})\ar[d]^{K_*} \\
& \cat{X}(\cat{A},\cat{C}')\ar[d]^{G^*} \\
\cat{B}' \ar[r]_-{\currylo'}& \cat{X}(\cat{A}',\cat{C}')
}
\]
commutes.

On the other hand, the aritation $\currylo$ corresponds to a 1-cell
\[
\cat{A} \toby{\curryhi} [\cat{B},\cat{C}]
\]
in $\cat{X}$ and there is a similar $(\curryhi)'$ corresponding to $\currylo'$. Viewing the aritations this way, we might be tempted to define a morphism of aritations to consist of a functor $F \from \cat{B}' \to \cat{B}$ and 1-cells $G \from \cat{A} \to \cat{A}'$ and $K \from \cat{C} \to \cat{C}'$ such that
\[
\xymatrix
@R=15pt{
\cat{A}\ar[r]^-{\curryhi}\ar[dd]_{G} & [\cat{B},\cat{C}]\ar[d]^{K_*} \\
& [\cat{B},\cat{C}']\ar[d]^{F^*} \\
\cat{A}' \ar[r]_-{(\curryhi)'}& [\cat{B}',\cat{C}']
}
\]
commutes.

If $\cat{X}$ admits tensors, meaning that for each $\cat{B} \in \CAT$ the 2-functor $[\cat{B}, -] \from \cat{X} \to \cat{X}$ admits a left adjoint $\cat{B} \tensor -$, then the aritation $\currylo$ corresponds to a 1-cell
\[
H \from \cat{B} \tensor \cat{A} \to \cat{C}
\]
and similarly we have $H'$ for $\currylo'$. Then a morphism of aritations could be defined to consist of a functor $F \from \cat{B} \to \cat{B}'$ and 1-cells $G \from \cat{A} \to \cat{A}'$ and $K \from \cat{C} \to \cat{C}'$ such that
\[
\xymatrix{
\cat{B} \tensor \cat{A}\ar[r]^-H\ar[d]_{F \tensor G} & \cat{C}\ar[d]^K \\
\cat{B}' \tensor \cat{A}'\ar[r]_-{H'} & \cat{C}
}
\]
commutes. Thus there are at least three different candidates for how to define a morphism of aritations. If in each case we allow the relevant square to commute up to equality, isomorphism or a 2-cell in either direction, there are 12 candidates.

Each of these candidates may be compatible (or fail to be compatible) with the structure--semantics adjunction in at least four different ways. A functor $F\from \cat{B} \to \cat{B}'$ induces an adjunction
\[
\xymatrix{
{\catover{\cat{B} }}\ar@<5pt>[r]_-{\perp}^-{F_!}\ & {\catover{\cat{B}'}}\ar@<5pt>[l]^-{F^*}
}
\]
where the left adjoint is given by composition with $F$, and the right adjoint by pullback along $F$. A 1-cell $G \from \cat{A} \to \cat{A}'$ induces an adjunction
\[
\xymatrix{
{\proth(\cat{A})}\ar@<5pt>[r]_-{\perp}^-{G_*}\ & {\proth(\cat{A}')}\ar@<5pt>[l]^-{G^!}
}
\]
with the left adjoint given by pushout along $G$, and the right adjoint given by first pre-composing with $G$ and then taking the factorisation.

Thus for each of the 12 notions of morphism of aritation, we can ask whether each of the four squares in
\[
\xymatrix{
{\catover{\cat{B} }}\ar@<5pt>[r]_-{\perp}^-{\str}\ar@<5pt>[d] & {\proth(\cat{A})^{\op}}\ar@<5pt>[l]^-{\sem}\ar@<5pt>[d] \\
{\catover{\cat{B}' }}\ar@<5pt>[r]_-{\perp}^-{\str}\ar@<5pt>[u] & {\proth(\cat{A}')^{\op}}\ar@<5pt>[l]^-{\sem}\ar@<5pt>[u]
}
\]
commutes, where the vertical arrows are induced by the functor and 1-cells comprising the morphism of aritations as described above.

\begin{question}
What is the most appropriate notion of morphism between aritations, and to what extent are such morphisms compatible with structure--semantics adjunctions?
\end{question}

\section{Structure--semantics adjunctions in more general contexts}
\label{sec:open-more-general}
Recall that a setting consists of a 2-category with cotensors and a factorisation system, and an aritation in a setting $\cat{X}$ consists of $\cat{B} \in \CAT$, together with $\cat{A},\cat{C} \in \cat{X}$ and $\currylo \from \cat{B} \to \cat{X}(\cat{A},\cat{C})$. From this data, we automatically obtain an adjunction
 \[
 \xymatrix{
 {\catover{\cat{B} }}\ar@<5pt>[r]_-{\perp}^-{\str} & {\proth(\cat{A})^{\op}}\ar@<5pt>[l]^-{\sem}.
 }
 \]
The process by which this adjunction is constructed, as described in Sections~\ref{sec:sem-general} and~\ref{sec:str-general}, does not make use of any special properties of $\CAT$ other than it being a symmetric monoidal closed category. Thus we can repeat this process to obtain notions of proto-theory, aritation and structure--semantics adjunctions for any symmetric monoidal category. Let us make this more precise.
 
 \begin{defn}
 Let $\cat{V}$ be a symmetric monoidal closed category. Then a \demph{$\cat{V}$-setting} consists of a $\cat{V}$-category $\cat{X}$ that is cotensored over $\cat{V}$ and whose underlying ordinary category is equipped with a factorisation system $(\ofsfont{E},\ofsfont{N})$.
 \end{defn}
 
 \begin{defn}
 Let $\cat{V}$ be a symmetric monoidal closed category and $(\cat{X},\cat{E},\cat{N})$ a $\cat{V}$-setting. Write $\cat{X}_0$ for the underlying ordinary category of $\cat{X}$ and let $A$ be an object of $\cat{X}$ (and hence also of $\cat{X}_0$). Then we write $\proth(A)$ for the full subcategory of $A/\cat{X}_0$ consisting of those morphisms out of $A$ that are in $\ofsfont{E}$.
 \end{defn}
 
 \begin{defn}
 Let $\cat{V}$ be a symmetric monoidal closed category and $\cat{X}$ a $\cat{V}$-setting. Then a \demph{$\cat{V}$-aritation} in $\cat{X}$ consists of objects $B \in \cat{V}$ and $A, C \in \cat{X}$ together with a morphism
 \[
 \currylo \from B \to \cat{X}(A,C)
 \]
 in $\cat{V}$, or equivalently
 \[
 \curryhi \from A \to [B,C]
 \]
 in $\cat{X}_0$.
 \end{defn}
 
 Given a symmetric monoidal closed category $\cat{V}$, a $\cat{V}$-setting $\cat{X}$ and a $\cat{V}$-aritation $\currylo \from B \to \cat{X}(A,C)$ in $\cat{X}$, we may define functors
 \begin{itemize}
 \item $\iota \from \proth(A) \to A/\cat{X}_0$ 
 \item $G \from (A/\cat{X}_0)^{\op} \to \cat{V}/\cat{X}(A,C)$ 
 \item $\currylo^* \from \cat{V} / \cat{X}(A,C) \to \cat{V}/ B$
 \end{itemize}
 exactly as in Definition~\ref{defn:sem-factors}, and these have adjoints defined exactly as in Definition~\ref{defn:str-factors}. Composing these three adjunctions gives
 \[
 \xymatrix{
 {\cat{V}/B}\ar@<5pt>[r]_-{\perp}^-{\str} & {\proth(A)^{\op}}\ar@<5pt>[l]^-{\sem}.
 }
 \]
 
Thus the formalism of proto-theories, aritations and structure--semantics adjunctions makes sense in a more general context than the one we have focused on in this thesis. However, the intuition of proto-theories as consisting of operations of different arities no longer makes sense in this general context, so it is not clear how it is best to interpret these more general structure--semantics adjunctions.

We saw in Section~\ref{sec:chu} that aritations in $\CAT$ are a special case of Chu spaces. But in fact a Chu space in any closed symmetric monoidal category $\cat{V}$ is a $\cat{V}$-aritation in the $\cat{V}$-setting $\cat{V}$ as defined above. Thus every Chu space gives rise to a structure--semantics adjunction.

Chu spaces in $\Set$ encompass an enormous range of mathematical structures; all algebraic structures, relational structures and topological spaces can be regarded as Chu spaces. A study of the structure--semantics adjunctions arising from all the various types of Chu space is beyond the scope of this thesis, but let us briefly consider the adjunction obtained by regarding a topological space $X$ as a Chu space.

We regard $X$ as a Chu space by taking the set of points to be the underlying set of $X$ and the set of states to be the set $\mathcal{O}(X)$ of open sets of $X$. We take the set of truth values to be $2= \{0,1\}$ and the pairing
\[
X \times \mathcal{O}(X) \to 2
\]
to be the map sending $(x, U)$ to $1$ if $x \in U$ and $0$ otherwise.

We think of this as an aritation in $\Set$, equipped with the usual surjection/injection factorisation system. Thus a $\Set$-proto-theory with arities $\mathcal{O}(X)$ is a surjection out of $\mathcal{O}(X)$, and we have an adjunction
 \[
 \xymatrix{
 {\Set/X}\ar@<5pt>[r]_-{\perp}^-{\str} & {\proth(\mathcal{O}(X))^{\op}}\ar@<5pt>[l]^-{\sem}.
 }
 \]
 Given a set $Y$ and a function $g \from Y \to X$ we obtain $\str(g) \from \mathcal{O}(X) \to \thr(g)$ as the image factorisation
 \[
 \xymatrix{
 \mathcal{O}(X)\ar[r]\ar[d]_{\str(g)} & 2^X\ar[d]^{g^*} \\
 \thr(g)\ar[r] & 2^Y.
 }
 \]
We can thus interpret $\thr(g)$ as consisting of those subsets $V$ of $Y$ of the form $g^{-1}(U)$ for some open $U \subseteq X$. But these are precisely the open sets for the topology induced by $g$; that is, the smallest topology on $Y$ such that $g$ is continuous.

Thus we recover an important construction in topology from a structure--semantics adjunction. This suggests that this very general notion of structure--semantics adjunction is potentially significant in contexts besides $\CAT$, and is worthy of further study.
\begin{question}
What is the general significance of aritations, proto-theories and structure--semantics adjunctions in closed symmetric monoidal categories other than $\CAT$?
\end{question}

\section{Complete topological proto-theories and profinite groups}
\label{sec:open-profinite}
In Chapters~\ref{chap:canonical2},~\ref{chap:topology} and~\ref{chap:complete}, we explored an analogy between proto-theories and groups, in which complete topological proto-theories corresponded to profinite groups, culminating in Table~\ref{tab:analogy-comp}. In particular we have given several characterisations of the category of complete topological proto-theories on a locally small category $\cat{B}$ with small limits and enough subobjects, mirroring similar characterisations of the category of profinite groups, as displayed in Table~\ref{tab:prof-chars-left-over}.
\begin{table}[h]
\centering
\begin{tabular}{p{7cm}p{7cm}}
\toprule
$\profGp$& $\proth_{ct}(\cat{B}^{\op})$ \\
\midrule
The smallest subcategory of $\TopGp$ containing $\FinGp$ and closed under small limits & The smallest subcategory of $\proth_t(\cat{B}^{\op})$ containing $\monad(\cat{B})$ and closed under large limits \\
The smallest replete reflective subcategory of $\TopGp$ containing $\FinGp$ & The smallest replete reflective subcategory of $\proth_t(\cat{B}^{\op})$ containing $\monad(\cat{B})$ \\
The category of algebras for the codensity monad of the inclusion $\FinGp \incl \TopGp$ & The category of algebras for the codensity monad of the inclusion $\monad(\cat{B}) \incl \proth_t(\cat{B}^{\op})$ \\
\bottomrule
\end{tabular}
\caption{Characterisations of $\proth_{ct}(\cat{B}^{\op})$ and $\profGp$.}
\label{tab:prof-chars-left-over}
\end{table}

Note, however that several of the characterisations of $\profGp$ from Remark~\ref{rem:prof-gp-chars} do not yet have analogues for proto-theories, namely that $\profGp$ is:
\begin{enumerate}
\item the full subcategory of $\TopGp$ consisting of the compact, Hausdorff, totally disconnected groups;
\item the category of algebras for the codensity monad of $\FinGp \incl \Gp$;
\item the category of algebras for the codensity monad of the forgetful functor $\FinGp \to \Top$; and
\item the category of algebras for the codensity monad of the forgetful functor $\FinGp \to \Set$.
\end{enumerate}
The most obvious analogue of the first of these is the assertion that the complete topological proto-theories are those topological proto-theories $L \from \cat{B}^{\op} \to \cat{L}$ for which each hom-space of $\cat{L}$ is compact, Hausdorff and totally disconnected. This is evidently false; any monad whose Kleisli category is not locally finite is a counterexample, since infinite discrete spaces are not compact. However, there may still be some similar characterisation of the complete topological proto-theories in terms of topological properties of their hom-spaces. We can also ask whether the corresponding analogues of the other characterisations above hold in the proto-theoretic case.

\begin{question}
Can the category $\proth_{ct}(\cat{B}^{\op})$ of complete topological proto-theories on a locally small category $\cat{B}$ with small limits and enough subobjects be characterised as any of the following:
\begin{enumerate}
\item the full subcategory of $\proth_t(\cat{B}^{\op})$ consisting of those topological proto-theories whose hom-spaces satisfy some topological property;
\item the category of algebras for the codensity monad of $\kle \from \monad(\cat{B}) \incl \proth(\cat{B}^{\op})$;
\item the category of algebras for the codensity monad for the functor $\monad(\cat{B}) \to [\cat{B}^{\op} \times \cat{B}, \TOP]$ that sends a monad $\mnd{T} = (T, \eta, \mu)$ to the discrete-space-valued functor $\cat{B}(-, T-) \from \cat{B}^{\op} \times \cat{B} \to \TOP$; or
\item the category of algebras for the codensity monad for the functor $\monad(\cat{B}) \to [\cat{B}^{\op} \times \cat{B}, \SET]$ that sends a monad $\mnd{T} = (T, \eta, \mu)$ to $\cat{B}(-, T-) \from \cat{B}^{\op} \times \cat{B} \to \SET$?
\end{enumerate}
\end{question}

Recall that, for a suitable category $\cat{B}$, the category $\proth_{ct}(\cat{B}^{\op})$ of complete topological proto-theories is defined to be the reflective subcategory of $\proth_t(\cat{B}^{\op})$ corresponding to the idempotent monad generated by the structure--semantics adjunction. In particular, we can only talk about complete topological proto-theories on $\cat{B}$ if this adjunction is idempotent, which we saw (Theorem~\ref{thm:str-sem-top-idempotent}) is the case when $\cat{B}$ has finite products and enough subobjects. However, it is not clear how crucial this condition is; it may still be the case that the structure--semantics adjunction is idempotent for many categories that do not have enough subobjects.

As noted in Theorem~\ref{thm:str-sem-monad-codensity-top}, one can show that the structure--semantics monad on $\proth_t(\cat{B}^{\op})$ is the codensity monad of the inclusion $\monad(\cat{B}) \incl \proth_t(\cat{B})$. Thus, the idempotency of the structure--semantics adjunction is equivalent to idempotency of this codensity monad. In~\cite{deleanu83}, Deleanu gives a necessary and sufficient condition for a codensity monad to be idempotent.

Specifically, if a functor $U \from \cat{M} \to \cat{C}$ has a pointwise codensity monad $(T, \eta, \mu)$, then that monad is idempotent if and only if, for every $m \in \cat{M}$ and $c \in \cat{C}$ the map
\[
\eta_m^* \from \cat{C}(Tc, Um) \to \cat{C}(c, Um)
 \]
 is an injection.

This could potentially be used to give a different proof of idempotency of this monad, possibly under weaker assumptions on $\cat{B}$ than having enough subobjects.

\begin{question}
Are there other conditions on a category $\cat{B}$ that ensure that the structure--semantics monad on $\proth_t(\cat{B}^{\op})$ is idempotent?
\end{question}

\cleardoublepage
\phantomsection
\addcontentsline{toc}{chapter}{\bibname}
\bibliography{bib}

\end{document}